\documentclass[11pt,b5paper,notitlepage]{article}
\usepackage[b5paper, margin={0.5in,0.65in}]{geometry}

\usepackage{amsmath,amscd,amssymb,amsthm,mathrsfs,amsfonts,layout,indentfirst,graphicx,caption,mathabx, stmaryrd,appendix,calc,imakeidx,upgreek,appendix} 
\usepackage[dvipsnames]{xcolor}
\usepackage{palatino}  
\usepackage{slashed} 
\usepackage{mathrsfs} 
\usepackage{extarrows} 
\usepackage{enumitem} 

\usepackage{fancyhdr} 
\usepackage{verbatim}
\usepackage{halloweenmath}
\usepackage{simpler-wick}

\usepackage{tikz-cd}
\usepackage[nottoc]{tocbibind}   

\makeindex

\usepackage{lipsum}
\let\OLDthebibliography\thebibliography
\renewcommand\thebibliography[1]{
	\OLDthebibliography{#1}
	\setlength{\parskip}{0pt}
	\setlength{\itemsep}{2pt} 
}

\allowdisplaybreaks  
\usepackage{latexsym}
\usepackage{chngcntr}
\usepackage[colorlinks,linkcolor=blue,anchorcolor=blue, linktocpage,
]{hyperref}
\hypersetup{ urlcolor=cyan,
	citecolor=[rgb]{0,0.5,0}}

\counterwithin{figure}{section}

\theoremstyle{definition}
\newtheorem{df}{Definition}[section]
\newtheorem{eg}[df]{Example}

\newtheorem{rem}[df]{Remark}
\newtheorem{ass}[df]{Assumption}
\newtheorem{cv}[df]{Convention}

\newtheorem{sett}[df]{Setting}
\theoremstyle{plain}
\newtheorem{thm}[df]{Theorem}

\newtheorem{pp}[df]{Proposition}
\newtheorem{co}[df]{Corollary}
\newtheorem{lm}[df]{Lemma}
\newtheorem{que}[df]{Question}

\newcommand{\fk}{\mathfrak}
\newcommand{\mc}{\mathcal}
\newcommand{\wtd}{\widetilde}
\newcommand{\wht}{\widehat}
\newcommand{\wch}{\widecheck}
\newcommand{\ovl}{\overline}

\newcommand{\tr}{\mathrm{t}} 

\newcommand{\End}{\mathrm{End}} 
\newcommand{\idt}{\mathbf{1}}
\newcommand{\Hom}{\mathrm{Hom}}
\newcommand{\Conf}{\mathrm{Conf}}
\newcommand{\Res}{\mathrm{Res}}

\newcommand{\SV}{\mathscr{V}}
\newcommand{\Span}{\mathrm{Span}}

\newcommand{\scr}{\mathscr}

\newcommand{\im}{\mathbf{i}}

\newcommand{\sgm}{\varsigma}
\newcommand{\SX}{{S_{\fk X}}}
\newcommand{\DX}{D_{\fk X}}

\newcommand{\mbf}{\mathbf}

\newcommand{\blt}{\bullet}

\newcommand{\Vbb}{\mathbb V}
\newcommand{\Ubb}{\mathbb U}
\newcommand{\Xbb}{\mathbb X}
\newcommand{\Wbb}{\mathbb W}
\newcommand{\Mbb}{\mathbb M}
\newcommand{\Gbb}{\mathbb G}
\newcommand{\Cbb}{\mathbb C}
\newcommand{\Nbb}{\mathbb N}
\newcommand{\Zbb}{\mathbb Z}
\newcommand{\Pbb}{\mathbb P}

\newcommand{\Ebb}{\mathbb E}
\newcommand{\cbf}{\mathbf c}
\newcommand{\wt}{\mathrm{wt}}

\newcommand{\btl}{\blacktriangleleft}
\newcommand{\btr}{\blacktriangleright}

\newcommand{\pr}{\mathrm {pr}}

\newcommand{\ibf}{\mathbf 1}
\newcommand{\wbf}{\mathbf w}

\newcommand{\srm}{\mathrm{s}}
\newcommand{\nrm}{\mathrm{n}}

\newcommand{\<}{\left\langle}
\renewcommand{\>}{\right\rangle}
\newcommand{\MO}{\mathcal{O}}
\newcommand{\MU}{\mathcal{U}}
\newcommand{\ME}{\mathcal{E}}

\newcommand{\MC}{\mathcal{C}}
\newcommand{\MB}{\mathcal{B}}
\newcommand{\fm}{\mathfrak{m}}
\newcommand{\MF}{\mathcal{F}}
\newcommand{\fx}{\mathfrak{X}}
\newcommand{\pt}{\mathrm{pt}}
\usepackage{tipa} 

\newcommand{\ST}{\mathscr{T}}
\newcommand{\SJ}{\mathscr{J}}
\newcommand{\SW}{\mathscr{W}}
\newcommand{\MV}{\mathcal{V}}
\usepackage{tipx}

\newcommand{\MD}{\mathcal{D}}
\newcommand{\MS}{\mathcal{S}}
\newcommand{\FA}{\mathfrak{A}}
\newcommand{\FC}{\mathfrak{C}}
\newcommand{\FP}{\mathfrak{P}}
\newcommand{\cl}{\mathrm{cl}}
\newcommand{\mk}{\mathfrak m}
\newcommand{\pre}{\mathrm{pre}}
\newcommand{\bk}[1]{\langle {#1}\rangle}
\newcommand{\bigbk}[1]{\big\langle {#1}\big\rangle}
\newcommand{\Bigbk}[1]{\Big\langle {#1}\Big\rangle}

\newcommand{\bbs}{\boxbackslash}
\newcommand{\fq}{{\mathfrak Q}}
\newcommand{\Mod}{\mathrm{Mod}}

\usepackage{calligra}
\DeclareMathOperator{\shom}{\mathscr{H}\text{\kern -3pt {\calligra\large om}}\,}
\DeclareMathOperator{\sext}{\mathscr{E}\text{\kern -3pt {\calligra\large xt}}\,}
\DeclareMathOperator{\Rel}{\mathscr{R}\text{\kern -3pt {\calligra\large el}~}\,}
\DeclareMathOperator{\sann}{\mathscr{A}\text{\kern -3pt {\calligra\large nn}}\,}
\DeclareMathOperator{\send}{\mathscr{E}\text{\kern -3pt {\calligra\large nd}}\,}
\DeclareMathOperator{\stor}{\mathscr{T}\text{\kern -3pt {\calligra\large or}}\,}

\usepackage{aurical}
\DeclareMathOperator{\VVir}{\text{\Fontlukas V}\text{\kern -0pt {\Fontlukas\large ir}}\,}

\numberwithin{equation}{section}

\title{Analytic Conformal Blocks of $C_2$-cofinite Vertex Operator Algebras I: Propagation and Dual Fusion Products}
\author{{\sc Bin Gui, Hao Zhang}
}
\date{}
\begin{document}\sloppy 
	\pagenumbering{arabic}
	\setcounter{section}{-1}

	\maketitle

\newcommand\blfootnote[1]{%
	\begingroup
	\renewcommand\thefootnote{}\footnote{#1}%
	\addtocounter{footnote}{-1}%
	\endgroup
}


\begin{abstract}
This is the first paper of a three-part series in which we develop a theory of conformal blocks for $C_2$-cofinite vertex operator algebras (VOAs) that are not necessarily rational. The ultimate goal of this series is to prove a sewing-factorization theorem (and in particular, a factorization formula) for conformal blocks over holomorphic families of compact Riemann surfaces, associated to grading-restricted generalized modules of $C_2$-cofinite VOAs.

In this paper, we prove that if $\Vbb=\bigoplus_{n\in\Nbb}\Vbb(n)$ is a $C_2$-cofinite VOA, if $\fx$ is a compact Riemann surface with $N$ incoming and $M$ outgoing marked points, each  equipped with a local coordinate, and if $\Wbb$ is a grading-restricted generalized $\Vbb^{\otimes N}$-module, then the ``dual fusion product" $(\bbs_\fx(\Wbb),\gimel)$ satisfying a natural universal property exists. Here, $\bbs_\fx(\Wbb)$ is a grading-restricted $\Vbb^{\otimes M}$-module, and the linear functional $\gimel:\Wbb\otimes_\Cbb\bbs_\fx(\Wbb)\rightarrow\Cbb$ is a conformal block associated to $\fx$. Indeed, we prove a more general version of this result without assuming $\Vbb$ to be $C_2$-cofinite. Our main method is the propagation of partial conformal blocks, a generalization of the standard propagation of conformal blocks.

Assuming that $\Vbb$ is $C_2$-cofinite, our (dual) fusion product recovers known constructions in genus-0 cases: When $N=2,M=1$, our (dual) fusion product agrees with that of Huang-Lepowsky-Zhang. When $N=1,M=2$, our dual fusion product coincides with Li's regular representations.
\end{abstract}

\tableofcontents





	
	

	

\section{Introduction}

\nocite{HLZ1,HLZ2,HLZ3,HLZ4,HLZ5,HLZ6,HLZ7,HLZ8}
This is the first paper in a three-part series in which we develop a theory of conformal blocks for $C_2$-cofinite VOAs that are not necessarily rational. In this introduction, we first give a general overview of the series and present the main theorem to be proved in the third part \cite{GZ3}. We then explain what will be accomplished in this paper.

$C_2$-cofiniteness and rationality are two crucial conditions on a vertex operator algebra $\Vbb=\bigoplus_{n\in\Nbb}\Vbb(n)$ that rigorize physicists' notion of ``rational chiral conformal field theory". The $C_2$-cofinite condition, introduced in the seminal work \cite{Zhu-modular-invariance}, guarantees that $\Vbb$ has finitely many irreducibles (up to equivalence) and that the spaces of conformal blocks are finite-dimensional. It is also crucial to the proof (using methods from differential equations) that the sewing of conformal blocks is convergent. See \cite{Zhu-modular-invariance,AN03-finite-dimensional,Miy-modular-invariance,Hua-differential-genus-0,Hua-differential-genus-1,NT-P1_conformal_blocks,Fio-genus-1,DGT2,Gui-sewingconvergence} for instance. If a $C_2$-cofinite VOA $\Vbb$ is also rational, then the genus-1 conformal blocks satisfy a modular invariance property \cite{Zhu-modular-invariance,DLM-modular-invariance,Hua-differential-genus-1}. If $\Vbb$ is strongly rational (i.e., $C_2$-cofinite, rational, self-dual, and $\dim\Vbb(0)=1$), then the category $\Mod(\Vbb)$ of grading-restricted (generalized) $\Vbb$-modules forms a braided tensor category \cite{HL-tensor-1,HL-tensor-2,HL-tensor-3,Hua-tensor-4,Hua-differential-genus-0,NT-P1_conformal_blocks} which is rigid and modular \cite{Hua-rigidity-modularity}.

Crucial to the proofs and the applications of the above result about $\Mod(\Vbb)$ is the proof of the \emph{sewing-factorization theorem} of conformal blocks, mainly in genera $0$ and $1$. For instance, the genus-$0$ sewing-factorization theorem implies that the tensor functor of $\Mod(\Vbb)$ satisfies associativity and pentagon axioms, and (together with the braiding defined by the connections on bundles of conformal blocks) satisfies hexagon axioms \cite{Hua-tensor-4,NT-P1_conformal_blocks,Hua-rigidity-modularity}. Modular invariance \cite{Zhu-modular-invariance,DLM-modular-invariance,Hua-differential-genus-1} is essentially the sewing-factorization theorem in genus-$1$, and is the key to proving the rigidity and modularity of $\Mod(\Vbb)$. 


\subsection{Sewing-factorization in the rational world}\label{lb84}

The goal of this series of papers is to state and prove a sewing-factorization theorem for conformal blocks associated to $C_2$-cofinite VOAs in the complex-analytic setting. (See \cite{DGK-presentations,DGK2} for an approach from the algebro-geometric point of view.) We first give a brief account of this theorem in the special case where $\Vbb$ is both $C_2$-cofinite and rational. Let
\begin{align}
\fx=(C|x_1,\dots,x_N;\eta_1,\dots,\eta_N\Vert y',y'';\xi,\varpi)
\end{align}
be a compact Riemann surface $C$ with distinct marked points $x_1,\dots,x_N,y',y''\in C$. Each $\eta_i$ is a \textbf{local coordinate} of $C$ at $x_i$, i.e. a univalent (i.e. holomorphic and injective) map from a neighborhood $U_i$ of $x_i$ to $\Cbb$ satisfying $\eta_i(x_i)=0$. Similarly, $\xi$ and $\varpi$ are local coordinates at $y'$ and $y''$ respectively. We do not assume that $C$ is connected. But we do assume that each connected component of $C$ contains at least one of $x_1,\dots,x_N$. 

The pair of points $y',y''$ are for sewing. Let $W'$ and $W''$ be neighborhoods of $y',y''$ such that $\xi(W')$ equals $\mc D_r=\{z\in\Cbb:|z|<r\}$ and $\varpi(W'')$ equals $\mc D_\rho$ for some $r,\rho>0$, and that $x_1,\dots,x_N,W',W''$ are mutually disjoint. Suppose $q\in \mc D_{r\rho}^\times=\{z\in\Cbb:0<|z|<r\rho\}$. Then we can cut off small closed discs centered at $y'$ and $y''$, and glue the remaining parts of $\mc D_r$ and $\mc D_\rho$ via the rule that $p'\in\mc D_r$ and $p''\in\mc D_\rho$ are identified iff
\begin{align}
\xi(p')\varpi(p'')=q.
\end{align}
In this way, we get a new pointed surface
\begin{align}
\mc S_q\fx=(\mc C_q|x_1,\dots,x_N;\eta_1,\dots,\eta_N)
\end{align}
Depending on whether $y',y''$ belong to the same component of $C$ or not, we call the sewing either a \textbf{self-sewing} or a \textbf{disjoint sewing}. See Fig. \ref{fig2}. (As we will see, at least in the beginning, self-sewing is not an appropriate consideration for conformal blocks of irrational VOAs.)
\begin{figure}[h]
	\centering
\scalebox{0.9}{

\tikzset{every picture/.style={line width=0.75pt}} 

\begin{tikzpicture}[x=0.75pt,y=0.75pt,yscale=-1,xscale=1]

\draw   (107.52,25.9) .. controls (129.12,24.3) and (157.12,26.7) .. (175.52,46.7) .. controls (193.92,66.7) and (188.32,89) .. (179.52,105.1) .. controls (170.72,121.2) and (122.72,119.5) .. (122.72,106.7) .. controls (122.72,93.9) and (138.72,96.37) .. (139.52,85.9) .. controls (140.32,75.43) and (125.92,72.37) .. (115.52,72.37) .. controls (105.12,72.37) and (93.12,73.17) .. (94.72,85.17) .. controls (96.32,97.17) and (104.32,88.3) .. (105.12,106.7) .. controls (105.92,125.1) and (54.72,108.23) .. (48.32,98.7) .. controls (41.92,89.17) and (37.12,73.1) .. (47.52,55.5) .. controls (57.92,37.9) and (85.92,27.5) .. (107.52,25.9) -- cycle ;
\draw    (65.12,68.26) .. controls (75.82,58.51) and (84.28,61.11) .. (91.6,69.56) ;
\draw    (68.5,66.31) .. controls (74.7,71.51) and (81.46,70.86) .. (87.66,66.31) ;
\draw    (136.27,66.02) .. controls (148.94,56.02) and (158.94,58.69) .. (167.6,67.36) ;
\draw    (140.27,64.02) .. controls (147.6,69.36) and (155.6,68.69) .. (162.94,64.02) ;
\draw  [fill={rgb, 255:red, 0; green, 0; blue, 0 }  ,fill opacity=1 ] (107.4,46.78) .. controls (107.4,45.88) and (108.13,45.14) .. (109.03,45.14) .. controls (109.94,45.14) and (110.67,45.88) .. (110.67,46.78) .. controls (110.67,47.68) and (109.94,48.42) .. (109.03,48.42) .. controls (108.13,48.42) and (107.4,47.68) .. (107.4,46.78) -- cycle ;
\draw  [fill={rgb, 255:red, 0; green, 0; blue, 0 }  ,fill opacity=1 ] (140.2,50.78) .. controls (140.2,49.88) and (140.93,49.14) .. (141.83,49.14) .. controls (142.74,49.14) and (143.47,49.88) .. (143.47,50.78) .. controls (143.47,51.68) and (142.74,52.42) .. (141.83,52.42) .. controls (140.93,52.42) and (140.2,51.68) .. (140.2,50.78) -- cycle ;
\draw  [fill={rgb, 255:red, 0; green, 0; blue, 0 }  ,fill opacity=1 ] (98.6,106.78) .. controls (98.6,105.88) and (99.33,105.14) .. (100.23,105.14) .. controls (101.14,105.14) and (101.87,105.88) .. (101.87,106.78) .. controls (101.87,107.68) and (101.14,108.42) .. (100.23,108.42) .. controls (99.33,108.42) and (98.6,107.68) .. (98.6,106.78) -- cycle ;
\draw  [fill={rgb, 255:red, 0; green, 0; blue, 0 }  ,fill opacity=1 ] (132.2,105.98) .. controls (132.2,105.08) and (132.93,104.34) .. (133.83,104.34) .. controls (134.74,104.34) and (135.47,105.08) .. (135.47,105.98) .. controls (135.47,106.88) and (134.74,107.62) .. (133.83,107.62) .. controls (132.93,107.62) and (132.2,106.88) .. (132.2,105.98) -- cycle ;
\draw [color={rgb, 255:red, 74; green, 144; blue, 226 }  ,draw opacity=1 ] [dash pattern={on 2.25pt off 2.25pt}]  (100.31,114.09) .. controls (103.54,117) and (107.51,118.76) .. (111.6,119.53) .. controls (119.55,121.03) and (127.95,118.82) .. (132.18,114.18) ;
\draw [shift={(133.36,112.64)}, rotate = 134.19] [color={rgb, 255:red, 74; green, 144; blue, 226 }  ,draw opacity=1 ][line width=0.75]    (6.56,-1.97) .. controls (4.17,-0.84) and (1.99,-0.18) .. (0,0) .. controls (1.99,0.18) and (4.17,0.84) .. (6.56,1.97)   ;
\draw [shift={(98.86,112.64)}, rotate = 39.48] [color={rgb, 255:red, 74; green, 144; blue, 226 }  ,draw opacity=1 ][line width=0.75]    (6.56,-1.97) .. controls (4.17,-0.84) and (1.99,-0.18) .. (0,0) .. controls (1.99,0.18) and (4.17,0.84) .. (6.56,1.97)   ;
\draw   (303.18,39.08) .. controls (324.04,26.23) and (339.47,39.5) .. (357.3,47.2) .. controls (375.12,54.89) and (384.35,49.9) .. (393.94,53.28) .. controls (403.52,56.67) and (407.47,66.81) .. (406.9,77.63) .. controls (406.34,88.45) and (398.45,94.54) .. (380.97,93.19) .. controls (363.5,91.84) and (374.77,89.13) .. (355.04,92.51) .. controls (335.31,95.89) and (334.75,109.42) .. (308.25,101.98) .. controls (281.76,94.54) and (282.32,51.93) .. (303.18,39.08) -- cycle ;
\draw   (437.45,56.05) .. controls (449.95,44.05) and (465.95,55.55) .. (479.45,53.05) .. controls (482.55,52.47) and (485.61,50.98) .. (488.75,49.21) .. controls (499.32,43.24) and (510.7,34.09) .. (526.68,46) .. controls (547.41,61.45) and (545.68,93) .. (523.95,101.55) .. controls (502.23,110.09) and (489.45,95.05) .. (477.45,89.05) .. controls (465.45,83.05) and (448.95,92.55) .. (440.95,86.05) .. controls (432.95,79.55) and (424.95,68.05) .. (437.45,56.05) -- cycle ;
\draw    (303.27,61.56) .. controls (313.65,52.29) and (321.85,54.76) .. (328.95,62.79) ;
\draw    (306.55,59.7) .. controls (312.56,64.65) and (319.12,64.03) .. (325.13,59.7) ;
\draw    (322.77,83.56) .. controls (333.15,74.29) and (341.35,76.76) .. (348.45,84.79) ;
\draw    (326.05,81.7) .. controls (332.06,86.65) and (338.62,86.03) .. (344.63,81.7) ;
\draw    (495.77,73.06) .. controls (506.15,63.79) and (514.35,66.26) .. (521.45,74.29) ;
\draw    (499.05,71.2) .. controls (505.06,76.15) and (511.62,75.53) .. (517.63,71.2) ;
\draw  [fill={rgb, 255:red, 0; green, 0; blue, 0 }  ,fill opacity=1 ] (347.9,61.78) .. controls (347.9,60.88) and (348.63,60.14) .. (349.53,60.14) .. controls (350.44,60.14) and (351.17,60.88) .. (351.17,61.78) .. controls (351.17,62.68) and (350.44,63.42) .. (349.53,63.42) .. controls (348.63,63.42) and (347.9,62.68) .. (347.9,61.78) -- cycle ;
\draw  [fill={rgb, 255:red, 0; green, 0; blue, 0 }  ,fill opacity=1 ] (398.9,74.28) .. controls (398.9,73.38) and (399.63,72.64) .. (400.53,72.64) .. controls (401.44,72.64) and (402.17,73.38) .. (402.17,74.28) .. controls (402.17,75.18) and (401.44,75.92) .. (400.53,75.92) .. controls (399.63,75.92) and (398.9,75.18) .. (398.9,74.28) -- cycle ;
\draw  [fill={rgb, 255:red, 0; green, 0; blue, 0 }  ,fill opacity=1 ] (439.4,71.28) .. controls (439.4,70.38) and (440.13,69.64) .. (441.03,69.64) .. controls (441.94,69.64) and (442.67,70.38) .. (442.67,71.28) .. controls (442.67,72.18) and (441.94,72.92) .. (441.03,72.92) .. controls (440.13,72.92) and (439.4,72.18) .. (439.4,71.28) -- cycle ;
\draw  [fill={rgb, 255:red, 0; green, 0; blue, 0 }  ,fill opacity=1 ] (494.4,83.78) .. controls (494.4,82.88) and (495.13,82.14) .. (496.03,82.14) .. controls (496.94,82.14) and (497.67,82.88) .. (497.67,83.78) .. controls (497.67,84.68) and (496.94,85.42) .. (496.03,85.42) .. controls (495.13,85.42) and (494.4,84.68) .. (494.4,83.78) -- cycle ;
\draw [color={rgb, 255:red, 74; green, 144; blue, 226 }  ,draw opacity=1 ] [dash pattern={on 2.25pt off 2.25pt}]  (404.31,79.59) .. controls (407.54,82.5) and (411.51,84.26) .. (415.6,85.03) .. controls (419.95,85.85) and (429.71,83.98) .. (435.14,79.01) ;
\draw [shift={(436.45,77.64)}, rotate = 141.28] [color={rgb, 255:red, 74; green, 144; blue, 226 }  ,draw opacity=1 ][line width=0.75]    (6.56,-1.97) .. controls (4.17,-0.84) and (1.99,-0.18) .. (0,0) .. controls (1.99,0.18) and (4.17,0.84) .. (6.56,1.97)   ;
\draw [shift={(402.86,78.14)}, rotate = 42.16] [color={rgb, 255:red, 74; green, 144; blue, 226 }  ,draw opacity=1 ][line width=0.75]    (6.56,-1.97) .. controls (4.17,-0.84) and (1.99,-0.18) .. (0,0) .. controls (1.99,0.18) and (4.17,0.84) .. (6.56,1.97)   ;

\draw (89.47,35.82) node [anchor=north west][inner sep=0.75pt]  [font=\small]  {$x_{1}$};
\draw (122.47,40.42) node [anchor=north west][inner sep=0.75pt]  [font=\small]  {$x_{2}$};
\draw (85.57,91.82) node [anchor=north west][inner sep=0.75pt]  [font=\small]  {$y’$};
\draw (135.77,91.62) node [anchor=north west][inner sep=0.75pt]  [font=\small]  {$y''$};
\draw (332.47,45.32) node [anchor=north west][inner sep=0.75pt]  [font=\small]  {$x_{1}$};
\draw (386.07,62.32) node [anchor=north west][inner sep=0.75pt]  [font=\small]  {$y'$};
\draw (475.47,73.42) node [anchor=north west][inner sep=0.75pt]  [font=\small]  {$x_{2}$};
\draw (444.77,56.12) node [anchor=north west][inner sep=0.75pt]  [font=\small]  {$y''$};

\end{tikzpicture}

}
	\caption{~~Self-sewing and disjoint sewing}
	\label{fig2}
\end{figure}

Associate grading-restricted $\Vbb$-modules $\Wbb_1,\dots,\Wbb_N,\Mbb,\Mbb'$ to the marked points $x_1,\dots,x_N,y',y''$ respectively where $\Mbb'$ is the contragredient module of $\Mbb$. Let $\Wbb_\blt=\Wbb_1\otimes\cdots\otimes\Wbb_N$. Then a \textbf{conformal block} associated to $\fx$ and $\Wbb_\blt\otimes\Mbb\otimes\Mbb'$ (cf. \cite{Zhu-global,FB04}) is a linear functional $\uppsi:\Wbb_\blt\otimes\Mbb\otimes\Mbb'\rightarrow\Cbb$ that is ``invariant under the action of $\Vbb$". (See Def. \ref{lb25} for the precise definition.) Then the \textbf{(standard) sewing of $\uppsi$} is defined by taking contraction:
\begin{gather}
\mc S_q\uppsi:\Wbb_\blt\rightarrow\Cbb\qquad w_\blt\mapsto \wick{\uppsi(w_\blt\otimes q^{L(0)}\c1-\otimes \c1-)}
\end{gather}
provided that this series about the variable $q$ converges absolutely. (Cf. Sec. \ref{lb76} or \cite[Sec. 10]{Gui-sewingconvergence} for details.)

\begin{rem}
Note that $\mc S_q\uppsi$ depends on the choice of the argument $\arg q$. So sometimes it is more convenient to consider the \textbf{normalized sewing of $\uppsi$}, defined by
\begin{gather}
\wtd{\mc S}_q\uppsi:\Wbb_\blt\rightarrow\Cbb\qquad w_\blt\mapsto \wick{\uppsi(w_\blt\otimes q^{\wtd L(0)}\c1-\otimes \c1-)}
\end{gather}
where $\wtd L(0)$ is a suitable shift of $L(0)$ (or a shift of the semisimple part of $L(0)$ in case $\Vbb$ is not rational and hence $L(0)$ is not necessarily diagonalizable on $\Mbb$) so that $\wtd L(0)$ is diagonalizable on $\Mbb$ with eigenvalues in $\Nbb$. Note that $\mc S_q$ and $\wtd{\mc S}_q$ are equal when $q=1$ and $\arg q=0$. So these two types of sewing are closely related. In the main body of this paper, we will use $\wtd{\mc S}_q$ to deal with the propagation of (partial) conformal blocks. But in the introduction we stick to $\mc S_q$.
\end{rem}

Now we state the sewing-factorization theorem for a $C_2$-cofinite and rational VOA $\Vbb$. Let $\scr T_{\fx}^*(\Wbb_\blt\otimes\Mbb\otimes\Mbb')$ be the space of conformal blocks associated to $\fx$ and $\Wbb_\blt\otimes\Mbb\otimes\Mbb'$, which is finite-dimensional \cite{AN03-finite-dimensional,DGT2}. Likewise, let $\scr T_{\mc S_q\fx}^*(\Wbb_\blt)$ be the space of conformal blocks associated to $\fx$ and $\Wbb_\blt$. Let $\mc E$ be a set of representatives of equivalence classes of irreducible grading-restricted $\Vbb$-modules, which is a finite set.

\begin{thm}[\textbf{Sewing-factorization theorem}, cf. {\cite[Thm. 12.1]{Gui-sewingconvergence}}]  \label{lb77}
Choose $q\in\mc D_{r\rho}^\times=\{z\in\Cbb:0<|z|<r\rho\}$. Then we have a well-defined linear map
\begin{gather}\label{eq170}
\begin{gathered}
\fk S_q:\bigoplus_{\Mbb\in\mc E} \scr T_{\fx}^*(\Wbb_\blt\otimes\Mbb\otimes\Mbb')\rightarrow \scr T_{\mc S_q\fx}^*(\Wbb_\blt)\\
\bigoplus_{\Mbb\in\mc E}\uppsi_\Mbb\mapsto \sum_{\Mbb\in\mc E} \mc S_q\uppsi_\Mbb
\end{gathered}
\end{gather}
By ``well-defined" we mean that  $\mc S_q\uppsi_\Mbb(w_\blt)$ converges absolutely at $q$ for every $w_\blt\in\Wbb_\blt$, and that the linear functional $\mc S_q\uppsi_\Mbb:\Wbb_\blt\rightarrow\Cbb$ is a conformal block (i.e., is an element of $\scr T_{\mc S_q\fx}^*(\Wbb_\blt)$). 

Moreover, $\fk S_q$ is an isomorphism of vector spaces. In particular, we have
\begin{align}
\sum_{\Mbb\in\mc E} \dim\scr T_{\fx}^*(\Wbb_\blt\otimes\Mbb\otimes\Mbb')=\dim \scr T_{\mc S_q\fx}^*(\Wbb_\blt)  \label{eq165}
\end{align}
\end{thm}

\begin{rem}\label{lb78}
\eqref{eq165} is due to \cite[Thm. 7.0.1]{DGT2}. See the introduction of \cite{DGT2} for a brief review of the history of the proof of \eqref{eq165}. The rest of Thm. \ref{lb77} (namely, the well-definedness and the injectivity of the linear map $\fk S_q$) is due to \cite[Thm. 12.1]{Gui-sewingconvergence} and does not require $\Vbb$ to be rational.
\end{rem}

Formula \eqref{eq165} is commonly referred to as the \textbf{factorization property} in the literature of the algebro-geometric approach to conformal blocks (cf. \cite{TUY,BFM-conformal-blocks,NT-P1_conformal_blocks}). It gives an efficient way of computing the dimensions of spaces of conformal blocks by reducing the genus and the number of marked points on each component. Also, a version of Thm. \ref{lb77} was proved in \cite[Thm. 8.5.1]{DGT2} for ``infinitesimally small $q$".

\subsection{Sewing-factorization beyond rationality}

As pointed out in the introduction of \cite{HLZ1}, from the viewpoint of the representation theory of VOAs, it is unnatural to restrict attention to rational VOAs. In fact, most of the proof of the sewing-factorization theorem does not rely on the rationality of $\Vbb$, i.e., it does not require the grading-restricted $\Vbb$-modules to be completely reducible. The assumption of rationality appears to be more of a technical convenience than an essential requirement of the proof. Moreover, there are important $C_2$-cofinite irrational VOAs, such as the triplet $\mc W$-algebras \cite{AM-triplet}, their tensor products, and their subalgebras fixed by finite solvable automorphism groups \cite{Miy-C2-orbifold}. Thus, it is natural to consider generalizing the sewing-factorization theorem to $C_2$-cofinite VOAs that are not necessarily rational.

So now we assume that $\Vbb$ is $C_2$-cofinite but not necessarily rational. By Rem. \ref{lb78}, $\fk S_q$ is still well-defined and is injective. But $\fk S_q$ is not necessarily surjective. Thus we only have ``$\leq$" in \eqref{eq165}. For instance, if we let $\Mbb$ be a grading-restricted $\Vbb$-module which is not completely reducible, and if we choose $\uppsi\in\scr T_{\fx}^*(\Wbb_\blt\otimes\Mbb\otimes \Mbb')$, then $\mc S_q\uppsi$ is an element of $\scr T_{\mc S_q\fx}^*(\Wbb_\blt)$ which is not necessarily in the range of $\fk S_q$. Thus, a preliminary step to understanding the sewing-factorization property is to answer the following question: 

\begin{que}\label{lb79}
Is $\scr T_{\mc S_q\fx}^*(\Wbb_\blt)$ spanned by elements of the form $\mc S_q\uppsi$ where $\uppsi\in\scr T_{\fx}^*(\Wbb_\blt\otimes\Mbb\otimes \Mbb')$ and $\Mbb$ is a grading-restricted $\Vbb$-module?
\end{que}

The answer is known in the low genus cases:
\begin{enumerate}[align=left, label=Case \arabic*.]
\item If $\fx$ is the disjoint union of $\fk P_1=(\Pbb^1;0,z_1,\infty)$ and $\fk P_2=(\Pbb^1;0,z_2,\infty)$ (where $z_1,z_2\in\Cbb$ are non-zero) and if the sewing is along the point $0$ of $\fk P_1$ and $\infty$ of $\fk P_2$ (so that $\mc S_q\fk X$ is a sphere), then Question \ref{lb79} is answered \emph{affirmatively} by \cite{HLZ1,HLZ2}-\cite{HLZ8}.
\item If $\fk X$ is $(\Pbb^1;0,z,\infty)$ where $z\neq 0$, and if the sewing is along $0,\infty$ (so that $\mc S_q\fx$ is a torus), then the answer to Question \ref{lb79} is \emph{negative}, as suggested by the replacement of traces with pseudo-traces in the study of the modular invariance of genus-$1$ conformal blocks in \cite{Miy-modular-invariance,AN-pseudo-trace,Fio-genus-1,Hua-modular-C2}.
\end{enumerate}

It is surprising that the answer to Question \ref{lb79} depends on the type of the sewing: In Case 1, the sewing is disjoint, and in Case 2, the sewing is a self-sewing (cf. Fig. \ref{fig2}). Worse still, the pseudo-trace construction in Case 2 seems to lack a geometric meaning. On the other hand, \cite{Gui-permutation}, though focusing mainly on the rational case, suggests that the genus-$0$ conformal blocks of permutation-twisted $\Vbb^{\otimes N}$-modules (which are not necessarily of the form $\Wbb_1\otimes\cdots\otimes\Wbb_N$ or a direct sum of such modules !) should correspond to higher genus conformal blocks of untwisted $\Vbb$-modules via a branched covering of $\Pbb^1$. This correspondence gives us a hint on how to generalize the sewing-factorization property in Case 1 to the \emph{disjoint sewing} of higher genus surfaces. Motivated by these observations, we outline below several key features that our sewing-factorization theory for $C_2$-cofinite VOAs is designed to satisfy.
\begin{enumerate}[align=left, label=R\arabic*.]
\item It is easy to see why the rational version (i.e. Thm. \ref{lb77}) is a special case of our theory.
\item Our theory will generalize both Case 1 (Huang-Lepowsky-Zhang's vertex tensor category theory) and Case 2 (Miyamoto's pseudo-trace theory). In particular, it will give the pseudo-trace construction in \cite{Miy-modular-invariance,AN-pseudo-trace,Fio-genus-1,Hua-modular-C2} a geometric explanation.\footnote{Part of this geometric interpretation of pseudo-traces will be presented in the second paper of this series. A more comprehensive treatment will appear in a future work outside this three-part series.}
\item Self-sewing will not be considered. Only disjoint sewing is allowed. To compensate, we also consider disjoint sewing along several pairs of points (cf. Fig. \ref{fig3}).
\item We consider a grading-restricted $\Vbb^{\otimes N}$-module not necessarily \textbf{tensor-factorizable} (cf. \cite{Gui-permutation}), i.e., not necessarily a direct sum of those of the form $\Wbb_1\otimes\cdots\otimes\Wbb_N$ where each $\Wbb_i$ is a grading-restricted $\Vbb$-module. (This consideration is closely related to the pseudo-trace mentioned in Case 2.)
\item Our sewing-factorization theorem is compatible with the permutation-twisted/untwisted correspondence established in \cite{Gui-permutation}. More precisely, the translation of the permutation-twisted version of the genus-$0$ sewing-factorization theorem from \cite{HLZ1,HLZ2}-\cite{HLZ8} to the higher genus untwisted case via the construction in \cite{Gui-permutation} will be a special case of our sewing-factorization theorem.
\end{enumerate}

\subsection{The sewing-factorization theorem for $C_2$-cofinite VOAs}\label{lb83}

Fix an $(N+M)$-pointed compact Riemann surface and an $(K+M)$-pointed one
\begin{gather} \label{eq166}
\fx=(y'_1,\dots,y'_M|C_1|x_1,\dots,x_N)\qquad \fk Y=(y_1'',\dots,y_M''|C_2|\varkappa_1,\dots,\varkappa_K)
\end{gather}
Assume that each component of $C_1$ (resp. $C_2$) intersects $x_1,\dots,x_N$ (resp. $\varkappa_1,\dots,\varkappa_K$). We choose local coordinates $\eta_i,\mu_k,\xi_j,\varpi_j$ at each $x_i,\varkappa_k,y_j',y_j''$ respectively. The reason we write $y_\blt'$ and $y_\blt''$ on the left of $C_1,C_2$ is that we regard them as ``outgoing marked points". Those written respectively on the right of $C_1,C_2$ are regarded as ``incoming marked points".

Let $\Vbb$ be $C_2$-cofinite. Associate a grading-restricted $\Vbb^{\otimes N}$-module $\Wbb$ to the ordered marked points $x_1,\dots,x_N$. Associate a grading-restricted $\Vbb^{\otimes K}$-module $\Mbb$ to the ordered marked points $\varkappa_1,\dots,\varkappa_K$. The following theorem is the main result of this paper. Its proof follows automatically from the results established in this paper.

\begin{thm} \label{lb80}
There exist uniquely (up to equivalences) a grading-restricted $\Vbb^{\otimes M}$-module $\bbs_{\fx}(\Wbb)$ and a conformal block $\gimel\in\scr T_{\fx}^*(\Wbb\otimes \bbs_{\fx}(\Wbb))$ (so $\gimel:\Wbb\otimes\bbs_{\fx}(\Wbb)\rightarrow\Cbb$ is a linear functional ``invariant under the action of $\Vbb$") satisfying the following condition: For every grading-restricted $\Vbb^{\otimes M}$-module $\Xbb$ and every $\Gamma\in\scr T_{\fx}^*(\Wbb\otimes\Xbb)$ there is a unique morphism $T:\Xbb\rightarrow\bbs_\fx(\Wbb)$ such that $\Gamma=\gimel\circ(\idt\otimes T)$.

Similarly, such $\bbs_{\fk Y}(\Mbb)$ and $\daleth\in\scr T_{\fk Y}^*(\Mbb\otimes\bbs_{\fk Y}(\Mbb))$ exist uniquely (up to equivalences) for $\fk Y$.
\end{thm}
\begin{proof}
By Thm. \ref{lb46}, grading-restricted $\Vbb^{\otimes M}$-modules are equivalent to finitely-generated admissible $\Vbb^{\times M}$-modules. The existence of $(\bbs_\fx(\Wbb),\gimel)$ is due to Thm. \ref{lb57} and Thm. \ref{lb55}. The uniqueness is due to Rem. \ref{fusion-uniqueness}.
\end{proof}

We call $(\bbs_\fx(\Wbb),\gimel)$ (or simply call $\bbs_\fx(\Wbb)$) the \textbf{dual fusion product} associated to $\fx$ and $\Wbb$. Its contragredient module is denoted by $\boxtimes_\fx(\Wbb)$ and called the \textbf{fusion product} associated to $\fx$ and $\Wbb$. As an immediate consequence of Thm. \ref{lb80}, we have
\begin{align}\label{eq167}
\dim\Hom_{\Vbb^{\otimes M}}(\Xbb,\bbs_\fx(\Wbb))=\dim\scr T_\fx^*(\Wbb\otimes\Xbb)
\end{align}

For each $1\leq j\leq M$, choose neighborhoods $W_j'$ of $y_j'$ and  $W_j''$ of $y_j''$ on which $\xi_j$ and $\varpi_j$ are defined respectively. We assume that there exist $r_j,\rho_j>0$ such that
\begin{align}
\xi_j(W_j')=\mc D_{r_j}\qquad \varpi_j(W_j'')=\mc D_{\rho_j}
\end{align}
Note that $\xi_j(y_j')=\varpi_j(y_j'')=0$.

For each $j$, choose $q_j\in\Cbb$ satisfying
\begin{align}
0<|q_j|<r_j\rho_j
\end{align}
Remove small discs centered at $y_j',y_j''$ respectively, and glue the remaining part by the rule that $p_j'\in W_j'$ and $r_j''\in W_j''$ are identified iff
\begin{align}
\xi_j(p_j')\varpi_j(p_j'')=q_j
\end{align} 
By performing this gluing construction for all $1\leq j\leq M$ we obtain a new pointed surface
\begin{align}
\fx\#_{q_\blt}\fk Y=(\mc C_{q_\blt}|x_1,\dots,x_N,\varkappa_1,\dots,\varkappa_K)
\end{align}
with local coordinates $\eta_1,\dots,\eta_N,\mu_1,\dots,\mu_K$. See Figure \ref{fig3} for example.

\begin{figure}[h]
	\centering
\scalebox{0.95}{

\tikzset{every picture/.style={line width=0.75pt}} 

\begin{tikzpicture}[x=0.75pt,y=0.75pt,yscale=-1,xscale=1]

\draw   (196.18,105.08) .. controls (217.04,92.23) and (232.47,105.5) .. (250.3,113.2) .. controls (268.12,120.89) and (278.84,119.13) .. (286.94,119.28) .. controls (295.03,119.44) and (299.7,116.3) .. (300.03,124.77) .. controls (300.36,133.25) and (279.36,126.83) .. (279.36,133.5) .. controls (279.36,140.17) and (300.86,130) .. (302.11,139.5) .. controls (303.36,149) and (281.36,143.5) .. (281.36,148.5) .. controls (281.36,153.5) and (302.86,145) .. (302.86,155) .. controls (302.86,165) and (281.19,159.75) .. (273.97,159.19) .. controls (266.75,158.63) and (267.77,155.13) .. (248.04,158.51) .. controls (228.31,161.89) and (227.75,175.42) .. (201.25,167.98) .. controls (174.76,160.54) and (175.32,117.93) .. (196.18,105.08) -- cycle ;
\draw   (345.86,123.75) .. controls (345.61,113.5) and (377.11,124) .. (386.36,122.5) .. controls (395.61,121) and (397.14,120.17) .. (403.25,117.21) .. controls (409.36,114.25) and (425.2,102.09) .. (441.18,114) .. controls (457.16,125.91) and (460.18,161) .. (438.45,169.55) .. controls (416.73,178.09) and (403.95,163.05) .. (391.95,157.05) .. controls (379.95,151.05) and (351.36,169.25) .. (348.86,157) .. controls (346.36,144.75) and (361.36,151.25) .. (361.11,145.75) .. controls (360.86,140.25) and (343.11,150.25) .. (342.36,139.75) .. controls (341.61,129.25) and (360.36,138.5) .. (361.11,133.25) .. controls (361.86,128) and (346.11,134) .. (345.86,123.75) -- cycle ;
\draw    (196.27,127.56) .. controls (206.65,118.29) and (214.85,120.76) .. (221.95,128.79) ;
\draw    (199.55,125.7) .. controls (205.56,130.65) and (212.12,130.03) .. (218.13,125.7) ;
\draw    (215.77,149.56) .. controls (226.15,140.29) and (234.35,142.76) .. (241.45,150.79) ;
\draw    (219.05,147.7) .. controls (225.06,152.65) and (231.62,152.03) .. (237.63,147.7) ;
\draw    (410.27,141.06) .. controls (420.65,131.79) and (428.85,134.26) .. (435.95,142.29) ;
\draw    (413.55,139.2) .. controls (419.56,144.15) and (426.12,143.53) .. (432.13,139.2) ;
\draw  [fill={rgb, 255:red, 0; green, 0; blue, 0 }  ,fill opacity=1 ] (240.9,127.78) .. controls (240.9,126.88) and (241.63,126.14) .. (242.53,126.14) .. controls (243.44,126.14) and (244.17,126.88) .. (244.17,127.78) .. controls (244.17,128.68) and (243.44,129.42) .. (242.53,129.42) .. controls (241.63,129.42) and (240.9,128.68) .. (240.9,127.78) -- cycle ;
\draw  [fill={rgb, 255:red, 0; green, 0; blue, 0 }  ,fill opacity=1 ] (293.9,123.53) .. controls (293.9,122.63) and (294.63,121.89) .. (295.53,121.89) .. controls (296.44,121.89) and (297.17,122.63) .. (297.17,123.53) .. controls (297.17,124.43) and (296.44,125.17) .. (295.53,125.17) .. controls (294.63,125.17) and (293.9,124.43) .. (293.9,123.53) -- cycle ;
\draw  [fill={rgb, 255:red, 0; green, 0; blue, 0 }  ,fill opacity=1 ] (295.9,140.14) .. controls (295.9,139.23) and (296.63,138.5) .. (297.53,138.5) .. controls (298.44,138.5) and (299.17,139.23) .. (299.17,140.14) .. controls (299.17,141.04) and (298.44,141.77) .. (297.53,141.77) .. controls (296.63,141.77) and (295.9,141.04) .. (295.9,140.14) -- cycle ;
\draw  [fill={rgb, 255:red, 0; green, 0; blue, 0 }  ,fill opacity=1 ] (295.9,155.14) .. controls (295.9,154.23) and (296.63,153.5) .. (297.53,153.5) .. controls (298.44,153.5) and (299.17,154.23) .. (299.17,155.14) .. controls (299.17,156.04) and (298.44,156.77) .. (297.53,156.77) .. controls (296.63,156.77) and (295.9,156.04) .. (295.9,155.14) -- cycle ;
\draw  [fill={rgb, 255:red, 0; green, 0; blue, 0 }  ,fill opacity=1 ] (349.4,124.89) .. controls (349.4,123.98) and (350.13,123.25) .. (351.03,123.25) .. controls (351.94,123.25) and (352.67,123.98) .. (352.67,124.89) .. controls (352.67,125.79) and (351.94,126.52) .. (351.03,126.52) .. controls (350.13,126.52) and (349.4,125.79) .. (349.4,124.89) -- cycle ;
\draw  [fill={rgb, 255:red, 0; green, 0; blue, 0 }  ,fill opacity=1 ] (346.15,139.39) .. controls (346.15,138.48) and (346.88,137.75) .. (347.78,137.75) .. controls (348.69,137.75) and (349.42,138.48) .. (349.42,139.39) .. controls (349.42,140.29) and (348.69,141.02) .. (347.78,141.02) .. controls (346.88,141.02) and (346.15,140.29) .. (346.15,139.39) -- cycle ;
\draw  [fill={rgb, 255:red, 0; green, 0; blue, 0 }  ,fill opacity=1 ] (352.65,155.39) .. controls (352.65,154.48) and (353.38,153.75) .. (354.28,153.75) .. controls (355.19,153.75) and (355.92,154.48) .. (355.92,155.39) .. controls (355.92,156.29) and (355.19,157.02) .. (354.28,157.02) .. controls (353.38,157.02) and (352.65,156.29) .. (352.65,155.39) -- cycle ;
\draw  [fill={rgb, 255:red, 0; green, 0; blue, 0 }  ,fill opacity=1 ] (258.9,146.78) .. controls (258.9,145.88) and (259.63,145.14) .. (260.53,145.14) .. controls (261.44,145.14) and (262.17,145.88) .. (262.17,146.78) .. controls (262.17,147.68) and (261.44,148.42) .. (260.53,148.42) .. controls (259.63,148.42) and (258.9,147.68) .. (258.9,146.78) -- cycle ;
\draw  [fill={rgb, 255:red, 0; green, 0; blue, 0 }  ,fill opacity=1 ] (429.4,123.78) .. controls (429.4,122.88) and (430.13,122.14) .. (431.03,122.14) .. controls (431.94,122.14) and (432.67,122.88) .. (432.67,123.78) .. controls (432.67,124.68) and (431.94,125.42) .. (431.03,125.42) .. controls (430.13,125.42) and (429.4,124.68) .. (429.4,123.78) -- cycle ;
\draw  [fill={rgb, 255:red, 0; green, 0; blue, 0 }  ,fill opacity=1 ] (442.4,154.28) .. controls (442.4,153.38) and (443.13,152.64) .. (444.03,152.64) .. controls (444.94,152.64) and (445.67,153.38) .. (445.67,154.28) .. controls (445.67,155.18) and (444.94,155.92) .. (444.03,155.92) .. controls (443.13,155.92) and (442.4,155.18) .. (442.4,154.28) -- cycle ;
\draw  [fill={rgb, 255:red, 0; green, 0; blue, 0 }  ,fill opacity=1 ] (400.9,143.78) .. controls (400.9,142.88) and (401.63,142.14) .. (402.53,142.14) .. controls (403.44,142.14) and (404.17,142.88) .. (404.17,143.78) .. controls (404.17,144.68) and (403.44,145.42) .. (402.53,145.42) .. controls (401.63,145.42) and (400.9,144.68) .. (400.9,143.78) -- cycle ;
\draw [color={rgb, 255:red, 74; green, 144; blue, 226 }  ,draw opacity=1 ] [dash pattern={on 1.5pt off 1.5pt}]  (310.27,123.21) -- (337.77,123.65) ;
\draw [shift={(339.77,123.68)}, rotate = 180.91] [color={rgb, 255:red, 74; green, 144; blue, 226 }  ,draw opacity=1 ][line width=0.75]    (6.56,-1.97) .. controls (4.17,-0.84) and (1.99,-0.18) .. (0,0) .. controls (1.99,0.18) and (4.17,0.84) .. (6.56,1.97)   ;
\draw [shift={(308.27,123.18)}, rotate = 0.91] [color={rgb, 255:red, 74; green, 144; blue, 226 }  ,draw opacity=1 ][line width=0.75]    (6.56,-1.97) .. controls (4.17,-0.84) and (1.99,-0.18) .. (0,0) .. controls (1.99,0.18) and (4.17,0.84) .. (6.56,1.97)   ;
\draw [color={rgb, 255:red, 74; green, 144; blue, 226 }  ,draw opacity=1 ] [dash pattern={on 1.5pt off 1.5pt}]  (308.77,139.18) -- (335.77,139.18) ;
\draw [shift={(337.77,139.18)}, rotate = 180] [color={rgb, 255:red, 74; green, 144; blue, 226 }  ,draw opacity=1 ][line width=0.75]    (6.56,-1.97) .. controls (4.17,-0.84) and (1.99,-0.18) .. (0,0) .. controls (1.99,0.18) and (4.17,0.84) .. (6.56,1.97)   ;
\draw [shift={(306.77,139.18)}, rotate = 0] [color={rgb, 255:red, 74; green, 144; blue, 226 }  ,draw opacity=1 ][line width=0.75]    (6.56,-1.97) .. controls (4.17,-0.84) and (1.99,-0.18) .. (0,0) .. controls (1.99,0.18) and (4.17,0.84) .. (6.56,1.97)   ;
\draw [color={rgb, 255:red, 74; green, 144; blue, 226 }  ,draw opacity=1 ] [dash pattern={on 1.5pt off 1.5pt}]  (311.27,155.18) -- (339.27,155.18) ;
\draw [shift={(341.27,155.18)}, rotate = 180] [color={rgb, 255:red, 74; green, 144; blue, 226 }  ,draw opacity=1 ][line width=0.75]    (6.56,-1.97) .. controls (4.17,-0.84) and (1.99,-0.18) .. (0,0) .. controls (1.99,0.18) and (4.17,0.84) .. (6.56,1.97)   ;
\draw [shift={(309.27,155.18)}, rotate = 0] [color={rgb, 255:red, 74; green, 144; blue, 226 }  ,draw opacity=1 ][line width=0.75]    (6.56,-1.97) .. controls (4.17,-0.84) and (1.99,-0.18) .. (0,0) .. controls (1.99,0.18) and (4.17,0.84) .. (6.56,1.97)   ;

\draw (225.47,111.32) node [anchor=north west][inner sep=0.75pt]  [font=\small]  {$x_{1}$};
\draw (239.97,136.42) node [anchor=north west][inner sep=0.75pt]  [font=\small]  {$x_{2}$};
\draw (410.47,113.42) node [anchor=north west][inner sep=0.75pt]  [font=\small]  {$\varkappa _{1}$};
\draw (423.47,143.92) node [anchor=north west][inner sep=0.75pt]  [font=\small]  {$\varkappa _{3}$};
\draw (381.97,133.42) node [anchor=north west][inner sep=0.75pt]  [font=\small]  {$\varkappa _{2}$};
\draw (162,139) node [anchor=north west][inner sep=0.75pt]    {$\mathfrak{X}$};
\draw (459,142) node [anchor=north west][inner sep=0.75pt]    {$\mathfrak{Y}$};

\end{tikzpicture}
}
	\caption{~~An example of $\fx\#_{q_\blt}\fk Y$ where $N=2,K=3,M=3$. $\fx\#_{q_\blt}\fk Y$ has genus-$5$}
	\label{fig3}
\end{figure}

Write $Y(\idt\otimes\cdots\otimes\cbf\otimes\cdots\otimes\idt,z)=\sum_{n\in\Zbb} L_j(n)z^{-n-2}$ where the conformal vector $\cbf\in\Vbb(2)$ is at the $j$-th tensor component of $\idt\otimes\cdots\otimes\cbf\otimes\cdots\otimes\idt$. The following theorem and corollary will be proved in \cite[Thm. 3.4, 3.5]{GZ3}.

\begin{thm}[\textbf{Sewing-factorization}]\label{lb81}
There is a well-defined linear map
\begin{gather}
\begin{gathered}
\Psi_{q_\blt}: \Hom_{\Vbb^{\otimes M}}\big(\boxtimes_{\fk Y}(\Mbb),\bbs_\fx(\Wbb)\big) \rightarrow\scr T_{\fk X\#_{q_\blt}\fk Y}^*(\Wbb\otimes\Mbb)\\
T\mapsto \mc S_{q_\blt}\big((\gimel\circ T)\otimes \daleth\big)
\end{gathered}
\end{gather}
By ``well-defined" we mean that for each $w\in\Wbb,m\in\Mbb$,
\begin{align}
\begin{aligned}
\mc S_{q_\blt}\big((\gimel\circ T)\otimes \daleth\big)(w\otimes m):=&\wick{\gimel\big(w\otimes T(\c2-)\big)\cdot \daleth\big(m\otimes q_1^{L_1(0)}\cdots q_M^{L_M(0)}\c2-\big)}\\
=&\sum_{\alpha} \gimel\big(w\otimes T\wch e_\alpha\big)\cdot \daleth\big(m\otimes q_1^{L_1(0)}\cdots q_M^{L_M(0)}e_\alpha\big)
\end{aligned}
\end{align}
(where $\{e_\alpha\}$ is a homogeneous basis of $\bbs_{\fk Y}(\Mbb)$ and $\{\wch e_\alpha\}$ is its dual basis), as a formal series of $q_1,\dots,q_M$ and $\log q_1,\dots,\log q_M$, converges absolutely (in an appropriate sense) to a conformal block associated to $\fx\#_{q_\blt}\fk Y$ and $\Wbb\otimes\Mbb$. 

Moreover, $\Psi_{q_\blt}$ is an isomorphism of (finite-dimensional) vector spaces. 
\end{thm}

Thm. \ref{lb81} can be stated in the following equivalent way, thanks to Thm. \ref{lb80}.

\begin{co}[\textbf{Sewing-factorization}]\label{lb82}
There is a well-defined isomorphism of vector spaces
\begin{gather}
\begin{gathered}
\Phi_{q_\blt}: \scr T_\fx^*\big(\Wbb\otimes\boxtimes_{\fk Y}(\Mbb)\big) \rightarrow\scr T_{\fk X\#_{q_\blt}\fk Y}^*(\Wbb\otimes\Mbb)\\
\uppsi\mapsto \mc S_{q_\blt}\big(\uppsi\otimes \daleth\big)
\end{gathered}
\end{gather}
where, for each $w\in\Wbb,m\in\Mbb$,
\begin{align}
\mc S_{q_\blt}\big(\uppsi\otimes \daleth\big)(w\otimes m)=\wick{\uppsi\big(w\otimes \c2-\big)\cdot \daleth\big(m\otimes q_1^{L_1(0)}\cdots q_M^{L_M(0)}\c2-\big)}
\end{align}
\end{co}

\subsection{Self-sewing regained; the factorization formula}

Although we have only considered disjoint sewing in Sec. \ref{lb83}, we can still prove a sewing-factorization theorem for any type of sewing by transforming it to a disjoint sewing as in Fig. \ref{fig4}.

\begin{figure}[h]
	\centering
\scalebox{0.85}{

\tikzset{every picture/.style={line width=0.75pt}} 

\begin{tikzpicture}[x=0.75pt,y=0.75pt,yscale=-1,xscale=1]

\draw   (230.35,107.87) .. controls (233.06,126.69) and (205.12,113.86) .. (205.57,131.39) .. controls (206.02,148.93) and (235.83,134.45) .. (233.51,152.77) .. controls (231.19,171.1) and (192.1,153.19) .. (162.92,151.53) .. controls (133.73,149.87) and (103.29,182.28) .. (77.61,161.76) .. controls (51.93,141.23) and (67.25,96.32) .. (96.54,87.77) .. controls (125.82,79.22) and (138.23,113.9) .. (164.36,111.16) .. controls (190.49,108.42) and (227.65,89.05) .. (230.35,107.87) -- cycle ;
\draw    (87.77,130.02) .. controls (100.52,116.13) and (110.59,119.83) .. (119.32,131.87) ;
\draw    (91.8,127.24) .. controls (99.18,134.65) and (107.24,133.72) .. (114.62,127.24) ;
\draw    (140.77,133.56) .. controls (151.15,124.29) and (159.35,126.76) .. (166.45,134.79) ;
\draw    (144.05,131.7) .. controls (150.06,136.65) and (156.62,136.03) .. (162.63,131.7) ;
\draw  [fill={rgb, 255:red, 0; green, 0; blue, 0 }  ,fill opacity=1 ] (223.4,109.03) .. controls (223.4,108.13) and (224.13,107.39) .. (225.03,107.39) .. controls (225.94,107.39) and (226.67,108.13) .. (226.67,109.03) .. controls (226.67,109.93) and (225.94,110.67) .. (225.03,110.67) .. controls (224.13,110.67) and (223.4,109.93) .. (223.4,109.03) -- cycle ;
\draw  [fill={rgb, 255:red, 0; green, 0; blue, 0 }  ,fill opacity=1 ] (226.4,152.53) .. controls (226.4,151.63) and (227.13,150.89) .. (228.03,150.89) .. controls (228.94,150.89) and (229.67,151.63) .. (229.67,152.53) .. controls (229.67,153.43) and (228.94,154.17) .. (228.03,154.17) .. controls (227.13,154.17) and (226.4,153.43) .. (226.4,152.53) -- cycle ;
\draw   (499.35,108.87) .. controls (502.06,127.69) and (474.12,114.86) .. (474.57,132.39) .. controls (475.02,149.93) and (504.83,135.45) .. (502.51,153.77) .. controls (500.19,172.1) and (461.1,154.19) .. (431.92,152.53) .. controls (402.73,150.87) and (372.29,183.28) .. (346.61,162.76) .. controls (320.93,142.23) and (336.25,97.32) .. (365.54,88.77) .. controls (394.82,80.22) and (407.23,114.9) .. (433.36,112.16) .. controls (459.49,109.42) and (496.65,90.05) .. (499.35,108.87) -- cycle ;
\draw    (356.77,131.02) .. controls (369.52,117.13) and (379.59,120.83) .. (388.32,132.87) ;
\draw    (360.8,128.24) .. controls (368.18,135.65) and (376.24,134.72) .. (383.62,128.24) ;
\draw    (409.77,134.56) .. controls (420.15,125.29) and (428.35,127.76) .. (435.45,135.79) ;
\draw    (413.05,132.7) .. controls (419.06,137.65) and (425.62,137.03) .. (431.63,132.7) ;
\draw  [fill={rgb, 255:red, 0; green, 0; blue, 0 }  ,fill opacity=1 ] (492.4,110.03) .. controls (492.4,109.13) and (493.13,108.39) .. (494.03,108.39) .. controls (494.94,108.39) and (495.67,109.13) .. (495.67,110.03) .. controls (495.67,110.93) and (494.94,111.67) .. (494.03,111.67) .. controls (493.13,111.67) and (492.4,110.93) .. (492.4,110.03) -- cycle ;
\draw  [fill={rgb, 255:red, 0; green, 0; blue, 0 }  ,fill opacity=1 ] (495.4,153.53) .. controls (495.4,152.63) and (496.13,151.89) .. (497.03,151.89) .. controls (497.94,151.89) and (498.67,152.63) .. (498.67,153.53) .. controls (498.67,154.43) and (497.94,155.17) .. (497.03,155.17) .. controls (496.13,155.17) and (495.4,154.43) .. (495.4,153.53) -- cycle ;
\draw [color={rgb, 255:red, 74; green, 144; blue, 226 }  ,draw opacity=1 ] [dash pattern={on 1.5pt off 1.5pt}]  (236.25,111.65) .. controls (245.26,120.95) and (248.01,135.17) .. (238.63,151.19) ;
\draw [shift={(237.73,152.68)}, rotate = 302.15] [color={rgb, 255:red, 74; green, 144; blue, 226 }  ,draw opacity=1 ][line width=0.75]    (6.56,-1.97) .. controls (4.17,-0.84) and (1.99,-0.18) .. (0,0) .. controls (1.99,0.18) and (4.17,0.84) .. (6.56,1.97)   ;
\draw [shift={(234.73,110.18)}, rotate = 42.14] [color={rgb, 255:red, 74; green, 144; blue, 226 }  ,draw opacity=1 ][line width=0.75]    (6.56,-1.97) .. controls (4.17,-0.84) and (1.99,-0.18) .. (0,0) .. controls (1.99,0.18) and (4.17,0.84) .. (6.56,1.97)   ;
\draw   (535.38,112.21) .. controls (534.55,98.64) and (571.55,97.14) .. (578.05,108.64) .. controls (584.55,120.14) and (589.67,132.45) .. (581.61,147.29) .. controls (573.55,162.14) and (538.05,165.14) .. (537.55,150.64) .. controls (537.05,136.14) and (561.03,151.19) .. (561.03,131.38) .. controls (561.03,111.56) and (536.21,125.78) .. (535.38,112.21) -- cycle ;
\draw  [fill={rgb, 255:red, 0; green, 0; blue, 0 }  ,fill opacity=1 ] (540.4,111.03) .. controls (540.4,110.13) and (541.13,109.39) .. (542.03,109.39) .. controls (542.94,109.39) and (543.67,110.13) .. (543.67,111.03) .. controls (543.67,111.93) and (542.94,112.67) .. (542.03,112.67) .. controls (541.13,112.67) and (540.4,111.93) .. (540.4,111.03) -- cycle ;
\draw  [fill={rgb, 255:red, 0; green, 0; blue, 0 }  ,fill opacity=1 ] (540.9,151.03) .. controls (540.9,150.13) and (541.63,149.39) .. (542.53,149.39) .. controls (543.44,149.39) and (544.17,150.13) .. (544.17,151.03) .. controls (544.17,151.93) and (543.44,152.67) .. (542.53,152.67) .. controls (541.63,152.67) and (540.9,151.93) .. (540.9,151.03) -- cycle ;
\draw [color={rgb, 255:red, 74; green, 144; blue, 226 }  ,draw opacity=1 ] [dash pattern={on 1.5pt off 1.5pt}]  (505.15,110.71) -- (530.55,110.68) ;
\draw [shift={(532.55,110.68)}, rotate = 179.95] [color={rgb, 255:red, 74; green, 144; blue, 226 }  ,draw opacity=1 ][line width=0.75]    (6.56,-1.97) .. controls (4.17,-0.84) and (1.99,-0.18) .. (0,0) .. controls (1.99,0.18) and (4.17,0.84) .. (6.56,1.97)   ;
\draw [shift={(503.15,110.71)}, rotate = 359.95] [color={rgb, 255:red, 74; green, 144; blue, 226 }  ,draw opacity=1 ][line width=0.75]    (6.56,-1.97) .. controls (4.17,-0.84) and (1.99,-0.18) .. (0,0) .. controls (1.99,0.18) and (4.17,0.84) .. (6.56,1.97)   ;
\draw [color={rgb, 255:red, 74; green, 144; blue, 226 }  ,draw opacity=1 ] [dash pattern={on 1.5pt off 1.5pt}]  (507.65,152.71) -- (533.05,152.68) ;
\draw [shift={(535.05,152.68)}, rotate = 179.95] [color={rgb, 255:red, 74; green, 144; blue, 226 }  ,draw opacity=1 ][line width=0.75]    (6.56,-1.97) .. controls (4.17,-0.84) and (1.99,-0.18) .. (0,0) .. controls (1.99,0.18) and (4.17,0.84) .. (6.56,1.97)   ;
\draw [shift={(505.65,152.71)}, rotate = 359.95] [color={rgb, 255:red, 74; green, 144; blue, 226 }  ,draw opacity=1 ][line width=0.75]    (6.56,-1.97) .. controls (4.17,-0.84) and (1.99,-0.18) .. (0,0) .. controls (1.99,0.18) and (4.17,0.84) .. (6.56,1.97)   ;

\draw (277.95,127.54) node [anchor=north west][inner sep=0.75pt]  [font=\LARGE]  {$=$};

\end{tikzpicture}

}
	\caption{~~Transforming self-sewing to disjoint sewing}
	\label{fig4}
\end{figure}

Let $\Vbb$ be $C_2$-cofinite. Let $\fx=(y',y'';\xi,\varpi|C|x_1,\dots,x_N;\eta_1,\dots,\eta_N)$ be as in Sec. \ref{lb84}. We have moved the pair of points $y',y''$ to the left of $C$ to indicate that they may be viewed as ``outgoing points". For each $0<|q|<r\rho$, $\mc S_q\fx$ is as in Sec. \ref{lb84}. Associate a grading-restricted $\Vbb^{\otimes N}$-module $\Wbb$ to the marked points $x_1,\dots,x_N$. 

Let $\zeta$ be the standard coordinate of $\Cbb$. Let
\begin{align}
\fq=(\infty,0;1/\zeta,\zeta|\Pbb^1|1;\zeta-1)   \label{eq171}
\end{align}
be a $3$-pointed sphere with marked points $\infty,0,1$ (where $\infty,0$ are the outgoing ones) and local coordinates $1/\zeta,\zeta,\zeta-1$. Associate $\Vbb$ to the incoming point $1$. By Thm. \ref{lb80}, we have a dual fusion product $(\bbs_\fq(\Vbb),\daleth)$ where $\bbs_\fq(\Vbb)$ is a grading-restricted $\Vbb\otimes\Vbb$-module, and the linear functional $\daleth:\Vbb\otimes \bbs_\fq(\Vbb)\rightarrow\Cbb$ is a conformal block (i.e. $\daleth\in\scr T_\fq^*(\Vbb\otimes\bbs_\fq(\Vbb))$). (In fact, $\bbs_\fq(\Vbb)$ is a subspace of Li's \textbf{regular representation} of $\Vbb$ \cite{Li-regular-rep,LS-twisted-regular-rep}, see Exp. \ref{lb92}.) Recall that $\boxtimes_\fq(\Vbb)$ is the contragredient module of $\bbs_\fq(\Vbb)$. 

Let $\fx\#_{q_1,q_2}\fq$ denote the sewing of $\fx$ and $\fq$ along the pairs of points $(y',0)$ and $(y'',\infty)$ with parameters $q_1,q_2\in\Cbb$ satisfying $q_1q_2=q$. Then $\fx_{q_1,q_2}$ is almost equal to $\mc S_q\fx$ except that $\fx_{q_1,q_2}$ has an extra marked point. Associate $\Vbb$ to this extra point. Then by the propagation of conformal blocks \eqref{eq168} (cf. \cite{Zhu-global,FB04,Cod19,DGT1,Gui-propagation}), we have a canonical isomorphism $\scr T_{\mc S_q\fx}^*(\Wbb)\simeq\scr T_{\fx\#_q\fq}^*(\Wbb\otimes\Vbb)$. Thus, Cor. \ref{lb82} implies the following corollary that will be proved in \cite[Thm. 2.2]{GZ3}.

\begin{co}[\textbf{Sewing-factorization}]\label{lb85}
There is a well-defined isomorphism of vector spaces
\begin{gather}
\begin{gathered}
\fk S_q:\scr T_\fx^*\big(\Wbb\otimes\boxtimes_\fq(\Vbb)\big)\rightarrow\scr T_{\mc S_q\fx}^*(\Wbb)
\end{gathered}
\end{gather}
such that for every $w\in\Wbb$, the following converges absolutely (in an appropriate sense) to a conformal block $\fk S_q\uppsi$ associated to $\mc S_q\fx$ and $\Wbb$:
\begin{align}
\fk S_q\uppsi(w)=\wick{\uppsi(w\otimes \c2 -)\cdot \daleth(\idt\otimes q^{L_1(0)} \c2-)}\label{eq169}
\end{align}
\end{co}

\begin{rem}
If $\Vbb$ is $C_2$-cofinite and rational, using the universal property in Thm. \ref{lb55}, it is not hard to see that
\begin{align}
\bbs_\fq(\Vbb)=\bigoplus_{\Mbb\in\mc E}\Mbb'\otimes\Mbb
\end{align}
where $\mc E$ is a set of representatives of equivalence classes of irreducible $\Vbb$-modules, and 
\begin{gather}
\begin{gathered}
\daleth:\Vbb\otimes\Big(\bigoplus_{\Mbb\in\mc E}\Mbb'\otimes\Mbb\Big)\rightarrow\Cbb\\
v\otimes m'\otimes m\mapsto \<Y(v,1)m',m\> 
\end{gathered}
\end{gather}
satisfies $\daleth(\idt\otimes m'\otimes m)=\bk{m',m}$ if $m\in\Mbb\in\mc E$ and $m'\in\Mbb'$. Then \eqref{eq169} equals \eqref{eq170}.
\end{rem}

We emphasize that the \textbf{factorization formula}
\begin{align}
\boxed{~\dim \scr T_\fx^*\big(\Wbb\otimes\boxtimes_\fq(\Vbb)\big)=\dim \scr T_{\mc S_q\fx}^*(\Wbb)~}
\end{align}
implied by Cor. \ref{lb85} generalizes \eqref{eq165} and computes the dimensions of spaces of conformal blocks of $\mc S_q\fx$ in terms of those of $\fx$.

\subsection{Construction of the dual fusion product $\bbs_\fx(\Wbb)$}
An explicit construction of the dual fusion product $\bbs_\fx(\Wbb)$ is given in this paper. For the reader’s convenience, we briefly summarize it here.

Choose integers $N\geq 1,M\geq 0$. Consider an $(M,N)$-pointed compact Riemann surface with local coordinates
\begin{align}
\fk X=(y_1,\dots,y_M;\theta_1,\dots,\theta_M|C|x_1,\dots,x_N;\eta_1,\dots,\eta_N)
\end{align}
Namely, $x_\blt,y_\star$ are distinct marked points of the compact Riemann surface $C$. We view $x_\blt$ as the incoming points and $y_\star$ as the outgoing points. We assume that each component of $C$ intersects $\{x_1,\dots,x_N\}$. Each $\eta_i$ resp. $
\theta_j$ is a local coordinate of $C$ at $x_i$ resp. $y_j$.

Let $\Vbb$ be $C_2$-cofinite, and let $\Wbb$ be a grading-restricted $\Vbb^{\otimes N}$-module. Write
\begin{align}
Y_i(v,z)=Y(\idt\otimes\cdots\otimes v\otimes\cdots\otimes \idt,z)
\end{align}
where $v\in\Vbb$ is in the $i$-th component of $\idt\otimes\cdots\otimes v\otimes\cdots\otimes \idt$. The goal of this paper is to give an explicit construction of the dual fusion product $\bbs_\fx(\Wbb)$ (which is a grading-restricted $\Vbb^{\otimes M}$-module) and prove Thm. \ref{lb80}. Note that the existence of dual fusion products satisfying Theorem \ref{lb80} can be readily established using the fact that any left exact functor from a finite linear category to the category of finite-dimensional vector spaces is representable (cf. \cite[Cor. 1.10]{DSPS19-balanced}). However, the explicit construction provided in this paper will be essential in the third part of the series for proving the sewing-factorization theorem.

Our definition of $\bbs_\fx(\Wbb)$ as a vector space is due to Kong and Zheng \cite{KZ-conformal-block}: We set
\begin{align*}
\bbs_\fx(\Wbb)=\varinjlim_{a_1,\dots,a_M\in\Nbb}\scr T_{\fx,a_1,\dots,a_M}^*(\Wbb)
\end{align*}
where $(\scr T_{\fx,a_1,\dots,a_M}^*(\Wbb))_{a_1,\dots,a_M\in\Nbb}$ is an increasing system of subspaces of $\Wbb^*$ described as follows. For each grading-restricted $\Vbb^{\otimes M}$-module $\Mbb$ and each $a_1,\dots,a_M\in\Nbb$, if we let
\begin{align*}
\Omega_{a_\star}(\Mbb)=\{m\in \Mbb:Y_j(v)_km=0~\forall 1\leq j\leq M,\text{ homogeneous }v\in\Vbb,k\geq\wt(v)+a_j\}
\end{align*}
then for each conformal block $\upomega:\Wbb\otimes\Mbb\rightarrow\Cbb$ associated to $\fx$ and each $m\in\Omega_{a_\star}(\Mbb)$, $\upomega(-\otimes m):\Wbb\rightarrow\Cbb$ is a linear functional. We expect that all such linear functionals form the space $\scr T_{\fx,a_\star}^*(\Wbb)=\scr T_{\fx,a_1,\dots,a_M}^*(\Wbb)$. So we define $\scr T_{\fx,a_\star}^*(\Wbb)$ to be the set of linear functionals $\upphi:\Wbb\rightarrow\Cbb$ satisfying an ``invariance condition" that is strong enough and is satisfied by all linear functionals of the form $\upomega(-\otimes m)$. Such a linear functional is called a \textbf{partial conformal block of multi-level $a_1,\dots,a_M$} associated to $\fk X$ and $\Wbb$, because it is a conformal block when $M=0$. Roughly speaking, this invariance condition says that the actions of $Y_i(v,z)$ on $\Wbb$, for all $1\leq i\leq N$, can be extended holomorphically to the same holomorphic section on $C-\{x_\blt,y_\star\}$ which has the desired order of poles (determined by $a_j$) at each $y_j$. We refer the readers to Def. \ref{lb16} for the precise definition which involves the sheaf $\scr V_{\fx,a_\star}=\scr V_{\fx,a_1,\dots,a_M}$ on $C$ (cf. Def. \ref{lb86}), a generalization of the vertex algebra bundles in \cite[Ch. 6]{FB04}. 

$\scr V_{\fx,0,\dots,0}$ and $\scr T_{\fx,0,\dots,0}^*(\Wbb)$ were introduced in \cite[Sec. 7.2]{NT-P1_conformal_blocks} and \cite[Sec. 6.2]{DGT2} to study the factorization property for conformal blocks of $C_2$-cofinite and rational VOAs. When $a_1,\dots,a_M$ are not necessarily equal to $0$ but $C$ has genus-$0$, $\scr T_{\fx,a_\star}^*(\Wbb)$, or rather its (pre-)dual space $\scr T_{\fx,a_\star}(\Wbb)$, has appeared much earlier:

\begin{eg}
Let $\zeta$ be the standard coordinate of $\Cbb$. Choose $z\in\Cbb^\times=\Cbb-\{0\}$. Assume that $\fk X$ is the $3$-pointed sphere
\begin{align}
\fk P_z=(\infty;1/\zeta| \Pbb^1|z,0;\zeta-z,\zeta)
\end{align}
with incoming points $z,0$ and outgoing one $\infty$. Choose grading-restricted $\Vbb$-modules $\Wbb_1,\Wbb_2$. Then $\Wbb_1\otimes\Wbb_2$ is a grading-restricted $\Vbb\otimes\Vbb$-module. In the language of the vertex tensor category theory by Huang-Lepowsky \cite{HL-tensor-1,HL-tensor-2,HL-tensor-3,Hua-tensor-4} and Huang-Lepowsky-Zhang \cite{HLZ1,HLZ2}-\cite{HLZ8}, $\bbs_{\fk P_z}(\Wbb_1\otimes \Wbb_2)$ is equal to the $P(z)$ dual fusion (tensor) product $\Wbb_1\bbs_{P(z)}\Wbb_2$ in the category of grading-restricted $\Vbb$-modules. Moreover, $\scr T_{\fk P_z,a}^*(\Wbb_1\otimes\Wbb_2)$ is the set of all $\uplambda\in(\Wbb_1\otimes\Wbb_2)^*$ satisfying the $P(z)$-compatibility condition (cf. \cite[Sec. 5.2]{HLZ4} the paragraph after Rem. 5.33) and satisfying that for each homogeneous $v\in \Vbb$,
\begin{align}
Y_{P(z)}'(v)_k\cdot\uplambda=0\qquad\forall~k\geq\wt(v)+a
\end{align}
where $Y_{P(z)}'(v,z)=\sum_{k\in\Zbb}Y_{P(z)}'(v)_kz^{-k-1}$ is defined in \cite[Def. 5.3]{HLZ4}. (Note that the notation $Y'$ defined in Def. \ref{lb59} and used extensively in Ch. \ref{lb87} has a different meaning and, in particular, satisfies an \emph{upper} truncation property.)
\end{eg}

\begin{eg}\label{lb92}
Let $\fk X$ be $\fq=\eqref{eq171}$ and $\Wbb=\Vbb$. Then each $\scr T_{m,n}^*(\Vbb)$ is the $\Omega_{m,n}$-subspace of Li's regular representation of $\Vbb$ \cite{Li-regular-rep,LS-twisted-regular-rep}. Note that $\bbs_\fq(\Vbb)=\varinjlim_{n\in\Nbb}\scr T_{\fq,n,n}^*(\Vbb)$.  Let 
\begin{gather}
\wtd O_n(\Vbb)=\Span_\Cbb\big\{\Res_{z=0}~z^{-2n-2}Y((1+z)^{L(0)+n}u,z)vdz:u,v\in\Vbb  \big\}
\end{gather}
Then the (pre-)dual space $\scr T_{\fq,n,n}(\Vbb)$ of $\scr T_{\fq,n,n}^*(\Vbb)$, denoted by $\wtd A_n(\Vbb)$, is equal to
\begin{gather}
\scr T_{\fq,n,n}(\Vbb)=\wtd A_n(\Vbb)=\Vbb/\wtd O_n(\Vbb)
\end{gather}
When $n=0$, $\wtd A_0(\Vbb)$ (as a vector space) is equal to the Zhu algebra $A(\Vbb)=A_0(\Vbb)$ introduced by Zhu in \cite{Zhu-modular-invariance}. (That $\scr T_{\fq,0,0}(\Vbb)$ equals $A(\Vbb)$ was shown in \cite[Prop. 7.2.2 and A.2.7]{NT-P1_conformal_blocks}.) In general, we have
\begin{align}
A_n(\Vbb)=\wtd A_n(\Vbb)/\{L(0)v+L(-1)v:v\in\Vbb\}
\end{align}
where $A_n(\Vbb)$ is the level $n$ Zhu algebra introduced by Dong-Li-Mason in \cite{DLM-Zhu}. See 
\cite{Li-regular-AnV,Li-regular-bimodules} for details.
\end{eg}

Next, we explain how to define a weak $\Vbb^{\otimes M}$-module structure on $\bbs_\fx(\Wbb)$. We need to define $Y_j(v)_n=Y(\idt\otimes\cdots\otimes v\otimes\cdots\otimes\idt)_n$ on $\bbs_{\fx}(\Wbb)$ for each $v\in \Vbb$, show that $Y_j$ satisfies the Jacobi identity in the definition of weak VOA modules, and show that $Y_j$ commutes with $Y_k$ if $j\neq k$. This task is one of the most important and non-trivial steps towards our ultimate goal of proving the sewing-factorization theorem. It clearly has its counterparts in the vertex tensor category theory by Huang-Lepowsky and Huang-Lepowsky-Zhang (cf. \cite[Ch. 6]{HLZ4}) and in the theory of regular representations by Li (cf. \cite{Li-regular-rep}). But it also plays a role similar to that of constructing an $A(\Vbb)^{\otimes M}$-module structure on $\scr T_{\fx,0,\dots,0}(\Wbb)$ in the proof of the factorization property for conformal blocks of $C_2$-cofinite and rational VOAs, cf. \cite[Sec. 7.2]{NT-P1_conformal_blocks} and \cite[Sec. 6]{DGT2}. 

\cite{NT-P1_conformal_blocks,DGT2,KZ-conformal-block} all use some universal algebra of $\Vbb$ to treat this problem. In our paper, instead of using any associative algebra of $\Vbb$ (either the universal algebra or the higher level Zhu algebras), we carry out this task by using the (single and double) propagations of partial conformal blocks, which is similar to the methods in \cite{Zhu-global} and especially in \cite{Gui-propagation}. One may view our method as an analytic-geometric and higher-genus version of Huang-Lepowsky-Zhang's approach in \cite{HLZ4} and Li's approach in \cite{Li-regular-rep}. (Note that Nagatomo-Tsuchiya's approach also uses (single and double) propagations. See for example  \cite[Sec. 5.5]{NT-P1_conformal_blocks}, some arguments of which are used in the proof of \cite[Prop. 7.7.2.]{NT-P1_conformal_blocks}.) 

To explain our method, we first consider the case that $\fk X$ is $\fq=\eqref{eq171}$. Then $\Wbb$ is a $\Vbb$-module and is associated to the incoming marked point $1$. We need to define $Y_+(v,z)=Y(v\otimes \idt,z)$ and $Y_-(v,z)=Y(\idt\otimes v,z)$ if $v\in\Vbb$, where the vertex operations $Y_+$ and $Y_-$ are associated to the outgoing marked points $\infty$ and $0$ respectively. Choose any $\upphi\in\bbs_\fq(\Wbb)$, which is a linear functional $\Wbb\rightarrow\Cbb$. Using the strong residue theorem (Thm. \ref{lb88}), one can show that for each $w\in\Wbb$, the formal Laurent series
\begin{align}
\<\upphi, Y_\Wbb(v,z-1)w\> \label{eq172}
\end{align}
can be extended to a holomorphic function $\wr\upphi(v,w)$ on $\Cbb^\times-\{1\}=\Cbb-\{0,1\}$ with finite poles at $0,1,\infty$. Clearly $\wr\upphi(v,w)$ is bilinear with respect to $v,w$. Choose circles $C_-,C_+$ centered at $0$ with radii $<0$ and $>0$ respectively. Then
\begin{gather*}
\<Y_-(v)_n\upphi,w\>=\oint_{C_-}\wr\upphi(v,w)\cdot z^ndz\qquad
\<Y_+'(v)_n\upphi,w\>=\oint_{C_+}\wr\upphi(v,w)\cdot z^ndz
\end{gather*}
where $Y_+'(v,z)=\sum_n Y_+'(v)_nz^{-n-1}=Y_+(e^{zL(1)}(-z^{-2})^{L(0)}v,z^{-1})$. Now, if $v_1,v_2\in\Vbb$, then the expressions
\begin{subequations}\label{eq173}
\begin{gather}
\wr\upphi(v_1,Y_\Wbb(v_2,z_2-1)w)_{z_1}\qquad (\text{when }|z_2-1|\text{ is small})\\
\wr\upphi(Y(v_2,z_2-z_1)v_1,w)_{z_1}\qquad (\text{when }|z_2-z_1|\text{ is small})
\end{gather}
\end{subequations}
can be extended to the same holomorphic function $\wr^2\upphi(v_2,v_1,w)$ on $\Conf^2(\Cbb^\times-\{1\})=\{(z_1,z_2):z_1,z_2\in\Cbb^\times-\{1\},z_1\neq z_2\}$. By calculating some contour integrals of $\wr^2\upphi(v_2,v_1,w)$, one can show that $Y_\pm(v)_n\upphi$ belongs to $\bbs_\fq(\Wbb)$, that $Y_-$ and $Y_+$ satisfy the Jacobi identity, and that $Y_-$ commutes with $Y_+$.

For a general $\fk X$, the idea is the same, except that one needs more effort to prove that \eqref{eq172} can be extended to a global holomorphic section of a suitable holomorphic vector bundle (of possibly infinite rank) on $C-\{x_\blt,y_\star\}$, and that \eqref{eq173} can be extended to a global holomorphic section on $\Conf^2(C-\{x_\blt,y_\star\})$. These processes are called the \textbf{propagations of partial conformal blocks} and will be studied systematically in Ch. \ref{lb89}.

We end this introduction with a few remarks.

\begin{rem}
In this section, we have assumed that $\Vbb$ is $C_2$-cofinite for simplicity. But $\bbs_\fx(\Wbb)$ can be constructed without assuming that $\Vbb$ is $C_2$-cofinite. (Therefore, our geometric interpretation of $A_n(\Vbb)$ is not restricted to $C_2$-cofinite VOAs.) In the main body of this paper, we only assume that $\Vbb=\bigoplus_{n\in\Nbb}\Vbb(n)$ with $\dim \Vbb(n)<+\infty$. We consider a vector space $\Wbb$ with mutually commuting vertex operations $Y_1,\dots,Y_N$ such that $(\Wbb,Y_i)$ is a weak $\Vbb$-module for each $i$. Then $(\Wbb, Y_1,\dots,Y_N)$ is called a \textbf{weak $\Vbb^{\times N}$-module}. Similar to admissible (i.e. $\Nbb$-gradable) $\Vbb$-modules, an admissible (i.e. $\Nbb^N$-gradable) $\Vbb^{\times N}$-module is a weak $\Vbb^{\times N}$-module $\Wbb$ with $\Nbb^N$-grading $\Wbb=\bigoplus_{n_1,\dots,n_N\in\Nbb}\Wbb(n_1,\dots,n_N)$ compatible with the vertex operations $Y_1,\dots,Y_N$. See Def. \ref{grading3}. If a grading can be chosen such that each $\Wbb(n_1,\dots,n_N)$ has finite dimension, we call $\Wbb$ a \textbf{finitely-admissible $\Vbb^{\times N}$-module}. We shall study the propagation of partial conformal blocks for finitely-admissible $\Vbb^{\times N}$-modules. This finite-dimension condition allows one to sew conformal blocks. (As in \cite{Gui-propagation}, we understand propagation as a sewing construction (cf. Fig. \ref{fig1}) followed by an analytic continuation.) And we shall prove that $\bbs_\fx(\Wbb)$ is a weak $\Vbb^{\times M}$-module. 
\end{rem}

\begin{rem}
We have fixed local coordinates $\eta_\blt$ and $\theta_\star$ at the incoming points $x_\blt$ and the outgoing points $y_\star$ respectively. But $\bbs_\fx(\Wbb)$ can be realized in a coordinate-free way with the help of Huang's change-of-coordinate formulas \cite{Hua97}. The independence of $\theta_\star$ can be realized easily using the universal property in Thm. \ref{lb80}. Thus, we shall fix $\theta_\star$ and define $\bbs_\fx(\Wbb)$ as a set of linear functionals on $\scr W_\fx(\Wbb)$, a coordinate-free version of $\Wbb$. $\scr W_\fx(\Wbb)$ is a vector bundle over a single point such that for each choice of $\eta_\blt$ we have a trivialization $\mc U(\eta_\blt):\scr W_\fx(\Wbb)\xrightarrow{\simeq}\Wbb$, and that the transition functions are given by the exponentials of certain Virasoro operators. See Def. \ref{lb8} for details. The readers should notice that $\scr W_\fx(\Wbb)$ \emph{relies on the choice of grading of $\Wbb$}.
\end{rem}

The task of each section is already indicated by its title. In particular, in Ch. \ref{Ch1} we review the basic properties of conformal blocks and generalize some of the results to the setting considered in this paper. In Ch. \ref{lb89}, we establish the propagation of partial conformal blocks, which, as mentioned earlier, is used crucially in the construction of the dual fusion products in Ch. \ref{Ch3}. A technical subtlety concerning modules for products of $\Vbb$ will be addressed in Appendix Ch. \ref{lb90}. A geometric construction of higher level Zhu algebras using dual fusion products will be discussed in Appendix Ch. \ref{lb87}.

In part II of this series, we will study the connections on sheaves of conformal blocks associated to $C_2$-cofinite VOAs and holomorphic families of compact Riemann surfaces. We will address several convergence issues concerning conformal blocks. In part III we will prove the sewing-factorization theorem.

\subsection*{Acknowledgment}

We are grateful to Liang Kong and Hao Zheng for many enlightening conversations. In particular, we owe to them the definition of the vector space $\bbs_\fx(\Wbb)$ (the dual fusion product). We would also like to thank 
Chiara Damiolini, Angela Gibney, Yi-Zhi Huang, Haisheng Li, and Robert McRae for helpful discussions. B.G. is supported by NSFC Grant 12401159.

\begin{subappendices}

\subsection{Notations}
\begin{itemize}
    \item $\Nbb=\{n\in \Zbb:n\geq 0\}$, $\Zbb_+=\{n\in \Zbb:n\geq 1\}$.
    \item $\delta_{i,j}$ is the Kronecker symbol, which means $\delta_{i,j}=0$ if $i\ne j$ and $\delta_{i,j}=1$ if $i=j$.
    \item All neighborhoods are open. The closure of a subset $E$ is denoted by $E^\cl$. 
    \item All vector spaces are over $\Cbb$.
    \item If $X$ is a set and $Y\subset X$ is a subset, then $X-Y$ denotes $\{x\in X:x\notin Y\}$. And \index{Conf@$\Conf^n(X)$}
\begin{align}
\Conf^n(X)=\{(x_1\dots,x_n)\in X^n:x_i\neq x_j\text{ if }i\neq j\}  \label{eq69}
\end{align}
    \item If $r>0$, then $\MD_r:=\{z\in \Cbb:\vert z\vert <r\}$ is the open disc with radius $r$ and $\MD_r^\times :=\MD_r-\{0\}$. When there are several discs,  we write \index{D@$\mc D_r,\mc D_r^\times,\mc D_{r_\blt},\mc D_{r_\blt}^\times$}
\begin{align}
\mc D_{r_\blt}=\mc D_{r_1}\times\cdots\times\mc D_{r_N}\qquad \mc D_{r_\blt}^\times=\mc D_{r_1}^\times\times\cdots\times\mc D_{r_N}^\times
\end{align}
 \item If $\MF$ is a sheaf on the topological space $X$, then $\MF_x$ denotes the stalk of $\MF$ at $x\in X$ and $H^q(X,\MF)$ denotes the $q$-th sheaf cohomology group of $X$. In particular, $H^0(X,\mc F)=\mc F(X)$.
    \item If $X$ is a complex manifold, then $\MO_X$ denotes the structure sheaf of $X$. $\MO_{X,x}$ denotes the stalk of $\MO_X$ at $x$ and $\fm_{X,x}=\{f\in\mc O_{X,x}:f(x)=0\}$ is the maximal ideal of $\MO_{X,x}$. Then
\begin{align*}
(\mc O_{X,x}/\fk m_{X,x})\simeq\Cbb.
\end{align*}
If $\MF$ is an $\MO_X$-module and $x\in X$ then $\mc F_x$ \index{Fx@$\mc F_x$, the stalk of $\mc F$ at $x$} \index{Fx@$\mc F\lvert_x=\mc F\lvert x=\mc F_x/{\fk m_{X,x}\cdot\mc F_x}$}
\begin{align}\label{eq6}
    \MF\vert_x=\frac{\MF_x}{\fm_{X,x}\cdot \MF_x}\simeq\mc F_x\otimes_{\mc O_{X_x}}(\mc O_{X,x}/\fk m_{X,x})
\end{align}
    is the fiber of $\MF$ at $x$, which is a vector space. The residue class of $s\in\mc F$ in $\mc F|_x$ is denoted by $s(x)$ or $s|_x$:
\begin{align}
s(x)\equiv s|_x\qquad\in\mc F|_x  \label{eq9}
\end{align}
Equivalently,
\begin{align}
s(x)=s\otimes 1\in \mc F_x\otimes(\mc O_{X,x}/\fk m_{X,x})
\end{align}

    \item If $\ME$ is an $\mc O_X$-module (for example, a holomorphic vector bundle of finite rank on a complex manifold $X$), then $\ME^\vee$ and $\ME^*=\shom_{\mc O_X}(\mc E,\mc O_X)$ both denote the \textbf{dual sheaf} of $\ME$, which is the dual bundle of $\mc E$ when $\mc E$ is a vector bundle. Then for each open $U\subset X$,
\begin{align}
\phi\in\mc E^*(U)\qquad\Longleftrightarrow \qquad\phi:\mc E|_U\rightarrow \mc O_U\text{ is an $\mc O_U$-module morphism}   \label{eq18}
\end{align}
Similarly, if $V$ is a vector space, then $V^\vee$ and $V^*$ both denote the dual vector space of $V$.

If $\mc E$ and $\mc F$ are $\mc O_X$-modules, then $\mc E\otimes_{\mc O_X}\mc F$ is often written as $\mc E\otimes\mc F$ for short. (Note that when $\mc E,\mc F$ are vector bundles, then $\mc E\otimes \mc F$ is their tensor product bundle.)

\item Let $\mc E$ be an $\mc O_X$-module. Let $\phi\in\mc E^*(X)$, i.e. $\phi$ is an $\mc O_X$-morphism $\mc E\rightarrow\mc O_X$. For each $x\in X$, let
\begin{align}\label{eq19}
\phi(x)\equiv\phi|_x:=\phi\otimes\idt: \mc E_x\otimes (\mc O_{X,x}/\mk_{X,x})\rightarrow \mc O_{X,x}\otimes (\mc O_{X,x}/\mk_{X,x})\simeq\Cbb
\end{align}
Thus $\phi(x)$ is a linear functional on $\mc E|_x$. Equivalently, $\phi(x)$ is defined by descending $\phi:\mc E_x\rightarrow\mc O_{X,x}$ to $\mc E_x/{\mk_{X,x}\mc E_x}\rightarrow\Cbb$.

    \item Suppose that $S$ is a closed submanifold of $X$ with codimension $1$. Then for each $k\in \Zbb$, $\mc O_X(kS)$ is the $\mc O_X$-submodule of $\MO\vert_{X- S}$ consisting of sections of $\MO\vert_{X- S}$ with poles of order $\leq k$ at $S$. Denote 
    \begin{align*}
        \MO_X(\blt S):=\varinjlim_{k\in \Nbb}\MO_X(k S).
    \end{align*}
If $\ME$ is an $\MO_X$-module, we set \index{ES@$\mc E(kS),\mc E(\blt S)$}
    \begin{gather*}
\mc E(kS)=\mc E\otimes\mc O_X(kS)\qquad \mc E(\blt S)=\varinjlim_{k\in \Nbb}\ME(k S)
    \end{gather*}
    If $\mc E$ is locally free (i.e. a vector bundle), then the sections of $\mc E(\blt S)$ (resp. $\mc E(k S)$) can be viewed as sections of $\mc E|_{X- S}$ with finite poles (resp. with poles of orders at most $k$) at $S$.
    
    \item If $\pi:X\rightarrow Y$ is a holomorphic map between complex manifolds and $\ME$ is an $\MO_X$-module, then $\pi_*(\ME)$ denotes the pushforward of $\ME$. If $\MF$ is an $\MO_Y$-module, then $\pi^*(\MF)=\MF\otimes_{\MO_Y}\MO_X$ denotes the pullback of $\MF$. Moreover, suppose $\MF$ is a holomorphic vector bundle over $Y$, whose trivialization on $U\subset Y$ is 
    $$
    f:\MF\vert_U\xrightarrow{\simeq} F\otimes_\Cbb \MO_U,
    $$
    where $F$ is a vector space. Then the pullback $\pi^*\MF$ has a natural vector bundle structure, whose trivialization on $\pi^{-1}(U)\subset Y$ is 
    $$
    \pi^*f:\pi^*(\MF)\vert_{\pi^{-1}(U)}\xrightarrow{\simeq} F\otimes \MO_{\pi^{-1}(U)}.
    $$
    \item Suppose $W$ is a vector space and $z$ is a formal variable. Then 
    $$
    \begin{aligned}
    W[z]&:=\{\sum_{n=0}^N w_n z^n:w_n\in W,n\in \Nbb\}\\
    W[[z]]&:=\{\sum_{n\in \Nbb}w_nz^n:w_n\in W\}\\
    W((z))&:=\{\sum_{n=-N}^{\infty}w_nz^n:w_n\in W,N\in \Nbb\}\\
    W[[z^{\pm 1}]]&:=\{\sum_{n\in \Zbb}w_nz^n:w_n\in W\}\\
    W((z,w))&:=\{\sum_{k,l\geq N}a_{k,l}z^k w^l:a_{k,l}\in W,N\in \Zbb\}
    \end{aligned}
    $$
    For each $w=\sum_{n\in \Zbb}w_nz^n\in W[[z^{\pm 1}]]$, 
    $$
    \Res_{z=0}wdz:=w_{-1}.
    $$
    If $\Wbb$ is a commutative ring, then so is $W((z,w))$.
\item We use frequently the symbol \index{n@$n_\blt=(n_1,\dots,n_N).$}
\begin{align}\label{eq33}
n_\blt=(n_1,\dots,n_N).
\end{align}
\item Let $X$ be a complex manifold. Choose a formal power series 
\begin{align*}
f=\sum_{n_\blt\in\Nbb^N}a_{n_\blt}\cdot z_1^{n_1}\cdots z_N^{n_N}\qquad\in\scr O(X)[[z_1,\dots,z_N]]
\end{align*}
where each $a_{\blt}\in\scr O(X)$. Let $\Omega$ be an open subset of $\Cbb^N$. We say that \index{00@Converging absolutely and locally uniformly (a.l.u.)}
\begin{align}\label{eq34}
f\textbf{ converges absolutely and locally uniformly (a.l.u.) on } X\times\Omega
\end{align}
if for every compact subsets $K\subset X$ and $\Gamma\subset \Omega$ we have
\begin{align*}
\sup_{x\in K,z_\blt\in\Gamma}\sum_{n_\blt\in\Nbb^N} |a_{n_\blt}(x)|\cdot |z_1^{n_1}\cdots z_N^{n_N}|<+\infty
\end{align*}
\end{itemize}

\end{subappendices}

\section{Vertex operator algebras and conformal blocks}
\label{Ch1}
\subsection{Vertex operator algebras and their modules}
In this article, unless otherwise stated, we assume that a vertex operator algebra (VOA) has $\Nbb$-grading $\Vbb=\bigoplus_{n\in \Nbb}\Vbb(n)$ and $\dim \Vbb(n)<+\infty$. The \textbf{vacuum vector} is denoted by $\idt$, and the \textbf{conformal vector} \index{c@$\cbf$, the conformal vector} is denoted by $\cbf$. The vertex operator is written as $Y(v,z)=\sum_{n\in\Zbb}Y(v)_nz^{-n-1}$. The Virasoro operators are $L(n)=Y(\cbf)_{n+1}$. Recall that $\Vbb$ is called \textbf{$C_2$-cofinite} \index{00@$C_2$-cofinite VOAs}   if $\dim \Vbb/C_2(\Vbb)<\infty$, where $C_2(\Vbb)=\Span\{Y(u)_{-2}v:u,v\in \Vbb\}$.

We recall the notions of weak modules and admissible modules. Let $\Wbb$ be a vector space over $\Cbb$ with a linear map 
$$
\begin{aligned}
    \Vbb&\rightarrow (\End\Wbb)[[z^{\pm 1}]]\\
    u&\mapsto Y_\Wbb(u,z)=\sum_{n\in \Zbb}Y_\Wbb(u)_nz^{-n-1}.
\end{aligned}
$$
\begin{df}\label{finitelyadmissible}
We say $(\Wbb,Y_\Wbb)$ is a \textbf{weak $\Vbb$-module} if \index{00@Weak $\Vbb$-modules, finitely admissible $\Vbb$-modules} it satisfies:
        \begin{enumerate}
            \item [(a)] \textbf{Lower truncation:} For each $u\in \Vbb,w\in \Wbb$, $Y_\Wbb(u,z)w\in \Wbb((z))$.
            \item [(b)] \textbf{Vacuum:} $Y_\Wbb(\ibf,z)=\ibf_\Wbb$.
            \item [(c)] \textbf{Jacobi identity:} For any $u,v\in \Vbb$ and $m,n,h\in \Zbb$, 
        \begin{equation}\label{jacobi}
        \begin{aligned}
        &\sum_{l\in \Nbb}\binom{m}{l}Y_\Wbb(Y(u)_{n+l}\cdot v)_{m+h-l}\\
        =&\sum_{l\in \Nbb}(-1)^l \binom{n}{l}Y_\Wbb(u)_{m+n-l}Y_\Wbb(v)_{h+l}-\sum_{l\in \Nbb}(-1)^{l+n}\binom{n}{l}Y_\Wbb(v)_{n+h-l}Y_\Wbb(u)_{m+l}.
        \end{aligned}
        \end{equation}
        \end{enumerate}
        Set $L(n)=Y_\Wbb(\cbf)_{n+1}$ for a weak $\Vbb$-module $\Wbb$.

A weak $\Vbb$-module $\Wbb$ is called an \textbf{admissible module} if there exists a diagonalizable operator $\widetilde{L}(0)$ on $\Wbb$ with eigenvalues in $\Nbb$, satisfying the grading property
\begin{align}
[\widetilde{L}(0),Y_\Wbb(u)_n]=Y_\Wbb(L(0)u)_n-(n+1)Y_\Wbb(u)_n   \label{eq150}
\end{align}
We fix grading $\Wbb=\bigoplus_{n\in\Nbb}\Wbb(n)$ where
\begin{align}
\Wbb(n)=\{w\in\Wbb:\wtd L(0)w=nw\}
\end{align}
Moreover, if each eigenspace of $\widetilde{L}(0)$ is finite dimensional, then $\Wbb$ is called a \textbf{finitely-admissible module}.

\end{df}

\begin{df}
    Given an admissible $\Vbb$-module $\Wbb=\bigoplus_{n\in \Nbb} \Wbb(n)$, one can define its contragredient module $\Wbb^\prime$\index{00@Contragredient modules} as follows. As a graded vector space,
    \begin{equation}\label{grading2}
    \Wbb^\prime=\bigoplus_{n\in \Nbb}\Wbb(n)^*.
    \end{equation}
    The vertex operator $Y_{\Wbb^\prime}$ is defined by 
    \begin{align}\label{contra1}
    \<Y_{\Wbb^\prime}(v,z)w^\prime,w\>=\<w^\prime ,Y_\Wbb(e^{zL(1)}(-z^{-2})^{L(0)}v,z^{-1})w\>
    \end{align}
    for each $v\in \Vbb,w\in \Wbb,w^\prime\in \Wbb^\prime$.    Define $\wtd{L}(0)$ on $\Wbb^{\prime}$ by setting $\wtd{L}(0)w=nw$ for $w\in \Wbb(n)^*$. Then $\Wbb^\prime$ becomes an admissible $\Vbb$-module. We call $(\Wbb^\prime,Y_{\Wbb^\prime})$ the \textbf{contragredient module} of $\Wbb$.
\end{df}
 
 If $\Wbb$ is furthermore finitely admissible, then so is $\Wbb^\prime$.

\begin{df}\label{lb59}
    Suppose $(\Wbb,Y_\Wbb)$ is a weak $\Vbb$-module. Define $Y_\Wbb^\prime(v,z)\in\End(\Wbb)[[z^{\pm1}]]$ and $Y_\Wbb'(v)_k\in\End(\Wbb)$ by \index{Y@$Y_\Wbb^\prime(v,z),Y_\Wbb^\prime(v)_k$}
    \begin{gather}
        Y_\Wbb^\prime(v,z)=\sum_{n\in\Zbb}Y_\Wbb'(v)_n\cdot z^{-n-1}=Y_\Wbb\big(e^{zL(1)}(-z^{-2})^{L(0)}v,z^{-1}\big)  \label{eq138}
    \end{gather}
If $v$ is homogeneous, it is easy to compute that
\begin{align}
    Y_\Wbb'(v)_n =\sum_{k\in \Nbb} \frac{(-1)^{\wt(v)}}{k!}Y_\Wbb\big(L(1)^k v\big)_{-n-k-2+2\wt(v)}   \label{eq139}
\end{align}
By \eqref{contra1}, if $\Wbb$ is admissible, then $Y_{\Wbb'}'(v)_n$ is the transpose of $Y_\Wbb(v)_n$.
\end{df}

\subsection{Admissible $\Vbb_1\times \cdots \times \Vbb_N$-modules}

Let $\Vbb_1,\cdots,\Vbb_N$ be VOAs. In this section, we introduce the notion of finitely-admissible $\Vbb_1\times\cdots\times\Vbb_N$-modules, which is convenient for proving many analytic properties of conformal blocks. See Thm. \ref{lb46} for a relationship between such modules and grading-restricted generalized $\Vbb_1\otimes\cdots\otimes\Vbb_N$-modules when $\Vbb_1,\dots,\Vbb_N$ are $C_2$-cofinite.

\begin{df}\label{lb33}
Suppose $\Wbb$ is a weak $\Vbb_i$-module with vertex operator $Y_{\Wbb,i}$ (or $Y_i$ for short if the context is clear) for each $1\leq i\leq N$. \index{YW@$Y_{\Wbb,i}=Y_i$} Moreover, if $v\in\Vbb_i$, we set 
$$
Y_{\Wbb,i}(v,z)=\sum_{n\in \Zbb}Y_{\Wbb,i}(v)_n z^{-n-1}
$$
and $L_i(n):=Y_i(\cbf)_{n+1}$. $\Wbb$ is called a \textbf{weak $\Vbb_1\times \cdots \times \Vbb_N$-module} if $[Y_i(u)_m,Y_j(v)_n]=0$ for all $1\leq i\ne j\leq N$ and $m,n\in \Zbb,u\in \Vbb_i,v\in \Vbb_j$.\index{00@Weak $\Vbb_1\times \cdots \times \Vbb_N$-modules} In particular, if $i\ne j$, then $L_i(m)$ and $L_j(n)$ commute with each other.

If $\Mbb$ is also a weak $\Vbb_1\times\cdots\times\Vbb_N$-module, then a linear map $T:\Wbb\rightarrow\Mbb$ intertwining the actions of $\Vbb_1,\dots,\Vbb_N$ is called a \textbf{(homo)morphism of weak $\Vbb_1\times\cdots\times\Vbb_N$-modules}. All such maps form a vector space $\Hom_{\Vbb_1\times\cdots\times\Vbb_N}(\Wbb,\Mbb)$. \index{Hom@$\Hom_{\Vbb_1\times\cdots\times\Vbb_N}(\Wbb,\Mbb)$}
\end{df}

\begin{rem}\label{lb91}
A weak $\Vbb_1\otimes\cdots\otimes \Vbb_N$-module (where $\Vbb_1\otimes\cdots\otimes \Vbb_N$ is viewed as a single VOA) is obviously a weak
$\Vbb_1\times\cdots\times\Vbb_N$-module. A version of the converse can also be proved by checking the weak associativity \cite{LL-introduction}.\footnote{We are grateful to Haisheng Li for informing us of this fact.} We will discuss this in more details in Sec. \ref{lb34}.
\end{rem}

\begin{df}\label{lb15}
If $\Mbb$ is a weak $\Vbb_1\times\cdots\times\Vbb_N$-module and $\Ebb\subset \Mbb$ is a subset, then we say \textbf{$\Ebb$ generates $\Mbb$}\index{00@Generating subsets of weak $\Vbb^{\times N}$-modules} if $\Mbb$ is spanned by
\begin{align} \label{eq60}
\begin{aligned}
\{Y_{j_1}(u_1)_{n_1}\cdots Y_{j_k}(u_k)_{n_k}w: &k\in\Zbb_+,1\leq j_1,\cdots ,j_k\leq N,\\
&u_1\in\Vbb_{j_1},\dots,u_k\in\Vbb_{j_k},n_1,\cdots,n_k\in \Zbb,w\in \Ebb\}  
\end{aligned}
\end{align}
Moreover, if $\Ebb$ is finite, we say that  $\Mbb$ is \index{00@Finitely generated weak $\Vbb^{\times N}$-modules} \textbf{finitely generated}. 
\end{df}

\begin{df}\label{grading3}

Let $\Vbb_1,\cdots,\Vbb_N$ be VOAs and $\Wbb$ be a weak $\Vbb_1\times \cdots \times \Vbb_N$-module. We say that $\Wbb$ is an \textbf{admissible $\Vbb_1\times \cdots \times \Vbb_N$-module} if there exist simultaneously diagonalizable (and hence commuting) operators $\widetilde{L}_j(0)(1\leq j\leq N)$ on $\Wbb$ with eigenvalues in $\Nbb$ such that for each $1\leq i,j\leq N$ we have
\begin{subequations}
\begin{gather}
[\widetilde{L}_j(0),Y_i(v)_n]=\delta_{i,j}\big((Y_i(L(0)v)_n-(n+1)Y_i(v)_n\big)  \label{eq99}
\end{gather}
Then we have
\begin{gather}
[\wtd L(0),Y_i(v)_n]=Y_i(L(0)v)_n-(n+1)Y_i(v)_n\label{eq115}
\end{gather}
\end{subequations}
if we define \index{L0@$\wtd L_j(0),\widetilde{L}(0):=\widetilde{L}_1(0)+\cdots+\widetilde{L}_N(0) $}
\begin{align} 
\widetilde{L}(0):=\widetilde{L}_1(0)+\cdots+\widetilde{L}_N(0) \label{eq83}
\end{align}
Set \index{Wn@$\Wbb(n_\blt)=\Wbb(n_1,\dots,n_N)$} \index{Wn@$W(n)=\bigoplus_{n_1+\cdots+n_N=n}\Wbb(n_1,\dots,n_N)$} \index{W@$\Wbb^{\leq n}=\bigoplus_{k\leq n}\Wbb(k)$}
\begin{subequations}\label{eq32}
\begin{gather}
\Wbb(n_\blt)\equiv\Wbb(n_1,\dots,n_N)=\{w\in \Wbb:\wtd{L}_j(0)w=n_jw\quad(\forall 1\leq j\leq N)\}\\
\Wbb(n)=\bigoplus_{n_1+\cdots+n_N=n}\Wbb(n_1,\dots,n_N)=\{w\in \Wbb:\wtd L(0)w=nw\}  \label{eq114}\\
\Wbb^{\leq n}=\bigoplus_{k\leq n}\Wbb(k),\qquad \Vbb^{\leq n}=\bigoplus_{k\leq n}\Vbb(k)\label{eq127}
\end{gather}
\end{subequations}
Then we have $\Nbb$-grading and $\Nbb^N$-grading
\begin{align}
\Wbb=\bigoplus_{n\in\Nbb}\Wbb(n)=\bigoplus_{n_\blt\in\Nbb^N}\Wbb(n_\blt)
\end{align}
A vector $w\in\Wbb$ is called \textbf{$\wtd L_\blt(0)$-homogeneous} (or simply \textbf{homogeneous}) \index{00@Homogeneous, $\wtd L_\blt(0)$-homogeneous vectors} if it belongs to $\Wbb(n_\blt)$ for some $n_\blt\in\Nbb^N$. If $w\in\Wbb$, we write \index{wt@$\wtd\wt_j(w),\wtd\wt(w)$}
\begin{subequations}\label{eq113}
\begin{gather}
\wtd \wt_j(w)=n_j\qquad \text{if }\wtd L_j(0)w=n_jw\\
\wtd \wt(w)=n\qquad \text{if }\wtd L(0)w=nw
\end{gather}
\end{subequations}
\end{df}

\begin{df}
Let $\Wbb$ be an admissible $\Vbb_1\times\cdots\times\Vbb_N$-module. If each $\Wbb(n)$ is finite dimensional (equivalently, if each $\Wbb(n_1,\dots,n_N)$ is finite dimensional), then $\Wbb$ is called a \textbf{finitely admissible $\Vbb_1\times \cdots \times \Vbb_N$-module}.\index{00@(Finitely) admissible $\Vbb_1\times \cdots \times \Vbb_N$-modules}
\end{df}

\begin{eg}\label{lb47}
\begin{enumerate}[label=(\arabic*)]
    \item An admissible (resp. finitely admissible) $\Vbb^{\times 1}$-module is equivalent to an admissible (resp. finitely admissible) $\Vbb$-module defined in Definition \ref{finitelyadmissible}.
    \item Suppose $\Wbb_i$ is an admissible (resp. finitely admissible) $\Vbb_i$-module for each $1\leq i\leq N$. Then $\Wbb_\blt=\Wbb_1\otimes \cdots \otimes \Wbb_N$ is naturally an admissible (resp. finitely admissible) $\Vbb_1\times \cdots \times \Vbb_N$-module if we define $\wtd L_i(0)$ on $\Wbb_\blt$ to be $\idt\otimes\cdots\otimes \wtd L_i(0)\otimes\cdots\otimes\idt$ where $\wtd L_i(0)$ is at the $i$-th component.   
\item Let $\Wbb$ be an admissible  $\Vbb_1\times \cdots \times \Vbb_N$-module. Define a graded vector space\index{W@$\Wbb'$}
\begin{gather}
\Wbb'=\bigoplus_{n\in \Nbb} \Wbb(n)^*=\bigoplus_{n_1,\dots,n_N}\Wbb(n_1,\dots,n_N)^*
\end{gather}
For each $1\leq j\leq N$, define $Y_{\Wbb',j}:\Vbb_j\otimes\Wbb'\rightarrow\Wbb'((z))$ by
\begin{gather}\label{eq39}
    \<Y_{\Wbb^\prime,j}(v_j,z)w^\prime,w\>=\<w^\prime ,Y_{\Wbb,j}(e^{zL(1)}(-z^{-2})^{L(0)}v_j,z^{-1})w\>
\end{gather}
for each $v\in\Vbb_j,w\in\Wbb,w'\in\Wbb'$. Then $\Wbb'$, together with all $Y_{\Wbb',j}$ and all $\wtd L_j(0)$ (where $1\leq j\leq N)$, is an admissible $\Vbb_1\times\cdots\times\Vbb_N$-module if the $\wtd L_j(0)$ acting on each $\Wbb(n)^*$ is defined to be the transpose of $\wtd L_j(0)$ acting on $\Wbb(n)$. We call $\Wbb'$ the \textbf{contragredient module} of $\Wbb$.
\end{enumerate}
\end{eg}

\begin{cv}
Unless otherwise stated, the grading operators of $\Wbb_1\otimes\cdots\otimes\Wbb_N$ and $\Wbb'$ are always defined as in Exp. \ref{lb47}. The grading operator of $\Vbb$ is set to be $\wtd L(0)=L(0)$.
\end{cv}

\begin{rem}\label{lb48}
Suppose that $\Wbb=\bigoplus_{k=1}^\infty \Wbb_k$ where each $\Wbb_k$ is a finitely-admissible $\Vbb_1\times\cdots\times\Vbb_N$-module with grading operators $\wtd L^k_1(0),\dots,\wtd L^k_N(0)$, then $\Wbb$ is a finitely admissible $\Vbb_1\times\cdots\times\Vbb_N$-module since we can define its grading operators $\wtd L_1(0),\dots,\wtd L_N(0)$ to be
\begin{align}
\wtd L_i(0)w=(k+\wtd L_i^k(0))w\qquad\text{if }w\in \Wbb_k
\end{align}
\end{rem}

\subsection{Change of coordinate}
In order to give the definition of conformal blocks, we recall the change of coordinate formula discovered by \cite{Hua97}. Let $\MO_\Cbb$ be the structure sheaf of complex plane $\Cbb$ and $\MO_{\Cbb,0}$ be its stalk at 0. Define $\Gbb=\{\rho\in \MO_{\Cbb,0}:\rho(0)=0,\rho^\prime(0)\ne 0\}$. Then $\Gbb$ becomes a group if we define the multiplication $\rho_1\cdot \rho_2$ to be their composition $\rho_1\circ \rho_2$. \index{G@$\Gbb$}

Notice that for each $\rho\in \Gbb$, there exist unique $c_1,c_2,\cdots\in \Cbb$, such that 
    \begin{equation}\label{coordinatechange1}
    \rho(z)=\rho'(0)\cdot \exp\big(\sum_{n>0}c_n z^{n+1}\partial_z\big)z.
    \end{equation}
Now, choose an admissible $\Vbb^{\times N}$-module $\Wbb$. Define $\MU_j(\rho)\in \End(\Wbb)(1\leq j\leq N)$ to be \index{U@$\MU_j(\rho)$, $\MU(\rho)$, the change of coordinate operator}
\begin{equation}\label{coordinatechange3}
\MU_j(\rho)=\rho^\prime(0)^{\wtd{L}_j(0)}\cdot \exp\big(\sum_{n>0}c_n L_j(n)\big)
\end{equation}
which is a finite sum when acting on each vector $w\in \Wbb$. Since 
$$
[\wtd{L}_j(0),L_j(n)]=-nL_j(n),
$$
$L_j(n)$ lowers $\wtd{L}_j(0)$-weights, and hence lowers $\wtd{L}(0)$-weights. So (\ref{coordinatechange3}) actually defines $\MU_j(\rho)\in \End(\Wbb^{\leq n})$. Moreover, $\MU_i(\rho_1)$ commutes with $\MU_j(\rho_2)$ for $i\ne j$ and $\rho_1,\rho_2\in \Gbb$, since each operator in $\{\widetilde{L}_i(0),L_i(n):n\in \Zbb\}$ commutes with each one in $\{\widetilde{L}_j(0),L_j(n):n\in \Zbb\}$ if $i\ne j$.

\begin{thm}\label{coordinatechange5}
For a fixed admissible $\Vbb^{\times N}$-module $\Wbb$, $\MU_j$ gives a representation of $\Gbb$ on $\Wbb$ for each $1\leq j\leq N$, i.e., $\MU_j(\rho_1\circ \rho_2)=\MU_j(\rho_1)\MU_j(\rho_2)$ for $\rho_1,\rho_2\in \Gbb$. In particular, $\MU_j(\rho)$ is invertible for each $\rho\in \Gbb$.
\end{thm}

\begin{proof}
One may view $\Wbb$ as an admissible $\Vbb$-module defined by $Y_{\Wbb,i}$. Then this theorem follows from \cite[Sec. 4.2]{Hua97}. See also \cite[Ch. 6]{FB04} or \cite[Sec. 10]{GuiLec}.
\end{proof}

Suppose $X$ is a complex manifold and $\rho:X\rightarrow \Gbb,x\mapsto \rho_x$ is a function. 
\begin{df}
    $\rho$ is called a \textbf{holomorphic family of transformations} \index{00@holomorphic family of transformation}if for each $x\in X$, there exists a neighborhood $V$ of $x$ and a neighborhood $U$ of 0, such that $(z,y)\in U\times V\mapsto \rho_y(z)$ is a holomorphic function on $U\times V$.
\end{df}
If $\rho$ is a holomorphic family of transformations, then the coefficients $c_0,c_1,c_2,\cdots$ given by \eqref{coordinatechange1} depend holomorphically on $x\in X$. Therefore, (\ref{coordinatechange3}) gives an isomorphism of $\MO_X$-modules
\begin{equation}\label{coordinatechange4}
\MU_j(\rho):\Wbb^{\leq n}\otimes_\Cbb \MO_X\xrightarrow{\simeq} \Wbb^{\leq n}\otimes_\Cbb \MO_X
\end{equation}
sending each $\Wbb^{\leq n}$-valued function $w$ to the section $x\mapsto \MU_j(\rho_x)w(x)$. When $N=1$ so that $\Wbb$ is an admissible $\Vbb$-module, we write $\mc U_1(\rho)$ as $\mc U(\rho)$.

\begin{eg}\label{changeexample1}
    Suppose $\Wbb$ is an admissible $\Vbb$-module, $X=\Cbb^\times$ and $z\in \Cbb^\times$. Set $\upgamma_z$ as \index{zz@$\upgamma_z$}
    \begin{align*}
        \upgamma_z(t)=\frac{1}{z+t}-\frac{1}{z}.
    \end{align*}
    Then $\upgamma=\upgamma_z$ is a holomorphic family of transformations and it is easy to compute \index{U@$\MU(\upgamma_z)$}
    \begin{align}
        \MU(\upgamma_z)=e^{zL(1)}(-z^{-2})^{\wtd L(0)}.  \label{eq40}
    \end{align}
It is easy to see $\upgamma_z(zt)=z^{-1}\upgamma_1(t)$. Thus
\begin{align}
\mc U(\upgamma_z)z^{\wtd L(0)}=z^{-\wtd L(0)}\mc U(\upgamma_1).  \label{eq43}
\end{align}
\end{eg}

\subsection{Sheaves of VOAs for compact Riemann surfaces}
\label{sheafofvoa}
Let $\Vbb$ be a VOA and $C$ be a compact Riemann surface. We recall the construction of the sheaf $\SV_{C}$ associated to $\Vbb$ and $C$.
 
Let $U,V$ be open subsets of $C$ with univalent (i.e., injective and holomorphic)\index{00@Univalent functions} functions $\eta\in \MO(U),\mu\in \MO(V)$. Define a holomorphic family of transformations $\varrho(\eta\vert \mu):U\cap V\rightarrow \Gbb,p\mapsto \varrho(\eta\vert \mu)_p,$
where $\varrho(\eta\vert \mu)_p(z)=\eta\circ \mu^{-1}\big(z+\mu(p)\big)-\eta(p)$. Equivalently, 
\begin{align}\label{eq1}
\eta-\eta(p)=\varrho(\eta\vert \mu)_p(\mu-\mu(p))
\end{align}
and for univalent functions $\eta_i\in \MO(U_i)(i=1,2,3)$,
\begin{equation}\label{cocycle1}
\varrho(\eta_3\vert \eta_1)=\varrho(\eta_3\vert \eta_2)\varrho(\eta_2\vert \eta_1)
\end{equation}
on $U_1\cap U_2\cap U_3$. By (\ref{coordinatechange4}), we have an isomorphism of $\MO_{U\cap V}$-modules:\index{U@$\MU(\varrho(\eta\vert \mu))$, the transition function of sheaf of VOA}
$$
\MU(\varrho(\eta\vert \mu)):\Vbb^{\leq n}\otimes_\Cbb \MO_{U\cap V}\xrightarrow{\simeq} \Vbb^{\leq n}\otimes_\Cbb \MO_{U\cap V}.
$$
By Thm. \ref{coordinatechange5} and (\ref{cocycle1}), we have the cocycle condition
$$
\MU(\varrho(\eta_3\vert \eta_1))=\MU(\varrho(\eta_3\vert \eta_2))\MU(\varrho(\eta_2\vert \eta_1)).
$$
This allows us to define a locally free $\mc O_C$-module  (i.e. a finite rank holomorphic vector bundle) $\SV_C^{\leq n}$, such that the transition functions are given by $\MU(\varrho(\eta\vert \mu))$. Thus, for any open subset $U\subset C$ with a univalent function $\eta\in \MO(U)$, we have a trivialization\index{U@$\MU_\varrho(\eta)$, the trivialization of sheaf of VOA}
\begin{equation}\label{trivialization}
\MU_\varrho(\eta):\SV_C^{\leq n}\vert_U \xrightarrow{\simeq} \Vbb^{\leq n}\otimes_\Cbb \MO_U.
\end{equation}
Moreover, if $V\subset C$ is another open subset with a univalent function $\mu\in \MO(V)$, then we have transition function
$$
\MU_\varrho(\eta)\MU_\varrho(\mu)^{-1}=\MU(\varrho(\eta\vert \mu))
$$
on $U\cap V$. The \textbf{sheaf of VOA}\index{V@$\SV_C,\SV_C^{\leq n}$, sheaves of VOA} $\SV_C$ is the $\mc O_C$-module defined by $\SV_C=\bigcup_{n\in \Nbb} \SV_C^{\leq n}$, or more precisely,
$$
\SV_C=\varinjlim_{n\in \Nbb} \SV_C^{\leq n}.
$$

\begin{eg}\label{lb60}
    Let $\zeta$ be the standard coordinate of $\Cbb^\times$. Then for each $z\in \Cbb^\times$, 
    \begin{align*}
        \varrho(1/\zeta\vert \zeta)_z=\varrho(\zeta\vert 1/\zeta)_{1/z}=\upgamma_z.
    \end{align*}
    By Exp. \ref{changeexample1}, 
    \begin{align*}
       \big( \MU_\varrho(1/\zeta)\MU_\varrho(\zeta)^{-1}\big)_z=\MU(\varrho(1/\zeta\vert \zeta)_z)=\MU(\upgamma_z)=e^{z L(1)}(-z^{-2})^{L(0)}.
    \end{align*}
\end{eg}
\subsection{Conformal blocks for compact Riemann surfaces}
\begin{df}
$\fx=(C\big| x_1,x_2,\cdots,x_N;\eta_1,\eta_2,\cdots,\eta_N)$ is called an \textbf{$N$-pointed compact Riemann surface with local coordinates}\index{00@$N$-pointed compact Riemann surfaces (with local coordinates)} if 
\begin{enumerate}[label=(\arabic*)]
    \item $x_1,x_2,\cdots,x_N$ are distinct marked points on the compact Riemann surface $C$, and each connected component of $C$ contains one of these points.
    \item For each $1\leq i\leq N$, $\eta_i$ is a \textbf{local coordinate} \index{00@Local coordinates} near $x_i$, i.e., a univalent function on a neighborhood $U_i$ of $x_i$ satisfying $\eta_i(x_i)=0$.
\end{enumerate}
Forgetting the local coordinates, we call $(C\big| x_1,x_2,\cdots,x_N)$ an \textbf{$N$-pointed compact Riemann surface}.
\end{df}
Fix an $N$-pointed compact Riemann surface with local coordinates $\fx=(C\big| x_1,x_2,\cdots,x_N;\eta_1,\eta_2,\cdots,\eta_N)$. Let $\Wbb$ be a weak $\Vbb^{\times N}$-module. Let us recall the definition of the space of conformal block $\ST_\fx^*(\Wbb)$ (cf. \cite[Ch. 9]{FB04}).

Write \index{SX@$\SX=x_1+x_2+\cdots+x_N$}
\begin{align*}
 S_\fx=x_1+x_2+\cdots +x_N   
\end{align*}
and let $\omega_C$\index{zz@$\omega_C$, cotangent sheaf of $C$} be the cotangent sheaf of $C$. For each $\MO_C$-module $\ME$, set 
$$
\ME(\blt S_\fx)=\varinjlim_{k\in \Nbb} \ME(kS_\fx).
$$
More precisely, sections of $\ME(\blt S_\fx)$ are meromorphic sections of $\ME$ whose only possible poles are at $x_1,x_2,\cdots,x_N$. By tensoring with the identity map of $\omega_{U_i}$ with (\ref{trivialization}), we have the sheaf $\SV_C\otimes \omega_C(\blt S_\fx)$ whose trivialization on the neighborhood $U_i$ of $x_i$ is:
$$
\MU_\varrho(\eta_i):\SV_C\vert_{U_i}\otimes \omega_{U_i}(\blt S_\fx)\xrightarrow{\simeq} \Vbb\otimes_\Cbb \omega_{U_i}(\blt S_\fx).
$$
where $\MU_\varrho(\eta_i)$ is short for $\MU_\varrho(\eta_i)\otimes\idt$. This allows us to define the $i$-th residue action of each section $\sigma\in H^0\big(U_i,\SV_{C}\otimes \omega_C(\blt S_\fx)\big)$ on $\Wbb$:
\begin{subequations}\label{eq73}
\begin{equation}\label{eq80}
    \sigma *_i w:=\Res_{\eta_i=0}Y_{i}(\MU_\varrho(\eta_i)\sigma,\eta_i)w.
\end{equation}
The residue action\index{00@Residue action} of $\sigma\in H^0\big(C,\SV_{C}\otimes \omega_C(\blt S_\fx)\big)$ on $w\in \Wbb$ is defined by 
\begin{equation}\label{eq81}
\sigma \cdot w=\sum_{i=1}^N \sigma *_i w.
\end{equation}
\end{subequations}

\begin{df}\label{lb25}
The \textbf{space of coinvariants} $\ST_\fx(\Wbb)$ \index{T@$\ST_\fx(\Wbb)$, the space of coinvariants}is defined by 
    $$
    \ST_\fx(\Wbb)=\frac{\Wbb}{H^0\big(C,\SV_{C}\otimes \omega_C(\blt S_\fx)\big)\cdot \Wbb},
    $$
    where $H^0\big(C,\SV_{C}\otimes \omega_C(\blt S_\fx)\big)\cdot \Wbb$ is the subspace of $\Wbb$ spanned by $\{\sigma \cdot w:\sigma\in H^0\big(C,\SV_{C}\otimes \omega_C(\blt S_\fx)\big),w\in \Wbb\}$.

The \textbf{space of conformal blocks} $\ST_\fx^*(\Wbb)$ is defined to be the dual space of $\ST_\fx(\Wbb)$.\index{T@$\ST_\fx^*(\Wbb)$, the space of conformal blocks} Therefore, elements in $\ST_\fx^*(\Wbb)$, called \textbf{conformal blocks} are linear functionals $\upphi:\Wbb\rightarrow \Cbb$ that vanish on $H^0\big(C,\SV_{C}\otimes \omega_C(\blt S_\fx)\big)\cdot \Wbb$.
\end{df}

The above definition of conformal blocks depend on the choice of local coordinates. See Def. \ref{lb16} for a coordinate-free definition.

\subsection{Sheaves of VOAs for families of compact Riemann surfaces}
\begin{df}
    $\fx=(\pi:\MC\rightarrow \MB)$ is called a \textbf{family of compact Riemann surfaces}\index{00@Families with compact Riemann surfaces} if $\pi:\MC\rightarrow \MB$ is a surjective proper submersion between complex manifolds, and if each fiber $\MC_b=\pi^{-1}(b)$ is a compact Riemann surface.
\end{df}

Fix a family of compact Riemann surfaces $\fx=(\pi:\MC\rightarrow \MB)$. We recall the definition of the sheaf of VOAs on $\fx$. (See \cite[Sec. 5]{Gui-sewingconvergence} for details.) Let $U$ and $V$ be open subsets of $\MC$. $\eta:U\rightarrow \Cbb$ and $\mu:V\rightarrow \Cbb$ are holomorphic functions, which are univalent on each fiber of $U$ and $V$. This is equivalent to saying that $(\eta,\pi)$ and $(\mu,\pi)$ are biholomorphic maps from $U,V$ to open subsets of $\Cbb\times \MB$. Define a holomorphic family of transformation $\varrho(\eta\vert \mu):U\cap V\rightarrow \Gbb$ by 
$$
\varrho(\eta\vert \mu)_p(z)=\eta\circ (\mu,\pi)^{-1}\big(z+\mu(p),\pi(p)\big)-\eta(p).
$$
That this family is holomorphic is clear from the definition. An equivalent but more transparent definition is
$$
\eta-\eta(p)\vert_{(U\cap V)_{\pi(p)}}=\varrho(\eta\vert \mu)_p(\mu-\mu(p)\vert_{(U\cap V)_{\pi(p)}}).
$$
For functions $\eta_i\in \MO(U_i)(i=1,2,3)$ univalent on each fiber, we have
$$
\varrho(\eta_3\vert \eta_1)=\varrho(\eta_3\vert \eta_2)\varrho(\eta_2\vert \eta_1)
$$
on $U_1\cap U_2\cap U_3$. Similar to Subsec. \ref{sheafofvoa}, we can define an $\mc O_{\mc C}$-module  (i.e. a finite rank holomorphic vector bundle) $\SV_{\fx}^{\leq n}$ such that the transition functions are given by $\MU(\varrho(\eta\vert \mu))$. More precisely, if $U\subset \MC$ is an open subset with $\eta\in \MO(U)$ univalent on each fiber, then we have a trivialization
$$
\MU_\varrho(\eta):\SV_\fx^{\leq n}\vert_{U}\xrightarrow{\simeq }\Vbb^{\leq n}\otimes_{\Cbb} \MO_U
$$
such that the transition function is given by
$$
\MU_\varrho(\eta)\MU_\varrho(\mu)^{-1}=\MU(\varrho(\eta\vert \mu)).
$$
The sheaf of VOA $\SV_\fx$ is defined by 
$$
\SV_\fx=\varinjlim_{n\in \Nbb} \SV_\fx^{\leq n},
$$
which is a possible infinite-rank locally free $\mc O_{\mc C}$-module (i.e. vector bundle).

\subsection{Conformal blocks for families of compact Riemann surfaces}

\begin{df}\label{lb1}
We call 
\begin{align}
\fx=(\pi:\mc C\rightarrow\mc B|\sgm_\blt;\eta_\blt)=(\pi:\MC\rightarrow \MB\big| \varsigma_1,\cdots,\varsigma_N;\eta_1,\cdots,\eta_N)   
\end{align}
a \textbf{family of $N$-pointed compact Riemann surfaces with local coordinates}\index{00@Families of $N$-pointed compact Riemann surfaces (with local coordinates)} if 
\begin{enumerate}[label=(\arabic*)]
    \item  $\pi:\MC\rightarrow \MB$ is a family of compact Riemann surfaces.
    \item  For each $i$, $\sgm_i:\mc B\rightarrow\mc C$ is a \textbf{section}, namely, a holomorphic map satisfying $\pi\circ\sgm_i=\idt_{\mc B}$. Moreover, we assume:
    \begin{enumerate}[label=(\alph*)]
        \item $\sgm_i(\mc B)\cap\sgm_j(\mc B)=\emptyset$ if $i\neq j$.
        \item For each $b\in\mc B$, each connected component of $\MC_b$ contains $\varsigma_i(b)$ for some $1\leq i\leq N$.
    \end{enumerate}
    \item  $\eta_1,\cdots,\eta_N$ are \textbf{local coordinates} at $\varsigma_1(\mc B),\cdots,\varsigma_N(\mc B)$. Namely, for each $i$, there exists a neighborhood $U_i$ of $\varsigma_i(\MB)$ such that $\eta_i\in\mc O(U_i)$, and that $\eta_i$ restricts to a univalent function on $U_i\cap \MC_b$  and satisfies $\eta_i(\varsigma_i(b))=0$ for each $b\in \MB$. (Equivalently, $(\eta_i,\pi)$ is a biholomorphism from $U_i$ to a neighborhood of $\{0\}\times\mc B$ in $\Cbb\times\mc B$ sending $\sgm_i(\mc B)$ to $\{0\}\times\mc B$.)
\end{enumerate}
If the local coordinates are forgotten, then $(\pi:\MC\rightarrow \MB\big| \varsigma_1,\cdots,\varsigma_N)$ is called a \textbf{family of $N$-pointed compact Riemann surfaces}. Define the following divisor of $\mc C$ \index{SX@$\SX=\sgm_1(\mc B)+\cdots+\sgm_N(\mc B)$}
   \begin{align}
\SX=\sgm_1(\mc B)+\cdots+\sgm_N(\mc B)
   \end{align}
\end{df}
\begin{rem}\label{remarkfamily}
Suppose $\fx=(\pi:\MC\rightarrow \MB\big| \varsigma_1,\cdots,\varsigma_N;\eta_1,\cdots,\eta_N)$ is a family of $N$-pointed compact Riemann surfaces with local coordinates.
\begin{enumerate}[label=(\arabic*)]
    \item If we choose $\MB$ as a point (viewed as a $0$-dimensional manifold), then $\fx$ is exactly an $N$-pointed compact Riemann surface with local coordinates.
    \item For each $b\in \MB$, each fiber \index{Cb@$\mc C_b=\pi^{-1}(b)$}  \index{Xb@$\fk X_b$, the fiber of $\fk X$ at $b$}
    $$
    \fx_b:=(\MC_b\big| \varsigma_1(b),\cdots,\varsigma_N(b);\eta_1\vert_{\MC_b},\cdots,\eta_N\vert_{\MC_b})
    $$
    is an $N$-pointed compact Riemann surface with local coordinates.
    \item Define \textbf{relative tangent bundle}\index{00@$\Theta_{\MC/\MB}$, the relative tangent bundle} $\Theta_{\MC/\MB}$ as the subbundle of $\Theta_\MC$ containing sections of $\Theta_\MC$ killed by $d\pi:\Theta_{\MC}\rightarrow \Theta_{\MB}$. The \textbf{relative dualizing sheaf}\index{zz@$\omega_{\MC/\MB}$, the relative dualizing sheaf} $\omega_{\MC/\MB}$ is defined as the dual bundle of $\Theta_{\MC/\MB}$. If $U$ is an open subset of $\MC$, then the sections of $\omega_{\MC/\MB}(U)$ are of the form $fd\eta$, where $f\in \MO(U)$ and $\eta\in \MO(U)$ is univalent on each fiber of $U$. If $\mu\in \MO(U)$ is another function univalent on each fiber, then we have transformation rule 
\begin{align}
    fd\eta=f\frac{\partial \eta}{\partial \mu}d\mu.  \label{eq28}
\end{align}
    Moreover, we have natural equivalences
    $$
    \Theta_{\MC/\MB}\vert_{\MC_b}\simeq \Theta_{\MC_b},\quad \omega_{\MC/\MB}\vert_{\MC_b}\simeq \omega_{\MC_b}. 
    $$
\end{enumerate}

\end{rem}
 
Fix an $N$-pointed family $\fx=(\pi:\MC\rightarrow \MB\big| \varsigma_\blt;\eta_\blt)$ and an admissible $\Vbb^{\times N}$-module $\Wbb$. Recall that $\pi_*(\SV_\fx\otimes \omega_{\MC/\MB}(\blt S_\fx))$ is the direct image sheaf of $\SV_\fx\otimes \omega_{\MC/\MB}(\blt S_\fx)$. So each
\begin{align*}
\sigma\in \pi_*(\SV_\fx\otimes \omega_{\MC/\MB}(\blt S_\fx))(V)=(\SV_\fx\otimes \omega_{\MC/\MB}(\blt S_\fx))(\pi^{-1}(V))
\end{align*}
is a meromorphic section of $\SV_\fx\otimes \omega_{\MC/\MB}$ on $\pi^{-1}(V)$ with possible poles only at $\SX$. We define the residue action of $\pi_*(\SV_\fx\otimes \omega_{\MC/\MB}(\blt S_\fx))$ on $\Wbb\otimes_\Cbb \MO_\MB$ as follows. Since $(\eta_i,\pi)$ is a biholomorphism from $U_i$ to $(\eta_i,\pi)(U_i)$, we have two obvious equivalences
\begin{subequations}
  \begin{gather}
(\eta_i,\pi)_*\equiv ((\eta_i,\pi)^{-1})^*:\MO_{U_i}\xrightarrow{\simeq} \MO_{(\eta_i,\pi)(U_i)}\label{eq3}\\
    (\eta_i,\pi)_*\equiv ((\eta_i,\pi)^{-1})^*:\omega_{\MC/\MB}\vert_{U_i}\xrightarrow{\simeq} \omega_{(\eta_i,\pi)(U_i)/\pi(U_i)},\label{eq4}
  \end{gather}
\end{subequations}
where $\omega_{(\eta_i,\pi)(U_i)/\pi(U_i)}$ is the relative dualizing sheaf associated to the family $(\eta_i,\pi)(U_i)\rightarrow \pi(U_i)$ inherited from the projection $\Cbb\times \MB\rightarrow \MB$. We have the following \textbf{pushforward} maps, all denoted by $\MV_\varrho(\eta_i)$ \index{00@$\MV_\varrho(\eta_i)$, the pushforward} by abuse of notation:
\begin{subequations}
\begin{gather}
\MV_\varrho(\eta_i)=(\idt_{\Vbb}\otimes \eqref{eq3})\circ\MU_\varrho(\eta_i):\SV_{\fx}\vert_{U_i}\xrightarrow{\simeq}\Vbb\otimes_\Cbb \MO_{(\eta_i,\pi)(U_i)},\label{eq2}\\
\MV_\varrho(\eta_i)=\eqref{eq2}\otimes \eqref{eq4}:\SV_{\fx}\otimes \omega_{\MC/\MB}\vert_{U_i}\xrightarrow{\simeq}\Vbb\otimes_\Cbb\omega_{(\eta_i,\pi)(U_i)/\pi(U_i)},\\
\MV_\varrho(\eta_i):\SV_{\fx}\otimes \omega_{\MC/\MB}(\blt S_\fx)\vert_{U_i}\xrightarrow{\simeq}\Vbb\otimes_\Cbb\omega_{(\eta_i,\pi)(U_i)/\pi(U_i)}(\blt \{0\}\times \MB).
\end{gather}
\end{subequations}

Let $z$ be the standard coordinate of $\Cbb$.  If $w\in \Wbb\otimes_\Cbb \MO(V)$, the $i$-th residue action of $\sigma$ on $w$ is defined by 
\begin{subequations}\label{eq7}
\begin{equation}\label{residueaction3}
\sigma *_i w=\Res_{z=0}Y_i(\MV_\varrho(\eta_i)\sigma,z)w,
\end{equation}
and the action of $\sigma$ on $w\in \Wbb\otimes_\Cbb \MO(V)$ is defined by 
\begin{equation}\label{residueaction4}
\sigma\cdot w=\sum_{i=1}^N \sigma *_iw. 
\end{equation}
\end{subequations}
\begin{df}
The $\MO_\MB$-module
        $$
        \ST_\fx(\Wbb)=\frac{\Wbb\otimes_\Cbb \MO_\MB}{\pi_*(\SV_\fx\otimes \omega_{\MC/\MB}(\blt S_\fx))\cdot (\Wbb\otimes_\Cbb \MO_\MB)}
        $$
        is called the \textbf{sheaf of coinvariants}, whose dual $\MO_\MB$-module $\ST_\fx^*(\Wbb)$ is called the \textbf{sheaf of conformal blocks}. The global sections of $\ST_\fx^*(\Wbb)$ are called \textbf{conformal blocks} associated to $\fk X$ and $\Wbb$.   
\end{df}

Note that this definition does not rely on the $\wtd L_\blt(0)$-grading of $\Wbb$.

\section{Sewing and propagation of partial conformal blocks}\label{lb89}

\subsection{Partial conformal blocks for compact Riemann surfaces}
\label{dualtensorproduct1}

\begin{df}\label{lb3}
An \textbf{$(M,N)$-pointed compact Riemann surface} is an $(N+M)$-pointed compact Riemann surface  \index{00@$(M,N)$-pointed compact Riemann surface}
$$
\fx=(y_\star|C|x_\blt)=(y_1,\cdots,y_M\big| C\big| x_1,\cdots,x_N)
$$
where  the marked points are split into two groups. Those points  $x_\blt$ on the right are called \textbf{incoming points}, \index{00@Incoming points} and those points $y_\star$ on the left are called \textbf{outgoing points}. \index{00@Outgoing points}   Suppose that for each $j$ a local coordinate $\theta_j$ at $y_j$ is defined. We call
$$
\fx=(y_\star;\theta_\star|C|x_\blt)=(y_1,\cdots,y_M;\theta_1,\cdots,\theta_M\big| C\big| x_1,\cdots,x_N)
$$
an \textbf{$(M,N)$-pointed compact Riemann surface (with outgoing local coordinates)}. Set \index{SX@$\SX=x_1+x_2+\cdots+x_N$}  \index{DX@$D_\fx=y_1+\cdots+y_M$}
\begin{equation}\label{marked2}
S_\fx=x_1+\cdots+x_N,\qquad D_\fx=y_1+\cdots+y_M.
\end{equation}
\end{df}

\begin{ass}\label{ass1}
Unless otherwise stated, we assume that each connected component of $C$ contains at least one of the incoming marked points $x_\blt$ (but not just one of $x_\blt,y_\star$).
\end{ass}
This assumption ensures that certain cohomology groups vanish, cf. \eqref{eq14}.

\begin{df}\label{lb86}
For each $a_1,\cdots,a_M\in \Nbb$, define \index{VX@$\SV_{\fx,a_1,\cdots,a_M}^{\leq n}=\SV_{\fx,a_\star}^{\leq n}$}
\begin{gather*}
\SV_{\fx,a_1,\cdots,a_M}^{\leq n}=\SV_{\fx,a_\star}^{\leq n}:=\SV_{\fx}^{\leq n}\big(-(L(0)D_\fx+a_1y_1+\cdots+a_M y_M)\big)\\
\SV_{\fx,a_1,\cdots,a_M}=\SV_{\fx,a_\star}:=\varinjlim_{n\in \Nbb}\SV_{\fx,a_1,\cdots,a_M}^{\leq n}
\end{gather*}
using the data of $\fx,\Vbb$. More precisely, $\SV_{\fx,a_\star}^{\leq n}$ is a locally free $\MO_C$-submodule of $\SV_C^{\leq n}$ described as follows: outside $y_1,\cdots,y_M$, it is exactly $\SV_C^{\leq n}$; for each $1\leq j\leq M$, if $W_j$ is a neighborhood of $y_j$ on which $\theta_j$ is defined (and univalent), and if $W_j\cap\{y_1,\dots,y_M\}=\{y_j\}$, then $\SV_{\fx,a_\star}^{\leq n}\vert_{W_j}$ is generated by 
\begin{gather}
\MU_\varrho(\theta_j)^{-1}\theta_j^{a_j+L(0)}v   \label{eq65}
\end{gather}
for homogeneous vectors $v\in\Vbb^{\leq n}$. \index{vv@$\SV_{\fx,a_1,\cdots,a_M}$}
\end{df}

\begin{df}\label{lb8}
Choose an admissible $\Vbb^{\times N}$-module $\Wbb$, which is associated to the ordered incoming marked points $x_1,\cdots,x_N$. Define a vector bundle $\SW_\fx(\Wbb)$ \index{WX@$\SW_\fx(\Wbb)$}  over a point $\pt$ (whose structure sheaf is $\Cbb$) as follows. For each set of local coordinates $\eta_1,\cdots,\eta_N$ at the incoming marked points $x_1,\cdots,x_N$, we have a trivialization (i.e. an isomorphism of vector spaces) \index{U@$\mc U(\eta_\blt)$}
\begin{align*}
\MU(\eta_\blt):\SW_\fx(\Wbb)\xrightarrow{\simeq}\Wbb
\end{align*}
satisfying that if there is another set of local coordinates $\mu_1,\cdots,\mu_N$ at $x_1,\cdots,x_N$, then the transition function is
\begin{align}\label{eq47}
\MU(\eta_\blt)\MU(\mu_\blt)^{-1}=\mc U_1(\eta_1\circ\mu_1^{-1})\cdots \mc U_N(\eta_N\circ\mu_N^{-1})
\end{align}
Notice that $\mc U_i(\cdots)$ commutes with $\mc U_j(\cdots)$ if $i\neq j$.
\end{df}

\begin{df}\label{lb17}
Define a linear action of $H^0\big(C,\SV_{\fx,a_1,\cdots,a_M}\otimes \omega_C(\blt S_\fx)\big)$ on $\SW_\fx(\Wbb)$ (called the \textbf{residue action}) as follows. Choose $\sigma\in H^0\big(C,\SV_{\fx,a_1,\cdots,a_M}\otimes \omega_C(\blt S_\fx)\big)$ and $\wbf\in \SW_\fx(\Wbb)$. Choose a set of coordinates $\eta_\blt$ at $x_\blt$. Then 
\begin{subequations}
\begin{gather}
\sigma*_i\wbf=\MU(\eta_\blt)^{-1} \big(\sigma *_i \MU(\eta_\blt)\wbf\big) \label{eq12}\\
\sigma\cdot \wbf=\sum_{i=1}^N\sigma*_i\wbf
\end{gather}
\end{subequations}
Here $\MU(\eta_\blt)\wbf$ is an element of $\Wbb$, and $\sigma *_i \MU(\eta_\blt)\wbf$ is defined by \eqref{eq80}.
\end{df}

\begin{rem}\label{lb2}
The above definition of $\sigma*_i\wbf$ is independent of the choice of $\eta_\blt$. This was proved in \cite[Thm. 6.5.4]{FB04} (see also \cite[Thm. 3.2]{Gui-sewingconvergence}), as a consequence of  Huang's change of coordinate Theorem \cite{Hua97} (see also \cite[Sec. 10]{GuiLec}).
\end{rem}

\begin{df}[Cf. \cite{KZ-conformal-block}]\label{lb16}
Similar to the definition of conformal blocks, we define
\begin{align}\label{eq129}
\ST_{\fx,a_1,\cdots,a_M}(\Wbb)=\ST_{\fx,a_\star}(\Wbb)=\frac{\SW_\fx(\Wbb)}{H^0\big(C,\SV_{\fx,a_1,\cdots,a_M}\otimes \omega_C(\blt S_\fx)\big)\cdot \SW_\fx(\Wbb)}
\end{align}
Then $\ST_{\fx,a_1,\cdots,a_M}(\Wbb)$ is called a \textbf{truncated $\fx$-fusion product} of the admissible $\Vbb^{\times N}$-module $\Wbb$. Its dual space $\ST_{\fx,a_1,\cdots,a_M}^*(\Wbb)$ is called a \textbf{truncated dual $\fx$-fusion product} of $\Wbb$.\index{00@Truncated (dual) $\fx$-fusion product} Moreover, when $a_j^\prime\leq a_j$ for each $1\leq j\leq M$, there is a natural injective linear map $\ST_{\fx,a_1^\prime,\cdots,a_M^\prime}^*(\Wbb)\hookrightarrow \ST_{\fx,a_1,\cdots,a_M}^*(\Wbb)$. Define  the \textbf{dual $\fx$-fusion product} of $\Wbb$ to be \index{00@Dual fusion products $\boxbackslash_\fx(\Wbb)$}  \index{W@$\bbs_\fx(\Wbb)$}
\begin{gather}
\boxbackslash_{\fx}(\Wbb)=\varinjlim_{a_1,\cdots,a_M\in \Nbb}\ST_{\fx,a_1,\cdots,a_M}^*(\Wbb).
\end{gather}
Elements in $\boxbackslash_{\fx}(\Wbb)$, called \textbf{partial conformal blocks},\index{00@Partial conformal blocks} are linear functionals $\upphi:\SW_\fx(\Wbb)\rightarrow \Cbb$ vanishing on $H^0\big(C,\SV_{\fx,a_1,\cdots,a_M}\otimes \omega_C(\blt S_\fx)\big)\cdot \SW_\fx(\Wbb)$ for some $a_1,\cdots,a_M$. 

If $M=0$, we call $\upphi$ a \textbf{conformal block}, and write $\ST_{\fx,a_1,\cdots,a_M}^*(\Wbb)$ as $\ST_\fx^*(\Wbb)$.
\end{df}

We remark that $\scr V_{\fx,0,\dots,0}$ and $\scr T_{\fx,0,\dots,0}(\Wbb)$ have already been considered in \cite[Sec. 7.2]{NT-P1_conformal_blocks} and \cite[Sec. 6.2]{DGT2}.

\subsection{Partial conformal blocks for families of compact Riemann surfaces}\label{lb4}

\begin{df}\label{familydef}
A \textbf{family of $(M,N)$-pointed compact Riemann surfaces (with outgoing local coordinates)} is an  $(M+N)$-pointed family
$$
\fx=(\tau_\star;\theta_\star\big| \pi:\MC\rightarrow \MB\big| \varsigma_\blt)=(\tau_1,\cdots,\tau_M;\theta_1,\cdots,\theta_M\big| \pi:\MC\rightarrow \MB\big| \varsigma_1,\cdots,\varsigma_N)
$$
where $\tau_\star,\sgm_\blt$ are sections and each $\theta_j$ is a local coordinate at $\tau_j(\mc B)$. (Recall Def. \ref{lb1}.) We call $\sgm_\blt$ the \textbf{incoming (families of) marked points} and $\tau_\star$ the \textbf{outgoing (families of) marked points}. Define divisors of $\mc C$:\index{SX@$\SX=\sgm_1(\mc B)+\cdots+\sgm_N(\mc B)$}  \index{DX@$D_\fx=\tau_1(\MB)+\cdots+\tau_M(\MB)$}
\begin{equation}\label{marked} 
S_\fx=\varsigma_1(\MB)+\cdots+\varsigma_N(\MB),\qquad D_\fx=\tau_1(\MB)+\cdots+\tau_M(\MB).
\end{equation}
\end{df}

\begin{ass}\label{lb11}
Unless otherwise stated, we  assume that for each $b\in\mc B$, each connected component of $\MC_b$ contains at least one of the incoming marked points $\sgm_\blt(b)$.
\end{ass}

\begin{df}\label{lb5}
For each $a_1,\cdots,a_M\in \Nbb$, define \index{VX@$\SV_{\fx,a_1,\cdots,a_M}^{\leq n}=\SV_{\fx,a_\star}^{\leq n}$}
$$
\begin{aligned}
\SV_{\fx,a_1,\cdots,a_M}^{\leq n}=\SV_{\fx,a_\star}^{\leq n}:=\SV_{\fx}^{\leq n}&\big(-(L(0)D_\fx+a_1\tau_1+\cdots+a_M \tau_M)\big),\\
\SV_{\fx,a_1,\cdots,a_M}=\SV_{\fx,a_\star}:&=\varinjlim_{n\in \Nbb}\SV_{\fx,a_1,\cdots,a_M}^{\leq n}.
\end{aligned}
$$
More precisely, $\SV_{\fx,a_1,\cdots,a_M}^{\leq n}$ is a locally free $\MO_\MC$-submodule of $\SV_\fx^{\leq n}$ described as follows: outside $\tau_1(\MB),\cdots,\tau_M(\MB)$, it is exactly $\SV_\fx^{\leq n}$; for each $1\leq j\leq M$, let $W_j$ be a neighborhood of $\tau_j(\mc B)$ on which $\theta_j$ is defined such that $W_j$ intersects only $\tau_j(\mc B)$ among $\tau_1(\mc B),\dots,\tau_M(\mc B)$, then $\SV_{\fx,a_1,\cdots,a_M}^{\leq n}\vert_{W_j}$ is generated by 
$$
\MU_\varrho(\theta_j)^{-1}\theta_j^{a_j+L(0)}v
$$
for homogeneous $v\in\Vbb^{\leq n}$.
\end{df}

\begin{df}[Definition of $\SW_\fx(\Wbb)$]\label{lb7}
Choose an admissible $\Vbb^{\times N}$-modules $\Wbb$, which is associated to the ordered incoming marked points $\varsigma_1,\cdots,\varsigma_N$.  Choose an open subset $V\subset \MB$ small enough such that  the restricted family\index{XV@$\fk X_V$, the restriction of $\fk X$ to the open subset $V\subset\mc B$} \index{CV@$\mc C_V=\pi^{-1}(V)$}
\begin{gather}\label{restrictedfamily}
\begin{gathered}
\fx_V:=\Big(\tau_1\vert_V,\cdots,\tau_M\vert_V~\Big|~\pi:\MC_V\rightarrow V~\Big|~\varsigma_1\vert_V,\cdots,\varsigma_N\vert_V\Big)\\
\text{where }\mc C_V=\pi^{-1}(V)
\end{gathered}
\end{gather}
admits local coordinates $\eta_1,\cdots,\eta_N$ at $\varsigma_1(V),\cdots,\varsigma_N(V)$. If there is another set of local coordinates $\mu_1,\cdots,\mu_N$ at $\varsigma_1(V),\cdots,\varsigma_N(V)$ respectively, then we have a family of transformations $(\eta_i\vert \mu_i):V\rightarrow \Gbb$ such that
$$
(\eta_i\vert \mu_i)_b\circ \mu_i\vert_{\MC_b}=\eta_i\vert_{\MC_b}
$$
for each $b\in V$. By (\ref{coordinatechange4}), we have an isomorphism of $\MO_V$-modules
$$
\MU_i(\eta_i\vert \mu_i):\Wbb\otimes \MO_V\xrightarrow{\simeq}\Wbb\otimes \MO_V.
$$
Set 
$$
\MU(\eta_\blt\vert \mu_\blt):=\MU_1(\eta_1\vert \mu_1) \cdots \MU_N(\eta_N\vert \mu_N):\Wbb\otimes \MO_V\xrightarrow{\simeq}\Wbb\otimes \MO_V.
$$
By Thm. \ref{coordinatechange5} and the fact that $\MU_i(\rho_1)$ commutes with $\MU_j(\rho_2)$ for $i\ne j$, $\MU(\eta_\blt\vert \mu_\blt)$ satisfies the cocycle condition, and so we can define a locally free $\MO_\MB$-module \index{WX@$\SW_\fx(\Wbb)$} 
\begin{align*}
\SW_\fx(\Wbb)
\end{align*}
with a trivialization 
$$
\MU(\eta_\blt):\SW_\fx(\Wbb)\vert_V\xrightarrow{\simeq} \Wbb\otimes \MO_V
$$
for each set of local coordinates $\eta_\blt$ at $\varsigma_1(V),\cdots,\varsigma_N(V)$. The transition function of $\SW_\fx(\Wbb)$ is given by 
$$
\MU(\eta_\blt)\MU(\mu_\blt)^{-1}=\MU(\eta_\blt\vert \mu_\blt).
$$
\end{df}


\subsubsection{Restriction to fibers}

For each $b\in \MB$, the fiber 
$$
\fx_b=\big(\tau_1(b),\cdots,\tau_M(b);\theta_1\vert_{W_1\cap \MC_b},\cdots,\theta_M\vert_{W_M\cap \MC_b}\big| \MC_b\big| \varsigma_1(b),\cdots,\varsigma_N(b)\big)
$$
is an $(M,N)$-pointed compact Riemann surface with outgoing local coordinates. Here we assume $\varsigma_1(b),\cdots,\varsigma_N(b)$ are incoming marked points and $\tau_1(b),\cdots,\tau_M(b)$ are outgoing marked points. Define divisors of fibers \index{SX@$\SX(b)=S_{\fx_b}$}  \index{DX@$D_{\fk X}(b)=D_{\fk X_b}$}
$$
S_\fx(b)=S_{\fx_b}=\varsigma_1(b)+\cdots+\varsigma_N(b),\quad D_\fx(b)=D_{\fx_b}=\tau_1(b)+\cdots+\tau_M(b)
$$

\begin{rem}\label{resrem1}
By comparing the transition functions, it is easy to see that for each $n\in \Nbb$, $a_1,\cdots,a_M\in \Nbb$ and $b\in \MB$, there is a natural isomorphism of $\MO_{\MC_b}$-modules
\begin{subequations}\label{eq8}
\begin{align}
    \SV_{\fx,a_1,\cdots,a_M}^{\leq n}\vert_{\MC_b}\simeq \SV_{\fx_b,a_1,\cdots,a_M}^{\leq n},
\end{align}
and a natural isomorphism of vector spaces
\begin{align}
    \SW_\fx(\Wbb)\vert_{b}\simeq \SW_{\fx_b}(\Wbb).  \label{eq10}
\end{align}
    Therefore, we have 
\begin{align}
    \SV_{\fx,a_1,\cdots,a_M}^{\leq n}\otimes\omega_{\MC/\MB}(\blt S_\fx)\vert_{\MC_b}\simeq \SV_{\fx_b,a_1,\cdots,a_M}^{\leq n}\otimes \omega_{\MC_b}(\blt S_\fx(b)).
\end{align}
\end{subequations} \addtocounter{equation}{-1}
\end{rem}

\begin{pp}\label{grauertcor}
Let $n,a_1,\cdots,a_M\in \Nbb$. For each $b\in \MB$, there exists $k_0\in\Nbb$ such that for all $k\geq k_0$ there is an isomorphism of vector spaces 
\begin{subequations}\setcounter{equation}{3}
\begin{align}
   \pi_*\big(\SV_{\fx,a_1,\cdots,a_M}^{\leq n}\otimes \omega_{\MC/\MB}(k S_\fx)\big)\Big\vert_b \simeq H^0\big(\MC_b,\SV_{\fx_b,a_1,\cdots,a_M}^{\leq n}\otimes \omega_{\MC_b}(k S_\fx(b))\big) \label{eq5}
\end{align}
\end{subequations}
(recall the notation \eqref{eq6}) defined by the natural restriction map
\begin{align}\label{eq16}
\pi_*\big(\SV_{\fx,a_1,\cdots,a_M}^{\leq n}\otimes \omega_{\MC/\MB}(k S_\fx)\big)_b\rightarrow H^0\big(\MC_b,\SV_{\fx_b,a_1,\cdots,a_M}^{\leq n}\otimes \omega_{\MC_b}(k S_\fx(b))\big)
\end{align}
In particular, \eqref{eq16} is surjective.
\end{pp}

\begin{proof}
Since the map $\pi:\mc C\rightarrow\mc B$ is open, $\mc O_{\mc C}$ is $\mc O_{\mc B}$-flat by \cite[Sec. 3.20]{Fis76}. Since we have Asmp. \ref{ass1}, by Serre's vanishing theorem (cf. \cite[Prop. 5.2.7]{Huy} or \cite[Thm. IV.2.1]{BaSt}), there exists $k_0\in\Nbb$ such that for all $k\geq k_0$ we have
\begin{align}\label{eq14}
H^1\big(\MC_b,\SV_{\fx_b,a_1,\cdots,a_M}^{\leq n}\otimes \omega_{\MC_b}(k S_\fx(b))\big)=0.
\end{align}
Therefore, by the base change property for flat families of complex analytic spaces, the restriction map \eqref{eq16} defines an isomorphism \eqref{eq5}, cf. \cite[Cor. III.3.9]{BaSt}. (One can also appeal to Grauert's base change theorem \cite[Thm. III.4.7]{GPR94} or \cite[Thm. III.4.12]{BaSt}. See the proof of \cite[Thm. 5.5]{Gui-sewingconvergence} about how to use this theorem.)
\end{proof}

\begin{co}\label{grauertcor2}
Let $a_1,\cdots,a_M\in \Nbb$. Let $V$ be any Stein open subset of $\mc B$. Then for each $b\in V$, the elements of $\pi_*\big(\SV_{\fx,a_1,\cdots,a_M}\otimes \omega_{\MC/\MB}(\blt S_\fx)\big)(V)$ generate the  $\mc O_{\mc B,b}$-module  $\pi_*\big(\SV_{\fx,a_1,\cdots,a_M}\otimes \omega_{\MC/\MB}(\blt S_\fx)\big)_b$ (the stalk). Thus, their restrictions to $\MC_b$ form the vector space $H^0\big(\MC_b,\SV_{\fx_b,a_1,\cdots,a_M}\otimes \omega_{\MC_b}(\blt S_\fx(b))\big)$.
\end{co}
\begin{proof}
The first conclusion follows from  Cartan's theorem A (applied to each $\mc O_{\mc B}$-module $\pi_*\big(\SV_{\fx,a_1,\cdots,a_M}^{\leq n}\otimes \omega_{\MC/\MB}(k S_\fx)\big)$, which is coherent by Grauert's direct image theorem \cite[Sec. 10.4]{GR84}). The second conclusion follows from Prop. \ref{grauertcor}.
\end{proof}

\subsubsection{Sheaves of partial conformal blocks}

We assume the isomorphisms \eqref{eq8}.
\begin{df}
Define an $\mc O_{\mc B}$-linear action of $\pi_*\big(\SV_{\fx,a_1,\cdots,a_M}\otimes \omega_{\MC/\MB}(\blt S_\fx)\big)$ on $\SW_\fx(\Wbb)$ (called \textbf{residue action}) \index{00@Residue action} as follows. Choose any $V\subset \MB$ small enough such that there are local coordinates $\eta_1,\cdots,\eta_N$ at $\varsigma_1(V),\cdots,\varsigma_N(V)$. For each section $\sigma\in \pi_*\big(\SV_{\fx,a_1,\cdots,a_M}\otimes \omega_{\MC/\MB}(\blt S_\fx)\big)(V)=H^0\big(\MC_V,\SV_{\fx,a_1,\cdots,a_M}\otimes \omega_{\MC/\MB}(\blt S_\fx)\big)$ and $\wbf\in H^0\big(V,\SW_\fx(\Wbb)\big)$, we have $\MU(\eta_\blt)\wbf\in \Wbb\otimes \MO(V)$. Since $\SV_{\fx,a_1,\cdots,a_M}\otimes \omega_{\MC/\MB}(\blt S_\fx)$ is a subsheaf of $\SV_{\fx}\otimes \omega_{\MC/\MB}(\blt S_\fx)$, we can regard $\sigma$ as an element in $\pi_*\big(\SV_{\fx}\otimes \omega_{\MC/\MB}(\blt S_\fx)\big)(V)$, and the residue action of $\sigma$ on $\MU(\eta_\blt)\wbf$ is defined \eqref{eq7}. The residue actions of $\sigma$ on $\wbf$ are defined by
\begin{subequations}\label{eq13}
\begin{gather}
\sigma *_i \wbf:=\MU(\eta_\blt)^{-1}\big(\sigma *_i \MU(\eta_\blt)\wbf\big)\qquad\in H^0(V,\SW_\fx(\Wbb))  \label{eq50}\\
\sigma \cdot \wbf=\sum_{i=1}^N\sigma*_i \wbf
\end{gather}
\end{subequations}
where $\sigma *_i \MU(\eta_\blt)\wbf$ is defined by \eqref{residueaction3}.
\end{df}

\begin{rem}\label{lb13}
When restricted to each fiber $\mc C_b$, using the natural equivalences \eqref{eq8}, we have (recall notation \eqref{eq9})
\begin{align}
(\sigma *_i\wbf)(b)=\sigma(b) *_i \wbf(b) \label{eq11}
\end{align}
where $\sigma(b)\in\eqref{eq5}$, $\wbf(b),(\sigma *_i\wbf)(b)\in\eqref{eq10}$, and the RHS of \eqref{eq11} is defined by \eqref{eq12}, which is independent of the choice of $\eta_\blt$ due to Rem. \ref{lb2}. Therefore, the definition of residue actions in \eqref{eq13} is also independent of the choice of $\eta_\blt$.
\end{rem}

\begin{df}
Define the \textbf{truncated $\fx$-fusion product} (of \textbf{multi-level $a_\star$})
$$
\ST_{\fx,a_1,\cdots,a_M}(\Wbb)=\frac{\SW_\fx(\Wbb)}{\pi_*\big(\SV_{\fx,a_1,\cdots,a_M}\otimes \omega_{\MC/\MB}(\blt S_\fx)\big)\cdot\SW_\fx(\Wbb)}
$$
where the denominator  \index{J@$\scr J,\scr J^\pre$}
\begin{subequations}\label{eq38}
\begin{align}\label{eq15}
\scr J=\pi_*\big(\SV_{\fx,a_1,\cdots,a_M}\otimes \omega_{\MC/\MB}(\blt S_\fx)\big)\cdot\SW_\fx(\Wbb)    
\end{align} 
is the sheafification of the presheaf $\scr J^\pre$ associating to each open $V\subset\mc B$ the $\mc O(V)$-module
\begin{align}\label{eq37}
\scr J^\pre(V)=\pi_*\big(\SV_{\fx,a_1,\cdots,a_M}\otimes \omega_{\MC/\MB}(\blt S_\fx)\big)(V)\cdot\SW_\fx(\Wbb)(V)
\end{align}
\end{subequations}
The dual sheaf is denoted by \index{TX@$\ST_{\fx,a_1,\cdots,a_M}(\Wbb)$ and its dual sheaf $\ST^*_{\fx,a_1,\cdots,a_M}(\Wbb)$}
\begin{align*}
\ST_{\fx,a_1,\cdots,a_M}^*(\Wbb)=\big(\ST_{\fx,a_1,\cdots,a_M}(\Wbb) \big)^*
\end{align*}
and called the \textbf{truncated dual $\fk X$-fusion product}. Following Subsec. \ref{dualtensorproduct1}, we can define the dual $\fx$-fusion product $\boxbackslash_{\fx}(\Wbb)$, which is an $\MO_\MB$-module. Global sections of $\boxbackslash_{\fx}(\Wbb)$ are called \textbf{partial conformal blocks associated to $\fx$ and $\Wbb$}.
\end{df}

Note that the definition of $\scr T_{\fx,a_\blt}(\Wbb)$ relies on the choice of $\wtd L_\blt(0)$-grading of $\Wbb$ since $\scr W_\fx(\Wbb)$ does. (However, if $\fx$ admits local coordinates $\eta_\blt$ and the identification $\scr W_\fx(\Wbb)\simeq\Wbb\otimes_\Cbb\mc O_\MB$ via $\mc U(\eta_\blt)$ is assumed, then $\scr T_{\fx,a_\blt}(\Wbb)$ is clearly independent of the choice of $\wtd L_\blt(0)$.)

\begin{rem}\label{lb9}
Recall notation \eqref{eq18}. Then we know that the following are equivalent:
\begin{gather}
\begin{gathered}
\upphi\in \ST_{\fx,a_1,\cdots,a_M}^*(\Wbb)(\mc B)\\
\Leftrightarrow\upphi\text{ is an $\mc O_{\mc B}$-module morphism }\SW_\fx(\Wbb)\rightarrow \MO_\MB\text{ vanishing on }\scr J\\
\Leftrightarrow\upphi\text{ is an $\mc O_{\mc B}$-module morphism }\SW_\fx(\Wbb)\rightarrow \MO_\MB\text{ vanishing on }\scr J^\pre
\end{gathered}
\end{gather}
We call such $\upphi$ a \textbf{partial conformal block associated to $\fk X,\Wbb$ with multi-level $a_1,\dots,a_M$}.
\end{rem}

One should keep in mind that outgoing marked points and local coordinates are only used to define the sheaf $\SV_{\fx,a_1,\cdots,a_M}$. When there are no outgoing marked points and local coordinates, i.e., $M=0$, $\SV_{\fx,a_1,\cdots,a_M}$ is exactly the sheaf of VOA $\SV_\fx$ and dual fusion products are exactly conformal blocks.

For the convenience of discussion, we define \index{Jb@$J(b)$}
\begin{gather}\label{eq21}
J(b)=H^0\big(\MC_b,\SV_{\fx_b,a_1,\cdots,a_M}\otimes \omega_{\MC_b}(\blt S_\fx(b))\big)\cdot \SW_{\fx_b}(\Wbb)
\end{gather}

\subsubsection{Basic properties of partial conformal blocks}

\begin{pp}
    For each $b\in \MB$, the evaluation map
    \begin{equation}\label{familyfiber1}
        \SW_\fx(\Wbb)_b\rightarrow \SW_{\fk X}(\Wbb)|_b\simeq\SW_{\fx_b}(\Wbb),\quad w\mapsto w(b)
    \end{equation}
    descends to an isomorphism of vector spaces
    $$
    \ST_{\fx,a_1,\cdots,a_M}(\Wbb)\vert_b \xrightarrow{\simeq} \ST_{\fx_b,a_1,\cdots,a_M}(\Wbb).
    $$
\end{pp}
\begin{proof}
Write $\mk_b=\mk_{\mc B,b}$. Recall that $\scr W_{\fk X}(\Wbb)_b/\fk m_{\mc B,b}\scr W_{\fk X}(\Wbb)_b=\scr W_{\fk X}(\Wbb)|_b$. Then $\ST_{\fx,a_1,\cdots,a_M}(\Wbb)_b=\scr W_{\fk X}(\Wbb)_b/\scr J_b$ where $\scr J_b$ is the stalk of $\scr J$ at $b$, and hence 
\begin{align*}
\ST_{\fx,a_1,\cdots,a_M}(\Wbb)|_b=\frac{\ST_{\fx,a_1,\cdots,a_M}(\Wbb)_b}{\mk_b\ST_{\fx,a_1,\cdots,a_M}(\Wbb)_b}=\frac{\scr W_{\fk X}(\Wbb)_b}{\scr J_b+\fk m_b\scr W_{\fk X}(\Wbb)_b}.
\end{align*}
Clearly \eqref{familyfiber1} has kernel $\mk_b\scr W_{\fk X}(\Wbb)_b$. Since \eqref{eq16} is surjective (Prop. \ref{grauertcor}),  \eqref{familyfiber1} restricts to a surjective map 
\begin{gather}
\scr J_b\twoheadrightarrow J(b)  \label{eq17}
\end{gather}
and hence descends to a surjective map $(\scr J_b+\mk_b\scr W_{\fk X}(\Wbb)_b)/\mk_b\scr W_{\fk X}(\Wbb)_b\twoheadrightarrow J(b)$. We thus have a commutative diagram
\begin{equation*}
\begin{tikzcd}
\displaystyle\frac{\scr J_b+\mk_b\scr W_{\fk X}(\Wbb)_b}{\mk_b\scr W_{\fk X}(\Wbb)_b} \arrow[d, two heads] \arrow[r] & \displaystyle \frac{\scr W_{\fk X}(\Wbb)_b}{\mk_b\scr W_{\fk X}(\Wbb)_b}  \arrow[r] \arrow[d,"\simeq"] & \ST_{\fx,a_1,\cdots,a_M}(\Wbb)|_b \arrow[r] \arrow[d] & 0 \\
J(b) \arrow[r]                      & \scr W_{\fk X_b}(\Wbb) \arrow[r]           & \ST_{\fx_b,a_1,\cdots,a_M}(\Wbb) \arrow[r]           & 0
\end{tikzcd}
\end{equation*}
where the horizontal lines are exact sequences. By Five Lemma, the third vertical arrow is an isomorphism.
\end{proof}

\begin{rem}
Let $\upphi:\SW_\fx(\Wbb)\rightarrow \MO_{\mc B}$ be a homomorphism of $\MO_V$-modules. For each $b\in \mc B$, we have the restriction
\begin{align}\label{eq20}
\upphi|_b=\scr W_{\fk X_b}(\Wbb)\simeq \scr W_{\fk X}(\Wbb)|_b\rightarrow\Cbb
\end{align}
defined by \eqref{eq19}. On the one hand, we may ask whether $\upphi$ is a partial conformal block of $\fk X$ with multi-level $a_\star$; namely, whether $\upphi$ vanishes on $\scr J(V)$ for all open $V\subset\mc B$ (equivalently, vanishes on $\scr J^\pre(V)$ for all $V$). On the other hand, one may ask whether for each $b$, $\upphi|_b$ is a partial conformal block of $\fk X_b$ with multi-level $a_\star$, i.e. whether the linear map \eqref{eq20} vanishes on $J(b)=\eqref{eq21}$. This is affirmed by the following proposition.
\end{rem}

\begin{pp}\label{familyfiber5}
Let $\upphi:\SW_\fx(\Wbb)\rightarrow \MO_{\mc B}$ be an $\MO_{\mc B}$-module morphism. Then $\upphi\in \ST_{\fx,a_1,\cdots,a_M}^*(\Wbb)(\mc B)$ if and only if $\upphi\vert_b\in \ST_{\fx_b,a_1,\cdots,a_M}^*(\Wbb)$ for each $b\in \mc B$. 
\end{pp}

Thus, whether or not $\upphi$ is a partial conformal block for the family $\fk X$ \emph{can be checked fiberwisely}.

\begin{proof}
Suppose that for each $b\in \mc B$, we have $\upphi\vert_b\in \ST_{\fx_b,a_1,\cdots,a_M}^*(\Wbb)$, i.e. the linear map $\upphi|_b:\scr W_{\fk X_b}(\Wbb)\rightarrow\Cbb$ vanishes on $J(b)$. Then for each open $V\subset\mc B$ and each $\sigma\in\scr J^\pre(V)$, $\upphi$ vanishes on $\sigma$ because for each $b\in V$ we have $\sigma(b)\in J(b)$ and hence $\upphi(\sigma)|_b=\upphi|_b(\sigma|_b)=0$. This proves ``$\Leftarrow $". That ``$\Rightarrow$" is true follows from the surjectivity of \eqref{eq17}.
\end{proof}

\begin{pp}\label{localglobal1}
    Let $\Phi:\SW_\fx(\Wbb)\rightarrow \MO_\MB$ be an $\MO_\MB$-module morphism. Suppose that each connected component of $\MB$ contains a non-empty open subset $V$ such that the restriction $\Phi|_V:\SW_{\fx_V}(\Wbb)\rightarrow \MO_V$ is a partial conformal block for $\fk X_V,\Wbb$ with multi-level $a_\star$. Then $\Phi$ is a partial conformal block for $\fx$ and $\Wbb$ with multi-level $a_\star$.
\end{pp}
\begin{proof}
It suffices to assume that $\MB$ is connected. Fix a non-empty open subset $V$ such that $\Phi\vert_V\in \ST_{\fx,a_1,\cdots,a_M}^*(\Wbb)(V)$.
     
First assume $\MB$ is Stein. By Prop. \ref{familyfiber5}, it suffices to show $\Phi\vert_b\in \ST_{\fx_b,a_1,\cdots,a_M}^*(\Wbb)$ for each $b\in \MB$. For each $b\in \MB$ and $\sigma_b\in J(b)$, by Cor. \ref{grauertcor2}, we can find $\sigma\in \scr J^\pre(\mc B)$ whose restriction to $b$ is $\sigma_b$. Note that $\Phi(\sigma)\in\mc O(\mc B)$. By Prop. \ref{familyfiber5} (applied to the restricted family $\fk X_V$), $\Phi(\sigma)\vert_t=\Phi\vert_t(\sigma_t)$ equals $0$ for all $t\in V$. Since $\mc B$ is connected, we have $\Phi(\sigma)=0$ by complex analysis. So $\Phi|_b(\sigma_b)=\Phi(\sigma)|_b=0$. This finishes the proof.
     
For the general case, let $A$ be the set of all $b\in \MB$ such that $b$ has a neighborhood $U$ such that $\Phi\vert_U$ is a partial conformal block with multi-level $a_\star$. Then $A$ is non-empty and open. For any $b\in \MB-A$, let $U$ be a connected Stein neighborhood of $b$. By the previous paragraph, $\Phi\vert_U$ is a partial conformal block if $U$ has a non-empty open subset $\wtd U$ such that $\Phi\vert_{\wtd U}$ is a partial conformal block with multi-level $a_\star$. Thus $U$ must be disjoint from $A$. This proves that $\MB-A$ is open. So $\MB=A$ and we are done.
\end{proof}

\subsection{Sewing partial conformal blocks}\label{lb76}

\subsubsection{Sewing a Riemann surface $\wtd C$ along several pairs of points} \label{lb6}

Choose an $(M,N+2R)$-pointed compact Riemann surface (recall Def. \ref{lb3})
\begin{align}
\wtd \fx=(y_1,\cdots,y_M;\theta_1,\dots,\theta_M\big|\wtd C\big|x_1,\cdots,x_N\big\Vert\varsigma_1^\prime,\cdots,\varsigma_R^\prime,\varsigma_1'',\cdots,\varsigma_R'')    \label{eq30}
\end{align}
where each $\theta_i$ is a local coordinate at $y_i$.\footnote{The double vertical line $\Vert$ in \eqref{eq30} emphasizes that the points after it are for sewing. We will sometimes write it as a single vertical line or a comma.} We will \emph{not} assume Asmp. \ref{ass1}. Instead, we shall assume the weaker Asmp. \ref{sewingass}.

We assume $\wtd \fx$ has local coordinates $\xi_1,\cdots,\xi_R$ at $\varsigma_1^\prime,\cdots,\varsigma_R^\prime$ and $\varpi_1,\cdots,\varpi_R$ at $\varsigma_1'',\cdots,\varsigma_R''$. Moreover, we assume $\xi_1,\cdots,\xi_R,\varpi_1,\cdots,\varpi_R$ are defined on open neighborhoods $V_1^\prime,\cdots,V_R^\prime,V_1'',\cdots,V_R''$, with biholomorphisms
\begin{align}\label{geosew2}
    \xi_i:V_i^\prime\xrightarrow{\simeq } \MD_{r_i}\qquad \varpi_i:V_i'' \xrightarrow{\simeq } \MD_{\rho_i}
\end{align}
for some $r_i,\rho_i>0$. 
\begin{ass}
We assume that $y_1,\dots,y_M,x_1,\dots,x_N,V_1',\dots,V_R',V_1'',\dots,V_R''$ are mutually disjoint.
\end{ass}

We can \textbf{sew $\wtd \fx$ along the pairs of points $\sgm_i',\sgm_i''$} \index{00@Sewing Riemann surfaces along pairs of points} for all $1\leq i\leq R$ to get a \textbf{family of $(M,N)$-pointed nodal curves} \index{00@Sewing compact Riemann surfaces} \index{00@{Families of $(M,N)$-pointed nodal curves}}
\begin{gather}\label{geosew5}
\begin{gathered}
    \fx=(y_1,\cdots,y_M;\theta_1,\dots,\theta_M\big|\pi:\MC\rightarrow \MB\big|x_1,\cdots,x_N)\\
\text{where }\MB=\MD_{r_\blt \rho_\blt} =\MD_{r_1\rho_1} \times \cdots \times \MD_{r_R \rho_R}
\end{gathered}
\end{gather}
If $b_\blt=(b_1,\dots,b_R)\in\mc B$ satisfies $b_1\cdots b_R\neq 0$, then the fiber $\fk X_{b_\blt}$ is obtained by removing the closed disks
\begin{gather}
F_{i,b_\blt}'=\Big\{p_i'\in V_i':|\xi_i(p_i')|\leq \frac{|b_i|}{\rho_i}\Big\}\qquad F_{i,b_\blt}''=\Big\{p_i''\in V_i'':|\varpi_i(p_i'')|\leq \frac{|b_i|}{r_i}\Big\}
\end{gather}
and gluing $V_i'-F_{i,b_\blt}'$ and $V_i''-F_{i,b_\blt}''$ (for all $i$) by the rule
\begin{gather}\label{eq45}
p_i'\in V_i'-F_{i,b_\blt}'\text{ is identified with }p_i''\in V_i''-F_{i,b_\blt}''\quad\Longleftrightarrow\quad \xi_i(p_i')\varpi_i(p_i'')=b_i
\end{gather}
This gives $\mc C_{b_\blt}$ which, together with the marked points $y_\star,x_\blt$ (which remain after sewing) and the local coordinates of $\wtd{\fk X}$, forms the pointed surface $\fk X_{b_\blt}$. Letting $b_i\rightarrow 0$ for some $i$, we get the pointed nodal curve $\fk X_{b_\blt}$ if $b_1\cdots b_R=0$.

\begin{ass}\label{sewingass}
    For each $b_\blt=(b_1,\dots,b_R)\in \MB$ such that $b_1\cdots b_R\neq 0$, each connected component of $\MC_{b_\blt}=\pi^{-1}(b_\blt)$ intersects $\{x_1,\cdots,x_N\}$. In other words, we assume that $\fk X_{b_\blt}$ satisfies Asmp. \ref{ass1}.
\end{ass}

\subsubsection{Details of the sewing construction}

Let us describe the construction of the family $\fx$ in details. The general construction can be found in \cite[Sec. 3]{Gui-sewingconvergence}. We make the identifications
\begin{gather}\label{eq22}
\begin{gathered}
V_i'=\mc D_{r_i}\qquad \text{via }\xi_i\\
V_i''=\mc D_{\rho_i}\qquad \text{via }\varpi_i
\end{gathered}
\end{gather}
The $\xi_i$ and $\varpi_i$ become  the standard coordinate of $\Cbb$ (i.e. the identity maps):
\begin{gather*}
\xi_i:\mc D_{r_i}\xrightarrow{=}\mc D_{r_i}\qquad \varpi_i:\mc D_{\rho_i}\xrightarrow{=}\mc D_{\rho_i}
\end{gather*}
We shall freely switch the orders of Cartesian products. Define 
\begin{gather*}
    \MD_{r_\blt}:=\MD_{r_1}\times \cdots \times \MD_{r_M}\qquad\MD_{\rho_\blt}:=\MD_{\rho_1}\times \cdots \times \MD_{\rho_M}\\
q_i=\xi_i\varpi_i:\MD_{r_i}\times \MD_{\rho_i}\rightarrow \MD_{r_i\rho_i}\qquad (z,w)\mapsto zw
\end{gather*}
Define also $W_i$ and its open subsets $W_i',W_i''$ by
\begin{subequations}\label{eq44}
\begin{gather}
W_i=\MD_{r_i}\times\MD_{\rho_i}\times \prod_{j\ne i} \MD_{r_j\rho_j}\\
W_i'=\MD_{r_i}^\times \times \MD_{\rho_i}\times \prod_{j\ne i} \MD_{r_j\rho_j}\\
W_i''=\MD_{r_i} \times \MD_{\rho_i}^\times\times \prod_{j\ne i} \MD_{r_j\rho_j}
\end{gather}
\end{subequations}
Then we can extend $\xi_i,\varpi_i,q_i$ constantly to 
\begin{subequations}\label{eq29}
\begin{gather}
    \xi_i:W_i\rightarrow \MD_{r_i}  \qquad (z,w,*)\mapsto z\\
    \varpi_i:W_i\rightarrow \MD_{\rho_i}\qquad (z,w,*)\mapsto w\\
q_i:W_i\rightarrow \MD_{r_i\rho_i} \qquad (z,w,*)\mapsto zw
\end{gather}
\end{subequations}
Then we have open holomorphic embeddings
\begin{subequations}
\begin{gather}
(\xi_i,\varpi_i,\idt):W_i\xrightarrow{=} \MD_{r_i}\times \MD_{\rho_i}\times \prod_{j\ne i} \MD_{r_j\rho_j}  \label{eq23}\\
(\xi_i,q_i,\idt):W_i'\rightarrow \MD_{r_i}\times \MD_{r_i\rho_i}\times \prod_{j\ne i} \MD_{r_j\rho_j}\simeq \MD_{r_i}\times \MD_{r_\blt\rho_\blt}   \label{eq24}\\
(\varpi_i,q_i,\idt):W_i''\rightarrow \MD_{\rho_i}\times \MD_{r_i\rho_i}\times \prod_{j\ne i} \MD_{r_j\rho_j}\simeq \MD_{\rho_i}\times \MD_{r_\blt\rho_\blt}   \label{eq25}
\end{gather}
\end{subequations}
The image of \eqref{eq24} resp. \eqref{eq25} is precisely the subset of all $(z_i,b_1,\dots,b_R)\in \MD_{r_i}\times \MD_{r_\blt\rho_\blt}$ resp. $(w_i,b_1,\dots,b_R)\in \MD_{\rho_i}\times \MD_{r_\blt\rho_\blt}$ satisfying 
\begin{align*}
\frac{\vert b_i\vert }{\rho_i}<\vert z_i\vert <r_i \qquad \text{resp.}\qquad \frac{\vert b_i\vert }{r_i}<\vert w_i\vert <\rho_i.
\end{align*}
So closed subsets $F_i'\subset \MD_{r_i}\times \MD_{r_\blt\rho_\blt}$ and $F_i''\subset \MD_{\rho_i}\times \MD_{r_\blt\rho_\blt}$ can be chosen such that we have biholomorphisms
\begin{subequations}
\begin{gather}\label{geosew1}
(\xi_i,q_i,\idt):W_i'\xrightarrow{\simeq} \MD_{r_i}\times \MD_{r_\blt\rho_\blt}-F_i'\\
 (\varpi_i,q_i,\idt):W_i''\xrightarrow{\simeq} \MD_{\rho_i}\times \MD_{r_\blt\rho_\blt}-F_i''
\end{gather}
\end{subequations}
By the identifications \eqref{eq22}, we can write the above maps as
\begin{subequations}\label{eq26}
\begin{gather}
(\xi_i,q_i,\idt):W_i'\xrightarrow{\simeq} V_i'\times \MD_{r_\blt\rho_\blt}-F_i'\qquad \subset\wtd C\times\mc D_{r_\blt\rho_\blt}\\
 (\varpi_i,q_i,\idt):W_i''\xrightarrow{\simeq} V_i''\times \MD_{r_\blt\rho_\blt}-F_i''\qquad \subset \wtd C\times\mc D_{r_\blt\rho_\blt}
\end{gather}
\end{subequations}
In particular, we view $F_i'$ and $F_i''$ as closed subsets of $\wtd C\times \mc D_{r_\blt\rho_\blt}$.

The complex manifold $\MC$ is defined by 
\begin{gather}
\MC=\big(W_1\sqcup\cdots\sqcup W_R\big)\bigsqcup \big(\wtd C\times \MD_{r_\blt \rho_\blt}-\bigcup_{i=1}^R( F_i'\cup F_i'')\big)\Big/\sim
\end{gather}
Here, the equivalence $\sim$ is defined by identifying each subsets $W_i',W_i''$ of $W_i$ with the corresponding open subsets of $\wtd C\times \MD_{r_\blt \rho_\blt}-\bigcup_{i=1}^R( F_i'\cup F_i'')$ via the biholomorphisms \eqref{eq26}.

$\pi:\MC\rightarrow \MB$ is defined as follows. The projection
\begin{align*}
    \wtd C\times \MD_{r_\blt \rho_\blt}\rightarrow \MD_{r_\blt \rho_\blt}=\MB
\end{align*}
agrees with
\begin{align*}
q_i:W_i=\MD_{r_i}\times \MD_{\rho_i}\times \prod_{j\ne i} \MD_{r_j\rho_j}\rightarrow \MD_{r_\blt \rho_\blt}=\MB
\end{align*}
when restricted to $W_i^\prime$ and $W_i''$. These two maps give a well-defined surjective holomorphic map $\pi:\MC\rightarrow \MB$.

Extend $x_i,y_j$ constantly to $\MD_{r_\blt \rho_\blt}=\MB\rightarrow \wtd C\times \MD_{r_\blt \rho_\blt}$, whose image is disjoint from $F_i^\prime$ and $F_i''$ for $1\leq i\leq R$. So $x_i,y_j$ can be extended to sections of $\pi:\MC\rightarrow \MB$.
The local coordinate $\theta_j$ of $\wtd \fx$ at $y_j$ extend constantly to that of $\fx$, also denoted by $\theta_j$. (If a local coordinate $\eta_i$ of $\wtd\fx$ at $x_i$ is chosen, we can also extend it constantly to one $\eta_i$ on $\fk X$.) This completes the definition of \eqref{geosew5}.
 \begin{rem}\label{geosew6}
We warn the readers that $\fx$ is NOT a family of compact Riemann surfaces. When $b_1\cdots b_R=0$, the fiber $\MC_{b_\blt}$ of $b_\blt =(b_1,\cdots,b_R)$ is not a compact Riemann surface, but a \textbf{nodal curve}. (See \cite[Sec. 2]{Gui-sewingconvergence}, or see \cite[Ch. 10]{ACG11} for a relatively complete story about families of nodal curves.) We will not consider such fibers in this article.
 \end{rem}
 \begin{df}
The set
\begin{align*}
\Sigma=\{x\in \mc C:\pi\text{ is not a submersion at }x\}
\end{align*}
is called the \textbf{critical locus}\index{00@Critical locus $\Sigma$} of $\fk X$. In other words, $x$ belongs to $\Sigma$ iff $x$ is not a smooth point of the fiber $\mc C_{\pi(x)}$. Write
\begin{align}
W=\bigsqcup_{i=1}^R W_i \qquad   W'=\bigsqcup_{i=1}^R W_i'\qquad W''=\bigsqcup_{i=1}^R W_i''
\end{align}
It is not hard to see that $\pi$ is a submersion outside $W$, and for each $i$ we have
\begin{align}
W_i\cap\Sigma= \big(\{0\}\times\{0\}\big)\times\prod_{j\ne i} \MD_{r_j\rho_j}\qquad \subset\MD_{r_i}\times \MD_{\rho_i}\times \prod_{j\ne i} \MD_{r_j\rho_j}
\end{align}
Thus, we have
\begin{align}
\Sigma=W-(W'\cup W'')=\bigsqcup_{i=1}^R(W_i-(W_i'\cup W_i''))  \label{eq27}
\end{align}
It is clear that the \textbf{discriminant locus} $\Delta=\pi(\Sigma)$\index{00@Discriminant locus $\Delta=\pi(\Sigma)$} satisfies
\begin{align}
\Delta\xlongequal{\mathrm{def}}\pi(\Sigma)=\{(b_1,\dots,b_R)\in\mc D_{r_\blt\rho_\blt}:b_1\cdots b_R=0\}=\mc D_{r_\blt\rho_\blt}-\mc D_{r_\blt\rho_\blt}^\times
\end{align}
 \end{df}

\subsubsection{The sheaf $\scr V_{\fk X,a_1,\dots,a_M}$}

Since $\fk X=\eqref{geosew5}$ is not a smooth family, the definition of $\scr V_{\fk X,a_\star}$ in Sec. \ref{lb4} does not apply to the current situation. Let us explain how to define this sheaf. The idea is similar to that in \cite{DGT1,DGT2}. We follow the approach in \cite[Sec. 5]{Gui-sewingconvergence}.

Define
\begin{align*}
\scr V_{\fk X,a_1,\dots,a_M}=\varinjlim_{n\in\Nbb}\scr V_{\fk X,a_1,\dots,a_M}^{\leq n}
\end{align*}
where each $\scr V_{\fk X,a_1,\dots,a_M}^{\leq n}$ is an $\mc O_{\mc C}$-module  defined as follows. Since $\pi:\mc C-\Sigma\rightarrow\mc B$ is a submersion, the sheaf $\scr V_{\fk X-\Sigma,a_1,\dots,a_M}^{\leq n}$ is defined as in Def. \ref{lb5}. Then $\scr V_{\fk X,a_1,\dots,a_M}^{\leq n}$ is an $\mc O_{\mc C}$-submodule of $\scr V_{\fk X-\Sigma,a_1,\dots,a_M}^{\leq n}$ which agrees with $\scr V_{\fk X-\Sigma,a_1,\dots,a_M}^{\leq n}$ outside $\Sigma$. To define $\scr V_{\fk X,a_1,\dots,a_M}^{\leq n}$ near $\Sigma$, it suffices to describe its restriction to each $W_i$. Recall $W_i-\Sigma=W_i'\cup W_i''$ by \eqref{eq27}. \index{VX@$\SV_{\fx,a_1,\cdots,a_M}^{\leq n}=\SV_{\fx,a_\star}^{\leq n}$}

\begin{df}
$\scr V_{\fk X,a_1,\dots,a_M}^{\leq n}|_{W_i}$ is the (automatically free) $\MO_W$-submodule of $\SV_{\fx-\Sigma,a_1,\cdots,a_M}^{\leq n}\vert_{W_i-\Sigma}$ generated by the sections whose restrictions to $W_i^\prime$ and $W_i''$ are 
\begin{align}\label{geosew7}
    \MU_\varrho(\xi_i)^{-1}(\xi_i^{L(0)}v) \qquad \text{resp.}\qquad \MU_\varrho(\varpi_i)^{-1}(\varpi_i^{L(0)}\MU(\upgamma_1)v)
\end{align}
where $\xi_i,\varpi_i$ are defined by \eqref{eq29} and $v\in \Vbb^{\leq n}$. This is well-defined (i.e. the two expressions in \eqref{geosew7} agree on $W_i'\cap W_i''$). See \cite[Sec. 5]{Gui-sewingconvergence}, especially \cite[Lem. 5.2]{Gui-sewingconvergence}.
\end{df}

Let us recall the definition of the \textbf{relative dualizing sheaf} $\omega_{\MC/\MB}$ \index{zz@$\omega_{\MC/\MB}$, the relative dualizing sheaf} which is similar to that of $\scr V_{\fk X,a_\star}^{\leq n}$. When restricted to $\MC-\Sigma$, $\omega_{\MC/\MB}$ is equal to the usual cotangent sheaf defined before. When restricted to each $W_i$ , $\omega_{\MC/\MB}\vert_{W_i}$ is generated freely by the sections whose restrictions to $W_i^\prime$ and $W_i''$ are 
\begin{align}\label{geosew9}
    \xi_i^{-1}d\xi_i\qquad \text{resp.}\qquad -\varpi_i^{-1}d\varpi_i.
\end{align}
Again, this definition is well-defined, since the above two expressions agree on $W_i'\cap W_i''$ by an easy computation of change of coordinates (recall \eqref{eq29}) using \eqref{eq28}.

\begin{pp}\label{nodal1}
Prop. \ref{grauertcor} and  Cor. \ref{grauertcor2} hold verbatim for the family $\fx=\eqref{geosew5}$ defined by sewing $\wtd{\fk X}=\eqref{eq30}$ as in Subsec. \ref{lb6}, except that we assume that the point $b$ is not inside the discriminant locus $\Delta=\pi(\Sigma)$.
\end{pp}

\begin{proof}
The smooth family $\fk X_{\mc B-\Delta}$ satisfies Asmp. \ref{ass1} (cf. Asmp. \ref{sewingass}). Therefore,  applying Prop. \ref{grauertcor} to the family $\fk X_{\mc B-\Delta}$, together with same argument as the proof of Cor. \ref{grauertcor2} (using Grauert's direct image theorem and Cartan's theorem A), proves our goal.
\end{proof}

\subsubsection{Normalized sewing}
Associate an admissible $\Vbb^{\times N}$-module $\Wbb$ to $x_1,\cdots,x_N$ of $\wtd \fx$. Recall (cf. Subsec. \ref{lb6}) that the local coordinates $\xi_\blt$ at $\sgm_\blt'$ and $\varpi_\blt$ at $\sgm_\blt''$ are fixed, but the local coordinates for $x_\blt$ are not fixed. Recall that the sheaf $\SW_{\wtd\fx}(\Wbb)$ is defined in Sec. 2.1.

Associate a \emph{finitely} admissible $\Vbb^{\times R}$-module $\Mbb$ to $\varsigma_1',\cdots,\varsigma_R'$ and the contragredient module $\Mbb^\prime$ to $\varsigma_1'',\cdots,\varsigma_R''$. We identify the vector spaces (recall Def. \ref{lb8})
\begin{equation}\label{sewingidentify1}
\SW_{\wtd \fx}(\Wbb\otimes \Mbb\otimes \Mbb^\prime)=\SW_{\wtd \fx}(\Wbb)\otimes \Mbb\otimes \Mbb'
\end{equation}
such that, for each set of local coordinates $\eta_\blt$ at $x_\blt$, the following diagram commutes: 
\begin{equation*}
    \begin{tikzcd}[column sep=-1ex,row sep=2.5ex]
\mathscr{W}_{\wtd \fx}(\Wbb\otimes \Mbb\otimes \Mbb^\prime) \arrow[rrrr, "="] \arrow[rrdd, "{\mathcal{U}(\eta_\blt,\xi_\blt,\varpi_\blt)}"'] &  &                                                    &  & \SW_{\wtd \fx}(\Wbb)\otimes \Mbb\otimes \Mbb^\prime  \arrow[lldd, "\MU(\eta_\blt)\otimes \ibf"] \\
                                                                                                                                                               &  &                                                    &  &                                                                                                                                                                      \\
                                                                                                                                                               &  & \Wbb\otimes \Mbb\otimes \Mbb^\prime &  &                                                                                                                                                                     
\end{tikzcd}
\end{equation*}
This is possible, since  it is not hard to see that the map $(\mc U(\eta_\blt)\otimes\idt)^{-1}\mc U(\eta_\blt,\xi_\blt,\varpi_\blt)$ is independent of the choice of $\eta_\blt$.

Recall the notation \eqref{eq32}. For each $n_\blt\in \Nbb^R$, let $\{m(n_\blt,a):a\in \FA_{n_\blt}\}$ be a basis of $\Mbb(n_\blt)$ and $\{\widecheck{m}(n_\blt,a):a\in \FA_{n_\blt}\}$ be the dual basis of $\Mbb(n_\blt)^*$. Note that $\FA_{n_\blt}$ is a finite set since $\Mbb$ is finitely admissible. \index{qq@$q_\blt^{\wtd L_\blt(0)}\btr\otimes \btl$}Define 
$$
q_\blt^{\wtd L_\blt(0)}\btr\otimes \btl =\sum_{n_\blt\in \Nbb^R}\sum_{a\in \FA_{n_\blt}}q_1^{n_1}\cdots q_R^{n_R}\cdot m(n_\blt,a)\otimes \widecheck{m}(n_\blt,a)\in (\Mbb\otimes \Mbb^\prime)[[q_1,\cdots,q_R]].
$$
For any partial conformal block
\begin{align*}
\uppsi:\SW_{\wtd \fx}(\Wbb\otimes \Mbb\otimes \Mbb')=\SW_{\wtd \fx}(\Wbb)\otimes \Mbb\otimes \Mbb'\rightarrow \Cbb
\end{align*}
associated to $\wtd \fx$ and the admissible $\Vbb^{\times(N+2R)}$-module $\Wbb\otimes \Mbb\otimes \Mbb'$ of multi-level $a_1,\dots,a_M$, we define a $\Cbb$-linear map
\begin{gather} \label{eq35}
\begin{gathered}
\wtd \MS \uppsi :\SW_{\wtd \fx}(\Wbb)\rightarrow \Cbb[[q_1,\cdots,q_R]]\\
w\mapsto \wtd \MS \uppsi(w)=\uppsi\big(w\otimes q_\blt^{\wtd L_\blt(0)}\btr\otimes \btl\big)
\end{gathered}
\end{gather}
$\wtd \MS \uppsi$ is called the \textbf{normalized sewing of $\uppsi$}.\index{S@$\wtd \MS\uppsi$, the normalized sewing} 

\begin{df}
We say \textbf{$\wtd \MS \uppsi$ converges a.l.u. on $\MB$} if for each $w\in \SW_{\wtd \fx}(\Wbb)$, $\wtd \MS \uppsi(w)$ converges a.l.u. on $\MB=\MD_{r_\blt \rho_\blt}$. (Recall Notation \eqref{eq34}.) In this case, we have 
\begin{align}
\wtd{\mc S}\uppsi(w)\in\mc O(\mc B).  \label{eq36}
\end{align}
\end{df}

\subsubsection{The sewing of a partial conformal block is a partial conformal block}

We continue our discussion from the previous subsection. The $\mc O_{\mc B}$-module $\SW_{\fx}(\Wbb)$ is defined in the same way as in Def. \ref{lb7}. As mentioned before Rem. \ref{geosew6}, if local coordinates $\eta_1,\dots,\eta_N$ of $\wtd\fx$ at $x_1,\dots,x_N$ are picked, then each $\eta_i$ can be extended constantly to local coordinate of $\fk X$ at the section $x_i$, also denoted by $\eta_i$. Thus, we can make the identification
\begin{align}
\scr W_{\fk X}(\Wbb)=\scr W_{\wtd\fx}(\Wbb)\otimes_\Cbb\mc O_{\mc B}
\end{align}
such that for each choice of $\eta_\blt$, the following diagram commutes:
\begin{equation*}
    \begin{tikzcd}[column sep=0ex,row sep=2.5ex]
\mathscr{W}_{\fx}(\Wbb) \arrow[rrrr, "="] \arrow[rrdd, "{\mathcal{U}(\eta_\blt)}"'] &  &                                                    &  & \scr W_{\wtd\fx}(\Wbb)\otimes_\Cbb\mc O_{\mc B}  \arrow[lldd, "\MU(\eta_\blt)\otimes \ibf"] \\
                                                                                                                                                               &  &                                                    &  &                                                                                                                                                                      \\
                                                                                                                                                               &  & \Wbb\otimes_\Cbb\mc O_{\mc B} &  &                                                                                                                                                                     
\end{tikzcd}
\end{equation*}
This is possible, since the lower left map $\mc U(\eta_\blt)$ composed with the inverse $(\mc U(\eta_\blt)\otimes\idt)^{-1}$ of the lower right map is independent of the choice of $\eta_\blt$.

\begin{rem}
With abuse of notations, we also denote $\wtd{\mc S}\uppsi\otimes\idt$ by $\wtd{\mc S}\uppsi$. Thus, in view of \eqref{eq35}, we have an $\mc O_{\mc B}$-module morphism
\begin{align}
\wtd{\mc S}\uppsi:\scr W_{\fk X}(\Wbb)=\scr W_{\wtd\fx}(\Wbb)\otimes_\Cbb\mc O_{\mc B}\rightarrow \Cbb[[q_1,\dots,q_R]]\otimes_\Cbb \mc O_{\mc B}
\end{align}
Thus, if $\wtd{\mc S}\uppsi$ converges a.l.u. on $\mc B$, by \eqref{eq36}, the above morphism becomes
\begin{align}
\wtd{\mc S}\uppsi:\scr W_{\fk X}(\Wbb)\rightarrow  \mc O_{\mc B}
\end{align}
\end{rem}

Recall that we are assuming Asmp.   \ref{sewingass}. Also, recall Rem. \ref{lb9}.
\begin{thm}\label{formalpartialconformal}
Suppose that $\wtd \MS\uppsi$ converges a.l.u. on $\MB$. Then the morphism $\wtd \MS\uppsi:\scr W_{\fk X}(\Wbb)\rightarrow\mc O_{\mc B}$ is a partial conformal block of multi-level $a_1,\dots,a_M$ outside the discriminant locus $\Delta$, i.e.
\begin{align*}
\wtd \MS\uppsi|_{\mc B-\Delta}\in\scr T_{\fk X_{\mc B-\Delta},a_1,\dots,a_M}^*(\Wbb)
\end{align*} 
Equivalently (cf.  Prop. \ref{familyfiber5}),  $\wtd \MS\uppsi|_b\in\scr T_{\fk X_b,a_\star}^*(\Wbb)$ for each $b\in\mc B-\Delta$.
\end{thm}
To prove Thm. \ref{formalpartialconformal}, we need the following analogue of \cite[Thm. 10.3]{Gui-sewingconvergence}. Recall \eqref{eq38} for the meanings of the notations $\scr J,\scr J^\pre$.
\begin{pp}\label{formal3}
    Suppose $\uppsi\in \ST_{\wtd \fx,a_1,\cdots,a_M}(\Wbb\otimes \Mbb\otimes \Mbb')$. Then $\wtd \MS \uppsi$ vanishes on
\begin{align*}
\scr J^\pre(\mc B)=   H^0\big(\MC,\SV_{\fx,a_1,\cdots,a_M}\otimes \omega_{\MC/\MB}(\blt S_\fx)\big)\cdot \SW_\fx(\Wbb)(\MB).
\end{align*}
\end{pp}

The proof is similar to  \cite[Thm. 10.4]{Gui-sewingconvergence}. So we omit some details to make the proof not extremely long. To explain the ideas, we assume $R=2$ in the proof for simplicity.

\begin{proof}
Step 1. We claim that for each unital commutative $\Cbb$-algebra $A$, for each   $u\in\Vbb$, and $f\in A[[\xi_1,\varpi_1,q_2]]$,  the following two elements of $(\Mbb\otimes_\Cbb\Mbb'\otimes_\Cbb A)[[q_1,q_2]]$ are equal:
\begin{align}\label{eq41}
\begin{aligned}
&\Res_{\xi_1=0}~Y_{\Mbb,1}\big(\xi_1^{L(0)}u,\xi_1\big)q_\blt^{\wtd L_\blt(0)}\btr\otimes \btl\cdot f\big(\xi_1,\frac{q_1}{\xi_1},q_2\big)\frac{d\xi_1}{\xi_1}\\
=&\Res_{\varpi_1=0}~q_\blt^{\wtd L_\blt(0)}\btr\otimes Y_{\Mbb',1}\big(\varpi_1^{L(0)}\mc U(\upgamma_1)u,\varpi_1\big)\btl\cdot f\big(\frac{q_1}{\varpi_1},\varpi_1,q_2\big)\frac{d\varpi_1}{\varpi_1}
\end{aligned}
\end{align}
See \cite[(10.2,10.3)]{Gui-sewingconvergence} for the meaning of the expressions $T\btr\otimes\btl$ and $\btr\otimes T\btl$ (which equals $T^\tr\btr\otimes\btl$ where $T^\tr$ is the transpose of $T$) if $T$ is a linear operator. Recall $\mc U(\upgamma_1)=e^{L(1)}(-1)^{L(0)}$ (cf. \eqref{eq40}).

The proof is similar to \cite[Lem. 10.2]{Gui-sewingconvergence}. First, one may simplify discussions by evaluating the two sides \eqref{eq41} with $\mbf m'\otimes\mbf m$ for each $\mbf m\in\Mbb,\mbf m'\in\Mbb'$ to get elements of $A[[q_1,q_2]]$. As in \cite[(10.13)]{Gui-sewingconvergence}, one proves
\begin{align}
Y_{\Mbb,1}\big(\xi_1^{L(0)}u,\xi_1\big)q_\blt^{\wtd L_\blt(0)}\btr\otimes \btl=q_\blt^{\wtd L_\blt(0)}\btr\otimes Y_{\Mbb',1}\big((q_1/\xi_1)^{L(0)}\mc U(\upgamma_1)u,q_1/\xi_1\big)\btl  \label{eq42}
\end{align}
as elements of $(\Mbb\otimes\Mbb')[[\xi_1^{\pm1},q_1^{\pm1},q_2^{\pm1}]]$. In other words,
\begin{align}
\bigbk{\mbf m',Y_{\Mbb,1}\big(\xi_1^{L(0)}u,\xi_1\big)q_\blt^{\wtd L_\blt(0)}\mbf m}=\bigbk{Y_{\Mbb',1}\big((q_1/\xi_1)^{L(0)}\mc U(\upgamma_1)u,q_1/\xi_1\big)q_\blt^{\wtd L_\blt(0)}\mbf m', \mbf m}
\end{align}
which follows from \eqref{eq39} and \eqref{eq43}, as indicated in the formulas before (10.13) of \cite{Gui-sewingconvergence}.

Now, we define an element of $(\Mbb\otimes\Mbb')[[\xi_1^{\pm1},\varpi_1^{\pm1},q_2^{\pm1}]]$ to be
\begin{subequations}
\begin{gather}
C(\xi_1,\varpi_1,q_2)=Y_{\Mbb,1}\big(\xi_1^{L(0)}u,\xi_1\big)(\xi_1\varpi_1)^{\wtd L_1(0)}q_2^{\wtd L_2(0)}\btr\otimes \btl
\end{gather}
By \eqref{eq42}, we have
\begin{gather}
C(\xi_1,\varpi_1,q_2)=(\xi_1\varpi_1)^{\wtd L_1(0)}q_2^{\wtd L_2(0)}\btr\otimes Y_{\Mbb',1}\big(\varpi_1^{L(0)}\mc U(\upgamma_1)u,\varpi_1\big)\btl
\end{gather}
\end{subequations}
Set $D=f\cdot C$. As shown in the proof of \cite[Lem. 10.2]{Gui-sewingconvergence}, we have the general fact about series:
\begin{align}
\begin{aligned}
\Res_{\xi_1=0}~ D\big(\xi_1,\frac{q_1}{\xi_1},q_2\big)\frac{d\xi_1}{\xi_1}=\Res_{\varpi_1=0}~D\big(\frac{q_1}{\varpi_1},\varpi_1,q_2\big)\frac{d\varpi_1}{\varpi_1}
\end{aligned}
\end{align}
This proves \eqref{eq41}. If we exchange $1$ and $2$, a similar description holds true.\\

Step 2. Define divisors of $\mc C$ and $\wtd C$:
\begin{gather*}
S_\fx=\sum_{i=1}^N x_i(\MB)\qquad S_{\wtd \fx}=\sum_{i=1}^N x_i\qquad T_{\wtd \fx}=\sum_j\varsigma_j'+\sum_j\sgm_j''
\end{gather*}
Choose $v\in H^0\big(\MC,\SV_{\fx,a_1,\cdots,a_M}\otimes \omega_{\MC/\MB}(\blt S_\fx)\big)$. We claim that we have a formal power series expansion
    \begin{equation}\label{formal1}
    v=\sum_{m,n\in \Nbb}v_{m,n}q_1^m q_2^n,
    \end{equation}
    where $v_{m,n}\in H^0\big(\wtd C,\SV_{\wtd \fx,a_1,\cdots,a_M}\otimes \omega_{\wtd C} (\blt (S_{\wtd \fx}+T_{\wtd \fx}))\big)$.
     
Choose any precompact open subset $U\subset \wtd C$ disjoint from $\varsigma_1',\varsigma_2',\varsigma_1'',\varsigma_2''$. We can find small enough $0<\epsilon<r,0<\lambda<\rho$ such that $U\times \MD_{\epsilon \lambda}\subset \MC$ is disjoint from the sewing parts. This means the restriction of $\pi:\MC\rightarrow \MB$ to $U\times \MD_{\epsilon\lambda}$ equals $\wtd \pi:\wtd C\times \MD_{r\rho}\rightarrow \MD_{r\rho}$. The section $v\vert_{U\times \MD_{\epsilon\lambda}}$ of $\SV_{\fx,a_1,\cdots,a_M}\otimes \omega_{\MC/\MB}(\blt S_\fx)$ can be regarded as a section of $\SV_{\wtd \fx\times \MD_{r\rho}}\otimes \omega_{\wtd C\times \MD_{r\rho}/\MD_{r\rho}}(\blt S_\fx)$. By taking power series expansion at $q_\blt =0$, $v\vert_{U\times \MD_{\epsilon\lambda}}$ can be regarded as an element in $\big(\SV_{\wtd \fx,a_1,\cdots,a_M}\otimes \omega_{\wtd C}(\blt (S_{\wtd \fx}+T_{\wtd \fx}))\big)(U)[[q_1,q_2]]$. This defines (\ref{formal1}) where $v_{m,n}\in H^0\big(\wtd C-\{\varsigma_1',\varsigma_2',\varsigma_1'',\varsigma_2''\},\SV_{\wtd \fx,a_1,\cdots,a_M}\otimes \omega_{\wtd C} (\blt (S_{\wtd \fx}+T_{\wtd \fx}))\big)$. 

Recall by \eqref{eq44} that
\begin{gather*}
W_1=\mc D_{r_1}\times\mc D_{\rho_1}\times\mc D_{r_2\rho_2}\qquad W_1'=\mc D_{r_1}^\times\times\mc D_{\rho_1}\times\mc D_{r_2\rho_2} \qquad W_1''=\mc D_{r_1}\times\mc D_{\rho_1}^\times\times\mc D_{r_2\rho_2}
\end{gather*}
and that exchanging the subscripts $1,2$ above gives the description of $W_2,W_2',W_2''$. By the description of $\SV_{\fx,a_1,\cdots,a_M}$ and $\omega_{\MC/\MB}$ on $W$ in \eqref{geosew7}, \eqref{geosew9},  $v|_{W_1-\Sigma}$ is a sum of sections whose restrictions to $W_1'$ resp. $W_1''$ under the trivializations $\mc U_\varrho(\xi_1)$ resp. $\mc U_\varrho(\varpi_1)$ are
\begin{align}\label{equation100}
f\big(\xi_1,\frac{q_1}{\xi_1},q_2\big)\xi_1^{L(0)}u\cdot \frac{d\xi_1}{\xi_1} \quad \text{resp.} \quad -f\big(\frac{q_1}{\varpi_1},\varpi_1,q_2\big)\varpi_1^{L(0)}\mc U(\upgamma_1)u\cdot\frac{d\varpi_1}{\varpi_1} 
\end{align}
where $u\in\Vbb$ and $f=f(\xi_1,\varpi_1,q_2)\in \mc O(W_1)$. Exchanging the roles of $1,2$ gives the description of $v|_{W_2-\Sigma}$. This shows  that the $v_{m,n}$ in \eqref{formal1} has poles of orders at most $m+1$ (resp. $n+1$) at $\varsigma_1',\varsigma_1''$ (resp. $\varsigma_2',\varsigma_2''$). This completes the construction of (\ref{formal1}). \\
     
Step 3. By setting $A=\Cbb$ in \eqref{eq41} and using the fact that $v|_{W_1-\Sigma}$ is a finite sum of vectors of the form \eqref{equation100}, we have the following equation of elements in $(\Mbb\otimes \Mbb')[[q_1,q_2]]$:
    \begin{equation}\label{formal2}
    \sum_{m,n\in \Nbb}\big(v_{m,n}\cdot q_\blt^{\wtd L_\blt(0)}\btr\otimes \btl+q_\blt^{\wtd L_\blt(0)}\btr\otimes v_{m,n}\cdot \btl\big)q_1^m q_2^n =0,
    \end{equation}
    where the residue action of $v_{m,n}$ on $\Mbb$ and $\Mbb'$ are as in \eqref{eq73} using local coordinates $\xi_1,\xi_2,\varpi_1,\varpi_2$. 
    
    On the other hand, since $\uppsi\in \ST_{\wtd \fx,a_\star}^*(\Wbb\otimes \Mbb\otimes \Mbb')$, for each $w\in \Wbb$, considered as a constant section of $\Wbb\otimes \MO(\MB)$, the element $A_{m,n}\in \Cbb[[q_1,q_2]]$ defined by 
    $$
    A_{m,n}:=\uppsi\big((v_{m,n}\cdot w)\otimes (q_\blt^{\wtd L_\blt(0)}\btr\otimes \btl)\big)+\uppsi\big(w\otimes (v_{m,n}\cdot q_\blt^{\wtd L_\blt(0)}\btr\otimes \btl)\big)+\uppsi\big(w\otimes ( q_\blt^{\wtd L_\blt(0)}\btr\otimes v_{m,n}\cdot\btl)\big)
    $$
    equals 0. By \eqref{formal1} and \eqref{formal2}, we have 
    $$
    0=\sum_{m,n\in \Nbb}A_{m,n}q_1^m q_2^n =\sum_{m,n}\uppsi\big((v_{m,n}\cdot w)\otimes (q_\blt^{\wtd L_\blt(0)}\btr\otimes \btl)\big)q_1^m q_2^n=\wtd \MS \uppsi(v\cdot w),
    $$
    which proves $\wtd \MS \uppsi$ vanishes on $\scr J^\pre(\mc B)$.
\end{proof}

\begin{proof}[\textbf{Proof of Thm. \ref{formalpartialconformal}}]
Choose any $b\in\mc B-\Delta$. Since $\scr W_{\fk X}(\Wbb)\simeq\Wbb\otimes_\Cbb\mc O_{\mc B}$ via $\mc U(\eta_\blt)$, $\scr W_{\fk X}(\Wbb)(\mc B)$  generates the fiber $\scr W_{\fk X}(\Wbb)|_b\simeq\scr W_{\fk X_b}(\Wbb)$. This fact, together with Prop. \ref{nodal1} and Cor. \ref{grauertcor2}, shows that $\scr J^\pre(\mc B)$ generates $J(b)=\eqref{eq21}$. Thus, since $\wtd{\mc S}\uppsi$ vanishes on $\scr J^\pre(\mc B)$ by Prop. \ref{formal3}, $\wtd{\mc S}\uppsi|_b$ vanishes on $J(b)$. So $\wtd \MS\uppsi|_b\in\scr T_{\fk X_b,a_\star}^*(\Wbb)$. 
\end{proof}

\subsection{The main example of sewing related to propagation}\label{sewingeg3}

    Let $\FC=(y_1,\cdots,y_M;\theta_1,\cdots,\theta_M\big|C\big|x_1,\cdots,x_N)$ be an $(M,N)$-pointed compact Riemann surface with outgoing local coordinates. Assume we have local coordinates $\eta_1,\cdots,\eta_N$ defined on mutually disjoint neighborhoods $V_1,\cdots,V_N$ of $x_1,\cdots,x_N$. Moreover, we choose $r_1,\dots,r_N>0$ and assume that
\begin{gather*}
\eta_i(V_i)=\MD_{r_i}\qquad\text{for each }1\leq i\leq N\\
\text{$y_1,\cdots,y_M,V_1,\cdots, V_N$ are mutually disjoint}
\end{gather*}

We let
\begin{gather*}
\fk P_1=(\Pbb^1|0,1,\infty)\qquad \fk P_2=\cdots=\fk P_N=(\Pbb^1|0,\infty)\\
\wtd{\fk X}=\fk C\sqcup \fk P_1\sqcup\fk P_2\sqcup\cdots\sqcup \fk P_N
\end{gather*}
where all marked points except $y_1,\cdots,y_M$ are incoming marked points. We will write $0,\infty$ as $0_i,\infty_i$ if we want to emphasize that they are the corresponding points in $\fk P_i$. Let $\zeta$ be the standard coordinate of $\Cbb$. Then $\FP_1$ is equipped with local coordinates $\zeta,\zeta-1,1/\zeta$ and $\FP_2,\dots,\FP_N$ are equipped with $\zeta,1/\zeta$. Then each incoming marked point of $\wtd \fx$ is equipped with a local coordinate. 

We sew $\wtd \fx$ along $N$ pairs of points $(x_1,\infty_1),\cdots ,(x_N,\infty_N)$ using local coordinates $(\eta_1,1/\zeta),\cdots,(\eta_N,1/\zeta)$ (cf. Fig. \ref{fig1})  to get a family $\fk X$ with base $\mc D_{r_\blt}$.

\begin{figure}[h]
	\centering
\scalebox{0.85}{

\tikzset{every picture/.style={line width=0.75pt}} 

\begin{tikzpicture}[x=0.75pt,y=0.75pt,yscale=-1,xscale=1]

\draw   (209.31,77.75) .. controls (229.31,78.55) and (246.91,96.95) .. (246.11,124.95) .. controls (245.31,152.95) and (229.64,155.09) .. (218.97,156.42) .. controls (208.3,157.76) and (176.1,144.48) .. (143.71,142.55) .. controls (111.32,140.61) and (86.51,161.55) .. (76.51,146.55) .. controls (66.51,131.55) and (66.11,111.35) .. (80.51,94.55) .. controls (94.91,77.75) and (116.31,98.55) .. (145.31,95.35) .. controls (174.31,92.15) and (189.31,76.95) .. (209.31,77.75) -- cycle ;
\draw    (96.97,120.42) .. controls (109.64,110.42) and (119.64,113.09) .. (128.3,121.76) ;
\draw    (100.97,118.42) .. controls (108.3,123.76) and (116.3,123.09) .. (123.64,118.42) ;
\draw    (147.64,119.76) .. controls (160.3,109.76) and (170.3,112.42) .. (178.97,121.09) ;
\draw    (151.64,117.76) .. controls (159.64,123.76) and (166.3,122.42) .. (174.3,117.76) ;
\draw  [fill={rgb, 255:red, 0; green, 0; blue, 0 }  ,fill opacity=1 ] (82.36,105.39) .. controls (82.36,104.49) and (83.1,103.76) .. (84,103.76) .. controls (84.9,103.76) and (85.64,104.49) .. (85.64,105.39) .. controls (85.64,106.3) and (84.9,107.03) .. (84,107.03) .. controls (83.1,107.03) and (82.36,106.3) .. (82.36,105.39) -- cycle ;
\draw  [fill={rgb, 255:red, 0; green, 0; blue, 0 }  ,fill opacity=1 ] (79.7,131.39) .. controls (79.7,130.49) and (80.43,129.76) .. (81.33,129.76) .. controls (82.24,129.76) and (82.97,130.49) .. (82.97,131.39) .. controls (82.97,132.3) and (82.24,133.03) .. (81.33,133.03) .. controls (80.43,133.03) and (79.7,132.3) .. (79.7,131.39) -- cycle ;
\draw  [fill={rgb, 255:red, 0; green, 0; blue, 0 }  ,fill opacity=1 ] (217.7,90.73) .. controls (217.7,89.82) and (218.43,89.09) .. (219.33,89.09) .. controls (220.24,89.09) and (220.97,89.82) .. (220.97,90.73) .. controls (220.97,91.63) and (220.24,92.36) .. (219.33,92.36) .. controls (218.43,92.36) and (217.7,91.63) .. (217.7,90.73) -- cycle ;
\draw  [fill={rgb, 255:red, 0; green, 0; blue, 0 }  ,fill opacity=1 ] (233.03,118.06) .. controls (233.03,117.16) and (233.76,116.42) .. (234.67,116.42) .. controls (235.57,116.42) and (236.3,117.16) .. (236.3,118.06) .. controls (236.3,118.96) and (235.57,119.7) .. (234.67,119.7) .. controls (233.76,119.7) and (233.03,118.96) .. (233.03,118.06) -- cycle ;
\draw  [fill={rgb, 255:red, 0; green, 0; blue, 0 }  ,fill opacity=1 ] (229.03,144.06) .. controls (229.03,143.16) and (229.76,142.42) .. (230.67,142.42) .. controls (231.57,142.42) and (232.3,143.16) .. (232.3,144.06) .. controls (232.3,144.96) and (231.57,145.7) .. (230.67,145.7) .. controls (229.76,145.7) and (229.03,144.96) .. (229.03,144.06) -- cycle ;
\draw   (283.24,63.05) .. controls (283.24,51.61) and (292.52,42.33) .. (303.95,42.33) .. controls (315.39,42.33) and (324.67,51.61) .. (324.67,63.05) .. controls (324.67,74.48) and (315.39,83.76) .. (303.95,83.76) .. controls (292.52,83.76) and (283.24,74.48) .. (283.24,63.05) -- cycle ;
\draw  [fill={rgb, 255:red, 0; green, 0; blue, 0 }  ,fill opacity=1 ] (288.36,65.39) .. controls (288.36,64.49) and (289.1,63.76) .. (290,63.76) .. controls (290.9,63.76) and (291.64,64.49) .. (291.64,65.39) .. controls (291.64,66.3) and (290.9,67.03) .. (290,67.03) .. controls (289.1,67.03) and (288.36,66.3) .. (288.36,65.39) -- cycle ;
\draw  [fill={rgb, 255:red, 0; green, 0; blue, 0 }  ,fill opacity=1 ] (316.36,61.39) .. controls (316.36,60.49) and (317.1,59.76) .. (318,59.76) .. controls (318.9,59.76) and (319.64,60.49) .. (319.64,61.39) .. controls (319.64,62.3) and (318.9,63.03) .. (318,63.03) .. controls (317.1,63.03) and (316.36,62.3) .. (316.36,61.39) -- cycle ;
\draw  [fill={rgb, 255:red, 0; green, 0; blue, 0 }  ,fill opacity=1 ] (299.03,50.06) .. controls (299.03,49.16) and (299.76,48.42) .. (300.67,48.42) .. controls (301.57,48.42) and (302.3,49.16) .. (302.3,50.06) .. controls (302.3,50.96) and (301.57,51.7) .. (300.67,51.7) .. controls (299.76,51.7) and (299.03,50.96) .. (299.03,50.06) -- cycle ;
\draw   (306.58,118.38) .. controls (306.58,106.94) and (315.85,97.67) .. (327.29,97.67) .. controls (338.73,97.67) and (348,106.94) .. (348,118.38) .. controls (348,129.82) and (338.73,139.09) .. (327.29,139.09) .. controls (315.85,139.09) and (306.58,129.82) .. (306.58,118.38) -- cycle ;
\draw  [fill={rgb, 255:red, 0; green, 0; blue, 0 }  ,fill opacity=1 ] (311.03,118.06) .. controls (311.03,117.16) and (311.76,116.42) .. (312.67,116.42) .. controls (313.57,116.42) and (314.3,117.16) .. (314.3,118.06) .. controls (314.3,118.96) and (313.57,119.7) .. (312.67,119.7) .. controls (311.76,119.7) and (311.03,118.96) .. (311.03,118.06) -- cycle ;
\draw  [fill={rgb, 255:red, 0; green, 0; blue, 0 }  ,fill opacity=1 ] (339.7,118.06) .. controls (339.7,117.16) and (340.43,116.42) .. (341.33,116.42) .. controls (342.24,116.42) and (342.97,117.16) .. (342.97,118.06) .. controls (342.97,118.96) and (342.24,119.7) .. (341.33,119.7) .. controls (340.43,119.7) and (339.7,118.96) .. (339.7,118.06) -- cycle ;
\draw   (275.24,168.38) .. controls (275.24,156.94) and (284.52,147.67) .. (295.95,147.67) .. controls (307.39,147.67) and (316.67,156.94) .. (316.67,168.38) .. controls (316.67,179.82) and (307.39,189.09) .. (295.95,189.09) .. controls (284.52,189.09) and (275.24,179.82) .. (275.24,168.38) -- cycle ;
\draw  [fill={rgb, 255:red, 0; green, 0; blue, 0 }  ,fill opacity=1 ] (281.03,163.39) .. controls (281.03,162.49) and (281.76,161.76) .. (282.67,161.76) .. controls (283.57,161.76) and (284.3,162.49) .. (284.3,163.39) .. controls (284.3,164.3) and (283.57,165.03) .. (282.67,165.03) .. controls (281.76,165.03) and (281.03,164.3) .. (281.03,163.39) -- cycle ;
\draw  [fill={rgb, 255:red, 0; green, 0; blue, 0 }  ,fill opacity=1 ] (308.36,173.39) .. controls (308.36,172.49) and (309.1,171.76) .. (310,171.76) .. controls (310.9,171.76) and (311.64,172.49) .. (311.64,173.39) .. controls (311.64,174.3) and (310.9,175.03) .. (310,175.03) .. controls (309.1,175.03) and (308.36,174.3) .. (308.36,173.39) -- cycle ;
\draw  [color={rgb, 255:red, 74; green, 144; blue, 226 }  ,draw opacity=1 ] [dash pattern={on 1.5pt off 1.5pt}]  (238.87,85.13) -- (275.74,73.05) ;
\draw [shift={(277.64,72.42)}, rotate = 161.85]  [color={rgb, 255:red, 74; green, 144; blue, 226 }  ,draw opacity=1 ][line width=0.75]    (6.56,-1.97) .. controls (4.17,-0.84) and (1.99,-0.18) .. (0,0) .. controls (1.99,0.18) and (4.17,0.84) .. (6.56,1.97)   ;
\draw [shift={(236.97,85.76)}, rotate = 341.85]  [color={rgb, 255:red, 74; green, 144; blue, 226 }  ,draw opacity=1 ][line width=0.75]    (6.56,-1.97) .. controls (4.17,-0.84) and (1.99,-0.18) .. (0,0) .. controls (1.99,0.18) and (4.17,0.84) .. (6.56,1.97)   ;
\draw  [color={rgb, 255:red, 74; green, 144; blue, 226 }  ,draw opacity=1 ] [dash pattern={on 1.5pt off 1.5pt}]  (255.64,117.76) -- (296.3,117.76) ;
\draw [shift={(298.3,117.76)}, rotate = 180]  [color={rgb, 255:red, 74; green, 144; blue, 226 }  ,draw opacity=1 ][line width=0.75]    (6.56,-1.97) .. controls (4.17,-0.84) and (1.99,-0.18) .. (0,0) .. controls (1.99,0.18) and (4.17,0.84) .. (6.56,1.97)   ;
\draw [shift={(253.64,117.76)}, rotate = 0]  [color={rgb, 255:red, 74; green, 144; blue, 226 }  ,draw opacity=1 ][line width=0.75]    (6.56,-1.97) .. controls (4.17,-0.84) and (1.99,-0.18) .. (0,0) .. controls (1.99,0.18) and (4.17,0.84) .. (6.56,1.97)   ;
\draw  [color={rgb, 255:red, 74; green, 144; blue, 226 }  ,draw opacity=1 ] [dash pattern={on 1.5pt off 1.5pt}]  (246.18,149.12) -- (271.1,158.39) ;
\draw [shift={(272.97,159.09)}, rotate = 200.41]  [color={rgb, 255:red, 74; green, 144; blue, 226 }  ,draw opacity=1 ][line width=0.75]    (6.56,-1.97) .. controls (4.17,-0.84) and (1.99,-0.18) .. (0,0) .. controls (1.99,0.18) and (4.17,0.84) .. (6.56,1.97)   ;
\draw [shift={(244.3,148.42)}, rotate = 20.41]  [color={rgb, 255:red, 74; green, 144; blue, 226 }  ,draw opacity=1 ][line width=0.75]    (6.56,-1.97) .. controls (4.17,-0.84) and (1.99,-0.18) .. (0,0) .. controls (1.99,0.18) and (4.17,0.84) .. (6.56,1.97)   ;

\draw (57.33,91.4) node [anchor=north west][inner sep=0.75pt]  [font=\small]  {$y_{1}$};
\draw (53.33,124.07) node [anchor=north west][inner sep=0.75pt]  [font=\small]  {$y_{2}$};
\draw (200.67,83.07) node [anchor=north west][inner sep=0.75pt]  [font=\small]  {$x_{1}$};
\draw (215.33,111.07) node [anchor=north west][inner sep=0.75pt]  [font=\small]  {$x_{2}$};
\draw (211.33,133.73) node [anchor=north west][inner sep=0.75pt]  [font=\small]  {$x_{3}$};
\draw (292.67,62.49) node [anchor=north west][inner sep=0.75pt]  [font=\small]  {$\infty $};
\draw (287.67,32.73) node [anchor=north west][inner sep=0.75pt]  [font=\small]  {$1$};
\draw (330,54.73) node [anchor=north west][inner sep=0.75pt]  [font=\small]  {$0$};
\draw (314,105.4) node [anchor=north west][inner sep=0.75pt]  [font=\small]  {$\infty $};
\draw (352.67,112.73) node [anchor=north west][inner sep=0.75pt]  [font=\small]  {$0$};
\draw (279.33,166.4) node [anchor=north west][inner sep=0.75pt]  [font=\small]  {$\infty $};
\draw (318.67,171.78) node [anchor=north west][inner sep=0.75pt]  [font=\small]  {$0$};

\end{tikzpicture}
}
	\caption{}
	\label{fig1}
\end{figure}

\begin{rem}
Let us visualize this sewing construction on each smooth fiber. Choose
\begin{align*}
b_\blt=(b_1,\dots,b_N)\in\mc D_{r_\blt}^\times\equiv\mc D_{r_1}^\times\times\cdots\times\mc D_{r_N}^\times=\mc B-\Delta
\end{align*}
The fiber $\fk X_{b_\blt}$ is obtained by discarding small discs around $x_1,\dots,x_N\in\fk C$ and $\infty_1\in\fk P_1,\dots,\infty_N\in\fk P_N$, and filling the $N$ holes of $\fk C$ using the remaining parts of $\fk P_1,\dots,\fk P_N$ by identifying each $\gamma_i\in\fk P_i$ outside the discarded part with the point $p_i=\eta_i^{-1}(b_i\gamma_i)$ of $\fk C$ (since $\eta_i(p_i)\cdot 1/\zeta(\gamma_i)=b_i$, cf. \eqref{eq45}). The $N$ holes of $\fk C$ have been filled. So $\fk X_{b_\blt}$ as a Riemann surface is equivalent to $C$. The original $x_i\in\fk C$ is discarded, and $0_i\in\fk P_i$ becomes the point $x_i$ on $C$. The original $1\in\fk P_1$ becomes $\eta_1^{-1}(b_1)$ on $C$.

The local coordinates at the marked points $y_1,\cdots,y_M,\eta_1^{-1}(b_1),x_1,\cdots,x_N$ of $\fk X_{b_\blt}$ are
\begin{gather}\label{eq46}
\begin{gathered}
\theta_j\text{ at }y_j\qquad b_1^{-1}\eta_1-1\text{ at }\eta_1^{-1}(b_1)\qquad b_i^{-1}\eta_i\text{ at }x_i
\end{gathered}
\end{gather}
\end{rem}

\begin{rem}
According to the above remark, when restricted to $\mc D_{r_\blt}^\times$, we have
    $$
    \fx_{\MD_{r_\blt}^\times}=\big(y_1,\cdots,y_M;\theta_1,\cdots,\theta_M\big|\pi:C\times \MD_{r_\blt}^\times \rightarrow \MD_{r_\blt }^\times\big|\mu,x_1,\cdots,x_N\big),
    $$
    where $\pi$ is the projection onto the $\MD_{r_\blt}$-component, $x_1,\cdots,x_N,y_1,\cdots,y_M$ are sections sending $b_\blt=(b_1,\dots,b_N)$ to $(x_1,b_\blt),\cdots,(x_N,b_\blt),(y_1,b_\blt),\cdots,(y_M,b_\blt)$, and $\mu$ sends $b_\blt$ to $(\eta_1^{-1}(b_1),b_\blt)$. The local coordinates can be determined fiberwisely by \eqref{eq46}.
\end{rem}

Choose a \emph{finitely} admissible $\Vbb^{\times N}$-module $\Wbb$ with contragredient module $\Wbb'$. Note that $\wtd{\fk X}$ has incoming marked points $x_1,\dots,x_N,0_1,\dots,0_N,1,\infty_1,\dots,\infty_N$ where the  $1$ in the middle belongs to $\fk P_1$. We associate:
\begin{gather}
\begin{gathered}
\Wbb\text{ to the set of points }x_1,\dots,x_N\\
\Wbb\text{ to the set of points }0_1,\dots,0_N\\
\Vbb\text{ to the point }1\\
\Wbb'\text{ to the set of points }\infty_1,\dots,\infty_N
\end{gathered}
\end{gather}
Identify (where all $\otimes$ are over $\Cbb$)
    \begin{gather*}
        \SW_{\FC}(\Wbb)=\Wbb,    \qquad    \SW_{\fk P_1\sqcup\cdots\sqcup\fk P_N}(\Wbb\otimes \Vbb\otimes \Wbb')=\Wbb\otimes \Vbb\otimes \Wbb'\\
        \SW_{\wtd{\fk X}}(\Wbb\otimes\Wbb\otimes\Vbb\otimes\Wbb')=\Wbb\otimes\Wbb\otimes\Vbb\otimes\Wbb'\\
        \SW_{\fx}(\Vbb\otimes \Wbb)=\Vbb\otimes \Wbb\otimes \MO_{\MD_{r_\blt}}
    \end{gather*}
via the trivializations defined by the chosen local coordinates. 

Let $\upphi:\Wbb\rightarrow \Cbb$ be an element of $\scr T_{\fk C,a_1,\dots,a_M}^*(\Wbb)$, and let
\begin{gather}
\begin{gathered}
       \upomega:\Wbb\otimes \Vbb\otimes \Wbb' \rightarrow \Cbb\\
        w\otimes u\otimes w'\mapsto \<Y_1(u,1)w,w'\>=\sum_{n\in \Zbb}\<Y_1(u)_n w,w'\>
\end{gathered}
\end{gather}
which is a conformal block associated to $\FP_1\sqcup \cdots\sqcup\FP_N$. Then $\uppsi:=\upphi\otimes \upomega$ is a partial conformal block for $\wtd \fx$ of multi-level $a_1,\dots,a_M$ and its normalized sewing equals 
     \begin{equation}\label{sewingeg1}
     \wtd \MS \uppsi(u\otimes w)=q_2^{\wtd \wt_2(w)}\cdots q_N^{\wtd \wt_N(w)}\sum_{n\in \Zbb}q_1^{\wt(u)+\wtd \wt_1(w)-n-1}\upphi\big(Y_1(u)_n w\big)
     \end{equation}
if the vectors are homogeneous. (Note that $\wtd\wt_1\big(Y_1(u)_nw\big)=\wt(u)+\wtd \wt_1(w)-n-1$.) The a.l.u. convergence of \eqref{sewingeg1} on $\MD_{r_\blt}$ is clearly equivalent to the absolute convergence of 
     \begin{equation}\label{sewingeg2}
         \sum_{n\in \Zbb}q_1^{\wt(u)+\wtd \wt_1(w)-n-1}\upphi\big(Y_1(u)_n w\big)
     \end{equation}
     on $\vert q_1\vert<r_1$. By Thm. \ref{formalpartialconformal}, we immediately conclude:
\begin{pp}\label{lb10}
If (\ref{sewingeg2}) converges absolutely on $\vert q_1\vert<r_1$ for all $u$ and $w$, then $\wtd\MS\uppsi|_{\mc D_{r_\blt}^\times}\in\scr T_{\fk X_{\mc D_{r_\blt}^\times}}^*(\Vbb\otimes\Wbb)$. Equivalently,  $\wtd \MS\uppsi|_{q_\blt=b}$ belongs to $\scr T_{\fk X_b,a_\star}^*(\Vbb\otimes\Wbb)$ for each $b\in\mc D_{r_\blt}^\times$.
\end{pp}

The local coordinates in \eqref{eq46} are not directly applicable to propagation. We define $\fk Y$ to be the same as $\fk X_{b_\blt}$ as $(M,N)$-pointed surface but with different local coordinates: 
\begin{gather}
\begin{gathered}
\fk Y_{b_1}=(y_1,\dots,y_M;\theta_1,\dots,\theta_M|C|\eta_1^{-1}(b_1),x_1,\dots,x_N)\\
\text{local coordinates: }\qquad\theta_j\text{ at }y_j\qquad \eta_1-b_1\text{ at }\eta_1^{-1}(b_1)\qquad \eta_i\text{ at }x_i
\end{gathered}
\end{gather}
Then by Prop. \ref{lb10} and \eqref{eq47}, under the identification $\scr W_{\fk Y_{b_1}}(\Vbb\otimes\Wbb)=\Vbb\otimes\Wbb$ via the trivialization defined by the local coordinates of $\fk Y_{b_1}$ we have:

\begin{co}\label{lb14}
Suppose that for each $u\in\Vbb,w\in\Wbb$,
\begin{align}
\wr\upphi(u\otimes w)=\sum_{n\in\Zbb} q_1^{-n-1}\upphi(Y_1(u)_nw)
\end{align}
converges absolutely when $0<|q_1|<r_1$. Choose $b_1\in\mc D_{r_1}^\times$. Then   $\wr\upphi|_{q_1=b_1}:\Vbb\otimes\Wbb\rightarrow\Cbb$ belongs to $\scr T_{\fk Y_{b_1},a_\star}^*(\Vbb\otimes\Wbb)$.

\end{co}

\subsection{Propagation of partial conformal blocks}
In order to give a module structure on dual fusion products, we introduce propagation of dual fusion products. Fix a family of $(M,N)$-pointed compact Riemann surfaces with outgoing local coordinates $\fx=\big(\tau_1,\cdots,\tau_M;\theta_1,\cdots,\theta_M\big|\pi:\MC\rightarrow \MB\big|\varsigma_1,\cdots,\varsigma_N\big)$ with divisors defined as (\ref{marked}). Recall Asmp. \ref{lb11}.

\begin{df}\label{propagatedfamily}
Define
\begin{gather*}
\wr \MC=\MC\times_\MB (\MC- S_\fx-D_\fx)\qquad\wr\MB=\MC-S_\fx-D_\fx
\end{gather*}
The \textbf{propagated family} of $\fx$ is an  $(M,N+1)$-pointed family \index{X@$\wr\fx$: propagation of $\fx$}
\begin{gather*}
\wr \fx=(\wr \tau_1,\cdots, \wr\tau_M;\wr\theta_1,\cdots,\wr\theta_M\big|\wr\pi:\wr \MC\rightarrow \wr\MB\big|\sigma,\wr \varsigma_1,\cdots ,\wr \varsigma_N),
\end{gather*}
defined to be the pullback of $\fk X$ along $\pi:\wr\mc B\rightarrow\mc B$ together with an extra incoming section
\begin{gather*}
\sigma:\wr\MB\rightarrow \wr \MC,x\mapsto(x,x)\qquad\text{is the diagonal map}
\end{gather*}
Let $\pr_1,\pr_2$ be respectively the projection of $\wr\mc C$ onto the first component $\mc C$ and the second one $\wr\mc B$. Then \index{zz@$\wr\sgm_i,\wr\tau_j,\wr\theta_j,\wr\pi$, the pullback of the sections $\sgm_i,\tau_j$, the local coordinate $\theta_j$, and the projection $\pi$ in propagation}
\begin{subequations}
\begin{gather}
\wr\pi=\pr_2:\wr \MC\rightarrow \wr\MB\\
\wr\tau_j:\wr\MB\rightarrow \wr \MC,x\mapsto (\tau_j(\pi(x)),x)\qquad \wr \varsigma_i:\wr \MB\rightarrow \wr \MC,x\mapsto(\varsigma_i(\pi(x)),x) \\
\wr\theta_j=\theta_j\circ \pr_1  \label{eq48}
\end{gather}
\end{subequations}
\end{df}

\begin{rem}\label{fiberpropagation}
Note that if $F:\mc B'\rightarrow\mc B$ and $G:\mc B''\rightarrow\mc B'$ are holomorphic maps, then the pullback of the family $\fk X$ along $G\circ F$ is equal to the pullback along $G$ of the pullback along $F$ of $\fk X$. Thus, from the fact that $\wr\fk X$ is a pullback of $\fk X$ (together with an extra incoming section) and that $\fk X_b$ is the pullback of $\fk X$ along $\{b\}\hookrightarrow \mc B$ for each $b\in\mc B$, one easily sees that
\begin{align}
(\wr\fk X)_{\mc C_b-\SX-\DX}=\wr(\fx_b)  \label{eq58}
\end{align}
where $(\wr\fk X)_{\mc C_b-\SX-\DX}$ is the pullback of $\wr\fk X$ along $\mc C_b-\SX-\DX\hookrightarrow\mc C-\SX-\DX$, i.e. the restriction of $\wr\fk X$ to $\mc C_b-\SX-\DX$.
\end{rem}

Now suppose $\Wbb$ is a \textit{finitely} admissible $\Vbb^{\times N}$-module. Associate $\Wbb$ to $\varsigma_1,\cdots,\varsigma_N$ in $\fx$ and associate $\Vbb\otimes\Wbb$ (which is a finitely admissible $\Vbb^{\times N}$-module) to $\sigma,\wr \varsigma_1,\cdots,\wr\varsigma_N$ in $\wr\fx$. 
\begin{pp}\label{equivalence3}
   For each $a_1,\cdots,a_M\in \Nbb$, there is a canonical isomorphism of $\MO_{\wr \MB}$-modules:
\begin{gather}
\Psi_{\fk X}:   \SW_{\wr\fx}(\Vbb\otimes \Wbb)\xlongrightarrow{\simeq} \big(\SV_{\fx,a_1,\cdots,a_M}\otimes_{\mc O_{\mc C}} \pi^*\SW_\fx(\Wbb)\big)\big|_{\MC-S_\fx-D_\fx}  \label{eq52}
\end{gather}
satisfying the following fact: For every open subset $U\subset\mc C-\SX-\DX$, every  $\mu\in\mc O(U)$ univalent on each fiber $U_b=U\cap\pi^{-1}(b)$ (where $b\in\mc B$), and every  local coordinates $\eta_1\dots,\eta_N$ of $\fk X$ at $\sgm_1,\dots,\sgm_N$, if we define the local coordinates of $\wr{\fk X}$ at the incoming sections $\sigma,\wr\sgm_1,\dots,\wr\sgm_N$ to be $\Delta\mu,\wr\eta_1,\dots,\wr\eta_N$ where
\begin{subequations}\label{eq49}
\begin{gather}
\Delta\mu(x,y)=\mu(x)-\mu(y)\qquad \forall (x,y)\in U\times_\MB U\\
\wr\eta_i=\eta_i\circ\pr_1\qquad\text{is the pullback of }\eta_i
\end{gather}
\end{subequations}
then the following diagram commutes.
   \begin{equation}\label{eq55}
\begin{tikzcd}[column sep=-0.5ex, row sep=3ex]
\mathscr{W}_{\wr\mathfrak{X}}(\mathbb{V}\otimes \mathbb{W})\big|_U \arrow[rrrr, "\Psi_\mathfrak{X}","\simeq"'] \arrow[rrdd, "{\mathcal{U}(\Delta\mu,\wr\eta_\bullet)}"'] &  &                                                    &  & \mathscr{V}_{\mathfrak{X},a_1,\cdots,a_M}\otimes_{\mc O_{\mc C}} \pi^*\mathscr{W}_\mathfrak{X}(\mathbb{W})\big|_U  \arrow[lldd, "\mathcal{U}_\varrho(\mu)\otimes \pi^*\mathcal{U}(\eta_\bullet)"] \\
                                                                                                                                                               &  &                                                    &  &                                                                                                                                                                      \\
                                                                                                                                                               &  & \mathbb{V}\otimes \mathbb{W} \otimes_\Cbb \mathcal{O}_U &  &                                                                                                                                                                     
\end{tikzcd}
\end{equation}
\end{pp}

\begin{proof}
Outside $\DX$, the sheaves $\scr V_{\fk X,a_\star}$ and $\scr V_{\fk X}$ are equal. So this proposition follows from \cite[Prop. 6.2]{Gui-propagation}.  
\end{proof} 

Our main result in this section is parallel to \cite[Thm. 7.1]{Gui-propagation}. 
To prove this result, we first  recall the following \textbf{strong residue theorem} for $\fk X$, cf. \cite[Thm. A.1]{Gui-propagation}.

\begin{sett}\label{lb12}
Assume that $\fx$ is equipped with local coordinates $\eta_1,\cdots,\eta_N$ at incoming marked points $\varsigma_1(\MB),\cdots,\varsigma_N(\MB)$. For each $j=1,\cdots, N$, choose a neighborhood $U_j\subset \MC_{\MB}$ of $\varsigma_j(\MB)$ on which a local coordinate $\eta_j$ is defined, and assume that $U_j$ intersects only $\varsigma_j(\MB)$ among $\varsigma_1(\MB),\cdots,\varsigma_N(\MB),\tau_1(\mc B),\dots,\tau_M(\mc B)$. 
\begin{gather}
\text{Identify $U_j=(\eta_j,\pi)(U_j)\subset \Cbb\times \MB$ via the biholomorphism $(\eta_j,\pi)$}
\end{gather}
so that $U_j$ becomes a neighborhood of $\{0\}\times \MB$. Set
\begin{align}
U_j^\circ:=U_j-S_\fx=U_j-(\{0\}\times \MB)
\end{align}
which is a subset of $\Cbb^\times \times \MB$. Let $z$ be the standard coordinate of $\Cbb$.
\end{sett}

\begin{thm}[\textbf{Strong residue theorem}] \label{strongresidue} \label{lb88}
 Let $\scr E$ be a holomorphic vector bundle of finite rank and $\scr E^\vee$ be its dual vector bundle. Assume Setting \ref{lb12}. Assume that for each $1\leq j\leq N$ we have trivialization
    \begin{equation}\label{strongresidue2}
    \scr E\vert_{U_j}\simeq E_j\otimes_\Cbb \MO_{U_j}
    \end{equation}
where $E_j$ is finite dimensional vector space, and write the corresponding dual trivialization as
    \begin{equation}\label{strongresidue3}
    \scr E^{\vee}\vert_{U_j}\simeq E_j^\vee\otimes_\Cbb \MO_{U_j}.
    \end{equation}
Identify $\scr E\vert_{U_j},\scr E^\vee\vert_{U_j}$ with the above trivializations. Consider
    \begin{equation}\label{strongresidue1}
    s_j=\sum_{n\in \Zbb}e_{j,n}\cdot z^n\in (E_j\otimes_\Cbb \MO(\MB))((z))
    \end{equation}
where $e_{j,n}\in E_j\otimes_\Cbb\mc O(\mc B)$.  For each $b\in \MB$, if $\nu_b\in H^0\big(\MC_b,\scr E^\vee\vert_{\MC_b}\otimes \omega_{\MC_b}(\blt S_\fx(b))\big)$ whose series expansion in $U_{j,b}=U_j\cap \MC_b$ is 
    $$
    \nu_b\vert_{U_{j,b}}(z)=\sum_n \phi_{j,n}z^n dz
    $$
    where $\phi_{j,n}\in E_j^\vee$, we define the \textbf{$j$-th residue pairing} to be
\begin{align}
    \Res_j\<s_j,\nu_b\>=\Res_{z=0}~\Bigbk{\sum_n e_{j,n}(b)z^n,\sum_n \phi_{j,n}z^n}dz
\end{align}
Then for arbitrary $s_1,\dots,s_N$ as in \eqref{strongresidue1},  the following are equivalent:
    \begin{enumerate}[label=(\arabic*)]
        \item There exists $s\in H^0\big(\MC,\scr E(\blt S_\fx)\big)$ whose series expansion at $\varsigma_j(\MB)$ is $s_j$ for $1\leq j\leq N$. Namely if we view $s\vert_{U_j}$ as $s\vert_{U_j}(b,z)$, then its series expansion at $z=0$ is \eqref{strongresidue1}.
        \item  For any $b\in \MB$ and any $\nu_b\in H^0\big(\MC_b,\scr E^\vee\vert_{\MC_b}\otimes \omega_{\MC_b}(\blt S_\fx(b))\big)$, 
        $$
        \sum_{j=1}^N \Res_j\<s_j,\nu_b\>=0.
        $$
    \end{enumerate}
\end{thm}

Recall \eqref{eq49} for the notations $\Delta\mu$ and $\wr\eta_\blt$.

\begin{thm}\label{lb71}
Let $\Wbb$ be a \textit{finitely} admissible $\Vbb^{\times N}$-module. Let $\upphi:\SW_\fx(\Wbb)\rightarrow \MO_\MB$ be a partial conformal block associated to $\fx$ and $\Wbb$ of multi-level $a_1,\dots,a_M$. Then there is a unique  $\MO_{\mc C-\SX}$-module morphism 
\begin{subequations}
\begin{align}
\wr\upphi:\scr V_{\fk X,a_1,\dots,a_M}\otimes_{\mc O_{\mc C}}\pi^*\scr W_{\fk X}(\Wbb)\big|_{\mc C-\SX}\rightarrow\mc O_{\mc C-\SX}    \label{eq63}
\end{align}
whose restriction (cf. the identification \eqref{eq52}) to $\wr\mc B=\mc C-\SX-\DX$:
\begin{align}
\wr \upphi:\SW_{\wr\fx}(\Vbb\otimes \Wbb)\rightarrow \MO_{\wr\MB}
\end{align}
\end{subequations}
satisfies the following property:

``Choose any open subset $V\subset\mc B$ such that for each $j$ the restricted family $\fk X_V$ has local coordinate $\eta_j$ at $\sgm_j(V)$, and choose a neighborhood $U_j$ of $\sgm_j(V)$ on which $\eta_j$ is defined such that $U_j$ intersects only $\sgm_j(V)$ among $\sgm_1(V),\dots,\sgm_N(V),\tau_1(V),\dots,\tau_M(V)$. 
\begin{subequations}
\begin{gather}
\text{Identify}\qquad U_j=(\eta_j,\pi)(U_j)\qquad\text{via }(\eta_j,\pi)
\end{gather}
so that $U_j$ is a neighborhood of $\{0\}\times V$ in $\Cbb\times V$. Let
\begin{align}
U_j^\circ:=U_j-\SX=U_j-(\{0\}\times V)
\end{align}
which is a subset of $\Cbb^\times\times V$. Let $z$ be the standard coordinate of $\Cbb$.
\begin{gather}
\text{Identify}\quad    \SW_\fx(\Wbb)\vert_V =\Wbb\otimes_\Cbb \MO_V\qquad \text{via }\MU(\eta_\blt)  \label{eq53}\\
\text{Identify}\quad       \SW_{\wr\fx}(\Vbb\otimes \Wbb)\vert_{U_j^\circ }=\Vbb\otimes \Wbb \otimes_\Cbb \MO_{U_j^\circ}\qquad \text{via }\MU(\Delta \eta_j,\wr\eta_\blt).  \label{sheafequivalence2}
\end{gather}
\end{subequations}
For each $u\in \Vbb,w\in \Wbb$,  consider $w$ as a constant section of $\Wbb\otimes_\Cbb \MO(U_j^\circ)$ and $u\otimes w$ as a constant section of $\Vbb\otimes \Wbb\otimes_\Cbb \MO(U_j^\circ)$. Then the following identity holds in $\MO(V)[[z^{\pm 1}]]$:
    \begin{equation}\label{eq82}
        \wr\upphi(u\otimes w)=\upphi(Y_j(u,z)w).
    \end{equation}
    Here $Y_j(u,z)w=\sum_{n\in \Zbb}Y_j(u)_nw\cdot z^{-n-1}$ is an element in $\Wbb((z))$ and $\wr\upphi(u\otimes w)\in \MO(U_j^\circ)$ is regarded as an element of $\MO(V)[[z^{\pm 1}]]$ by taking Laurent series expansion."

Moreover, we have $\wr\upphi\in\scr T_{\wr{\fk X},a_1,\dots,a_M}^*(\Vbb\otimes\Wbb)$.
\end{thm}

For each $u\otimes w\in \scr V_{\fk X,a_1,\dots,a_M}\otimes_{\mc O_{\mc C}}\pi^*\scr W_{\fk X}(\Wbb)\big|_{\mc C-\SX}$, we shall also write
\begin{align}
\wr\upphi(u,w)=\wr\upphi(u\otimes w).  \label{eq64}
\end{align}

This theorem can be proved in a similar way to that of \cite[Thm. 7.1]{Gui-propagation}, except that one should pay special attention to the outgoing sections $\DX$.  We include a proof below for the reader's convenience.  

\begin{proof}
Step 1. We show that if $V,U_\blt$ are chosen and a morphism $\upphi|_V:\scr W_{\fk X}(\Wbb)|_V\rightarrow \mc O_V$ satisfies \eqref{eq82} (for all $v$ and $w$) for a set of local coordinates $\eta_\blt$, then it satisfies \eqref{eq82} for any other one $\eta_\blt'$ at $\sgm_\blt(V)$. Indeed,  \eqref{eq82} is equivalent to that for each $\nu\in H^0(U_j^\circ,\scr  V_{\fk X}\otimes\omega_{\mc C/\mc B}(\blt\SX))$
\begin{align}
\Res_{\eta_j=0}~\wr\upphi(\nu\otimes w)=\upphi(\nu*_j w)
\end{align}
where the action $\nu*_j w$ is defined as in \eqref{residueaction3} and \eqref{eq50} and is independent of the choice of $\eta_j$ due to Rem. \ref{lb13}.\\

Step 2.  Clearly $\wr\upphi$ is determined by its restriction to $\wr\mc B$. We prove the uniqueness of the restricted propagation.  Note that
\begin{align}
\wr\mc B=\mc C-\SX-\DX=\bigcup_{b\in\mc B}\mc C_b-\SX-\DX\label{eq59}
\end{align}
Thus, for two possible propagations $\wr_1\upphi,\wr_2\upphi$ of $\upphi$,  it suffices to show that their restrictions to $\mc C_b-\SX-\DX$ are equal for each $b\in\mc B$. By \eqref{eq58}, this is equivalent to showing that $\wr_1\upphi,\wr_2\upphi$ are equal when restricted to the propagation $\wr(\fk X_b)$ of $\fk X_b$. Let $\Omega$ be the set of all $x\in\mc C_b-\SX-\DX$ on a neighborhood of which $\wr_1\upphi$ agrees with $\wr_2\upphi$. Then $\Omega$ is open and intersects any connected component of $\mc C_b$ by \eqref{eq82} and Asmp. \ref{lb11}. If $\Gamma$ is a connected open subset of $\mc C_b-\SX-\DX$ intersecting $\Omega$ such that the restriction $\scr W_{\wr(\fk X_b)}(\Vbb\otimes\Wbb)|_\Gamma$ is equivalent to $\Vbb\otimes\Wbb\otimes_\Cbb\mc O_\Gamma$, then by the fact that holomorphic functions on a connected Riemann surface are determined by their values on any non-empty open subset, we have $\Gamma\subset\Omega$. So $\Omega$ is closed, and hence must be $\mc C_b-\SX-\DX$. This proves the uniqueness.\\

Step 3. Now we prove the existence of propagation. If we can construct $\wr\upphi|_V$ for all open $V\subset\mc B$ satisfying the conditions as stated in this proposition, then by the uniqueness proved in Step 2, we may glue all these $\wr\upphi|_V$ together to get $\wr\upphi$. Thus,  we assume without loss of generality that $\mc B=V$  and choose $U_\blt,\eta_\blt$ as stated in this proposition.

By \eqref{eq53}, we have
\begin{align}\label{eq54}
        \SV_{\fx,a_\star}\otimes \pi^*\SW_\fx(\Wbb)=\SV_{\fx,a_\star}\otimes_\Cbb \Wbb.
\end{align}
By Prop. \ref{equivalence3}, we make identification
\begin{align}\label{eq61}
\SW_{\wr\fx}(\Vbb\otimes \Wbb)=(\SV_{\fx,a_\star}|_{\mc C-\SX-\DX})\otimes_\Cbb \Wbb
\end{align}
By the commutative diagram \eqref{eq55}, we have
\begin{gather}\label{eq56}
\begin{gathered}
\SW_{\wr\fx}(\Vbb\otimes \Wbb)|_{U_j^\circ}=\big(\SV_{\fx,a_\star}|_{U_j^\circ}\big)\otimes_\Cbb \Wbb\xlongequal{\pumpkin}\Vbb\otimes\Wbb\otimes\mc O_{U_j^\circ}\\
\pumpkin:\text{ both via \eqref{sheafequivalence2} and via }\mc U_\varrho(\eta_j)\otimes \idt
\end{gathered}
\end{gather}


For each $k\in \Nbb$, let $\scr E=(\SV_{\fx,a_\star}^{\leq k})^\vee$        be the dual bundle of $\SV_{\fx,a_\star}^{\leq k}$. Then the identification
        \begin{equation}\label{eq57}
        \SV_{\fx,a_\star}^{\leq k}\vert_{U_j}=\Vbb^{\leq k}\otimes_\Cbb \MO_{U_j} \quad \text{via }\MU_\varrho(\eta_j)
        \end{equation}
        is compatible with identification (\ref{strongresidue3}) in Thm. \ref{strongresidue} if we choose the $E_j$ in Thm. \ref{strongresidue}  to be $(\Vbb^{\leq k})^\vee$. Choose $w\in \Wbb$ and let $e_{j,n}\in (\Vbb^{\leq k})^\vee\otimes \MO(\MB)$ be determined by
\begin{align}
e_{j,n}:        u\in \Vbb^{\leq k}\mapsto \upphi(Y_j(u)_{-n-1}w)\in \MO(\MB).
\end{align}
        Define $s_j=\sum_{n\in \Zbb}e_{j,n}z^n$ for $1\leq j \leq N$. By the compatibility of the identifications \eqref{eq56} and \eqref{eq57}, for each $\nu_b\in H^0\big(\MC_b,\SV_{\fx_b,a_\star}\otimes \omega_{\MC_b}(\blt S_\fx(b))\big)$, the $j$-th residue pairing defined in Thm. \ref{strongresidue} equals
        $$
        \Res_j \<s_j,\nu_b\>=\<\upphi\vert_b,\nu_b *_j w\>
        $$
where $w$ is considered as a section of $\SW_{\fx_b}(\Wbb)=\Wbb$. Since $\upphi\vert_b$ vanishes on $H^0\big(\MC_b,\SV_{\fx_b,a_\star}\otimes \omega_{\MC_b}(\blt S_\fx(b))\big)\cdot \SW_{\fx_b}(\Wbb)$ for each $b\in \MB$ due to Prop. \ref{familyfiber5}, we see that
\begin{gather*}
\sum_{j=1}^N \Res_j \<s_j,\nu_b\>=\sum_{j=1}^N\<\upphi\vert_b,\sigma *_j w\>=0,
\end{gather*}
which implies that $s_1,\cdots,s_N$ satisfy (2) in Thm. \ref{strongresidue}. So there exists $s\in H^0\big(\MC,(\SV_{\fx,a_\star}^{\leq k})^\vee(\blt S_\fx)\big)$ as in (1) of Thm. \ref{strongresidue}. Its restriction to $\mc C-\SX$ defines an $\MO_{\MC-S_\fx}$-module morphism $\SV_{\fx,a_\star}^{\leq k}|_{\mc C-\SX-\DX}\otimes w\rightarrow \MO_{\MC-S_\fx-D_\fx}$. Since these morphisms are compatible for different $k$ and $w$, we can extend them $\MO_{\MC-S_\fx}$-linearly to a morphism $\wr\upphi:(\SV_{\fx,a_\star}|_{\mc C-\SX})\otimes_\Cbb \Wbb\rightarrow \MO_{\MC-S_\fx}$, which satisfies (\ref{eq82}).\\

Step 4. Finally, we prove that $\wr \upphi$ is a partial conformal block. By \eqref{eq59} and Prop. \ref{familyfiber5}, it suffices to check that $\wr\upphi|_{\mc C_b-\SX-\DX}$ is a partial conformal block of multi-level $a_\star$ associated to the family \eqref{eq58}. Thus, by focusing on $\fk X_b$ and its propagation $\wr(\fk X_b)=\eqref{eq58}$, we may replace $\fk X$ with $\fk X_b$ so that $\fk X$ is a single pointed surface. By  shrinking $U_j$, we may assume that $\eta_j(U_j)=\MD_{r_j}$ for some $r_j>0$. Since the LHS of \eqref{eq82} is a holomorphic function on the punctured disc $U_j^\circ=\mc D_{r_j}^\times$, the RHS of \eqref{eq82} converges absolutely on $\mc D_{r_j}^\times$. Thus, by Cor. \ref{lb14} and Prop. \ref{familyfiber5}, we have $\wr\upphi\in\scr T_{(\wr{\fk X})_{U_j^\circ},a_1,\dots,a_M}^*(\Vbb\otimes\Wbb)$ for each $j$. Hence $\wr\upphi\in\scr T_{\wr{\fk X},a_\star}^*(\Vbb\otimes\Wbb)$ by Prop. \ref{localglobal1}.
\end{proof}

Recall Def. \ref{lb15}.

\begin{co}\label{propagation3}
Choose $a_1,\dots,a_M\in\Nbb$. Let
$$
\fx=(y_1,\cdots,y_M;\theta_1,\cdots,\theta_M\big|C\big| x_1,\cdots,x_N,z_1,\cdots,z_L)
$$
be an $(M,N+L)$-pointed compact Riemann surface with outgoing local coordinates and 
$$
S_\fx=x_1+\cdots+x_N+z_1+\cdots+z_L,\quad D_\fx=y_1+\cdots+y_M.
$$
We assume $x_1,\cdots,x_N,z_1,\cdots,z_L$ are incoming marked points and $y_1,\cdots,y_M$ are outgoing ones. Instead of Asmp. \ref{ass1}, we assume a stronger condition:
\begin{align}\label{ass3}
\text{Each connected component of $C$ contains at least one of $x_1,\cdots,x_N$.}
\end{align}
Choose local coordinates $\eta_1,\cdots,\eta_N,\varpi_1,\cdots,\varpi_L$ at $x_1,\cdots,x_N,z_1,\cdots,z_L$. Associate a finitely admissible $\Vbb^{\times N}$-module $\Wbb$ to $x_1,\cdots,x_N$ and a finitely admissible $\Vbb^{\times L}$-module $\Mbb$ to $z_1,\cdots,z_L$. 
\begin{align*}
\text{Identify}\qquad\SW_\fx(\Wbb\otimes \Mbb)=\Wbb\otimes \Mbb\qquad \text{via }\MU(\eta_\blt,\varpi_\blt).
\end{align*}
Suppose that $\Ebb$ is a generating subset of $\Mbb$. Then any partial conformal block $\upphi\in\scr T_{\fk X,a_\star}^*(\Wbb\otimes\Mbb)$ is determined by its values on $\Wbb\otimes \Ebb$.
\end{co}

\begin{proof}
By induction on the $k$ in \eqref{eq60}, it suffices to show that    if $\upphi\in\scr T_{\fk X,a_\star}^*(\Wbb\otimes\Mbb)$ vanishes on $\Wbb\otimes \Ebb$, then $\upphi$ vanishes on $\Wbb\otimes Y_j(u)_n \Ebb$ for each $u\in \Vbb, 1\leq j\leq L,n\in \Zbb$.  As in \eqref{eq61}, we make identification
$$
\begin{aligned}
 \SW_{\wr\fx}(\Vbb\otimes \Wbb\otimes \Mbb)&=\big(\SV_{\fx,a_1,\cdots,a_M}\otimes \pi^*\SW_\fx(\Wbb\otimes \Mbb)\big)\vert_{C-S_\fx-D_\fx}\\
 &=\SV_{\fx,a_1,\cdots,a_M}\vert_{C-S_\fx-D_\fx}\otimes_\Cbb \Wbb\otimes \Mbb.
 \end{aligned}
$$
via Prop. \ref{equivalence3}. By (\ref{ass3}),  the connected component $C_j$ of $C$ containing $z_j$ contains one of $x_1,\dots,x_N$, say $x_1$. For each $w\in \Wbb,m\in \Ebb$, in $\Cbb((z))$ we have
$$
\upphi(Y_1(u,z)w\otimes m)=0,\quad \forall u\in \Vbb.
$$
Thus, by \eqref{eq82}, the $\wr\upphi:\scr W_{\wr\fk X}(\Vbb\otimes\Wbb\otimes\Mbb)\rightarrow\mc O_{C-\SX-\DX}$ vanishes on a neighborhood of $x_1$. As in the proof-Step 2 of Thm. \ref{lb71}, the set $\Omega$ of $x\in C_j-\SX-\DX$ on a neighborhood of which $\wr\upphi$ vanishes is an open and closed subset. So $\Omega=C_j-\SX-\DX$, i.e., $\wr\upphi$ vanishes on $C_j-\SX-\DX$. Therefore, by \eqref{eq82}, in $\Cbb((z))$ we have
$$
\upphi(w\otimes Y_j(u,z)m)=0,
$$
which finishes the proof.
\end{proof}

\begin{co}\label{propagation5}
Let $a_1,\dots,a_M\in\Nbb$. Let $\fx=(y_1,\cdots,y_M;\theta_1,\cdots,\theta_M\big|C\big|x_1,\cdots,x_N)$ be an $(M,N)$-pointed compact Riemann surface with outgoing local coordinates. Associate a finitely admissible $\Vbb^{\times N}$-module $\Wbb$ to $x_1,\cdots,x_N$ and let $\upphi\in \ST_{\fx,a_1,\cdots,a_M}^*(\Wbb)$. Identify
$$
 \SW_{\wr\fx}(\Vbb\otimes \Wbb)=(\SV_{\fx,a_1,\cdots,a_M}\vert_{C-S_\fx-D_\fx})\otimes_\Cbb \SW_\fx(\Wbb)
$$
by Prop. \ref{equivalence3}. Then for each $x\in C-S_\fx-D_\fx$, $\wr\upphi\vert_x$ is the unique linear map $\SV_{\fx,a_1,\cdots,a_M}\vert_x\otimes_\Cbb \SW_\fx(\Wbb)\rightarrow \Cbb$ which belongs to $\scr T_{(\wr\fk X)_x,a_\star}^*(\Vbb\otimes\Wbb)$ and satisfies 
\begin{equation}\label{eq51}
\wr\upphi\vert_x(\ibf\otimes w)=\upphi(w)
\end{equation}
for each $w\in \SW_\fx(\Wbb)$. Thus, we have a linear isomorphism
\begin{gather}
\scr T_{\fx,a_\star}^*(\Wbb)\xrightarrow{\simeq}\scr T_{(\wr\fx)_x,a_\star}^*(\Vbb\otimes\Wbb)\qquad\upphi\mapsto \wr\upphi|_x  \label{eq168}
\end{gather}
\end{co}
\begin{proof}
By Assumption \ref{ass1}, each connected component of $(\wr\fk X)_x$ contains one of the incoming marked points $x_1,\cdots,x_N$. Thus $(\wr\fx)_x$ satisfies  \eqref{ass3} if we set $L=1,z_1=x$. Note that $\ibf$ generates $\Vbb$ as a finitely admissible $\Vbb$-module. Thus, the uniqueness statement in the corollary follows from Cor. \ref{propagation3}.
     
We claim that $\wr\upphi(\ibf\otimes w)=\upphi(w)$ holds as a holomorphic function on $C-S_\fx-D_\fx$. By \eqref{eq82}, this relation holds on punctured disks around $x_1,\dots,x_N$. So it holds on $C-S_\fx-D_\fx$ by complex analysis.

Clearly \eqref{eq168} is injective. For each $\uppsi\in\scr T_{(\wr\fx)_x,a_\star}^*(\Vbb\otimes\Wbb)$, if we let $\upphi:\scr W_\fx(\Wbb)\rightarrow\Cbb$ be $\upphi(w)=\uppsi(\idt\otimes w)$ for each $w\in\scr W_\fx(\Wbb)$, then $\uppsi=\wr\upphi|_x$. So \eqref{eq168} is surjective.
\end{proof}

\subsection{Double propagation}
Fix an $(M,N)$-pointed compact Riemann surface with outgoing local coordinates $\fx=(y_1,\cdots,y_M;\theta_1,\cdots,\theta_M\big|C\big|x_1,\cdots,x_N)$. Here we assume $x_1,\cdots,x_N$ are incoming marked points and $y_1,\cdots,y_M$ are outgoing marked points as usual. Choose neighborhoods $W_1,\cdots,W_M$ of $y_1,\cdots,y_M$, on which the local coordinates $\theta_1,\cdots,\theta_M$ are defined. $S_\fx$ and $D_\fx$ are defined in (\ref{marked2}). 
\begin{df}
    The \textbf{double propagation} $\wr^2 \fx$ of $\fx$\index{zz@$\wr^2\fx$: double propagation of $\fx$} is defined by $\wr (\wr\fx)$, where $\wr\fx$ is the propagated family of $\fx$. It is a family of $(M,N+2)$-pointed compact Riemann surfaces.
\end{df}

\begin{rem}
Recall the notation \eqref{eq69}. We can write $\wr^2\fx$ in details as
\begin{align}
\begin{aligned}
\wr^2\fx=&(\wr^2 y_1,\cdots,\wr^2 y_M;\wr^2\theta_1,\dots,\wr^2\theta_M\big|\\
&\wr^2 \pi:C\times \Conf^2(C-S_\fx-D_\fx)\rightarrow \Conf^2(C-S_\fx-D_\fx)~\big|\\
        &\sigma_1,\sigma_2,\wr^2 x_1,\cdots,\wr^2 x_N)
\end{aligned}
\end{align}
where $\wr^2\pi$ is the projection onto the second component. The sections
\begin{subequations}
\begin{gather*}
        \wr^2 x_i,\wr^2 y_j, \sigma_k:\Conf^2(C-S_\fx-D_\fx)\rightarrow C\times \Conf^2(C-S_\fx-D_\fx)
\end{gather*}
are defined by
\begin{gather}
\wr^2 x_i(z_1,z_2)=(x_i,z_1,z_2)\\
\wr^2 y_j(z_1,z_2)=(y_j,z_1,z_2)\\
\sigma_k(z_1,z_2)=(z_k,z_1,z_2)
\end{gather}
$\wr^2 y_\blt$ are the outgoing sections, and $\sigma_\blt,\wr^2x_\blt$ are the incoming sections. 
\begin{gather}
\wr^2\theta_j=\theta_j\circ\pr_1
\end{gather}
(where $\pr_1:C\times\Conf^2(C-\SX-\DX)\rightarrow C$ is the projection onto the first component) is the  local coordinate of $\wr^2\fk X$ at $y_j\times \Conf^2(C-S_\fx-D_\fx)$.
\end{subequations}
\end{rem}

\begin{rem}
Suppose that the local coordinates $\eta_1,\cdots,\eta_N$ at the incoming marked points $x_1,\cdots,x_N$ are chosen. Then the local coordinates of $\wr^2\fx$ at the incoming marked points $\sigma_1,\sigma_2,\wr^2 x_1,\cdots,\wr^2 x_N$ can be described as follows. The local coordinates at $x_i\times \Conf^2(C-S_\fx-D_\fx)$ are defined by
\begin{subequations}
\begin{gather}
\wr^2 \eta_i=\eta_i\circ\pr_1
\end{gather}
Suppose $V$ is an open subset of $C-S_\fx-D_\fx$, which admits a univalent function $\mu\in \MO(V)$. Then the local coordinate $\Delta_k\mu$ of the restricted family $(\wr^2 \fx)_V$ at $\sigma_k(V)$ is defined by 
\begin{gather}
\Delta_k\mu(x,z_1,z_2)=\mu(x)-\mu(z_k),\quad k=1,2
\end{gather}
whenever this expression is definable. Thus, $\Delta_1\mu,\Delta_2\mu,\wr^2\eta_1,\cdots,\wr^2\eta_N$ are the local coordinates associated respectively to  $\sigma_1,\sigma_2,\wr^2x_1,\dots,\wr^2 x_N$, written for simplicity as
\begin{gather*}
(\Delta_\blt\mu,\wr^2\eta_\blt):=(\Delta_1\mu,\Delta_2\mu,\wr^2\eta_1,\cdots,\wr^2\eta_N)
\end{gather*}
\end{subequations}
\end{rem}

For each $a_1,\cdots,a_M$, define an infinite-rank locally free $\MO_{C^2}$-module \index{VX@$\SV_{\fx,a_1,\cdots,a_M}^{\boxtimes 2}$}
\begin{gather*}
\SV_{\fx,a_1,\cdots,a_M}^{\boxtimes 2}:=\pr_1^*\SV_{\fx,a_1,\cdots,a_M}\otimes \pr_2^*\SV_{\fx,a_1,\cdots,a_M}
\end{gather*}
where $\pr_i:C^2\rightarrow C$ is the projection onto the $i$-th component. If $V\subset C-S_\fx-D_\fx$ is an open subset and $\mu\in \MO(V)$ is univalent, then we have a trivialization
\begin{gather}
\pr_i^*\MU_\varrho(\mu):\pr_i^*\SV_{\fx,a_1,\cdots,a_M}\vert_{\pr_i^{-1}(V)}\xrightarrow{\simeq} \Vbb\otimes_\Cbb \MO_{\pr_i^{-1}(V)}
\end{gather}

Choose a finitely admissible $\Vbb^{\times N}$-module $\Wbb$. Associate $\Wbb$ to $x_1,\cdots,x_N$ of $\fx$ and $\Vbb^{\otimes 2}\otimes \Wbb$ (which is finitely admissible) to $\sigma_1,\sigma_2,\wr^2 x_1,\cdots,\wr^2 x_N$. The following proposition is analogous to  Prop. \ref{equivalence3}.

\begin{pp}\label{equivalence4}
    We have a canonical isomorphism of $\MO_{\Conf^2(C-S_\fx-D_\fx)}$-modules
    $$
    \SW_{\wr^2\fx}(\Vbb^{\otimes 2}\otimes \Wbb)\xrightarrow{\simeq} \SV_{\fx,a_1,\cdots,a_M}^{\boxtimes 2}\vert_{\Conf^2(C-S_\fx-D_\fx)}\otimes_\Cbb \SW_\fx(\Wbb)
    $$
    such that for any two disjoint open subsets $V_1,V_2\subset C-S_\fx-D_\fx$, any univalent $\mu_1\in \MO(V_1),\mu_2\in \MO(V_2)$, and any local coordinates $\eta_\blt$ at the incoming marked points $x_\blt$, the restriction of this isomorphism to $V_1\times V_2$ makes the following diagram commute.
    \begin{center}
\begin{tikzcd}[column sep=-1.5ex, row sep=2.75ex]
\mathscr{W}_{\wr^2\mathfrak{X}}(\mathbb{V}^{\otimes 2}\otimes \mathbb{W})\vert_{V_1\times V_2} \arrow[rrrr, "\simeq"] \arrow[rrdd, "{\mathcal{U}(\Delta_\blt\mu_\blt,\wr^2\eta_\bullet)}"'] &  &                                                    &  & \mathscr{V}_{\mathfrak{X},a_1,\cdots,a_M}^{\boxtimes 2}\vert_{V_1\times V_2}\otimes_\Cbb \mathscr{W}_\mathfrak{X}(\mathbb{W})  \arrow[lldd, "\pr_1^*\mathcal{U}_\varrho(\mu_1)\otimes\pr_2^*\mathcal{U}_\varrho(\mu_2) \otimes_\Cbb \mathcal{U}(\eta_\bullet)"] \\
                                                                                                                                                               &  &                                                    &  &                                                                                                                                                                      \\
                                                                                                                                                               &  & \Vbb^{\otimes 2}\otimes \mathbb{W} \otimes_\Cbb \mathcal{O}_{V_1\times V_2} &  &                                                                                                                                                                     
\end{tikzcd}
\end{center}
\end{pp}
\begin{proof}
    The proof is similar to Prop. \ref{equivalence3} and is exactly the same as \cite[Prop. 8.1]{Gui-propagation}. So we omit the proof.  
\end{proof}

Choose $a_1,\dots,a_M\in\Nbb$. Choose
\begin{gather*}
\upphi\in\scr T_{\fk X,a_1,\dots,a_M}^*(\Wbb).
\end{gather*}
By Thm. \ref{lb71}, we have double propagation
\begin{align*}
\wr^2\upphi=\wr(\wr\upphi)\in\scr T_{\wr^2\fk X,a_1,\dots,a_M}^*(\Vbb^{\otimes 2}\otimes \Wbb)
\end{align*}
By Prop. \ref{equivalence4}, we can view $\wr^2\upphi$ as an $\mc O_{\Conf^2(C-\SX-\DX)}$-module morphism
\begin{equation}\label{doublepropagation}
\wr^2\upphi:\SV_{\fx,a_1,\cdots,a_M}^{\boxtimes 2}\vert_{\Conf^2(C-S_\fx-D_\fx)}\otimes_\Cbb \SW_\fx(\Wbb)\rightarrow \MO_{\Conf^2(C-S_\fx-D_\fx)}.
\end{equation}
 
Our main result in this section is parallel to Thm. 8.2 in \cite{Gui-propagation} and is significant for the construction of $\Vbb$-module structures of dual tensor products. Before describing this result, we introduce some notations.

For each open subsets $V_1,V_2\subset C-D_\fx$ (not necessarily disjoint), we write
\begin{subequations}
\begin{gather}
    \Conf(V_\blt-S_\fx-D_\fx)=(V_1\times V_2)\cap \Conf^2(C-S_\fx-D_\fx)\label{eq66}
\end{gather}
If $v_i\in \SV_{\fx,a_1,\cdots,a_M}(V_i)(i=1,2)$ and $w\in \SW_\fx(\Wbb_\blt)$, write
\begin{gather}
    \wr^2\upphi(v_1,v_2,w)=\wr^2\upphi\big(\pr_1^*v_1\otimes \pr_2^*v_2\vert_{\Conf(V_\blt-S_\fx-D_\fx)}\otimes w\big)
\end{gather}
which is an element of $\MO(\Conf(V_\blt-S_\fx-D_\fx))$. Here, $v_2$ is for the first propagation, and $v_1$ is for the second one. (We are following the rule that the section for the last propagation is written on the leftmost side. ) We use the following two symbols
\begin{gather}
\wr^2\upphi(v_1,v_2,w)\vert_{p_1,p_2}\equiv  \wr\big(\wr\upphi(v_1,v_2,w)\vert_{p_2}\big)\vert_{p_1}
\end{gather}
to denote the value of the holomorphic function $\wr^2\upphi(v_1,v_2,w)$ at $(p_1,p_2)\in\Conf(V_\blt-\SX-\DX)$.
\end{subequations}

\begin{thm}\label{lb18}
For each $1\leq i\leq N$, choose a local coordinate $\eta_i$ at $x_i$ defined on a neighborhood $U_i\subset C-\DX$ of $x_i$. Let $V_1,V_2$ be non-necessarily disjoint open subsets of $C-\DX$. Choose univalent functions $\mu_k\in \MO(V_k)$ for $k=1,2$.  Identify
\begin{subequations}
\begin{gather}
\SW_\fx(\Wbb)=\Wbb\qquad\text{via }\MU(\eta_\blt)\\
\SV_{\fx,a_1,\cdots,a_M}\vert_{V_k}=\Vbb\otimes_\Cbb \MO_{V_k}\qquad\text{via }\MU_\varrho(\mu_k)
\end{gather}
\end{subequations}
(Note that $\scr V_{\fk X,a_\star}$ equals $\scr V_{\fk X}$ outside $\DX$.) Choose $v_k\in \Vbb\otimes\MO(V_k)$.   Choose $w\in \Wbb$ and $(p_1,p_2)\in \Conf(V_\blt-S_\fx-D_\fx)$. The following are true.
    \begin{enumerate}[label=(\arabic*)]
        \item Suppose that $V_1=U_i$ (where $1\leq i\leq N$), that $V_1$ contains only $p_1,x_i$ among $p_1,p_2,x_\blt$, that $\mu_1=\eta_i$, and that $V_1$ contains the closed disc with center $x_i$ and radius $\vert \eta_i(p_1)\vert$ under the coordinate $\eta_i$. Then
        $$
        \wr^2\upphi(v_1,v_2,w)\vert_{p_1,p_2}=\wr \upphi(v_2,Y_i(v_1,z)w)\vert_{p_2}\vert_{z=\eta_i(p_1)}
        $$
        where the series of $z$ on the RHS converges absolutely to the LHS, and $v_1$ is considered as an element of $\Vbb\otimes\Cbb((z))$ by taking power series expansion with respect to $z=\eta_i$ at $x_i$.
        \item Suppose that $V_1=V_2$, that $V_2$  contains only $p_1,p_2$ among $p_1,p_2,x_\blt$, that $\mu_1=\mu_2$, and that $V_2$ contains the closed disc with center $p_2$ and radius $\vert \mu_2(p_1)-\mu_2(p_2)\vert$ under the coordinate $\mu_2$. Then
        $$
        \wr^2\upphi(v_1,v_2,w_\blt)\vert_{p_1,p_2}=\wr\upphi(Y(v_1,z)v_2,w)\vert_{p_2}\vert_{z=\mu_2(p_1)-\mu_2(p_2)}
        $$
        where the series of $z$ on the RHS converges absolutely to the LHS, and $v_1$ is considered as an element of $\Vbb\otimes \Cbb((z))$ by taking power series expansion with respect to $z=\mu_2-\mu_2(p_2)$ at $p_2$.
        \item $\displaystyle         \wr^2\upphi(\ibf,v_2,w)\vert_{p_1,p_2}=\wr\upphi(v_2,w)\vert_{p_2}$.
        \item $\displaystyle         \wr^2 \upphi(v_1,v_2,w)\vert_{p_1,p_2}=\wr^2\upphi(v_2,v_1,w)\vert_{p_2,p_1}$.
    \end{enumerate}
\end{thm}

Notice that by \eqref{eq51} we have
\begin{align}
\wr\upphi(\idt,w)|_{p_2}=\upphi(w)  \label{eq62}
\end{align}

\begin{proof}
When $v_1,v_2$ are constant sections, (1) and (2) follow from Thm. \ref{lb71}. The general case follows immediately. (3) follows from Cor. \ref{propagation5} by considering the partial conformal block $\wr\upphi\vert_{p_2}$ associated to $(\wr\fx)_{p_2}$ and $\Vbb\otimes \Wbb$. To prove (4), consider the two partial conformal blocks of multi-level $a_\star$ associated to $(\wr^2\fx)_{p_1,p_2}$ defined by
\begin{gather*}
(v_1,v_2,w)\mapsto \wr^2\upphi(v_1,v_2,w)\vert_{p_1,p_2}\qquad (v_1,v_2,w)\mapsto \wr^2\upphi(v_2,v_1,w)\vert_{p_2,p_1}\\
(\text{where }v_1,v_2\in\Vbb,w\in\Wbb)
\end{gather*}
(Note that the second one belongs to $\scr T_{(\wr^2\fk X)_{p_1,p_2}}^*(\Vbb^{\otimes 2}\otimes\Wbb)$ because the linear functional $(v_2,v_1,w)\mapsto \wr^2\upphi(v_2,v_1,w)\vert_{p_2,p_1}$ belongs to $\scr T_{(\wr^2\fx)_{p_2,p_1}}^*(\Vbb^{\otimes 2}\otimes\Wbb)$.) By (3) and \eqref{eq62}, they are equal when $v_1=v_2=\idt$. By Cor. \ref{propagation3}, they are equal for all $v_1,v_2$.
\end{proof}

Let $\displaystyle \Conf^2(C-\SX)=\{(x,y)\in C-\SX:x\neq y\}$. With the help of Thm. \ref{lb18}-(4), we show:

\begin{pp}\label{anotherversionpropagation2}
The morphism $\wr^2\upphi$ in \eqref{doublepropagation} can be extended (necessarily uniquely) to an  $\MO_{\Conf^2(C-S_\fx)}$-module morphism
\begin{equation}\label{doublepropagation2}
\wr^2\upphi:\SV_{\fx,a_1,\cdots,a_M}^{\boxtimes 2}\vert_{\Conf^2(C-S_\fx)}\otimes_\Cbb \SW_\fx(\Wbb)\rightarrow \MO_{\Conf^2(C-S_\fx)}.
\end{equation}
\end{pp}
\begin{proof}
It suffices to prove that for any open sets $V_1,V_2\subset C-S_\fx$ and any  $v_i\in \SV_{\fx,a_1,\cdots,a_M}(V_i)$ (where $i=1,2$), the holomorphic function
\begin{equation}\label{proofdoublepropagation2}
f=\wr^2\upphi(v_1,v_2,w)
\end{equation}
on $(V_1-\SX-\DX)\times(V_2-\SX-\DX)-\Gamma$ extends to a holomorphic function on $(V_1-\SX)\times(V_2-\SX)-\Gamma$ where $\Gamma=\{(x,x):x\in C\}$. 

By \eqref{eq63} of Thm. \ref{lb71} (applied to the family $\wr\fk X$), $f$ is holomorphic on $(V_1-\SX)\times(V_2-\SX-\DX)-\Gamma$. By Thm. \ref{lb18}-(4) and \eqref{eq63}, $f$ is holomorphic on $(V_1-\SX-\DX)\times (V_2-\SX)-\Gamma$. Thus $f$ is holomorphic on $(V_1-\SX)\times(V_2-\SX)-\Gamma-Y$ where $Y=\{(y_i,y_j):1\leq i,j\leq M\}$. Since every closed complex submanifold of codimension at least $2$ is a removable singularity   (cf. e.g. \cite[Thm. 7.1.2]{GR84}), $f$ is holomorphic on $(V_1-\SX)\times(V_2-\SX)-\Gamma$.
\end{proof}

\section{Dual fusion products}
\label{Ch3}

Throughout this chapter, we assume the following setting.

\begin{sett}\label{lb24}
Fix an $(M,N)$-pointed compact Riemann surface with outgoing local coordinates $\fx=(y_1,\cdots,y_M;\theta_1,\cdots,\theta_M\big|C\big|x_1,\cdots,x_N)$ with incoming marked points $x_1,\cdots,x_N$ and outgoing ones $y_1,\cdots,y_M$ satisfying Asmp. \ref{ass1}.  Choose neighborhoods $W_1,\cdots,W_M$ of $y_1,\cdots,y_M$, on which the local coordinates $\theta_1,\cdots,\theta_M$ are defined. We assume that
\begin{align}
W_1,\dots,W_M,x_1,\dots,x_N\text{ are mutually disjoint}.
\end{align}
$S_\fx$ and $D_\fx$ are defined in (\ref{marked2}). Associate a finitely admissible $\Vbb^{\times N}$-module $\Wbb$ to $x_1,\cdots,x_N$.
\end{sett}


\subsection{$\bbs_{\fk X}(\Wbb)$ is a weak $\Vbb^{\times M}$-module}
Recall Def. \ref{lb16}. Choose any $\upphi\in\boxbackslash_{\fx}(\Wbb)$. Then there exist $a_1,\cdots, a_M$ such that $\upphi\in \ST_{\fx,a_1,\cdots ,a_M}^*(\Wbb)$. Fix $w\in \scr W_{\fk X}(\Wbb)$. By Thm. \ref{lb71} (and recall \eqref{eq64}), we have a morphism of $\MO_{C-S_\fx}$-modules:
\begin{equation}\label{action1}
\wr\upphi(-,w):\SV_{\fx,a_1,\cdots,a_M}\vert_{C-S_\fx}\rightarrow \MO_{C-S_\fx}.
\end{equation}
Recall that $\scr V_{\fk X,a_\star}$ is a subsheaf of $\scr V_{\fk X}=\scr V_C$ and that we have trivialization
\begin{gather}
\mc U_\varrho(\theta_j):\scr V_C|_{W_j}\xrightarrow{\simeq}\Vbb\otimes_\Cbb\mc O_{W_j}  \label{eq68}
\end{gather}
Define $Y_j(\cdot)_n:\Vbb\otimes_\Cbb\boxbackslash_{\fx}(\Wbb)\rightarrow \scr W_\fx(\Wbb)^*$  by 
    \begin{equation}\label{eq72}
\boxed{  ~  \<Y_j(v)_n\upphi,w\>=\Res_{\theta_j=0}~{\wr\upphi}\big(\mc U_\varrho(\theta_j)^{-1}v,w\big)\theta_j^n d\theta_j~}
    \end{equation}
    for each $v\in \Vbb\subset\Vbb\otimes_\Cbb\mc O(W_j)$, $w\in\scr W_{\fk X}(\Wbb)$, and $n\in \Zbb$.


\begin{pp}\label{lb20}
Let $\upphi\in\scr T_{\fk X,a_1,\dots,a_M}^*(\Wbb)$. Then for each homogeneous $v\in\Vbb$ we have
\begin{equation}\label{eq126}
Y_j(v)_n \upphi=0 \qquad \text{if } n\geq \wt(v)+a_j
\end{equation}
Thus $Y_j(v,z)\upphi$ belongs to $\scr W_{\fk X}(\Wbb)^*((z))$  if we write
\begin{align*}
Y_j(v,z)\upphi=\sum_{n\in\Zbb}Y_j(v)_n\upphi\cdot z^{-n-1}
\end{align*}
\end{pp}
  
A converse of this proposition is given by Cor. \ref{lb63}.

\begin{proof}
Choose any homogeneous vector $v\in\Vbb$. By \eqref{eq65},  $\theta_j^{\wt(v)+a_j}\mc U_\varrho(\theta_j)^{-1}v$ belongs to $\SV_{\fx,a_\star}(W_j)$. Thus, by (\ref{action1}), $\wr\upphi(\mc U_\varrho(\theta_j)^{-1}v,w)\theta_j^{\wt(v)+a_j}$ has removable singularity at $y_j$ (i.e. at $\theta_j=0$). This proves \eqref{eq126}.
\end{proof}

    By Prop. \ref{anotherversionpropagation2}, we have the double propagation 
\begin{gather}
    \wr^2 \upphi(w):\SV_{\fx,a_1,\cdots,a_M}^{\boxtimes 2}\vert_{\Conf^2(C-S_\fx)}\rightarrow \MO_{\Conf^2(C-S_\fx)}  \label{eq67}
\end{gather}
    where $w\in \scr W_\fx(\Wbb)$ is chosen. This morphism restricts to
        \begin{equation}\label{doubleprop3}
    \wr^2 \upphi(w):\SV_{C-S_\fx-D_\fx}^{\boxtimes 2}\rightarrow \MO_{\Conf^2(C-S_\fx-D_\fx)}.
    \end{equation}

\begin{rem}\label{chap3remark1}
For each open subset $V\subset C$ and each $v\in \SV_C(V),w\in\scr W_{\fk X}(\Wbb)$, $\wr\upphi(v,w)$ is understood as
\begin{subequations}
\begin{gather}
\wr\upphi(v\vert_{V-S_\fx-D_\fx},w)\qquad\in \MO(V-S_\fx-D_\fx)
\end{gather}
which has finite poles at $y_1,\dots,y_M$ by \eqref{action1} (cf. the proof of Prop. \ref{lb20}).
Similarly, if $V_1,V_2\subset C$ are open, for each $u\in \SV_C(V_1),v\in \SV_C(V_2)$, $\wr^2\upphi(u,v,w)$ is understood to be
\begin{gather}
\wr^2\upphi(u\vert_{V_1-S_\fx-D_\fx},v\vert_{V_2-S_\fx-D_\fx},w)\qquad\in \MO(\Conf(V_\blt-S_\fx-D_\fx))
\end{gather}
(recall \eqref{eq66} for the notation), which has finite poles at $y_1,\cdots,y_M$ by \eqref{eq67}. 
\end{subequations}

Note that by the uniqueness part of Thm. \ref{lb71}, the expressions $\wr\upphi(v,w)$ and $\wr^2\upphi(u,v,w)$ are independent of the choice of $a_1,\dots,a_M$ satisfying that $\upphi$ belongs to $\scr T_{\fk X,a_\star}^*(\Wbb)$.
\end{rem}

\begin{pp}\label{lb19}
For any homogeneous $v\in \Vbb$, $n\in \Zbb$ and $\upphi\in \ST_{\fx,a_1,\cdots,a_j,\cdots,a_M}^*(\Wbb)$, we have
\begin{align*}
Y_j(v)_n\upphi\in \ST_{\fx,a_1,\cdots,a_j',\cdots,a_M}^*(\Wbb)\qquad\text{where }a_j'=a_j+\max\{0,\wt(v)-n-1\}
\end{align*}
In particular, for each $v\in \Vbb$ and $n\in \Zbb$, 
\begin{gather*}
        Y_j(v)_n:\boxbackslash_\fx(\Wbb)\rightarrow \boxbackslash_\fx(\Wbb)
\end{gather*}
    \end{pp}

    \begin{proof}

Step 1. Fix $w\in\scr W_{\fk X}(\Wbb)$. Fix $1\leq j\leq M$,  $n\in\Zbb$, and $v\in\Vbb$. Then $\mc U_\varrho(\theta_j)^{-1}v\in\scr V_C(W_j)$.        For each $\fk u\in\SV_C(W)$ where $W\subset C-S_\fx-D_\fx$ is open,  note that $f=\wr^2\upphi(\fk u,\mc U_\varrho(\theta_j)^{-1}v,w)\in\mc O((W\times W_j)\cap\Conf^2(C))$. Define
\begin{gather}
\uppsi(\fk u)=\Res_{\theta_j=0} ~{\wr^2\upphi}\big(\fk u,\mc U_\varrho(\theta_j)^{-1}v,w\big)\theta_j^n d\theta_j \qquad \in\mc O(W) \label{eq71}
\end{gather}
where $\theta_j$ and $\Res$ are for the second variable of $f$. Recall that $\scr V_C$ equals $\scr V_{\fk X,a_1,\dots,a_j',\dots,a_M}$ outside $y_\star$. So $\uppsi$ is an $\mc O_{C-\SX-\DX}$-module morphism $\scr V_{a_1,\dots,a_j',\dots,a_M}|_{C-\SX-\DX}\rightarrow \mc O_{C-\SX-\DX}$. Let us prove that this morphism has removable singularity at $y_1,\dots,y_M$. Namely, we show
\begin{gather}\label{claim31}
\text{Claim: $\uppsi$  is an $\mc O_{C-\SX}$-module morphism }\scr V_{a_1,\dots,a_j',\dots,a_M}|_{C-\SX}\rightarrow\mc O_{C-\SX}
\end{gather}

Suppose that the claim is proved. Our goal is to prove that for each $\sigma\in H^0\big(C,\SV_{\fx,a_1,\cdots,a_j',\cdots,a_M}\otimes \omega_C(\blt S_\fx)\big)$, $Y_j(v)_n\upphi$ vanishes on the vector $\sigma\cdot w$ of $\scr W_{\fk X}(\Wbb)$ defined in Def. \ref{lb17}. Let $\uppsi$ also denote the $\mc O_{C-\SX}$-module morphism
\begin{align*}
\uppsi\otimes\idt: \scr V_{a_1,\dots,a_j',\dots,a_M}|_{C-\SX}\otimes\omega_{C-\SX}\rightarrow\omega_{C-\SX}
\end{align*}
So $\uppsi(\sigma)\in\omega_{C-\SX}(C-\SX)$. By Residue Theorem/Stokes Theorem, 
\begin{align*}
\sum_{i=1}^N \Res_{x_i} \uppsi(\sigma)=0.
\end{align*}

For each $1\leq i\leq N$, choose a local coordinate $\eta_i$ at $x_i$, and assume that $\eta_i$ is defined on a neighborhood $U_i$ containing only $x_i$ among all the incoming and outgoing marked points. Identify $\scr W_{\fk X}(\Wbb)=\Wbb$ via $\mc U(\eta_\blt)$. Notice that if $\fk u\in\SV_{C}(U_i)$, then
\begin{align}
&\uppsi(\fk u)\xlongequal[\text{Thm. \ref{lb18}-(1)}]{\eqref{eq71}}\Res_{\theta_j=0}~{\wr \upphi}\big(\mc U_\varrho(\theta_j)^{-1}v,Y_i(\mc U_\varrho(\eta_i)\fk u,\eta_i)w\big)\theta_j^n d\theta_j\nonumber\\
\xlongequal{\eqref{eq72}}&\<Y_j(v)_n \upphi,Y_i(\mc U_\varrho(\eta_i)\fk u,\eta_i)w\>.\label{eq70}
\end{align}
From this one concludes (recalling \eqref{eq73})
\begin{align*}
\Res_{x_i}\uppsi(\sigma)=\Res_{\eta_i=0}~\<Y_j(v)_n \upphi,Y_i(\mc U_\varrho(\eta_i)\sigma,\eta_i)w\>=\<Y_j(v)_n \upphi,\sigma*_iw\>
\end{align*}
and hence $\<Y_j(v)_n \upphi,\sigma\cdot w\>=\sum_{i=1}^N \<Y_j(v)_n \upphi,\sigma*_iw\>=0$. This finishes the proof of the proposition.\\

Step 2. Let us prove the claim in Step 1. First, let $i\neq j$ and choose a neighborhood  $W$ of $y_i$. By Prop. \ref{anotherversionpropagation2}, if $\fk u\in\scr V_{\fk X,a_1,\dots,a_j',\dots,a_M}(W)$ then $\wr^2\upphi(\fk u,\mc U_\varrho(\theta_j)^{-1}\theta_j^{a_j+L(0)}v,w)$ is holomorphic on $(W\times W_j)\cap\Conf^2(C)$, and hence $\uppsi(\fk u)$ has removable singularities at $y_i$.
This proves that $\uppsi$ is an $\mc O_{C-\SX-y_j}$-module morphism $\scr V_{\fk X,a_1,\dots,a_j',\dots,a_M}|_{C-\SX-y_j}\rightarrow\mc O_{C-\SX-y_j}$.

It remains to show that the morphism $\uppsi$ has removable singularity at $y_j$. Identify
\begin{subequations}\label{eq77}
\begin{gather}
\scr V_C|_{W_j}=\Vbb\otimes_\Cbb\mc O_{W_j} \qquad\text{via }\mc U_\varrho(\theta_j)\\
W_j=\theta_j(W_j)\qquad\text{via }\theta_j
\end{gather}
\end{subequations}
Let $z,\zeta$ both denote the standard coordinates of $\Cbb$, which are equivalent to $\theta_j$. In the following, when discussing two-variable meromorphic functions, we let $\zeta$ (resp. $z$) be the first (resp. second) complex variable.

To complete the proof of the claim, it suffices to show that for each homogeneous $u\in \Vbb$, considered as a constant section of $\SV_C(W_j)$, the holomorphic function $\uppsi(u)=\upphi(u)(\zeta)$ on $W_j-\{y_j\}$ has poles of order at most $\wt(u)+a_j'$ at $y_j$. Set $f=f(\zeta,z)$ to be 
        $$
        f=\wr^2\upphi(u,v,w)\in \MO(\Conf^2(W_j-\{y_j\})).
        $$
where $\zeta$ is for $u$ and $z$ is for $v$.      By Prop. \ref{anotherversionpropagation2}, 
        \begin{equation}\label{proofpole}
            \zeta^{a_j+\wt(u)}z^{a_j+\wt(v)}f(\zeta,z)\in \MO(\Conf^2(W_j)).
        \end{equation}
Choose anticlockwise circles $C_1,C_2,C_3$ in $W_j$ surrounding $y_j$ with radii $r_1<r_2<r_3$. For each $z\in C_2$, choose a circle $C(z)$ centered at $z$ with radius less than $r_3-r_2$ and $r_2-r_1$. Let $m\in \Zbb$. By Cauchy's theorem/Residue theorem, 
        \begin{align}\label{proofpole2}
            \begin{aligned}
       & \Res_{\zeta=0}~\zeta^m\uppsi(u)d\zeta=\oint_{C_3}\oint_{C_2}\zeta^m z^n f(\zeta,z)\frac{dz}{2\pi\im} \frac{d\zeta}{2\pi\im}=\oint_{C_2}\oint_{C_3}\zeta^m z^n f(\zeta,z)\frac{d\zeta}{2\pi\im} \frac{dz}{2\pi\im} \\
        =&\oint_{C_2}\oint_{C_1}\zeta^m z^n f(\zeta,z)\frac{d\zeta}{2\pi\im} \frac{dz}{2\pi\im} +\oint_{C_2}\oint_{C(z)}\zeta^m z^n f(\zeta,z)\frac{d\zeta}{2\pi\im} \frac{dz}{2\pi\im}.
        \end{aligned}
        \end{align}

        For fixed $z\in C_2$, by (\ref{proofpole}), $\zeta^{a_j+\wt(u)}f(\zeta,z)$ has removable singularity at $\zeta=0$ when $z$ is away from $0$. So the first term on the RHS of (\ref{proofpole2}) equals $0$ whenever $m\geq a_j+\wt(u)$. By Thm. \ref{lb18}-(2), the second term equals
        \begin{align}
&\oint_{C_2}\oint_{C(z)}\zeta^m z^n f(\zeta,z)\frac{d\zeta}{2\pi\im} \frac{dz}{2\pi\im} =\oint_{C_2}\oint_{C(z)}\zeta^m z^n \wr^2 \upphi(u,v,w)\frac{d\zeta}{2\pi\im} \frac{dz}{2\pi\im}  \nonumber\\
            =&\oint_{C_2}\oint_{C(z)}\zeta^m z^n \wr \upphi(Y(u,\zeta-z)v,w)\frac{d\zeta}{2\pi\im} \frac{dz}{2\pi\im}  \nonumber\\
            =&\sum_{l\in \Nbb}\binom{m}{l}\oint_{C_2}\oint_{C(z)}(\zeta-z)^l z^{m+n-l} \wr \upphi(Y(u,\zeta-z)v,w)\frac{d\zeta}{2\pi\im} \frac{dz}{2\pi\im}  \nonumber\\
            =&\sum_{l\in \Nbb}\binom{m}{l}\oint_{C_2} z^{m+n-l} \wr \upphi(Y(u)_lv,w) \frac{dz}{2\pi\im} =\sum_{l\in \Nbb}\binom{m}{l} \<Y(Y(u)_lv)_{m+n-l}\upphi,w\>.   \label{proofpole3}
        \end{align}
        By (\ref{eq126}), (\ref{proofpole3}) equals $0$ whenever
        $$
        m+n-l\geq \wt(Y(u)_lv)+a_j=\wt(u)+\wt(v)-l-1+a_j,
        $$
        and hence when 
        $$
        m\geq \wt(u)+\wt(v)+a_j-n-1.
        $$
        In conclusion, when $m\geq a_j+\wt(u)+\max\{0,\wt(v)-n-1\}=\wt(u)+a_j'$, (\ref{proofpole2}) equals $0$. This finishes the proof of our claim.
    \end{proof}

Prop. \ref{lb20} and  \ref{lb19} tell us $(\boxbackslash_\fx(\Wbb),Y_j)$ is a \textbf{linear representation} of $\Vbb$ in the following sense.

    \begin{df}
    Let $\Xbb$ be a vector space and
$$
\begin{aligned}
    \Vbb&\rightarrow (\End(\Xbb))[[z^{\pm 1}]]\\
    u&\mapsto Y_\Xbb(u,z)=\sum_{n\in \Zbb}Y_\Xbb(u)_nz^{-n-1}
\end{aligned}
$$
be a linear map. If for each $v\in \Vbb$ and $w\in \Xbb$, 
$$
Y_\Xbb(v,z)w\in \Xbb((z)),
$$
then we call $(\Xbb,Y_\Xbb)$ (or simply $\Xbb$) a \textbf{linear representation} of $\Vbb$.
    \end{df}

To prove that $(\boxbackslash_\fx(\Wbb),Y_j)$ is a weak $\Vbb$-module for each $1\leq j\leq M$ we need the following criterion. Let $\Xbb^\circ$ be a subspace of the dual space $\Xbb^*$. We say that $\Xbb^\circ$ is a \textbf{dense subspace} of $\Xbb^*$, if the only vector $w\in \Xbb$ satisfying $\<w',w\>=0$ for all $w'\in \Xbb^\circ$ is $0$.
\begin{pp}\label{criterion}
    Let $\Xbb$ be a linear representation of $\Vbb$ with $Y_\Xbb(\ibf,z)=\ibf_\Xbb$. Let $\Xbb^\circ$ be a dense subspace of $\Xbb^*$. Assume that for each $u,v\in \Vbb$, $w\in \Xbb$, $w'\in \Xbb^\circ$, there exists $\epsilon>0$ and $f=f(\zeta,z)\in \MO(\Conf^2(\MD_\epsilon^\times))$, such that for any $n\in \Zbb$ and $z\in \MD_\epsilon^\times$, the LHS of the following (as Laurent series of $z$) converges absolutely to the RHS:
\begin{subequations}
    \begin{align}
        \<Y_\Xbb(v,z)Y_\Xbb(u)_nw,w'\>&=\Res_{\zeta=0}f(\zeta,z)\zeta^n d\zeta,\label{condition1}\\
        \<Y_\Xbb(Y(u)_nv,z)w,w'\>&=\Res_{\zeta-z=0}f(\zeta,z)(\zeta-z)^n d\zeta,\label{condition2}
    \end{align}
    and for any $n\in \Zbb$ and $\zeta\in \MD_\epsilon^\times$, the LHS of the following converges absolutely to the RHS:
    \begin{equation}\label{condition3}
        \<Y_\Xbb(u,\zeta)Y_\Xbb(v)_nw,w'\>=\Res_{z=0}f(\zeta,z)z^n dz.
    \end{equation}
\end{subequations}
    Then $(\Xbb,Y_\Xbb)$ is a weak $\Vbb$-module.
\end{pp}

\begin{proof}
Choose anticlockwise circles $C_1,C_2,C_3$ in $\MD_\epsilon^\times $ surrounding $0$ with radii $r_1<r_2<r_3$. For each $z\in C_2$, choose a circle $C(z)$ centered at $z$ with radius less than $r_3-r_2$ and $r_2-r_1$. Choose $m,n\in \Zbb$. By Cauchy's theorem in complex analysis, we have $P(z)=Q(z)-R(z)$, where
    $$
    \begin{aligned}
        P(z)&=\oint_{C(z)}f(\zeta,z)\zeta^m (\zeta-z)^n \frac{d\zeta}{2\pi\im},\\
        Q(z)&=\oint_{C_3}f(\zeta,z)\zeta^m (\zeta-z)^n \frac{d\zeta}{2\pi\im},\\
        R(z)&=\oint_{C_1}f(\zeta,z)\zeta^m (\zeta-z)^n \frac{d\zeta}{2\pi\im}.
    \end{aligned}
    $$
    By (\ref{condition2}), we can compute
\begin{align}
        & P(z)=\oint_{C(z)}f(\zeta,z)\sum_{l\in \Nbb}\binom{m}{l}z^{m-l}(\zeta-z)^{n+l}\frac{d\zeta}{2\pi\im}  \nonumber\\
\xlongequal{\eqref{condition2}}&\sum_{l\in \Nbb}\binom{m}{l}z^{m-l} \<Y_\Xbb(Y(u)_{n+l}v,z)w,w'\>.  \label{eq74}
\end{align}
where the RHS converges absolutely. Similarly,
\begin{align}
&R(z)=\oint_{C_1}f(\zeta,z)\sum_{l\in \Nbb}(-1)^{n-l}\binom{n}{l}z^{n-l}\zeta^{m+l} \frac{d\zeta}{2\pi\im}  \nonumber\\
\xlongequal{\eqref{condition1}}&\sum_{l\in \Nbb}(-1)^{n-l}\binom{n}{l}z^{n-l}\<Y_\Xbb(v,z)Y_\Xbb(u)_{m+l}w,w'\>. \label{eq75}
\end{align}
    Since $P(z)=Q(z)-R(z)$ holds for all $z\in C_2$, for each $h\in \Zbb$ we have 
    \begin{equation}\label{criterion1}
    \oint_{C_2}P(z)z^h\frac{dz}{2\pi\im}=\oint_{C_2}Q(z)z^h\frac{dz}{2\pi\im}-\oint_{C_2}R(z)z^h\frac{dz}{2\pi\im}.
    \end{equation}
We compute
\begin{align}
        &\oint_{C_2}Q(z)z^h\frac{dz}{2\pi\im}=\oint_{C_2} \oint_{C_3}\zeta^mf(\zeta,z)(\zeta-z)^n z^h \frac{d\zeta}{2\pi\im}\frac{dz}{2\pi\im} \nonumber\\
=&\oint_{C_3}\oint_{C_2}\zeta^m f(\zeta,z)(\zeta-z)^n z^h \frac{dz}{2\pi\im}\frac{d\zeta}{2\pi\im} \nonumber\\
        =&\sum_{l\in \Nbb}\oint_{C_3}\oint_{C_2}\zeta^{m+n-l}\cdot(-1)^l \binom{n}{l}f(\zeta,z) z^{h+l}\frac{dz}{2\pi\im}\frac{d\zeta}{2\pi\im} \nonumber\\
\xlongequal{\eqref{condition3}}&\sum_{l\in \Nbb}(-1)^l \binom{n}{l}\<Y_\Xbb(u)_{m+n-l}Y_\Xbb(v)_{h+l}w,w'\>  \label{eq76}
\end{align}
Substituting \eqref{eq74}, \eqref{eq75}, and \eqref{eq76} into \eqref{criterion1}, we get
\begin{align*}
&\sum_{l\in \Nbb}\binom{m}{l}\<Y_\Xbb(Y(u)_{n+l}v)_{m+h-l}w,w'\>\\
        =&\sum_{l\in \Nbb}(-1)^l \binom{n}{l}\<Y_\Xbb(u)_{m+n-l}Y_\Xbb(v)_{h+l}w,w'\>\\
&-\sum_{l\in \Nbb}(-1)^{n-l}\binom{n}{l}\<Y_\Xbb(v)_{n+h-l}Y_\Xbb(u)_{m+l}w,w'\>.
\end{align*}
Since $\Xbb^\circ$ is dense in $\Xbb^*$, the Jacobi identity \eqref{jacobi} holds for $Y_\Xbb$.    This, together with the assumption $Y_\Xbb(\ibf,z)=\ibf_\Xbb$, proves that $\Xbb$ is a weak module.
\end{proof}

\begin{lm}\label{lb23}
Choose $\upphi\in \boxbackslash_{\fx}(\Wbb)$ and $n\in \Zbb$. Choose $v\in\Vbb$. Identify $W_j=\theta_j(W_j)$ via $\theta_j$ so that $\theta_j$ becomes the standard coordinate $z$. Then for each section $\fk u\in\scr V_{C-\SX-\DX}$ and $w\in\scr W_{\fk X}(\Wbb)$,
    \begin{equation}\label{lemma52}
    \wr(Y_j(v)_n\upphi)(\fk u,w)=\Res_{z=0}~{\wr^2\upphi}\big(\fk u,\mc U_\varrho(\theta_j)^{-1}v,w\big)z^n dz
    \end{equation}
where $z$ is for the second variable of ${\wr^2\upphi}(\fk u,\mc U_\varrho(\theta_j)^{-1}v,w)$.
\end{lm}
\begin{proof}
When $\fk u$ is defined on a neighborhood $U_i$ of $x_i$ on which $\eta_i$ is defined, 
\begin{align*}
 & {\wr(Y_j(v)_n\upphi)}(\fk u,w)\xlongequal{\eqref{eq82}}(Y_j(v)_n\upphi)\big(Y_i(\mc U_\varrho(\eta_i)\fk u,\eta_i)w\big)\\
\xlongequal{\eqref{eq72}}&\Res_{z=0}~{\wr\upphi}\big(\mc U_\varrho(\theta_j)^{-1}v,Y_i(\mc U_\varrho(\eta_i)\fk u,\eta_i)w\big)z^n dz\\
\xlongequal{\text{Thm.\ref{lb18}-(1)}}&\Res_{z=0}~{\wr^2 \upphi}\big(\fk u,\mc U_\varrho(\theta_j)^{-1}v,w\big)z^n dz.
\end{align*}
(Note that in the above derivation, we have exchanged the order of $\Res_{z=0}$ and the infinite sum in the Laurent series about the variable $\eta_i$. This is legitimate because ${\wr^2 \upphi}(\fk u,\mc U_\varrho(\theta_j)^{-1}v,w)$ is holomorphic, or alternatively because of Thm. \ref{lb71}.)    So, in this case, \eqref{lemma52} holds on $U_i-\SX-\DX$. As in Proof-Step 2 of Thm. \ref{lb71}, one shows that if $\Omega$ denotes the set of all $x\in C-\SX-\DX$ satisfying that there is a neighborhood $W$ of $x$ such that \eqref{lemma52} holds for all $\fk u\in\scr V_C(W)$, then $\Omega$ is both open and closed in $C-\SX-\DX$ and is intersecting $U_1,\dots,U_N$; one thus concludes $\Omega=C-\SX-\DX$ thanks to Asmp. \ref{ass1}.
\end{proof}

\begin{co}\label{lb22}
Choose $\upphi\in \boxbackslash_{\fx}(\Wbb)$ and $m,n\in \Zbb$. Assume the identifications \eqref{eq77}. Choose $u,v\in \Vbb$, considered as constant sections of $\Vbb\otimes_\Cbb\mc O(W_j)$. Choose any $w\in \scr W_{\fk X}(\Wbb)$. Then
    \begin{equation}\label{lemma51}
        \<Y_j(u)_mY_j(v)_n\upphi,w\>=\Res_{\zeta=0}\Res_{z=0}~{\wr^2\upphi}(u,v,w)\zeta^mz^n dzd\zeta.
    \end{equation}
    Here $\wr^2\upphi(u,v,w)$ is considered as a holomorphic function $\wr^2\upphi(u,v,w)(\zeta,z)$ on $\Conf^2(W_j-\{y_j\})$, and the variables $\zeta$ and $z$ are for $u,v$ respectively.
\end{co}

\begin{proof}
In \eqref{lemma52}, set $\fk u=u\in\scr V_C(W_j)$. Apply $\Res_{\zeta=0}(-)d\zeta$ to \eqref{lemma52} and use \eqref{eq72}.
\end{proof}

\begin{lm}\label{lb21}
For each $1\leq j\leq M$,    $(\boxbackslash_{\fx}(\Wbb),Y_j)$ is a weak $\Vbb$-module.
\end{lm}

\begin{proof}
By Prop. \ref{lb20} and \ref{lb19},  $(\boxbackslash_{\fx}(\Wbb),Y_j)$ is a linear representation of $\Vbb$. We shall check that $\Xbb=\bbs_\fx(\Wbb)$ and $Y_\Xbb=Y_j$ satisfy the conditions in Prop. \ref{criterion}.
     
The natural linear map $\scr W_{\fk X}(\Wbb)\rightarrow \boxbackslash_\fx(\Wbb)^*$ clearly has dense range $\fk X^\circ$. Moreover, by Cor. \ref{propagation5}, for each $\fk w\in\scr W_{\fk X}(\Wbb)$, $\wr\upphi(\ibf,\fk w)$ is a constant function with value $\upphi(\fk w)$. So 
    $$
    \<Y_j(\ibf)_n\upphi,\fk w\>=\Res_{\theta_j=0}\wr\upphi(\ibf,\fk w)\theta_j^n d\theta_j=\upphi(\fk w)\delta_{n,-1}
    $$
    which proves $Y_j(\ibf,z)=\ibf_{\boxbackslash_\fx(\Wbb)}$.
     
Assume the identifications \eqref{eq77}.  Choose $u,v\in \Vbb,\upphi\in \boxbackslash_\fx(\Wbb),\fk w\in \scr W_\fx(\Wbb)$. Choose $\epsilon>0$ such that $\MD_\epsilon\subset W_j$ and let $\zeta,z$ be standard coordinates of $W_j$, which are equivalent to $\theta_j$. Let $f\in \MO(\Conf^2(\MD_\epsilon^\times))$ be 
    $$
    f(\zeta,z)=\wr^2\upphi(u,v,\fk w)(\zeta,z).
    $$
Then by Cor. \ref{lb22} and Thm. \ref{lb18}-(4), \eqref{condition1} and \eqref{condition3} hold for $w=\upphi$ and $w'$ the corresponding vector of $\fk w$ in $\Xbb^*=\bbs_\fx(\Wbb)^*$ and for all $n$. To verify \eqref{condition2}, we compute  
\begin{align*}
 &\Res_{z=0}\Res_{\zeta-z=0}f(\zeta,z)\cdot (\zeta-z)^n z^m d\zeta dz\\
    =&\Res_{z=0}\Res_{\zeta-z=0}~{\wr^2\upphi}(u,v,\fk w)(\zeta,z)\cdot (\zeta-z)^n z^m d\zeta dz\\
\xlongequal{\text{Thm. \ref{lb18}-(2)}}&\Res_{z=0}\Res_{\zeta-z=0}~{\wr\upphi}(Y(u,\zeta-z)v,\fk w)(\zeta,z)\cdot (\zeta-z)^n z^m d\zeta dz\\
    =&\Res_{z=0}~{\wr\upphi}(Y(u)_nv,\fk w)(z)\cdot  z^m  dz\\
    =&\<Y_j(Y(u)_n v)_m\upphi,\fk w\>
\end{align*}
This finishes the proof.
\end{proof}

\begin{thm}\label{lb43}
    $(\boxbackslash_\fx(\Wbb),Y_1,\cdots,Y_M)$ is a weak $\Vbb^{\times M}$-module.
\end{thm}
\begin{proof}
For each $1\leq j\leq M$, assume the identifications \eqref{eq77}.   Since we have proved Lemma \ref{lb21},  it remains to show that $Y_i$ commutes with $Y_j$ for $i\ne j$. Let $\zeta,z$ be respectively the standard coordinate of $W_i$ and $W_j$. Choose $u,v\in \Vbb$, considered as constant sections of $\SV_C(W_i),\SV_C(W_j)$ respectively. Choose $w\in\scr W_{\fk X}(\Wbb)$. Then by Lem. \ref{lb23},
\begin{align*}
\wr(Y_j(v)_n\upphi)(u,w)=\Res_{z=0}~{\wr^2\upphi}(u,v,w)z^n dz.
\end{align*}
Apply $\Res_{\zeta=0}(-)d\zeta$ to both sides.  Then, by \eqref{eq72}, we have
\begin{align}
\<Y_i(u)_mY_j(v)_n\upphi,w\>=\Res_{\zeta=0}\Res_{z=0}~{\wr^2 \upphi}(u,v,w)(\zeta,z)\cdot \zeta^m z^n dzd\zeta.
\end{align}
    Similarly, we have
\begin{align*}
\<Y_j(v)_nY_i(u)_m\upphi,w\>=\Res_{z=0}\Res_{\zeta=0}~{\wr^2 \upphi}(v,u,w)(z,\zeta)\cdot \zeta^m z^n d\zeta dz.
\end{align*}
The above two expressions are equal by Thm. \ref{lb18}-(4). Therefore $Y_i(u)_m$ commutes with $Y_j(v)_n$.
\end{proof}

Recall Def. \ref{lb40} for the definition of generalized modules. 

\begin{co}\label{lb44}
Suppose that for each $a_1,\dots,a_M\in\Nbb$, $\scr T_{\fk X,a_1,\dots,a_M}^*(\Wbb)$ is finite-dimensional. Then $\bbs_\fx(\Wbb)$ is a generalized $\Vbb^{\otimes M}$-module.
\end{co}

\begin{proof}
By Thm. \ref{lb43} and Prop. \ref{lb42} (together with Prop. \ref{lb19}).
\end{proof}

\subsection{The canonical conformal block $\gimel$ associated to $\boxbackslash_{\fx}(\Wbb)$}  \label{lb58}

Recall Setting \ref{lb24} in which a finitely admissible $\Vbb^{\times N}$-module $\Wbb$ is associated to $x_1,\dots,x_N$, and the local coordinates $\theta_1,\dots,\theta_M$ are associated to the outgoing marked points. Associate the weak $\Vbb^{\times M}$-module $\boxbackslash_\fx(\Wbb)$ (cf. Thm. \ref{lb43}) to $y_1,\cdots,y_M$, and view $\fk X$ as an $(M+N)$-pointed surface.

\begin{thm}\label{lb26}
Choose local coordinates $\eta_1,\dots,\eta_N$ of $C$ at $x_1,\dots,x_N$.  Define a linear map
    \begin{gather}\label{eq79}
\gimel:\Wbb\otimes \boxbackslash_\fx(\Wbb)\rightarrow \Cbb\qquad
        w\otimes \upphi\mapsto \upphi\big(\mc U(\eta_\blt)^{-1}w\big).
    \end{gather}
Then we have $\gimel\in \ST_\fx^*(\Wbb\otimes \boxbackslash_\fx(\Wbb))$ in the sense of Def. \ref{lb25}.
\end{thm}

We call $\gimel$ the \textbf{canonical conformal block associated to $\bbs_{\fk X}(\Wbb)$} (more precisely, associated to $\fk X$ and $\Wbb\otimes\bbs_\fx(\Wbb)$).

\begin{proof}
Identify $\scr W_{\fk X}(\Wbb)$ with $\Wbb$ via $\mc U(\eta_\blt)$.   Choose any $\upphi\in \boxbackslash_\fx(\Wbb)$ and $w\in \Wbb$. Denote the tensor product of $\wr \upphi(\cdot,w):\scr V_{C-\SX-\DX}\rightarrow \mc O_{C-\SX-\DX}$ and $\idt:\omega_{C-\SX-\DX}\rightarrow\omega_{C-\SX-\DX}$ also by
\begin{align*}
\wr\upphi(\cdot,w):\scr V_C\otimes\omega_C|_{C-\SX-\DX}\rightarrow\omega_{C-\SX-\DX}
\end{align*}
Choose $\sigma\in H^0\big(C,\SV_C\otimes \omega_C(\bullet S_\fx+\bullet D_\fx)\big)$. Recall notations \eqref{eq73}. Then for each $1\leq i\leq N$,
\begin{align}
\upphi(\sigma*_i w)\xlongequal{\eqref{eq80}} \Res_{x_i}\upphi\big(Y_i(\mc U_\varrho(\eta_i)\sigma,\eta_i)w \big)\xlongequal{\eqref{eq82}} \Res_{x_i}~{\wr\upphi}(\sigma,w).    \label{eq123}
\end{align}
For each  $1\leq j\leq M$,
\begin{align}
(\sigma*_j\upphi)(w)\xlongequal{\eqref{eq80}} \<\Res_{y_j} Y_j(\mc U_\varrho(\theta_j)\sigma,\theta_j)\upphi,w\>  \xlongequal{\eqref{eq72}} \Res_{y_j}~{\wr\upphi}(\sigma,w).   \label{eq124}
\end{align}
Therefore, by residue theorem/Stokes theorem and that $\wr\upphi(\sigma,w)\in \omega_C(C-\SX-\DX)$,
\begin{align*}
&\gimel(\sigma\cdot(w\otimes\upphi))=\sum_{i=1}^N\gimel((\sigma*_i w)\otimes\upphi)+\sum_{j=1}^M\gimel(w\otimes (\sigma*_j\upphi))\\
=&\sum_{i=1}^N\upphi(\sigma*_i w)+\sum_{j=1}^M (\sigma*_j\upphi)(w)=\sum_{i=1}^N \Res_{x_i}~{\wr\upphi}(\sigma,w)+\sum_{j=1}^M\Res_{y_j}~{\wr\upphi}(\sigma,w)
\end{align*}
equals zero.
\end{proof}

\begin{rem}\label{lb64}
From the above proof, it is clear that Thm. \ref{lb26} is equivalent to that for each $\sigma\in H^0(C,\scr V_C\otimes\omega_C(\blt \SX+\blt\DX))$, $\upphi\in\bbs_\fx(\Wbb)$, and $w\in \scr W_\fx(\Wbb)$, we have
\begin{align}
\sum_{j=1}^M (\sigma*_j\upphi)(w)=-\sum_{i=1}^N \upphi(\sigma*_i w)  \label{eq142}
\end{align}
Using this formula, one easily shows the following converse of Prop. \ref{lb20}:
\end{rem}

\begin{co}\label{lb63}
Let $\upphi\in\bbs_\fx(\Wbb)$ and $a_1,\dots,a_M\in\Nbb$. Suppose that $\upphi$ satisfies \eqref{eq126} for each $1\leq j\leq M$, $n\in\Zbb$, and homogeneous $v\in \Vbb$. Then $\upphi\in\scr T_{\fx,a_1,\dots,a_M}^*(\Wbb)$.
\end{co}

\begin{proof}
Choose any $\sigma\in H^0(C,\scr V_{\fx,a_\star}\otimes\omega_C(\blt\SX))$ and $w\in\scr W_\fx(\Wbb)$. By \eqref{eq126}, $(\sigma*_j\upphi)(w)$ (which can be computed by the middle of \eqref{eq124}) equals $0$. So the RHS of \eqref{eq142} equals $0$. This proves that $\upphi$ vanishes on $\sigma\cdot w$. So $\upphi$ belongs to $\scr T_{\fx,a_1,\dots,a_M}^*(\Wbb)$.
\end{proof}

\begin{rem}\label{lb56}
Let $\Xbb$ be an admissible $\Vbb^{\times M}$-module. Then we clearly have a linear isomorphism (the partial trivialization)
\begin{align}\label{eq78}
\mc U(\cdot,\theta_\star):\scr W_{\fk X}(\Wbb\otimes \Xbb)\xrightarrow{\simeq}\scr W_\fx(\Wbb)\otimes \Xbb
\end{align}
such that for each local coordinates $\eta_1,\dots,\eta_N$ of $x_1,\dots,x_N$, the diagram
\begin{equation}\label{eq128}
\begin{tikzcd}[column sep=0ex]
\scr W_{\fk X}(\Wbb\otimes \Xbb) \arrow[rr,"\simeq"',"{\mc U(\cdot,\theta_\star)}"] \arrow[rd,"{\mc U(\eta_\blt,\theta_\star)}"'] &   & \scr W_\fx(\Wbb)\otimes \Xbb \arrow[ld,"\mc U(\eta_\blt)\otimes\idt"] \\
                        & \Wbb\otimes\Xbb &             
\end{tikzcd}
\end{equation}
commutes. We identify the two sides of \eqref{eq78} via $\mc U(\cdot,\theta_\star)$.

Now, assume that $\bbs_\fx(\Wbb)$ is an admissible $\Vbb^{\times M}$-module. (This is true when $\Vbb$ is $C_2$-cofinite and $\Wbb$ is finitely-generated; see Thm. \ref{lb57}). Then by Thm. \ref{lb26}, the linear map
\begin{align}
\gimel:\scr W_\fx(\Wbb)\otimes \boxbackslash_\fx(\Wbb)\rightarrow \Cbb\qquad
        w\otimes \upphi\mapsto \upphi(w)
\end{align}
belongs to $\scr T_{\fk X}^*(\Wbb\otimes\bbs_\fx(\Wbb))$ in the sense of Def. \ref{lb16}. We also call this $\gimel$ the \textbf{canonical conformal block associated to $\bbs_\fx(\Wbb)$}.
\end{rem}

\subsection{Universal property of $(\bbs_\fx(\Wbb),\gimel)$}

\begin{df}
A \textbf{weakly-admissible $\Vbb^{\times M}$-module $\Mbb$} \index{00@Weakly-admissible modules} is a weak $\Vbb^{\times M}$-module satisfying that for each $m\in\Mbb$  there exist $a_1,\dots,a_M\in\Nbb$ such that for each homogeneous $v\in\Vbb$ and each $1\leq j\leq M$ we have
\begin{align}
Y_j(v)_n m=0 \qquad \text{if } n\geq \wt(v)+a_j  \label{eq118}
\end{align}
\end{df}

\begin{eg}
By \eqref{eq99}, every admissible $\Vbb^{\times M}$-module is weakly-admissible. By Prop. \ref{lb20} and Thm. \ref{lb43}, $\bbs_\fx(\Wbb)$ is a weakly-admissible $\Vbb^{\times M}$-module.
\end{eg}

The goal of this section is to prove Thm. \ref{lb50}. For that purpose, we need an explicit method of computing \eqref{eq72} in terms of the residue action of some global meromorphic section of $\scr V_C\otimes\omega_C$. Recall $W_1,\dots,W_M$ in Setting \ref{lb24}.

\begin{lm}\label{lb51}
Choose $b_1,\dots,b_M\in\Zbb$. Choose $E\in\Nbb$. Then there exists $T\in\Nbb$ such that for each $n\in\Zbb$, $v\in\Vbb^{\leq E}$, $1\leq j\leq M$,  there exists $\sigma\in H^0(C,\scr V_C^{\leq E}\otimes\omega_C(T\SX+\blt\DX))$   satisfying 
\begin{subequations}\label{eq122}
\begin{gather}
\mc U_\varrho(\theta_j)\sigma\big|_{W_j}\equiv v\cdot  \theta_j^nd\theta_j\quad\mod\quad H^0\big(W_j,\Vbb^{\leq E}\otimes_\Cbb\omega_C(-b_jy_j)\big) \\
\mc U_\varrho(\theta_k)\sigma\big |_{W_k}\equiv 0\quad\mod\quad H^0\big(W_k,\Vbb^{\leq E}\otimes_\Cbb\omega_C(-b_ky_k)\big)\qquad(\forall k\neq j)
\end{gather}
\end{subequations}
\end{lm}

The following Mittag-Leffler type argument is standard and has appeared in \cite{AN03-finite-dimensional,KZ-conformal-block,DGT2}. We follow the proof of \cite[Thm. 12.1]{Gui-sewingconvergence}.

\begin{proof}
It suffices to assume $n< b_j$, since the case $n\geq b_j$ is trivial if we set $\sigma=0$. Define divisor $\varDelta=-\sum_{k=1}^M b_ky_k$. By Asmp. \ref{ass1} and Serre's vanishing theorem (cf. \cite[Prop. 5.2.7]{Huy} or \cite[Thm. IV.2.1]{BaSt}), there exists $T\in \Nbb$ such that for all $t\geq T$,
    \begin{equation}
    H^1\big(C,\SV_C^{\leq E}\otimes \omega_C(tS_\fx+\varDelta)\big)=0  \label{eq121}
    \end{equation}
Fix $1\leq j\leq M$. Define $\varDelta'=-ny_j-\sum_{k\neq j}b_ky_k$. Then $\varDelta'\geq\varDelta$. Consider the short exact sequence
\begin{align*}
0\rightarrow \SV_C^{\leq E}\otimes \omega_C(TS_\fx+\varDelta)\rightarrow \SV_C^{\leq E}\otimes \omega_C(TS_\fx+\varDelta')\rightarrow\scr G\rightarrow 0
\end{align*} 
where $\scr G$ is the quotient sheaf of the previous two sheaves, which is an $\mc O_C$-module with support in $y_j$. By \eqref{eq121}, we have a long exact sequence
\begin{align}
\begin{aligned}
&0\rightarrow H^0\big(C,\SV_C^{\leq E}\otimes \omega_C(TS_\fx+\varDelta)\big)\rightarrow H^0\big(C,\SV_C^{\leq E}\otimes \omega_C(TS_\fx+\varDelta')\big)\\
\rightarrow&H^0(C,\scr G)\rightarrow 0
\end{aligned}
\end{align}
Define $\beta\in H^0(C,\scr G)$ to be (the equivalence class of) $\mc U_\varrho(\theta_j)^{-1}v\cdot\theta_j^nd\theta_j$ on $W_j$ and $0$ on $C-\{y_j\}$. Then any lift $\sigma\in H^0\big(C,\SV_C^{\leq E}\otimes \omega_C(TS_\fx+\varDelta')\big)$ of $\beta$ satisfies \eqref{eq122}. 
\end{proof}

\begin{pp}\label{lb52}
Let $a_1,\dots,a_M\in\Nbb$. Let $E\in\Nbb$ and $b_1=a_1+E,\dots,b_M=a_M+E$. Choose $T\in\Nbb$, $n\in\Zbb$, $v\in\Vbb^{\leq E}$, $1\leq j\leq M$, and $\sigma$ be as in Lem. \ref{lb51}. Then for each $\upphi\in\scr T_{\fk X,a_\star}^*(\Wbb)$ and $w\in\scr W_\fx(\Wbb)$ we have
\begin{align}
\<Y_j(v)_n\upphi,w\>=-\sum_{i=1}^N \upphi (\sigma*_i w)  \label{eq125}
\end{align}
\end{pp}

\begin{proof}
By \eqref{eq124} and \eqref{eq142}, the RHS of \eqref{eq125} equals $\sum_{k=1}^M\<\Res_{y_k} Y_k(\mc U_\varrho(\theta_k)\sigma,\theta_k)\upphi,w\>$.
Using \eqref{eq122} and \eqref{eq126}, one finds that this expression equals the LHS of \eqref{eq125}.
\end{proof}

\begin{thm}[Universal property]\label{lb50}
Choose local coordinates $\eta_1,\dots,\eta_N$ of $C$ at $x_1,\dots,x_N$. Then for each weakly-admissible $\Vbb^{\times M}$-module $\Mbb$ and each $\Gamma\in\scr T_{\fk X}^*(\Wbb\otimes\Mbb)$ (recall Def. \ref{lb25}), there exists a unique $T\in\Hom_{\Vbb^{\times M}}(\Mbb,\bbs_\fx(\Wbb))$ such that the following diagram commutes:
\begin{equation}
\begin{tikzcd}[row sep=tiny, column sep=large]
\Wbb\otimes\Mbb \arrow[rd,"\Gamma"] \arrow[dd,"\idt\otimes T"'] &   \\
                        & \Cbb \\
\Wbb\otimes\bbs_\fx(\Wbb) \arrow[ru,"\gimel"']            &  
\end{tikzcd}
\end{equation}
\end{thm}

\begin{rem}
As in Sec. \ref{lb58}, when considering $\scr T_{\fk X}^*(\Wbb\otimes\Mbb)$, we are assigning $\Mbb$ to the marked points $y_1,\dots,y_M$. So the $\fk X$ in $\scr T_{\fk X}^*(\Wbb\otimes\Mbb)$ has $N+M$ (incoming) marked points.
\end{rem}

\begin{proof}
By \eqref{eq79}, the only element $\upphi\in\bbs_\fx(\Wbb)$ annihilated by $\gimel(w\otimes\cdot)$ for all $w\in\Wbb$ is $0$. So $T$ must be unique. Let us prove the existence of $T$. Identify $\scr W_\fx(\Wbb)$ with $\Wbb$ via $\mc U(\eta_\blt)$.

Define a linear map $T:\Mbb\rightarrow\Wbb^*$ such that for each $m\in\Mbb$, 
\begin{gather*}
T(m):\Wbb\rightarrow\Cbb\qquad w\mapsto \Gamma(w\otimes m)
\end{gather*}
Let us prove that $T(\Mbb)\subset\bbs_\fx(\Wbb)$. Choose any $m\in\Mbb$. Choose $a_1,\dots,a_M\in\Nbb$ satisfying \eqref{eq118}. Choose  $\sigma\in H^0\big(C,\SV_{\fx,a_\star}\otimes \omega_C(\bullet S_\fx)\big)$ and $w\in \Wbb$. Since $\Gamma$ is a conformal block associated to $\fk X$ and $\Wbb\otimes\Mbb$, $\Gamma$ vanishes on $\sigma\cdot (w\otimes m)=\sigma\cdot w\otimes m+w\otimes\sigma\cdot m$. (Note that we are also viewing $\sigma$ as an element of $H^0(C,\scr V_\fx\otimes\omega_C(\blt\SX+\blt\DX))$.) So 
    \begin{equation}\label{eq119}
    T(m)(\sigma\cdot w)=\sum_{k=1}^M\Gamma\big(w\otimes(\sigma*_k m)\big)=-\sum_{i=1}^N \Gamma\big((\sigma*_i w)\otimes m\big)
    \end{equation}
By \eqref{eq118} and the local expression of $\sigma$ near $y_1,\cdots, y_M$, we have $\sigma*_k m=0$ and hence $\eqref{eq119}=0$. This proves $T(m)\in \ST_{\fx,a_\star}^*(\Wbb)\subset\bbs_\fx(\Wbb)$.

We now prove that $T$ is a weak $\Vbb^{\times M}$-module morphism. Choose any $E\in\Nbb$,  $v\in\Vbb^{\leq E}$, $n\in\Zbb$, $1\leq j\leq M$. Let $\upphi=T(m)$, and let $\sigma$ be as in Lem. \ref{lb51}. Then
\begin{align*}
\<Y_j(v)_n(T(m)),w\>\xlongequal{\eqref{eq125}}-\sum_{i=1}^N \<T(m), \sigma*_i w\>=\text{the RHS of }\eqref{eq119}
\end{align*}
Using \eqref{eq122} and \eqref{eq118}, one finds that the middle of \eqref{eq119} equals
\begin{align*}
\Gamma\big(w\otimes Y_j(v)_n m \big)=\<T(Y_j(v)_nm),w\>
\end{align*}
This proves that $T$ intertwines the actions of $Y_j(u)_n$.
\end{proof}

\begin{rem}\label{fusion-uniqueness}
It is clear that the pair $(\bbs_\fx(\Wbb),\gimel)$ is uniquely determined by the universal property in Thm. \ref{lb50}. Namely, if $\Xbb$ is a weakly-admissible $\Vbb^{\times M}$-module, $\daleth\in\scr T_{\fk X}^*(\Wbb\otimes\Xbb)$, and $(\Xbb,\daleth)$ satisfies the same property as $(\bbs_\fx(\Wbb),\gimel)$ in Thm. \ref{lb50}, then there is a (necessarily unique) isomorphism $\Phi:\Xbb\rightarrow\bbs_\fx(\Wbb)$ such that $\daleth=\gimel\circ(\idt\otimes\Phi)$. 
\end{rem}

\begin{co}\label{lb53}
Choose local coordinates $\eta_1,\dots,\eta_N$ of $C$ at $x_1,\dots,x_N$. Then for each weakly-admissible $\Vbb^{\times M}$-module $\Mbb$, we have an isomorphism of vector spaces
\begin{gather}
\begin{gathered}
\Hom_{\Vbb^{\times M}}(\Mbb,\bbs_\fx(\Wbb))\xlongrightarrow{\simeq} \scr T_\fx^*(\Wbb\otimes\Mbb)\\
T\mapsto \gimel\circ(\idt\otimes T)
\end{gathered}
\end{gather}
\end{co}

\begin{proof}
Immediate from Thm. \ref{lb50}.
\end{proof}

\subsection{$C_2$-cofiniteness implies $\dim\scr T_{\fk X,a_1,\dots,a_M}^*(\Wbb)<+\infty$}

In the remaining part of this chapter, we assume that $\Vbb$ is $C_2$-cofinite. Notice Thm. \ref{lb46} for many equivalent descriptions of grading-restricted $\Vbb^{\otimes N}$-modules. By Cor. \ref{lb44}, in order to show that $\boxbackslash_\fx(\Wbb)$ is a generalized $\Vbb^{\otimes M}$-module, we need to show that each $\scr T_{\fk X,a_\star}^*(\Wbb)$ is finite-dimensional. For that purpose, we need a preparatory result:

\begin{lm}\label{lb49}
Let $\Vbb_1,\dots,\Vbb_N$ be $C_2$-cofinite, and let $\Ebb\subset\Vbb_1\otimes\cdots\otimes\Vbb_N$ be the finite subset of homogeneous vectors in Thm. \ref{Miy}. Let $\Wbb$ be a finitely-generated admissible $\Vbb_1\times\cdots\times\Vbb_N$-module. Then for any $n\in \Nbb$, there exists $\nu(n)\in \Nbb$ such that any $\wtd L_\blt(0)$-homogeneous vector $w\in \Wbb$ satisfying $\wtd\wt(w)>\nu(n)$ is a finite sum of vectors of the form $Y_i(u_i)_{-l}w^\circ$ where $1\leq i\leq N$, $u=u_1\otimes \cdots \otimes u_N\in \Ebb,l>n$. 

\end{lm}

Recall Def. \ref{grading3} for the meanings of $\wtd\wt$ and $\wtd\wt_i$ and homogeneous vectors. 

\begin{proof}
By Thm. \ref{lb46}, $\Wbb$ is a finitely-generated weak $\Vbb_1\otimes\cdots\otimes\Vbb_N$-module. It suffices to consider the case that $\Wbb$ is generated by a single homogeneous vector $w_0$. Let $T$ be the set of  vectors of the form (\ref{finiteness1}) satisfying $n_k\leq nN-N+1$. So $T$ is a finite subset of $\Wbb$. Set $\nu(n)=\max\{\wtd\wt(w_1):w_1\in T\}$. If $w\in \Wbb$ is homogeneous and $\wtd \wt(w)>\nu(n)$, then we can also write $w$ as a sum of nonzero homogeneous vectors of the form (\ref{finiteness1}) whose $\wtd L_\blt(0)$-weights are equal to $\wtd\wt_\blt(w)$, but now $n_k$ must be greater than $nN-N+1$. So $w$ is a sum of vectors of the form $Y_\Wbb(u)_{-K}w_2$ where $u=u_1\otimes \cdots \otimes u_N\in \Ebb,K>nN-N+1$ and $w_2$ is homogeneous. Note that 
    $$
    Y_\Wbb(u)_{-K}w_2=\sum_{k_1+\cdots +k_N=K-1+N} Y_1(u_1)_{-k_1}\cdots Y_N(u_N)_{-k_N}w_2.
    $$
For each $(k_1,\cdots,k_N)$ satisfying $\sum_{i=1}^N k_i=K-1+N$, there exists $1\leq i\leq N$ such that 
    $$
    k_i\geq \frac{K-1+N}{N}>n.
    $$
This finishes the proof if we let $l=k_i$.
\end{proof}

In the following, we set $\Vbb_1=\cdots=\Vbb_N=\Vbb$ and fix $\Ebb\subset\Vbb^{\otimes N}$ as in Lem. \ref{lb49}.

\begin{thm}\label{lb45}
    Let $\Vbb$ be a $C_2$-cofinite VOA. Let $\Wbb$ be a finitely-generated admissible $\Vbb^{\times N}$-module. Then for each $a_1,\cdots,a_M\in \Nbb$, $\ST_{\fx,a_1,\cdots,a_M}(\Wbb)$ is finite dimensional.
\end{thm}

The proof of this theorem is similar to the proof that the spaces of conformal blocks have finite dimensions \cite{AN03-finite-dimensional,KZ-conformal-block,DGT2}. We include a proof for the readers' convenience.  Our approach follows \cite[Thm. 7.4]{Gui-sewingconvergence}. Recall $\Wbb(n)=\eqref{eq114}$. Then
\begin{align*}
\Wbb^{\leq n}=\bigoplus_{k\in\Nbb,k\leq n}\Wbb(k)\qquad \text{(cf. \eqref{eq127})}
\end{align*}
is finite-dimensional by Thm. \ref{lb46}. In the following proof, we assume:

\begin{sett}
In addition to Setting \ref{lb24}, we choose local coordinates $\eta_1,\dots,\eta_N$ of $C$ at $x_1,\dots,x_N$ defined on neighborhoods $U_1,\dots,U_N$. Assume that $U_1,\dots,U_N$ and $y_1,\dots,y_M$ are mutually disjoint. Identify $\scr W_{\fk X}(\Wbb)=\Wbb$ via $\mc U(\eta_\blt)$. Then $\ST_{\fx,a_1,\cdots,a_M}(\Wbb)=\Wbb/\SJ$ where
\begin{gather*}
    \SJ=H^0\big(C,\SV_{\fx,a_1,\cdots,a_M}\otimes \omega_C(\bullet S_\fx)\big)\cdot \Wbb
\end{gather*}
\end{sett}

\begin{proof}
Let $E=\max\{\wt(v):v\in \Ebb\}$. By Asmp. \ref{ass1} and Serre's vanishing theorem (cf. \cite[Prop. 5.2.7]{Huy} or \cite[Thm. IV.2.1]{BaSt}), there exists $k_0\in \Nbb$ such that 
    \begin{equation}\label{vanishing1}
    H^1\big(C,\SV_{\fx,a_1,\cdots,a_M}^{\leq E}\otimes \omega_C(kS_\fx)\big)=0
    \end{equation}
for all $k\geq k_0$. Fix an arbitrary $k\in \Nbb$ satisfying $k\geq E+k_0$. We shall prove that for any $n>\nu(k)$, any vector of $\Wbb(n)$ is a finite sum of elements of $\Wbb^{\leq n-1}$ mod $\SJ$. If this claim is true, then $\Wbb^{\leq\nu(k)}\rightarrow\Wbb/\scr J$ is surjective, and hence $\Wbb/\scr J$  is  finite-dimensional. 
     
    Choose $w\in \Wbb(n)$. By Lem. \ref{lb49}, $w$ is a sum of vectors of the form $Y_i(u_i)_{-l}w^\circ$, where $1\leq i\leq N$, $u=u_1\otimes \cdots \otimes u_N\in \Ebb$, and $l>k$. Then, since $\wt(u_i)\geq0$, we have
    \begin{equation}\label{eq117}
    \wtd\wt(w^\circ)\xlongequal{\eqref{eq115}}n-\wt(u_i)-l+1\leq n-k\leq n-E-k_0
    \end{equation}
    It suffices to show that each $Y_i(u_i)_{-l}w^\circ$ is a sum of elements of $\Wbb^{\leq n-1}\mod \SJ$. Thus we may assume for simplicity that $w=Y_i(u_i)_{-l}w^\circ$ for some $i$. From now on the $i$ is fixed.  
     
Consider the short exact sequence of $\MO_C$-modules
\begin{gather}
    0\rightarrow \SV_{\fx,a_1,\cdots,a_M}^{\leq E}\otimes \omega_C(k_0S_\fx)\rightarrow \SV_{\fx,a_1,\cdots,a_M}^{\leq E} \otimes \omega_C(lS_\fx)\rightarrow \mathscr{G}\rightarrow 0
\end{gather}
    where $\mathscr{G}$ is the quotient of the previous two sheaves. (Note that the support of $\scr G$ is a subset of $\{x_1,\dots,x_N\}$.) By (\ref{vanishing1}), we have an exact sequence
\begin{align}\label{eq116}
\begin{aligned}
&0\rightarrow H^0\big(C,\SV_{\fx,a_1,\cdots,a_M}^{\leq E}\otimes \omega_C(k_0S_\fx)\big)\rightarrow H^0\big(C,\SV_{\fx,a_1,\cdots,a_M}^{\leq E} \otimes \omega_C(lS_\fx)\big)\\
\rightarrow & H^0\big(C,\mathscr{G}\big)\rightarrow 0
\end{aligned}
\end{align}
Define an element $\sigma\in H^0\big(C,\mathscr{G}\big)$ as follows. $\mc U_\varrho(\eta_i)\sigma\vert_{U_i}$ is the equivalence class represented by $u_i\cdot \eta_i^{-l}d\eta_i$, and $\sigma\vert_{C-\{x_i\}}=0$. This makes $\sigma$ a well-defined global section of $\scr G$.


By the exactness of (\ref{eq116}), we can find a lift $\wht{\sigma}\in H^0\big(C,\SV_{\fx,a_1,\cdots,a_M}^{\leq E} \otimes \omega_C(lS_\fx)\big)$ of $\sigma$ and find some $v_j\in \Vbb^{\leq E}\otimes \MO_C(k_0S_\fx)(U_j)$ for each $1\leq j\leq N$ such that
\begin{subequations}
\begin{gather}
        \mc U_\varrho(\eta_i)\wht{\sigma}\vert_{U_i}=u_i\cdot \eta_i^{-l}d\eta_i+v_i\cdot d\eta_i\\
        \mc U_\varrho(\eta_j)\wht{\sigma}\vert_{U_j}=v_j\cdot d\eta_j\quad (\text{if }j\ne i)
\end{gather}
\end{subequations}
It follows that $\wht{\sigma}\cdot w^\circ\in \SJ$ equals $w+w_\Delta$ where 
    $$
    w_\Delta=\sum_{j=1}^N \big(\mc U_\varrho(\eta_j)^{-1}v_jd\eta_j\big)*_i w^\circ=\sum_{j=1}^N\Res_{\eta_j=0}~ Y_j(v_j,\eta_j)w^\circ d\eta_j.
    $$
    Thus $w=-w_\Delta\mod \SJ$. Note that for each $1\leq j\leq N$, residue action of $\mc U_\varrho(\eta_j)^{-1}v_jd\eta_j$ on $w^\circ$ increases the $\wtd L(0)$-weight by at most $E+k_0-1$. By (\ref{eq117}), $w_\Delta\in \Wbb^{\leq n-1}$. So our proof is complete.
\end{proof}

\subsection{$C_2$-cofiniteness implies that $\bbs_{\fk X}(\Wbb)$ is a grading-restricted $\Vbb^{\otimes M}$-module}

\begin{df}
Let $\Ubb$ be a VOA. A generalized $\Ubb$-module $\Mbb$ is called of \textbf{finite-length} (cf. \cite[Def. 1.2]{Hua-projectivecover}) if there is a chain of generalized submodules $0=\Mbb_0\subset\Mbb_1\subset\cdots\subset\Mbb_l=\Mbb$ where $\Mbb_i/\Mbb_{i-1}$ is an irreducible (generalized) $\Ubb$-module (i.e. a generalized $\Ubb$-module such that $0$ and $\Mbb_i/\Mbb_{i-1}$ are the only $\Ubb$-invariant subspaces). 
\end{df}

\begin{rem}
If $\Ubb$ is $C_2$-cofinite, then a generalized $\Ubb$-module is finitely-generated iff it is grading-restricted (cf. Thm. \ref{lb46}), iff it is of finite length (by \cite[Prop. 4.3]{Hua-projectivecover}).
\end{rem}

\begin{lm}\label{lb54}
Let $\Ubb$ be a $C_2$-cofinite VOA. Let $\Mbb$ be a generalized $\Ubb$-module. Suppose that each projective object $\Pbb$ in the category of finite-length generalized $\Ubb$-modules satisfies
\begin{align*}
\dim \Hom_\Ubb(\Pbb,\Mbb)<+\infty.
\end{align*}
Then $\Mbb$ is of finite length.
\end{lm}

We prove this lemma by mimicking the proof of \cite[Thm. 4.5]{Hua-projectivecover}.

\begin{proof}
Let $\mc S$ be the set of finite-length (equivalently, finitely-generated) generalized $\Ubb$-submodules of $\Mbb$. Our goal is to show that $\Mbb\in\mc S$. 

Assume that $\Mbb\notin\mc S$. Then we have a chain in $\mc S$:
\begin{gather*}
0=\Mbb_0\subsetneq\Mbb_1 \subsetneq\Mbb_2\subsetneq \cdots  \tag{$\star$} \label{eqa1}
\end{gather*}
Since each $\Mbb_i$ cannot be $\Mbb$, the above chain is infinite. Each $\Mbb_i/\Mbb_{i-1}$ has finite length since  $\Mbb_i\in\mc S$. Thus we can find finitely many intermediate modules between $\Mbb_{i-1}$ and $\Mbb_i$ forming a finite increasing chain such that the quotient of any member by the previous one in this chain is irreducible. Therefore, inserting all these intermediate modules into the chain \eqref{eqa1}, we may assume at the beginning that each $\Mbb_i/\Mbb_{i-1}$ is irreducible. 

Since $\Ubb$ is $C_2$-cofinite, there are only finitely many equivalence classes of irreducible $\Ubb$-modules (which are grading-restricted by Thm. \ref{lb46}), cf. \cite[Prop. 4.2]{Hua-projectivecover} or the end of \cite[Sec. 12]{Gui-sewingconvergence}. Thus we can find an infinite subset $B\subset\Zbb_+$ such that the  members of $\{\Mbb_i/\Mbb_{i-1}\}_{i\in B}$ are all isomorphic to a nonzero irreducible $\Ubb$-module $\Xbb$. By \cite[Thm. 3.23]{Hua-projectivecover}, there exists a projective cover $(\Pbb,p)$ of $\Xbb$ in the category of finite-length generalized $\Ubb$-modules. So there exists a morphism $p_i\in\Hom_\Ubb(\Pbb,\Mbb_i)$ such that the following diagram commutes and the horizontal line is exact:
\begin{equation}
\begin{tikzcd}
            &             & \Pbb \arrow[d,"p_i"'] \arrow[rd, two heads, "p"] &             &   \\
0 \arrow[r] & \Mbb_{i-1} \arrow[r] & \Mbb_i \arrow[r,"\pi_i"]            & \Xbb \arrow[r] & 0
\end{tikzcd}
\end{equation}
Since $p=\pi_i\circ p_i$ is surjective, $p_i(\Pbb)$ does not lie inside $\Mbb_{i-1}$.

Since $\dim\Hom_\Ubb(\Pbb,\Mbb)<+\infty$ and $B$ is infinite, $\{p_i\}_{i\in B}$ must be linearly dependent. So there exist $k\geq 1$, $i_1<\cdots<i_k<i_{k+1}$ in $B$, and $\lambda_{i_1},\dots\lambda_{i_k}\in\Cbb$ such that 
\begin{align*}
p_{i_{k+1}}=\lambda_{i_1}p_{i_1}+\cdots+\lambda_{i_k}p_{i_k}
\end{align*}
Then $p_{i_{k+1}}(\Pbb)$ is inside $\Mbb_{i_k}$, and hence inside $\Mbb_{i_{k+1}-1}$. This gives a contradiction.
\end{proof}

Recall Thm. \ref{lb46}.

\begin{thm}\label{lb57}
Assume that $\Vbb$ is $C_2$-cofinite. Let $\Wbb$ be a finitely-generated admissible $\Vbb^{\times N}$-module. Then $\boxbackslash_{\fx}(\Wbb)$ is a finitely-generated admissible $\Vbb^{\times M}$-module.
\end{thm}

\begin{proof}
By Cor. \ref{lb44} and Thm. \ref{lb45}, $\bbs_\fx(\Wbb)$ is a generalized $\Vbb^{\otimes M}$-module. If $\Pbb$ is a finite-length generalized $\Vbb^{\otimes M}$-module, then by Cor. \ref{lb53}, the vector space $\Hom_{\Vbb^{\otimes M}}(\Pbb,\bbs_\fx(\Wbb))$ is isomorphic to $\scr T_{\fk X}^*(\Wbb\otimes\Pbb)$, which has finite dimension by Thm. \ref{lb45}. (Recall that a conformal block is a partial conformal block associated to a pointed surface with no outgoing marked points.) Therefore, by Lem. \ref{lb54}, $\bbs_\fx(\Wbb)$ is a grading-restricted generalized $\Vbb^{\otimes M}$-module, equivalently (Thm. \ref{lb46}), a finitely-generated admissible $\Vbb^{\times M}$-module.
\end{proof}

In the following, we assume the setting of Rem. \ref{lb56}, which applies here since Thm. \ref{lb57} establishes that $\bbs_\fx(\Wbb)$ is an admissible $\Vbb^{\times M}$-module. Then under the identification
\begin{align}
\scr W_\fx(\Wbb\otimes\bbs_\fx(\Wbb))=\scr W_\fx(\Wbb)\otimes\bbs_\fx(\Wbb)\qquad\text{via }\mc U(\cdot,\theta_\star)
\end{align}
each element of $\scr T_\fx^*(\Wbb\otimes\bbs_\fx(\Wbb))$ (in particular, $\gimel$) is a linear functional on $\scr W_\fx(\Wbb)\otimes\bbs_\fx(\Wbb)$. The same is true if $\bbs_\fx(\Wbb)$ is replaced by any admissible $\Vbb^{\times M}$-module.

\begin{thm}[Universal property]\label{lb55}
Assume that $\Vbb$ is $C_2$-cofinite.  Then for each finitely-generated admissible $\Vbb^{\times M}$-module $\Mbb$ and each $\Gamma\in\scr T_{\fk X}^*(\Wbb\otimes\Mbb)$, there exists a unique $T\in\Hom_{\Vbb^{\times M}}(\Mbb,\bbs_\fx(\Wbb))$ such that the following diagram commutes:
\begin{equation}
\begin{tikzcd}[row sep=tiny, column sep=large]
\scr W_\fx(\Wbb)\otimes\Mbb \arrow[rd,"\Gamma"] \arrow[dd,"\idt\otimes T"'] &   \\
                        & \Cbb \\
\scr W_\fx(\Wbb)\otimes\bbs_\fx(\Wbb) \arrow[ru,"\gimel"']            &  
\end{tikzcd}
\end{equation}
\end{thm}

\begin{proof}
What we shall prove follows immediately from Thm. \ref{lb50} if we identify $\scr W_\fx(\Wbb)\otimes\Mbb$ with $\Wbb\otimes\Mbb$ and $\scr W_\fx(\Wbb)\otimes\bbs_\fx(\Wbb)$ with $\Wbb\otimes\bbs_\fx(\Wbb)$ via $\mc U(\eta_\blt)\otimes\idt$, thanks to the commutative diagram \eqref{eq128}.
\end{proof}

\appendix

\section{Modules of $\Vbb_1\times \cdots \times \Vbb_N$ and $\Vbb_1\otimes \cdots \otimes \Vbb_N$} \label{lb90}

Let $\Vbb_1,\cdots,\Vbb_N$ be VOAs. In this chapter, we do NOT assume that each graded subspace $\Vbb_i(n)$ of $\Vbb_i=\bigoplus_{n\in\Nbb}\Vbb_i(n)$ is finite-dimensional.

\subsection{Admissible (i.e. $\Nbb^N$-gradable) and $\Nbb$-gradable modules}\label{lb34}

The goal of this section is to address the issue raised in Rem. \ref{lb91}: whether a weak $\Vbb_1\times\cdots\times\Vbb_N$-module is a weak $\Vbb_1\otimes\cdots\otimes\Vbb_N$-module. As pointed out in Rem. \ref{lb91}, this is true and is actually known to experts. It can be proved by checking the weak associativity of the vertex operators defined by \eqref{eq31}. This ``formal variable approach" assumes some familiarity with the techniques (developed e.g. in \cite{Kac-beginners} or \cite{LL-introduction}) of handling the subtleties in the identities of formal variables. 

In this section, we prove a slightly weaker result (Thm. \ref{lb37}) using complex-analytic methods in the spirit of \cite{FHL93} and \cite{GuiLec}. An advantage of this approach is that once the formal series is shown to converge a.l.u. to a holomorphic function, the variables of this function can be safely changed to some other more convenient variables. As another advantage, one can check that many algebraic and formal operations actually commute using the fact that taking residues commute with taking a.l.u. convergent infinite sums of holomorphic functions. Since this approach does not seem to be as common as the formal variable approach in the VOA literature, we address all the subtleties in this approach with sufficient details for the readers' convenience, although the experienced readers can certainly fill in the details by their own efforts.

\begin{df}\label{lb41}
A weak $\Vbb_1\times \cdots \times \Vbb_N$-module $(\Wbb,Y_1,\dots,Y_N)$ (cf. Def. \ref{lb33}) is called  an \textbf{$\Nbb$-gradable $\Vbb_1\times\cdots\times\Vbb_N$-module} \index{00@$\Nbb$-gradable $\Vbb_1\times\cdots\times\Vbb_N$-module} if there is a diagonalizable linear operator $\wtd L(0)$ on $\Wbb$ with eigenvalues in $\Nbb$ such that for all $1\leq i\leq N$, $v_i\in\Vbb_i$, $n\in\Zbb$, we have
\begin{align}
[\wtd L(0),Y_i(v_i)_n]=Y_i(L(0)v_i)_n-(n+1)Y_i(v_i)_n.  \label{eq111}
\end{align}
We let
\begin{gather*}
\Wbb(n)=\{w\in\Wbb:\wtd L(0)w=nw\}\\
\Wbb'=\bigoplus_{n\in\Nbb}\Wbb(n)^*
\end{gather*}
A vector $w$ in $\Wbb$ resp. $\Wbb'$ is called \textbf{$\wtd L(0)$-homogeneous}, or simply \textbf{homogeneous}, if it belongs to $\Wbb(n)$ resp. $\Wbb(n)^*$ for some $n\in\Nbb$.
\end{df}

\begin{eg}
If $\Wbb$ is an admissible $\Vbb_1\times\cdots\times\Vbb_N$-module, then $\Wbb$ is an $\Nbb$-gradable $\Vbb_1\times\cdots\times\Vbb_N$-module if $\wtd L(0)$ is defined by \eqref{eq83}.
\end{eg}

\begin{eg}
If $N=1$, a weak $\Vbb$-module $\Wbb$ is admissible iff it is $\Nbb$-gradable. 
\end{eg}

\subsubsection{Main result}\label{lb32}

Let $\Wbb$ be an $\Nbb$-gradable $\Vbb_1\times\cdots\times\Vbb_N$-module with grading operator $\wtd L(0)$. For each $k\in \Zbb$, $v_1\in \Vbb_1,\cdots,v_N\in \Vbb_N,w\in \Wbb$, define
\begin{subequations}\label{eq31}
\begin{align}
Y_\Wbb(v_1\otimes \cdots \otimes v_N)_k w=\sum_{k_1+\cdots k_N=k+1-N}Y_1(v_1)_{k_1}\cdots Y_N(v_N)_{k_N}w
\end{align}
It is not hard to show that the RHS is  finite sum. Let
\begin{align}
Y_\Wbb(v_1\otimes \cdots \otimes v_N,z)=\sum_{k\in \Zbb}Y_\Wbb(v_1\otimes \cdots \otimes v_N)_k z^{-k-1}.
\end{align}
\end{subequations}

\begin{thm}\label{lb37}
\eqref{eq31} defines a 1-1 correspondence between the following notions:
\begin{enumerate}[label=(\alph*)]
\item $\Nbb$-gradable $\Vbb_1\times\cdots\times\Vbb_N$-modules.
\item Admissible (i.e. $\Nbb$-gradable) $\Vbb_1\otimes\cdots\otimes\Vbb_N$-modules.
\end{enumerate}
More precisely, let $(\Wbb,Y_1,\dots,Y_N)$ be an $\Nbb$-gradable $\Vbb_1\times\cdots\times\Vbb_N$-module with grading operator $\wtd L(0)$. Then $(\Wbb,Y_\Wbb)$, together with $\wtd L(0)$ defined by \eqref{eq83}, is an admissible $\Vbb_1\otimes \cdots \otimes \Vbb_N$-module. Moreover, every admissible $\Vbb_1\otimes \cdots \otimes \Vbb_N$-module (with suitable $\wtd L(0)$) can be realized in this way.
\end{thm}

Note that in (b) we are treating $\Vbb_1\otimes\cdots\otimes\Vbb_N$ as a single VOA. Clearly (b)$\Rightarrow$(a) by  setting $Y_i(v,z)=Y_\Wbb(\idt\otimes\cdots\otimes v\otimes\cdots\otimes\idt,z)$ (where $v$ is at the $i$-th tensor component). So we shall only prove (a)$\Rightarrow$(b).

\begin{co}\label{lb38}
Let $\Wbb$ be a weak $\Vbb_1\times\cdots\times\Vbb_N$-module. Suppose that $\Wbb$ is spanned by some weak  $\Vbb_1\times\cdots\times\Vbb_N$-submodules which are $\Nbb$-gradable. Then $\Wbb$ is a weak $\Vbb_1\otimes\cdots\otimes\Vbb_N$-module.
\end{co}

\begin{proof}
The only thing to check is the Jacobi identity for $Y_\Wbb$. By Thm. \ref{lb37}, the Jacobi identity holds when $Y_\Wbb$ is acting on each $\Nbb$-gradable weak $\Vbb_1\times\cdots\times\Vbb_N$-submodule.
\end{proof}

Our ultimate interest is not in weak or admissible $\Vbb_1\otimes\cdots\otimes\Vbb_N$-modules, but in generalized modules or even grading-restricted (generalized) modules of $\Vbb_1\otimes\cdots\otimes\Vbb_N$ which are extensively studied in the literature. This will be discussed in the next section.

\subsubsection{Another perspective on $Y_\Wbb$}\label{lb28}

Let $\Wbb$ be an $\Nbb$-gradable $\Vbb_1\times\cdots\times\Vbb_N$-module. For each $n\in\Nbb$, define the canonical projection
\begin{align}
P_n:\ovl\Wbb=\prod_{k\in\Nbb}\Wbb(k)\rightarrow \Wbb(n)=\{w\in\Wbb:\wtd L(0)w=nw\}
\end{align}
Then for each $v\in\Vbb_i$, $1\leq i\leq N$, and $m,n\in\Zbb$, $P_nY_i(v,z)P_m$ is an element of $\Hom(\Wbb(m),\Wbb(n))[z^{\pm1}]$ since, when $v$ is homogeneous,
\begin{align*}
P_n Y_i(v,z)P_m=z^{n-m-\wt v}\cdot P_nY_i(v)_{-n+m+\wt v-1}P_m
\end{align*} 
From this, it is easy to see that for each homogeneous $w\in\Wbb,w'\in\Wbb'$ and homogeneous $v_1\in\Vbb_1,\dots,v_N\in\Vbb_N$,
\begin{subequations}\label{eq84}
\begin{align}
&\<w',Y_1(v_1,z_1)\cdots Y_N(v_N,z_N)w\>  \nonumber\\
\xlongequal{\mathrm{def}}&\sum_{k_1,\dots,k_N}\<w',Y_1(v_1)_{k_1}\cdots Y_N(v_N)_{k_N}w\>z_1^{-k_1-1}\cdots z_N^{-k_N-1}  \label{eq85}\\
        =&\sum_{n_1,\dots,n_N\in\Nbb}\<w',P_{n_1}Y_1(v_1,z_1)P_{n_2} Y_2(v_2,z_2)\cdots P_{n_N}Y_N(v_N,z_N)w\>  \label{eq86}
\end{align}
\end{subequations}
holds in $\Cbb[[z_1^{\pm1},\dots,z_N^{\pm1}]]$. Moreover, on any open subset of $\Cbb^N$, the a.l.u. convergence of \eqref{eq85} is equivalent to that of \eqref{eq86}.\footnote{We warn the readers that these two a.l.u. convergences are NOT a priori equivalent when some $v_i$ is not homogeneous since, in that case, not all summands in \eqref{eq86} are monomials  of $z_1,\dots,z_N$.  See \cite[Subsec. 7.3]{GuiLec} for a detailed explanation.} (Here, the mutual commutativity of $Y_1,\dots,Y_N$ is not needed.)

\begin{rem}\label{lb27}
From \eqref{eq86}, we know that the powers of $z_N$ in \eqref{eq84} are bounded from below, and the powers of $z_1$ are bounded from above. From \eqref{eq85} and that $Y_1,\dots,Y_N$ mutually commute, we know that for each permutation $\sigma$ of $1,\dots, N$, we have
\begin{align}
&\<w',Y_1(v_1,z_1)\cdots Y_N(v_N,z_N)w\>\nonumber\\
=&\<w',Y_{\sigma(1)}(v_{\sigma(1)},z_{\sigma(1)})\cdots Y_{\sigma(N)}(v_{\sigma(N)},z_{\sigma(N)})w\>
\end{align}
These two facts together imply immediately that
\begin{align}
\<w',Y_1(v_1,z_1)\cdots Y_N(v_N,z_N)w\>\in\Cbb[z_1^{\pm1},\dots,z_N^{\pm1}]
\end{align}
and that \eqref{eq86} has only finitely many non-zero summands (so there is no convergence issue in \eqref{eq86}). Let $Y_\Wbb$ be defined by \eqref{eq31}. Then   in $\Cbb[z^{\pm1}]$ we clearly have
\begin{align}
\<w',Y_\Wbb(v_1\otimes\cdots\otimes v_N,z)w \>=\<w',Y_1(v_1,z_1)\cdots Y_N(v_N,z_N)w\>\big|_{z_1=\cdots=z_N=z}.
\end{align}
\end{rem}

\subsubsection{Some convergence properties}

Assume the setting of Subsec. \ref{lb32}. To prove that $Y_\Wbb$ satisfies Jacobi identity, we need the following convergence lemmas to ensure that taking residues commutes with certain infinite sums. Assume $N=2$, which is sufficient for the proof of Thm. \ref{lb37} by induction. The first lemma is proved in a similar way as \cite[Prop. 3.5.1]{FHL93}:

\begin{lm}\label{lb29}
         For each $w\in \Wbb,w'\in \Wbb',u_1,v_1\in \Vbb_1,u_2,v_2\in \Vbb_2$, the sums
\begin{subequations}
\begin{gather}
\sum_{n_1,n_2,n_3\in\Nbb} \<w',Y_1(u_1,z_1)P_{n_1}Y_2(u_2, \wtd z_1)P_{n_2}Y_1(v_1,z_2)P_{n_3}Y_2(v_2,\wtd z_2)w\>  \label{locality1}\\
\sum_{n_1,n_2,n_3\in\Nbb} \<w',Y_1(v_1,z_2)P_{n_1}Y_2(v_2, \wtd z_2)P_{n_2}Y_1(u_1,z_1)P_{n_3}Y_2(u_2,\wtd z_1)w\>  \label{locality2}\\
\sum_{n_1,n_2,n_3\in\Nbb} \<w',Y_1(u_1,z_1)P_{n_1}Y_1(v_1,z_2)P_{n_2}Y_2(u_2, \wtd z_1)P_{n_3}Y_2(v_2,\wtd z_2)w\>  \label{eq104}
\end{gather}
\end{subequations}
converge a.l.u. on $\Omega_1,\Omega_2,\Omega_1$ respectively where
\begin{subequations}
\begin{gather}
\Omega_1=\{(z_1,\wtd z_1,z_2,\wtd z_2)\in \Cbb^4:0<\vert z_2\vert <\vert z_1\vert,0<\vert \wtd z_2\vert <\vert \wtd z_1\vert\}\\
\Omega_2=\{(z_1,\wtd z_1,z_2,\wtd z_2)\in \Cbb^4:0<\vert z_1\vert <\vert z_2\vert,0<\vert \wtd z_1\vert <\vert \wtd z_2\vert\}
\end{gather}
\end{subequations}
and can be extended to the same holomorphic function $\varphi$ on 
\begin{gather}
\Omega=\{(z_1,\wtd z_1,z_2,\wtd z_2)\in (\Cbb^\times)^4:z_1\ne z_2,\wtd z_1\ne \wtd z_2\}  \label{eq91}
\end{gather}
such that $(z_1-z_2)^T (\wtd z_1-\wtd z_2)^T \varphi$ is holomorphic on $(\Cbb^\times)^4$ for some $T\in \Nbb$ independent of $w,w'$.
\end{lm}

\begin{proof}

Step 1. It suffices to assume that $u_1,v_1,u_2,v_2,w,w'$ are homogeneous. Then, as in \eqref{eq84}, we can also understand \eqref{locality1} as
\begin{align}
&\<w',Y_1(u_1,z_1)Y_2(u_2,\wtd z_1)Y_1(v_1,z_2)Y_2(v_2,\wtd z_2)w\>\nonumber\\
=&\sum_{k_1,k_2,k_3,k_4\in\Nbb}\<w',Y_1(u_1)_{k_1}Y_2(u_2)_{k_2}Y_1(v_1)_{k_3}Y_2(v_2)_{k_4}w\>\nonumber\\
&\qquad\qquad\qquad\cdot z_1^{-k_1-1}\wtd z_1^{-k_2-1}z_2^{-k_3-1}\wtd z_2^{-k_4-1}  \label{eq89}
\end{align}
On any open subset of $\Cbb^4$ the a.l.u. convergences of \eqref{locality1} and \eqref{eq89} are equivalent.  \eqref{locality2} and \eqref{eq104} can be understood in a similar way. Then it follows immediately that \eqref{locality1} and \eqref{eq104} are equal as elements of $\Cbb[[z_1^{\pm 1},\wtd z_1^{\pm 1},z_2^{\pm 1},\wtd z_2^{\pm 1}]]$ (and hence as holomorphic functions whenever the a.l.u. convergence holds) because $Y_1$ and $Y_2$ commute.\\[-1ex]

Step 2. By the locality of vertex operators $Y_1,Y_2$, we can choose $T\in \Nbb$ such that 
    \begin{equation}\label{eq96}
    (z_1-z_2)^T [Y_1(u_1,z_1),Y_1(v_1,z_2)]=0,\quad (\wtd z_1-\wtd z_2)^T [Y_2(u_2,\wtd z_1),Y_2(v_2,\wtd z_2)]=0
    \end{equation}
    as elements in $\End(\Wbb)[[z_1^{\pm 1},z_2^{\pm 1}]]$ and $\End(\Wbb)[[\wtd z_1^{\pm 1},\wtd z_2^{\pm 1}]]$. Let $f(z_1,\wtd z_1,z_2,\wtd z_2)=\eqref{locality1}$ and $g(z_1,\wtd z_1,z_2,\wtd z_2)=\eqref{locality2}$. Define 
    \begin{equation}\label{locality4}
    \psi(z_1,\wtd z_1,z_2,\wtd z_2)=(z_1-z_2)^T (\wtd z_1-\wtd z_2)^T  f(z_1,\wtd z_1,z_2,\wtd z_2)
    \end{equation}
viewed as an element in the $\Cbb[z_1^{\pm 1},\wtd z_1^{\pm 1},z_2^{\pm 1},\wtd z_2^{\pm 1}]$-module $\Cbb[[z_1^{\pm 1},\wtd z_1^{\pm 1},z_2^{\pm 1},\wtd z_2^{\pm 1}]]$. Then  (\ref{eq96}) and the commutativity of $Y_1$ and $Y_2$ imply that
    $$
    \psi=(z_1-z_2)^T (\wtd z_1-\wtd z_2)^T  g(z_1,\wtd z_1,z_2,\wtd z_2)
    $$

By \eqref{eq111}, the powers of $\wtd z_2$ (resp. $z_1$) in \eqref{locality1} are bounded from below (resp. above), and the same is true for $z_2$ (resp. $\wtd z_1$) since we can exchange the first two vertex operators and the second two in \eqref{locality1}. This implies that $f$ belongs to the ring $\Cbb((z_1^{-1},\wtd z_1^{-1},z_2,\wtd z_2))$, and so does $\psi$. Similarly, $g$ belongs to $\Cbb((z_1,\wtd z_1,z_2^{-1},\wtd z_2^{-1}))$, and so does $\psi$. Therefore
\begin{gather*}
\psi(z_1,\wtd z_1,z_2,\wtd z_2) \in \Cbb[z_1^{\pm 1},\wtd z_1^{\pm 1},z_2^{\pm 1},\wtd z_2^{\pm 1}].
\end{gather*}

Step 3. Now \eqref{locality4} can be viewed as a relation in the ring $R=\Cbb((z_1^{-1},\wtd z_1^{-1},z_2,\wtd z_2))$. In this ring, $(z_1-z_2)^T$ and $(\wtd z_1-\wtd z_2)^T$ have inverses
\begin{gather}
    (z_1-z_2)^{-T}=\sum_{j\in \Nbb}\binom{-T}{j}(-1)^j z_1^{-T-j}z_2^j  \label{eq87}\\
    (\wtd z_1-\wtd z_2)^{-T}=\sum_{j\in \Nbb}\binom{-T}{j}(-1)^j \wtd z_1^{-T-j}\wtd z_2^j \label{eq88}
\end{gather}
Therefore, in $R$ we have $f=\eqref{eq87}\cdot\eqref{eq88}\cdot\psi$. Clearly \eqref{eq87} and \eqref{eq88} converge a.l.u. on $\Omega_1$. So does $f$ because $\psi$ is a Laurent polynomial. Since the RHS of \eqref{eq87} and \eqref{eq88} converge a.l.u. on $\Omega_1$ to the LHS as holomorphic functions, $f$ as a formal Laurent series of $z_1,z_2,\wtd z_1,\wtd z_2$ converges a.l.u. on $\Omega_1$ to $\varphi$. Here
\begin{gather*}
\varphi=(z_1-z_2)^{-T} (\wtd z_1-\wtd z_2)^{-T}\psi\qquad\in\mc O(\Omega)
\end{gather*}
Therefore, by Step 1, \eqref{locality1} converges a.l.u. on $\Omega_1$ to $\varphi$. Similarly, \eqref{locality2} converges a.l.u. on $\Omega_2$ to $\varphi$.
\end{proof}

The following lemma is well-known when $\Wbb(n_2)$ is finite-dimensional. Without assuming the finite-dimensionality, the proof is more subtle.

\begin{lm}\label{lb30}
In Lem. \ref{lb29}, fix any $n_2\in\Nbb$. Then
\begin{subequations}
\begin{gather}
\sum_{n_1,n_3\in\Nbb} \<w',Y_1(u_1,z_1)P_{n_1}Y_1(v_1,z_2)P_{n_2}Y_2(u_2, \wtd z_1)P_{n_3}Y_2(v_2,\wtd z_2)w\> \label{eq94}\\
\sum_{n_1,n_3\in\Nbb}\<w',Y_1(P_{n_1}Y(u_1,z_1-z_2)v_1,z_2)P_{n_2}Y_2(P_{n_3}Y(u_2,\wtd z_1-\wtd z_2)v_2,\wtd z_2)w\> \label{eq95}
\end{gather}
\end{subequations}
converge a.l.u. on $\Omega_1,\Omega_3$ respectively where
\begin{align}
\Omega_3=\{(z_1,\wtd z_1,z_2,\wtd z_2)\in\Cbb^4:0<|z_1-z_2|<|z_2|,0<|\wtd z_1-\wtd z_2|<|\wtd z_2|\}
\end{align}
and can be extended to the same holomorphic function $\omega_{n_2}$ on $\Omega$.
\end{lm}
\begin{proof}
Step 1. It suffices to assume that all the vectors are homogeneous. \eqref{eq111} implies that \eqref{eq94} belongs to $\Cbb((z_1^{-1},\wtd z_1^{-1},z_2,\wtd z_2))$, and that
\begin{subequations}\label{eq101}
\begin{gather}
\sum_{n_1,n_3\in\Nbb} \<w',Y_1(v_1,z_2)P_{n_1}Y_1(u_1,z_1)P_{n_2}Y_2(v_2,\wtd z_2)P_{n_3}Y_2(u_2, \wtd z_1)w\> \label{eq98}
\end{gather}
 belongs to $\Cbb((z_1,\wtd z_1,z_2^{-1},\wtd z_2^{-1}))$. By \eqref{eq96}, when multiplied by $(z_1-z_2)^T(\wtd z_1-\wtd z_2)^T$,  \eqref{eq94} and \eqref{eq98} are equal as elements of the $\Cbb[z_1^{\pm1},\wtd z_1^{\pm1},z_2^{\pm1},\wtd z_2^{\pm1}]$-module $\Cbb[[z_1^{\pm1},\wtd z_1^{\pm1},z_2^{\pm1},\wtd z_2^{\pm1}]]$. So this element must be in $\Cbb[z_1^{\pm1},\wtd z_1^{\pm1},z_2^{\pm1},\wtd z_2^{\pm1}]$. Using this fact, one shows as in the proof of Lem. \ref{lb29} that \eqref{eq94} and \eqref{eq98} converge a.l.u. on $\Omega_1,\Omega_2$ respectively to the same function $\omega_{n_2}\in\mc O(\Omega)$. A similar argument shows that
\begin{gather}
\sum_{n_1,n_3\in\Nbb} \<w',Y_1(v_1,z_2)P_{n_1}Y_1(u_1,z_1)P_{n_2}Y_2(u_2, \wtd z_1)P_{n_3}Y_2(v_2,\wtd z_2)w\>\label{eq102}\\
\sum_{n_1,n_3\in\Nbb} \<w',Y_1(u_1,z_1)P_{n_1}Y_1(v_1,z_2)P_{n_2}Y_2(v_2,\wtd z_2)P_{n_3}Y_2(u_2, \wtd z_1)w\>
\end{gather}
\end{subequations}
converge a.l.u. on $\{|z_1|<|z_2|,|\wtd z_1|>|\wtd z_2|\}$ and $\{|z_1|>|z_2|,|\wtd z_1|<|\wtd z_2|\}$ respectively to $\omega_{n_2}$. \\[-1ex]

Step 2. It remains to show that \eqref{eq95} converges a.l.u. on $\Omega_3$ to $\omega_{n_2}$. This is equivalent to showing that for each $r,\rho>0$, \eqref{eq95} converges a.l.u. to $\omega_{n_2}$ on the multi-annulus
\begin{align}
\Omega_{3,r,\rho}=\{(z_1,\wtd z_1,z_2,\wtd z_2)\in\Cbb^4:0<|z_1-z_2|<r<|z_2|,0<|\wtd z_1-\wtd z_2|<\rho<|\wtd z_2|\}
\end{align}
The advantage of working with holomorphic functions on $\Omega_{3,r,\rho}$ is that we can take Laurent series expansions with respect to the four variables $z_1-z_2,z_2,\wtd z_1-\wtd z_2,\wtd z_2$. By basic facts about Laurent series expansions of holomorphic functions on multi-annuli (cf. e.g. \cite[Lem. 7.13]{GuiLec}), the RHS of the following converges a.l.u. on $\Omega_{3,r,\rho}$ to the LHS:
\begin{align*}
\omega_{n_2}=\sum_{k_1,k_2\in\Zbb} a_{k_1,k_2}(z_2,\wtd z_2)\cdot (z_1-z_2)^{-k_1-1}(\wtd z_1-\wtd z_2)^{-k_2-1}
\end{align*}
where
\begin{align*}
a_{k_1,k_2}=\Res_{z_1-z_2=0}\Res_{\wtd z_1-\wtd z_2=0}~&\omega_{n_2}(z_1,\wtd z_1,z_2,\wtd z_2)\\
&\cdot (z_1-z_2)^{k_1}(\wtd z_1-\wtd z_2)^{k_2}d(\wtd z_1-\wtd z_2)d(z_1-z_2)
\end{align*}
is a holomorphic function on $\Gamma_{r,\rho}=\{(z_2,\wtd z_2)\in\Cbb^2:|z_2|>r,|\wtd z_2|>\rho\}$. The proof will be completed if we can show that
\begin{align}
a_{k_1,k_2}(z_2,\wtd z_2)=\<w',Y_1(Y(u_1)_{k_1}v_1,z_2)P_{n_2}Y_2(Y(u_2)_{k_2}v_2,\wtd z_2)w\>  \label{eq97}
\end{align}

Step 3. We fix $(z_2,\wtd z_2)\in\Gamma_{r,\rho}$ and verify \eqref{eq97}. Choose circles $C_{-1},C_1,\wtd C_{-1},\wtd C_{1}$ in $\Cbb$ with center $0$ and radii $<|z_2|,>|z_2|,<|\wtd z_2|,>|\wtd z_2|$ respectively. Then
\begin{align}
&a_{k_1,k_2}(z_2,\wtd z_2)\nonumber\\
=&\sum_{i,j=-1,1} (-1)^{i+j}\oint_{C_i}\oint_{C_j}\omega_{n_2}(z_1,\wtd z_1,z_2,\wtd z_2)\cdot (z_1-z_2)^{k_1}(\wtd z_1-\wtd z_2)^{k_2}\frac {d\wtd z_1}{2\pi\im}\frac{d z_1}{2\pi\im}  \label{eq100}
\end{align}
by residue theorem. The four summands in \eqref{eq100} can be computed by \eqref{eq94} and \eqref{eq101}: By the fact that contour integrals commute with a.l.u. convergent series of holomorphic functions, \eqref{eq102} implies that
\begin{align}
&\oint_{C_{-1}}\oint_{C_1}\omega_{n_2}(z_1,\wtd z_1,z_2,\wtd z_2)\cdot (z_1-z_2)^{k_1}(\wtd z_1-\wtd z_2)^{k_2}\frac {d\wtd z_1}{2\pi\im}\frac{d z_1}{2\pi\im}\nonumber\\
=&\sum_{l_1,l_2\in\Nbb}{k_1\choose l_1}{k_2\choose l_2}\oint_{C_{-1}}\oint_{C_1}\omega_{n_2}z_1^{l_1}(-z_2)^{k_1-l_1}\wtd z_1^{k_2-l_2}(-\wtd z_2)^{l_2}\frac{d\wtd z_1}{2\pi\im}\frac{d z_1}{2\pi\im}\nonumber\\
=&\sum_{l_1,l_2\in\Nbb}{k_1\choose l_1}{k_2\choose l_2}\sum_{n_1,n_3\in\Nbb}(-z_2)^{k_1-l_1}(-\wtd z_2)^{l_2} \nonumber\\
~~&\cdot \<w',Y_1(v_1,z_2)P_{n_1}Y_1(u_1)_{l_1}P_{n_2}Y_2(u_2)_{k_2-l_2}P_{n_3}Y_2(v_2,\wtd z_2)w\>\nonumber\\
=&\sum_{l_1,l_2\in\Nbb}{k_1\choose l_1}{k_2\choose l_2}(-z_2)^{k_1-l_1}(-\wtd z_2)^{l_2}\nonumber\\
~~&\cdot \<w',Y_1(v_1,z_2)Y_1(u_1)_{l_1}P_{n_2}Y_2(u_2)_{k_2-l_2}Y_2(v_2,\wtd z_2)w\>
\end{align}
where the last term is a finite sum by \eqref{eq111}. The other three integrals in \eqref{eq100} can be calculated in the same way. This computes the LHS of \eqref{eq97}.

By the Jacobi identity \eqref{jacobi} for $Y_i$ (where $i=1,2$), for each $m,n\in\Nbb$, $k\in\Zbb$, and homogeneous $u,v\in\Vbb$, we have in $\End_\Cbb(\Wbb)[z^{\pm1}]$ that
\begin{align*}
&P_mY_i\big(Y(u)_kv,z\big)P_n\\
=&\sum_{l\in\Nbb} {k\choose l}(-z)^lP_mY_i(u)_{k-l}Y_i(v,z)P_n-\sum_{l\in\Nbb}{k\choose l}(-z)^{k-l}P_mY_i(v,z)Y_i(u)_lP_n
\end{align*}
from which one easily computes the RHS of \eqref{eq97} and finds that it equals the LHS.
\end{proof}

\subsubsection{Proof of Thm. \ref{lb37}}

\begin{proof}[\textbf{Proof of Thm. \ref{lb37}}]
By induction on $N$, it suffices to assume $N=2$. Clearly, $Y_\Wbb(\ibf\otimes \ibf,z)=\ibf_\Wbb$ and $Y_\Wbb$ satisfies the lower truncation property. Moreover, it is easy to show
    \begin{equation}\label{app7}
     [\wtd{L}(0),Y_\Wbb(u_1\otimes u_2,z)]=Y_\Wbb\big(L(0)(u_1\otimes u_2),z\big)+z\partial_z Y_\Wbb(u_1\otimes u_2,z).
    \end{equation}
using \eqref{eq111}. So it remains to check the Jacobi identity (\ref{jacobi}) for $Y_\Wbb$.  Let $u_1,v_1\in \Vbb_1,u_2,v_2\in \Vbb_2$ be homogeneous and 
\begin{gather*}
A(z)=Y_\Wbb(u_1\otimes u_2,z),\quad B(z)=Y_\Wbb(v_1\otimes v_2,z).
\end{gather*}
Our goal is to prove for each $m,n,k\in\Zbb$
\begin{subequations}\label{eq109}
\begin{align}
\begin{aligned}\label{eq90}
    &\sum_{l\in \Nbb}\binom{m}{l}(A_{n+l}B)_{m+k-l}\\
    =&\sum_{l\in \Nbb}(-1)^l \binom{n}{l} A_{m+n-l}B_{k+l}-\sum_{l\in \Nbb}(-1)^{n+l}\binom{n}{l}B_{n+k-l}A_{m+l}.
\end{aligned}
\end{align}
 where, for each $k,l\in\Zbb$,
\begin{align}
 (A_kB)_l=Y_\Wbb(Y(u_1\otimes u_2)_{k}(v_1\otimes v_2))_l 
\end{align}
\end{subequations}

Step 1. Choose homogeneous $w\in\Wbb,w'\in\Wbb'$. By Subsec. \ref{lb28}, for $z\in\Cbb^\times$ we have
\begin{gather*}
\bigbk{w',A(z)w}=\sum_{n\in\Nbb}\bigbk{w',Y_1(u_1,z)P_nY_2(u_2,z)w}\\
\bigbk{w',B(z)w}=\sum_{n\in\Nbb}\bigbk{w',Y_1(v_1,z)P_nY_2(v_2,z)w}
\end{gather*}
where the RHS contain only finitely many non-zero summands.  Let $\Omega=\eqref{eq91}$ and
\begin{align*}
\varphi=\varphi(z_1,\wtd z_1,z_2,\wtd z_2)\qquad\in\mc O(\Omega)
\end{align*}
be defined by Lem. \ref{lb29}. Let
\begin{gather*}
\Gamma=\Conf^2(\Cbb^\times)\\
\Gamma_1=\{(z_1,z_2)\in \Cbb^\times \times \Cbb^\times :0<\vert z_2\vert <\vert z_1\vert\},\\
\Gamma_2=\{(z_1,z_2)\in \Cbb^\times \times \Cbb^\times :0<\vert z_1\vert <\vert z_2\vert\}\\
\Gamma_3=\{(z_1,z_2)\in\Cbb^\times\times\Cbb^\times:0<|z_1-z_2|<|z_2|\}\\
f(z_1,z_2)=\varphi(z_1,z_1,z_2,z_2)
\end{gather*}
Then $f\in\mc O(\Gamma)$. By Lem. \ref{lb29}, the RHS of the following converge a.l.u. on the given region to the LHS:
\begin{subequations}
\begin{gather}
f=\sum_{n\in\Nbb}\bigbk{w',A(z_1)P_n B(z_2)w}\quad \text{on }\Gamma_1 \label{eq107}\\
f=\sum_{n\in\Nbb}\bigbk{w',B(z_2)P_n A(z_1)w}\quad \text{on }\Gamma_2 \label{eq108}\\
f=\sum_{n_1,n_2,n_3\in\Nbb}\<w',Y_1(u_1,z_1)P_{n_1}Y_1(v_1,z_2)
P_{n_2}Y_2(u_2,z_1)P_{n_3}Y_2(v_2,z_2)w\>\quad\text{on }\Gamma_1  \label{eq105}
\end{gather}
\end{subequations}

Step 2. Recall the function $\omega_{n_2}\in\mc O(\Omega)$ in Lem. \ref{lb30}. Define $g_{n_2}\in\mc O(\Gamma)$ by
\begin{align*}
g_{n_2}(z_1,z_2)=\omega_{n_2}(z_1,z_1,z_2,z_2)
\end{align*} 
Our goal in this step is to show that the RHS of the following converges a.l.u. to the LHS:
\begin{align}
f(z_1,z_2)=\sum_{n_2\in\Nbb} g_{n_2}(z_1,z_2)\qquad\text{on }\Gamma \label{eq103}
\end{align}
We already know that this is true on $\Gamma_1$ due to \eqref{eq105}. To extend \eqref{eq103} from $\Gamma_1$ to $\Gamma$, we follow the proof of \cite[Thm. 8.4]{GuiLec}.

Consider the functions $F\in\mc O(\Gamma\times\Cbb^\times)$ and $G_{n_2}\in\mc O(\Gamma)$ defined by
\begin{gather*}
F(z_1,z_2,q)=q^{\wt(u_2)+\wt(v_2)+\wtd\wt(w_2)}\varphi(z_1,qz_1,z_2,qz_2)\\
G_{n_2}(z_1,z_2)=\Res_{q=0}~F(z_1,z_2,q)q^{-n_2-1}dq
\end{gather*}
By complex analysis, the RHS of the following converges a.l.u. on $\Gamma\times\Cbb^\times$ to the LHS:
\begin{align*}
F(z_1,z_2,q)=\sum_{n_2\in\Zbb} G_{n_2}(z_1,z_2)q^{n_2}
\end{align*}
Note that $f(z_1,z_2)=F(z_1,z_2,1)$. So \eqref{eq103} is true if we can prove that $G_{n_2}$ equals $g_{n_2}$ on $\Gamma$. (Here, we set $g_{n_2}=0$ if $n_2<0$.)

By the uniqueness of analytic continuation, it suffices to show that $G_{n_2}=g_{n_2}$ on $\Gamma_1$. Fix $(z_1,z_2)\in\Gamma_1$. By \eqref{eq104} and \eqref{eq111}, we have
\begin{align*}
&F(z_1,z_2,q)=\sum_{n_1,n_2,n_3\in\Nbb}\bigbk{w',Y_1(u_1,z_1)P_{n_1}Y_1(v_1,z_2)P_{n_2}q^{\wtd L_0}Y_2(u_2,z_1)P_{n_3}Y_2(v_2,z_2)w}\\
=&\sum_{n_1,n_2,n_3\in\Nbb} q^{n_2}\<w',Y_1(u_1,z_1)P_{n_1}Y_1(v_1,z_2)P_{n_2}Y_2(u_2,z_1)P_{n_3}Y_2(v_2,z_2)w\>
\end{align*} 
which, by \eqref{eq94}, equals $\sum_{n_2\in\Nbb} g_{n_2}(z_1,z_2)q^{n_2}$. This finishes the proof of \eqref{eq103}.\\[-1ex]

Step 3. Let us compute 
\begin{align*}
h_k(z_2):=\Res_{z_1=z_2}(z_1-z_2)^k f(z_1,z_2)dz_1=\oint_{C(z_2)} (z_1-z_2)^k f(z_1,z_2)\frac{dz_1}{2\pi\im}
\end{align*}
where $C(z_2)$ is any circle centered at $z_2$ with radius $<|z_2|$. By \eqref{eq103}, the sum of \eqref{eq95} over all $n_2\in\Nbb$ converges a.l.u. on $\Gamma_3$ to $f$ if we set $\wtd z_1=z_1,\wtd z_2=z_2$. Therefore, by the fact that contour integrals commute with a.l.u. convergent series \eqref{eq103} and \eqref{eq95}, we have
\begin{align}
&h_k(z_2)\nonumber\\
=&\sum_{n_2\in\Nbb}\sum_{n_1,n_3\in\Nbb}\oint_{C(z_2)} (z_1-z_2)^k\big\langle w',Y_1(P_{n_1}Y(u_1,z_1-z_2)v_1,z_2)\nonumber\\
&\qquad\qquad\qquad\quad\cdot P_{n_2}Y_2(P_{n_3}Y(u_2,z_1-z_2)v_2, z_2)w\big\rangle \frac{dz_1}{2\pi\im}\nonumber\\
=&\sum_{n_2\in\Nbb}\sum_{n_1,n_3\in\Nbb}\sum_{
\begin{subarray}{c}
k_1,k_2\in\Zbb\\
k=k_1+k_2+1
\end{subarray}
}\big\langle w',Y_1(P_{n_1}Y(u_1)_{k_1}v_1,z_2) P_{n_2}Y_2(P_{n_3}Y(u_2)_{k_2}v_2, z_2)w\big\rangle \nonumber\\
=&\sum_{n_2\in\Nbb}\sum_{m\in\Zbb} \big\langle w',Y_1(Y(u_1)_mv_1,z_2)\cdot P_{n_2}Y_2(Y(u_2)_{k-1-m}v_2,z_2)w\big\rangle \nonumber\\
=&\<w',Y_\Wbb(Y(u_1\otimes u_2)_k(v_1\otimes v_2),z_2)w\>  \label{eq106}
\end{align}
where the second last expression has finitely many nonzero summands (cf. Rem. \ref{lb27}).\\[-1ex]

Step 4. By complex analysis, the RHS of the following converges a.l.u. to the LHS:
\begin{align}
f=\sum_{k\in\Zbb} h_k(z_2)\cdot (z_1-z_2)^{-k-1} \qquad\text{on }\Gamma_3  \label{eq110}
\end{align}
For each $z_2\in\Cbb^\times$, if we let $C_-,C_+$ be circles centered at $0$ with radii $<|z_2|$ and $>|z_2|$ respectively, then for each $m,n\in\Zbb$ we have
\begin{align*}
&\oint_{C(z_2)}f(z_1,z_2)z_1^m(z_1-z_2)^n\frac{dz_1}{2\pi\im}\\
=&\oint_{C_+}f(z_1,z_2)z_1^m(z_1-z_2)^n\frac{dz_1}{2\pi\im}-\oint_{C_-}f(z_1,z_2)z_1^m(z_1-z_2)^n\frac{dz_1}{2\pi\im}
\end{align*}
These three integrals can be computed respectively by \eqref{eq110}, \eqref{eq107}, \eqref{eq108} in the same way as \eqref{eq74} and \eqref{eq75}: we get
\begin{align*}
&\sum_{l\in\Nbb}{m\choose l}z_2^{m-l}h_{n+l}(z_2)\\
=&\sum_{l\in\Nbb}(-z_2)^l{n\choose l}\bigbk{w',A_{m+n-l}B(z_2)w}-\sum_{l\in\Nbb}{n\choose l}(-z_2)^{n-l}\bigbk{w',B(z_2)A_{m+l}w}
\end{align*}
Substitute \eqref{eq106} into the LHS, multiply both sides by $z_2^h$, and apply $\Res_{z_2=0}(\cdot)dz_2$. Then we get \eqref{eq109}.
\end{proof}

\subsection{Generalized $\Vbb_1\otimes\cdots\otimes\Vbb_N$-modules}  \label{lb35}

Let $\Ubb$ be a VOA whose eigenspaces of $L(0)$ are not necessarily finite-dimensional. Recall the following definitions from \cite{Hua-projectivecover}.

\begin{df}\label{lb40}
A weak $\Ubb$-module $\Wbb$ is called a \textbf{generalized $\Ubb$-module} \index{00@Generalized $\Ubb$-module: lower-truncated, grading-restricted} if $\Wbb$ is spanned by the generalized eigenvectors of $L(0)$, or equivalently, if we have a decomposition
    \begin{equation}\label{gendef1}
    \Wbb=\bigoplus_{n\in \Cbb}\Wbb_{[n]}
    \end{equation}
    where $\Wbb_{[n]}$ \index{Wn@$\Wbb_{[n]},\Wbb_{[n_1,\dots,n_N]}$} is the subspace of generalized eigenvectors of $L(0)$ with eigenvalue $n\in \Cbb$, i.e. $\Wbb_{[n]}$ is the set of all $w\in\Wbb$ satisfying $(L(0)-n)^kw=0$ for some $k\in\Zbb_+$ (possibly depending on $n$ and $w$).

A generalized $\Ubb$-module $\Wbb$ is called a \textbf{grading-restricted (generalized) module}  if $\Wbb_{[n]}=0$ when $\Re(n)$ is sufficiently negative, and if each $\Wbb_{[n]}$ is finite-dimensional.
\end{df}

Applying the above definitions to the case $\Ubb=\Vbb_1\otimes\cdots\otimes\Vbb_N$, one wants to know whether a weak $\Vbb_1\times\cdots\times\Vbb_N$-module is a generalized module or even a grading-restricted module of $\Vbb_1\otimes\cdots\otimes\Vbb_N$. This is a main goal of this section.

\subsubsection{General results}

For each $1\leq i\leq N,n\in\Zbb$ and each homogeneous $v_i\in\Vbb_i$, we call $Y_i(v_i)_n$ a \textbf{raising operator} resp. \textbf{lowering operator} if the \textbf{degree of vertex operator}
\begin{align}
\deg\big(Y_i(v_i)_n\big):= \wt(v_i)-n-1
\end{align}
is $\geq0$ resp. $\leq 0$. 

\begin{rem}
Let $\Wbb$ be a weak $\Vbb_1\times\cdots\times\Vbb_N$-module. For each $s_1,\dots,s_N\in\Cbb$, define \index{Wn@$\Wbb_{[n]},\Wbb_{[n_1,\dots,n_N]}$}
\begin{align}
\Wbb_{[s_1,\dots,s_N]}=\{w\in\Wbb:\forall 1\leq i\leq N, \exists k\in\Zbb_+\text{ such that }(L_i(0)-s_i)^kw=0\}
\end{align}
Then by Jacobi identity \eqref{jacobi}, for each $1\leq i,j\leq N$ we have
\begin{subequations}\label{eq157}
\begin{gather}
[L_j(0),Y_i(v_i)_n]=\delta_{i,j}(Y_i(L(0)v_i)_n-(n+1)Y_i(v_i)_n\big)  \label{eq140} \\
[L(0),Y_i(v_i)_n]=Y_i(L(0)v_i)_n-(n+1)Y_i(v_i)_n
\end{gather}
\end{subequations}
where $L(0):=L_1(0)+\cdots+L_N(0)$. Let $\Wbb_{[s]}$ be the generalized eigenspace of $L(0)$ with eigenvalue $s$. Thus, if $Y_i(v_i)_n$ has degree $d_i=\wt(v_i)-n-1$, then
\begin{gather}
Y_i(v_i)_n \Wbb_{[s_1,\dots,s_i,\dots,s_N]}\subset\Wbb_{[s_1,\dots,s_i+d_i,\dots,s_N]}, \qquad
Y_i(v_i)_n \Wbb_{[s]}\subset\Wbb_{[s+d_i]}  \label{eq112}
\end{gather}
\end{rem}

\begin{lm}\label{lb36}
Let $\fk A$ be the subalgebra of $\End(\Wbb)$ generated by all $Y_i(v_i)_n$ where $1\leq i\leq N,v_i\in\Vbb_i,n\in\Zbb$. Let $\fk A_+$ resp. $\fk A_-$ be the unital subalgebra generated by all raising operators resp. lowering operators. Then $\fk A=\Span_\Cbb(\fk A_+\cdot\fk A_-)$. 
\end{lm}

\begin{proof}
One needs to show that a product of $k\in\Zbb_+$ vertex operators can be written as a linear combination of those in $\fk A_+\cdot\fk A_-$. This follows easily by induction on $k$ and by the formula for $[Y_i(v_i)_m,Y_j(v_j)_n]$ implied by the Jacobi identity.
\end{proof}

The following criterion on generalized $\Vbb_1\otimes\cdots\otimes\Vbb_N$-modules can be applied to dual fusion products. (See Cor. \ref{lb44}.)   

\begin{pp}\label{lb42}
Let $\Wbb$ be a weak $\Vbb_1\times\cdots\times\Vbb_N$-module generated by elements of $\{\mc T_\alpha:\alpha\in\mc A\}$ where each $\mc T_\alpha$ is a finite-dimensional subspace invariant under the action of every lowering operator $Y_i(v_i)_n$ (where $1\leq i\leq N,n\in\Zbb$ and $v_i\in\Vbb_i$ is homogeneous). Then $\Wbb$ is a generalized $\Vbb_1\otimes\cdots\otimes\Vbb_N$-module.
\end{pp}

\begin{proof}

Since $\mc T_\alpha$ is finite-dimensional and  invariant under $L_1(0),\dots,L_N(0)$, by linear algebra, $\mc T_\alpha$ is spanned by common generalized eigenvectors of $L_1(0),\dots,L_N(0)$. Thus, there is a finite set $E\subset\Cbb^N$ such that 
\begin{align*}
\mc T_\alpha=\bigoplus_{(s_1,\dots,s_N)\in E+\Nbb^N} (\mc T_\alpha\cap\Wbb_{[s_1,\dots,s_N]})
\end{align*}
and that any two elements of $E$ do not differ by an element of $\Zbb^N$.

Let $\mc W_\alpha=\fk A\cdot \mc T_\alpha$. By Lem. \ref{lb36}, we have $\mc W_\alpha=\fk A_+\cdot\mc T_\alpha$. Thus, by \eqref{eq112},
\begin{align*}
\mc W_\alpha=\bigoplus_{(s_1,\dots,s_N)\in E+\Nbb^N} \mc W_{\alpha,[s_1,\dots,s_N]}
\end{align*}
where $\mc W_{\alpha,[s_1,\dots,s_N]}=\mc W_\alpha\cap\Wbb_{[s_1,\dots,s_N]}$. For each $1\leq i\leq N$, define a linear operator $\wtd L_i(0)$ on $\mc W_\alpha$ satisfying for each $\mbf e=(e_1,\dots,e_N)\in E$ that
\begin{align*}
\wtd L_i(0)=s_i-e_i \qquad\text{ on }\mc W_{\alpha,[s_1,\dots,s_N]}\text{ if } (s_1,\dots,s_N)\in\mbf e+\Nbb^N
\end{align*} 
Then $\wtd L_1(0),\dots,\wtd L_N(0)$ make $\mc W_\alpha$ an admissible $\Vbb_1\times\cdots\times\Vbb_N$-module. By assumption, $\Wbb=\Span_\alpha\mc W_\alpha$. Therefore, by Cor.  \ref{lb38}, $\Wbb$ is a weak $\Vbb_1\otimes\cdots\otimes\Vbb_N$-module. Since each $\mc W_\alpha$ is spanned by generalized eigenvectors of $L(0)=L_1(0)+\cdots+L_N(0)$, so is $\Wbb$. So $\Wbb$ is a generalized $\Vbb_1\otimes\cdots\otimes\Vbb_N$-module.
\end{proof}

\subsubsection{The $C_2$-cofinite case}

The following theorem follows immediately from \cite[Lem. 2.4]{Miy-modular-invariance} and the fact that a tensor product of $C_2$-cofinite VOAs is $C_2$-cofinite. It shows, in particular, that if $\Vbb$ is a $C_2$-cofinite VOA then each graded subspace $\Vbb(n)$ is finite-dimensional.

\begin{thm}\label{Miy}
    Let $\Vbb_1,\dots,\Vbb_N$ be $C_2$-cofinite. Then there exists a finite subset $\Ebb\subset \Vbb_1\otimes\cdots\otimes\Vbb_N$ such that any weak $\Vbb_1\otimes\cdots\otimes\Vbb_N$-module $\Wbb$ generated by a vector $w_0$ is spanned by 
    \begin{equation}\label{finiteness1}
    Y_{\Wbb}(v_k)_{-n_k}Y_\Wbb(v_{k-1})_{-n_{k-1}}\cdots Y_\Wbb(v_1)_{-n_1}w_0
    \end{equation}
    where $k\in \Nbb,v_1,\cdots,v_k\in \Ebb$ and $n_1<n_2<\cdots<n_k$. Moreover, $\Ebb$ can be chosen as a subset whose elements are of the form $u_1\otimes \cdots \otimes u_N$, where $u_1,\cdots,u_N$ are homogeneous. 
\end{thm}

\begin{thm}\label{lb46}
Assume that $\Vbb_1,\dots,\Vbb_N$ are $C_2$-cofinite. Then the following statements for $\Wbb$ are equivalent:
\begin{enumerate}[label=(\alph*)]
\item $\Wbb$ is a finitely-generated admissible $\Vbb_1\times\cdots\times\Vbb_N$-module.
\item $\Wbb$ is a finitely-generated finitely-admissible $\Vbb_1\times\cdots\times\Vbb_N$-module.
\item $\Wbb$ is a finitely-generated admissible $\Vbb_1\otimes\cdots\otimes\Vbb_N$-module.
\item $\Wbb$ is a finitely-generated finitely-admissible $\Vbb_1\otimes\cdots\otimes\Vbb_N$-module.
\item $\Wbb$ is a finitely-generated generalized $\Vbb_1\otimes\cdots\otimes\Vbb_N$-module.
\item $\Wbb$ is a grading-restricted (generalized) $\Vbb_1\otimes\cdots\otimes\Vbb_N$-module.
\end{enumerate}
\end{thm}

\begin{proof}
(e,f)$\Rightarrow$(b): Assume (e,f). Then there exists a finite subset $F\subset\Cbb$ such that the finite-dimensional subspace $\mc T=\bigoplus_{t\in F}\Wbb_{[t]}$ generates $\Wbb$. Since $\Wbb$ is grading-restricted, we can extend $F$ to a larger finite set such that if $t\in F,n\in\Nbb$ and $\Wbb_{[t-n]}\neq 0$ then $t-n\in F$. Thus $\mc T$ is invariant under $\fk A_-$. As in the proof of Prop. \ref{lb42}, we can find a finite subset $E\subset\Cbb^N$, any two elements of which do not differ by an element of $\Zbb^N$, such that
\begin{align*}
\mc T=\bigoplus_{(s_1,\dots,s_N)\in E+\Nbb^N} (\mc T\cap\Wbb_{[s_1,\dots,s_N]})
\end{align*}
By Lem. \ref{lb36}, $\Wbb=\fk A_+\mc T$. So
\begin{align*}
\Wbb=\bigoplus_{(s_1,\dots,s_N)\in E+\Nbb^N}\Wbb_{[s_1,\dots,s_N]}
\end{align*}
and by Thm. \ref{Miy} and \eqref{eq112}, each $\Wbb_{[s_1,\dots,s_N]}$ is finite-dimensional. For each $1\leq i\leq N$, define a linear operator $\wtd L_i(0)$ on $\Wbb$ satisfying for each $\mbf e=(e_1,\dots,e_N)\in E$ that
\begin{align*}
\wtd L_i(0)=s_i-e_i \qquad\text{ on }\Wbb_{[s_1,\dots,s_N]}\text{ if } (s_1,\dots,s_N)\in\mbf e+\Nbb^N
\end{align*} 
Then $\wtd L_1(0),\dots,\wtd L_N(0)$ make $\Wbb$ a finitely admissible $\Vbb_1\times\cdots\times\Vbb_N$-module.

(b)$\Rightarrow$(a): Obvious. 

(a)$\Rightarrow$(c): Immediate from Thm. \ref{lb37}.

It remains to show that (c,d,e,f) are equivalent. Note that if they are equivalent whenever $N=1$, then they are equivalent for any $N$ by considering $\Ubb=\Vbb_1\otimes\cdots\otimes\Vbb_N$ as a single VOA. Thus, in the following, we assume $N=1$ and write $\Vbb_1=\Vbb$.

(e)$\Rightarrow$(f): By \eqref{eq112} and Thm. \ref{Miy}. (f)$\Rightarrow$(e): By \cite[Prop. 4.3]{Hua-projectivecover}, (f) implies that $\Wbb$ has finite length, and hence is finitely-generated.

(e,f)$\Rightarrow$(d): Apply (e,f)$\Rightarrow$(b) to the case $N=1$.

(d)$\Rightarrow$(c): Obvious.

(c)$\Rightarrow$(e): Assume (c). By Thm. \ref{Miy} and \eqref{eq99}, $\Wbb$ is a finitely-admissible $\Vbb$-module. Thus, by \eqref{eq112}, $\Wbb$ is spanned by some finite-dimensional $L(0)$-invariant subspace. So $\Wbb$ is spanned by generalized eigenvectors of $L(0)$. This proves (e).
\end{proof}

\section{A geometric construction of higher level Zhu algebras}\label{lb87}

In this chapter, we fix a VOA $\Vbb$ and $n\in \Nbb$.  Our goal is to show that the higher level Zhu algebra $A_n(\Vbb)$ (originally introduced in \cite{Zhu-modular-invariance,DLM-Zhu}) can be constructed in a natural way from dual fusion products and propagation of partial conformal blocks. (That $A_0(\Vbb)$ can be realized from dual fusion products was indicated by \cite[Prop. 7.2.2 and A.2.7]{NT-P1_conformal_blocks}.) We want to impress upon the audience that many computations about $A_n(\Vbb)$ in the literature have geometric meanings, and that these geometric interpretations can be formulated in a precise way. (In particular, the algebra structure of $A_n(\Vbb)$ is implied by the module structures of dual fusion products.) It would be interesting to see if some other associative algebras and/or their modules related to $\Vbb$ (e.g. the $A_m(\Vbb)$-$A_n(\Vbb)$ bimodule $A_{m,n}(\Vbb)$ constructed by Dong-Jiang \cite{DJ-bimodules}, the associative algebras $A^n(\Vbb)$ and $A^\infty(\Vbb)$ defined by Huang in \cite{Hua-associative}) can also be given geometric meanings in terms of the dual fusion products. See also \cite{Li-regular-Zhu,Li-regular-AnV,Li-regular-bimodules} for closely related materials.

The results of this chapter are not used elsewhere in the series of three papers.

\subsection{Preliminary}

Throughout this chapter, we let $\zeta$ denote the standard coordinate of $\Cbb$, and fix  a $(2,1)$-pointed compact Riemann surface with local coordinates
\begin{align}
\fq=\big(\infty,0;1/\zeta,\zeta~\big|~\Pbb^1~\big|~1;\zeta-1\big)
\end{align}
where $1$ is the incoming point (with local coordinate $\zeta-1$) and $\infty,0$ are the outgoing ones (with local coordinates $1/\zeta,\zeta$ respectively). Associate the vacuum module $\Vbb$ to the incoming marked point $1$, and identify
\begin{align*}
\scr W_\fq(\Vbb)=\Vbb\qquad\text{via }\mc U(\zeta)
\end{align*} 
By Thm. \ref{lb43}, $(\bbs_\fq(\Vbb),Y_+,Y_-)$  is a weak $\Vbb\times\Vbb$-module where $Y_+$ \index{Y@$Y_+,Y_-$} is the vertex operator for the marked point $\infty$ and $Y_-$ is for $0$. Write the Virasoro operators as
\begin{align}
L_+(k)=Y_+(\cbf)_{k+1}\qquad L_-(k)=Y_-(\cbf)_{k+1}
\end{align}
Recall Def. \ref{lb59} for the meaning of $Y'$. 

The residue actions $*_+=*_\infty$ and $*_-=*_0$ of $H^0\big(\Pbb^1, \SV_{\fq,n,n}\otimes \omega_{\Pbb^1}(\bullet 1)\big)$ on $\Vbb$ are defined as in \eqref{eq80} (cf. Def. \ref{lb17}). The following lemma is helpful for the computation of these actions.

\begin{lm}\label{lb73}
Let
\begin{subequations}
\begin{gather}
f=\sum_{k\in\Zbb} f_k\zeta^k\qquad\in\Cbb((\zeta))\\
g=\sum_{k\in\Zbb} g_k\zeta^k\qquad\in\Cbb((\zeta^{-1}))
\end{gather}
\end{subequations}
be respectively sections of $\mc O_{\Pbb^1}(\blt 0+\blt\infty)$ on a neighborhood $U_+$ of $\infty$ and a neighborhood $U_-$ of $0$ such that $0\notin U_+$ and $\infty\notin U_-$. Let $v\in\Vbb$. Let
\begin{gather*}
\tau=g\cdot \mc U_\varrho(\zeta)^{-1}vd\zeta\qquad\in H^0(U_+,\scr V_\fq\otimes\omega_{\Pbb^1}(\blt \infty))\\
\sigma=f\cdot \mc U_\varrho(\zeta)^{-1}vd\zeta\qquad\in H^0(U_-,\scr V_\fq\otimes\omega_{\Pbb^1}(\blt 0))
\end{gather*} 
Let $(\Wbb,Y_+,Y_-)$ be a weak $\Vbb\times\Vbb$-module associated to $0,\infty$. Then for each $w\in\Wbb$,
\begin{subequations}\label{eq145}
\begin{gather}
\tau*_+ w=-\sum_{k\in\Zbb}g_k Y'_+(v)_k w  \label{eq131}\\
\sigma *_- w=\sum_{k\in\Zbb} f_k Y_-(v)_kw  \label{eq130}
\end{gather}
\end{subequations}
\end{lm}

\begin{proof}
\eqref{eq130} is obvious.  
$\mc U_\varrho(1/\zeta)\tau$, as a $\Vbb$-valued holomorphic $1$-form, equals 
\begin{align*}
\mc U_\varrho(1/\zeta)\tau=g\cdot e^{\zeta L_1}(-\zeta^{-2})^{L_0}vd\zeta
\end{align*}
thanks to Exp. \ref{lb60}. From this, one easily computes that
\begin{align*}
\tau *_+ w=\Res_{1/\zeta=0}~Y_+(\mc U_\varrho(1/\zeta)\tau,1/\zeta)w
\end{align*}
equals the RHS of \eqref{eq131}.
\end{proof}

\subsection{Construction of $A_n(\Vbb)$ from the dual fusion product $\bbs_\fq(\Vbb)$}

The goal of this section is to prove Thm. \ref{zhumain}.

\subsubsection{$\wtd A_n(\Vbb)$, $A_n(\Vbb)$, and the multiplications $\diamond_L,\diamond_R$}

\begin{df}
Define a vector space \index{AV@$\wtd A_n(\Vbb)$}
\begin{gather}\label{eq134}
\wtd A_n(\Vbb)\xlongequal{\text{def}}\ST_{\fq,n,n}(\Vbb)\xlongequal{\eqref{eq129}}\frac{\Vbb}{H^0\big(\Pbb^1, \SV_{\fq,n,n}\otimes \omega_{\Pbb^1}(\bullet 1)\big)\cdot \Vbb}\\
\text{so}\qquad\wtd A_n(\Vbb)^*=\scr T_{\fq,n,n}^*(\Vbb) \nonumber
\end{gather}
Define bilinear maps $\diamond_L,\diamond_R:\Vbb\times \Vbb\rightarrow \wtd A_n(\Vbb)^{**}$ \index{zz@$u\diamond_L v,u\diamond_R v$} such that for each $\upphi\in \wtd A_n(\Vbb)^*$ and homogeneous $u,v\in \Vbb$, 
\begin{subequations}\label{zhu6}
\begin{gather}
\bk{\upphi,u\diamond_L v}=\<Y_+'(u)_{\wt u-1}\cdot \upphi,v\>, \label{eq132}\\
\bk{\upphi,u\diamond_R v}=\<Y_-(v)_{\wt v-1}\cdot \upphi,u\>.  \label{eq133}
\end{gather}
\end{subequations}
Notice that by Prop. \ref{lb19} and \eqref{eq139}, 
\begin{align}\label{eq135}
Y_+'(u)_{\wt u-1}\cdot \wtd A_n(\Vbb)^*\subset \wtd A_n(\Vbb)^*\qquad\quad Y_-(v)_{\wt v-1}\cdot \wtd A_n(\Vbb)^*\subset \wtd A_n(\Vbb)^*
\end{align}
\end{df}

\begin{eg}\label{lb62}
For each $\upphi\in\wtd A_n(\Vbb)^*$ we have
\begin{gather}
\<\upphi,\idt\diamond_L v\>=\<\upphi,v\diamond_R\idt\>=\<\upphi,v\>  \label{eq148}\\
\<\upphi,\cbf\diamond_L v\>=\<L_+(0)\upphi,v\>\qquad\quad
\<\upphi,v\diamond_R\cbf\>=\<L_-(0)\upphi,v\>  \label{eq149}
\end{gather}
\end{eg}

\begin{lm}\label{lb61}
$\diamond_L,\diamond_R$ are bilinear maps $\Vbb\times\Vbb\rightarrow\wtd A_n(\Vbb)$.
\end{lm}

\begin{proof}
By Prop. \ref{lb52}, the RHS of \eqref{eq132} can be written as $\upphi(m)$ for some $m\in\Vbb$ independent of $\upphi$. This shows that $u\diamond_L v$ equals the equivalence class of $m$ in $\wtd A_n(\Vbb)$. A similar thing can be said about $u\diamond_R v$.
\end{proof}

\begin{pp}\label{lb66}
$\diamond_L$ and $\diamond_R$ descend to bilinear maps
\begin{align}
\diamond_L:\Vbb\times \wtd A_n(\Vbb)\rightarrow \wtd A_n(\Vbb)\qquad \diamond_R:\wtd A_n(\Vbb)\times\Vbb\rightarrow\wtd A_n(\Vbb)   \label{eq136}
\end{align}
Moreover, if $u,v,w\in\Vbb$, then the following relation holds in $\wtd A_n(\Vbb)$:
\begin{align}
(u\diamond_L v)\diamond_R w=u\diamond_L(v\diamond_R w)  \label{eq137}
\end{align}
\end{pp}

\begin{proof}
By \eqref{zhu6} and \eqref{eq135}, if $v$ belongs to the denominator of the RHS of \eqref{eq134} (equivalently, if $v$ is killed by every $\upphi\in\wtd A_n(\Vbb)^*$), then so does $u\diamond_Lv$ and $v\diamond_R u$ for every $u\in \Vbb$. Thus, \eqref{eq136} follows immediately from Lem. \ref{lb61}. Since $Y_+'$ and $Y_-$ commute by Thm. \ref{lb43}, for each $\upphi\in\wtd A_n(\Vbb)^*$ and each homogeneous $u,v,w\in\Vbb$,  \eqref{zhu6} implies that
\begin{align*}
&\bigbk{\upphi,(u\diamond_L v)\diamond_R w}=\bigbk{Y_-(w)_{\wt w-1}\cdot \upphi,u\diamond_Lv}=\bigbk{Y_+'(u)_{\wt u-1}\cdot Y_-(w)_{\wt w-1}\cdot \upphi,v}\\
=&\bigbk{Y_-(w)_{\wt w-1}\cdot Y_+'(u)_{\wt u-1}\cdot \upphi,v}=\bigbk{Y_+'(u)_{\wt u-1}\cdot \upphi,v\diamond_R w}=\bigbk{\upphi,u\diamond_L(v\diamond_Rw)}
\end{align*}
\end{proof}

If one can show that $\diamond_L=\diamond_R$, then these two bilinear maps descend to $\diamond:\wtd A_n(\Vbb)\times\wtd A_n(\Vbb)\rightarrow\wtd A_n(\Vbb)$. Unfortunately, $\diamond_L$ and $\diamond_R$ are in general not equal unless when $n=0$. To get equal maps, one needs to pass to a quotient $A_n(\Vbb)$ of $\wtd A_n(\Vbb)$.

\begin{df}
Define
\begin{subequations}\label{eq152}
\begin{gather}
Z_n(\Vbb)=\big\{\cbf\diamond_L v-v\diamond_R\cbf:v\in \wtd A_n(\Vbb)\big\}\qquad\subset \wtd A_n(\Vbb)\\[0.5ex]
A_n(\Vbb)=\wtd A_n(\Vbb)/Z_n(\Vbb)
\end{gather}
\end{subequations}
Thus, by Exp. \ref{lb62}, the dual space of $A_n(\Vbb)$ is
\begin{align}
A_n(\Vbb)^*=\big\{\upphi\in \wtd A_n(\Vbb)^*: L_+(0)\upphi= L_-(0)\upphi  \big\} \label{eq146}
\end{align}
\end{df}

\begin{lm}\label{zhu11}
$\diamond_L$ and $\diamond_R$ descend to bilinear maps
\begin{align}
\diamond_L:\Vbb\times A_n(\Vbb)\rightarrow A_n(\Vbb)\qquad \diamond_R: A_n(\Vbb)\times\Vbb\rightarrow A_n(\Vbb)   \label{eq143}
\end{align}
\end{lm}

     \begin{proof}
We need to show that if $v\in\Vbb$ is homogeneous, then $v\diamond_L-$ and $-\diamond_R v$ send $Z_n(\Vbb)$ into $Z_n(\Vbb)$. It suffices to prove
         \begin{align}\label{zhu10}
             Y_+'(v)_{\wt v-1}\cdot A_n(\Vbb)^*\subset A_n(\Vbb)^*\qquad\quad Y_-(v)_{\wt v-1}\cdot A_n(\Vbb)^*\subset A_n(\Vbb)^*
         \end{align}
Choose $\upphi\in\wtd A_n(\Vbb)^*$  killed by $L_+(0)-L_-(0)$. Then the formula
\begin{align*}
[L_\pm(0),Y_-(v)_{\wt v-1}]=[Y_\pm(\cbf)_1,Y_-(v)_{\wt v-1}]=0
\end{align*}
implied by \eqref{eq140} shows that $Y_-(v)_{\wt v-1}\upphi$ is also killed by $L_+(0)-L_-(0)$. This proves the second relation in \eqref{zhu10}. A similar argument together with \eqref{eq139} proves the first relation in \eqref{zhu10}.
     \end{proof}

\begin{pp}\label{lb69}
We have $Z_0(\Vbb)=0$, and hence $A_0(\Vbb)=\wtd A_0(\Vbb)$.
\end{pp}

\begin{thm}\label{lb68}
For each $u,v\in \Vbb$, the equivalence classes of $u\diamond_L v$ and $u\diamond_R v$ in $A_n(\Vbb)$ are equal. Thus (by Lem. \ref{zhu11}), $\diamond_L$ and $\diamond_R$ descend to the same bilinear map
\begin{align}
\diamond: A_n(\Vbb)\times A_n(\Vbb)\rightarrow A_n(\Vbb) \label{eq147}
\end{align}
\end{thm}

We defer the proofs of Prop. \ref{lb69} and Thm. \ref{lb68} to the next subsection.

   \begin{thm}\label{zhumain}
$A_n(\Vbb)$ is an associative algebra if we let the map $\diamond$ in Thm. \ref{lb68} be the multiplication. The equivalence class of $\idt$ in $A_n(\Vbb)$ is the unit. The equivalence class of $\cbf$ in $A_n(\Vbb)$ is a central element.
    \end{thm}

    \begin{proof}
The associativity of $\diamond$ follows from Prop. \ref{lb66}. That $\idt$ is the unity in $A_n(\Vbb)$ follows from \eqref{eq148}. That $\cbf$ is central in $A_n(\Vbb)$ follows from \eqref{eq152}.
    \end{proof}

\subsubsection{Proofs of Prop. \ref{lb69} and Thm. \ref{lb68}}

\begin{lm}\label{lb65}
Let $v\in\Vbb$. Then the relation $v\diamond_L\idt=\idt\diamond_R v$ holds in $\wtd A_n(\Vbb)$.
\end{lm}

\begin{proof}
By \eqref{zhu6}, it suffices to show that for each $\upphi\in\wtd A_n(\Vbb)^*$ we have
\begin{align}
-\<Y_+'(v)_{k}\cdot\upphi,\idt \>+\<Y_-(v)_{k}\cdot\upphi,\idt \>=0 \label{eq144}
\end{align}
where $k=\wt v-1$. Indeed, we shall prove \eqref{eq144} for all $k\in\Zbb$. Let $\sigma=\mc U_\varrho(\zeta)^{-1}v\cdot \zeta^{k}d\zeta$ which is in $H^0(\Pbb^1,\scr V_{\Pbb^1}\otimes\omega_{\Pbb^1}(\blt 0+\blt\infty))$. By \eqref{eq145}, the LHS of \eqref{eq144} equals $\bk{\sigma*_+\upphi,\idt}+\bk{\sigma*_-\upphi,\idt}$, and hence equals $-\bk{\upphi,\sigma \cdot \idt}$ by Rem. \ref{lb64}. And $\sigma\cdot \idt=\sum_{l\in\Nbb}{k\choose l}Y(v)_l\idt=0$.
\end{proof}

Note that if we set $v=\cbf$ in \eqref{eq144}, we get for all $\upphi\in\wtd A_n(\Vbb)^*$ that 
         \begin{align}\label{zhu16}
             \<L_+(0)\upphi,\ibf\>=\<L_-(0)\upphi,\ibf\>
         \end{align}

\begin{pp}\label{lb67}
Choose $v\in\Vbb$ and $\upphi\in \wtd A_n(\Vbb)^*$.  The following relation holds in $\Cbb[z^{\pm1}]$:
\begin{align}\label{eq154}
\<Y_+'(v,z)\upphi,\idt\>=\<Y_-(v,z)\upphi,\ibf\>
\end{align}
Its value at $z=1$ is
\begin{align}\label{eq156}
\<Y_-(v,1)\upphi,\ibf\>=\upphi(v)
\end{align}
\end{pp}

\begin{proof}
By \eqref{eq144}, \eqref{eq154} holds in $\Cbb[[z^{\pm1}]]$. The RHS of \eqref{eq154} belongs to $\Cbb((z))$ by the lower-truncation property for the weak $\Vbb\times\Vbb$-module $\bbs_\fq(\Vbb)$. Similarly, the LHS of \eqref{eq154} belongs to $\Cbb((z^{-1}))$ by the ``upper truncation property" (cf. \eqref{eq139}). So  $\eqref{eq154}\in\Cbb[z^{\pm1}]$. For each $k\in\Zbb$, by \eqref{eq72} we have
\begin{align*}
 \<Y_-(v)_k\upphi,\ibf\>=\Res_{z=0}~{\wr\upphi}(\mc U_\varrho(\zeta)^{-1}v,\idt)z^kdz
\end{align*}
Therefore, $\wr\upphi(\mc U_\varrho(\zeta)^{-1}v,\idt)$, a priori an element of $\mc O(\Cbb^\times-1)$ (cf. Thm. \ref{lb71}), is a Laurent polynomial of the standard complex variable $z$ and is equal to \eqref{eq154}. By \eqref{eq82}, in $\Cbb[[z-1]]$ we have (by taking the Laurent series expansion of the LHS at $z=1$)
\begin{align}
\wr \upphi(\MU_\varrho(\zeta)^{-1}v,\ibf)_z=\upphi\big(Y(v,z-1)\idt \big)  \label{eq141}
\end{align}
The RHS equals $\upphi\big(e^{(z-1)L(-1)}v\big)$, which is $\upphi(v)$ when $z=1$. This proves \eqref{eq156}.
\end{proof}

Formula \eqref{eq156} is useful because it looks similar to $\<w',Y_\Wbb(\cdot,1)w\>$. (See also Rem. \ref{lb72}.) Thus, many computations about $\<w',Y_\Wbb(v,1)w\>$ can be straightforwardly generalized to $\upphi(v)$, as shown in the proof of:

\begin{pp}\label{lb70}
Choose any $\upphi\in\wtd A_n(\Vbb)^*$. Then $\upphi$ belongs to $A_n(\Vbb)^*$ if and only if for each homogeneous $v\in\Vbb$ we have
\begin{align}
\<Y_-(v,z)\upphi,\ibf\>=z^{-\wt v}\cdot \upphi(v)   \label{eq155}
\end{align}
\end{pp}

\begin{proof}
Choose $\upphi\in\wtd A_n(\Vbb)^*$. Since $Y_+$ commutes with $Y_-$, we have in $\Cbb[[z^{\pm1}]]$ that
\begin{align}
\begin{aligned}
&(z\partial_z+\wt v)\<Y_-(v,z)\upphi,\ibf\>\xlongequal{\eqref{eq157}}\<[L_-(0),Y_-(v,z)]\upphi,\idt\>  \\
\xlongequal{\eqref{zhu16}}&\<(L_+(0)Y_-(v,z)-Y_-(v,z)L_-(0))\upphi,\idt \>=\<Y_-(v,z)(L_+(0)-L_-(0))\upphi,\idt \>  \label{eq160}
\end{aligned}
\end{align}
which, by Prop. \ref{lb67}, is an element of $\Cbb[z^{\pm1}]$. 

Suppose $\upphi\in A_n(\Vbb)^*$. Then $L_+(0)\upphi=L_-(0)\upphi$, and hence \eqref{eq160} is $0$. So $z^{\wt v}\<Y_-(v,z)\upphi,\idt\>$ is constant over $z$. Thus \eqref{eq155} follows from \eqref{eq156}.

Conversely, suppose that \eqref{eq155} is true. Then $z^{\wt v}\<Y_-(v,z)\upphi,\idt\>$ is constant over $z$, and hence \eqref{eq160} equals $0$. Thus, by \eqref{eq156}, 
\begin{align*}
\<(L_+(0)-L_-(0))\upphi,v\>=\<Y_-(v,1)(L_+(0)-L_-(0))\upphi,\idt \>=0
\end{align*}
So $L_+(0)\upphi-L_-(0)\upphi=0$, i.e., $\upphi\in A_n(\Vbb)^*$.
\end{proof}

\begin{proof}[\textbf{Proof of Prop. \ref{lb69}}]
It suffices to prove $\wtd A_0(\Vbb)^*= A_0(\Vbb)^*$. Choose any $\upphi\in\wtd A_0(\Vbb)^*=\scr T_{\fq,0,0}^*(\Vbb)$. By Prop. \ref{lb20} and \eqref{eq139}, we know that $z^{\wt v}\< Y_-(v,z)\upphi,\idt\>$ is in $\Cbb[[z]]$ and $z^{\wt v}\<Y_+'(v,z)\upphi,\idt\>$ is in $\Cbb[[z^{-1}]]$. So $z^{\wt v}\<Y_+'(v,z)\upphi,\idt\>$ is a constant, which together with \eqref{eq156} implies \eqref{eq155}. So $\upphi\in A_0(\Vbb)^*$ by Prop. \ref{lb70}.
\end{proof}

     \begin{proof}[\textbf{Proof of Thm. \ref{lb68}}]
Choose any homogeneous vectors $u,v\in \Vbb$. For each $\upphi\in A_n(\Vbb)^*$, 
        \begin{align*}
 & \<\upphi,\idt\diamond_R v\>=             \<Y_-(v)_{\wt v-1}\upphi,\ibf\>=\Res_{z=0}\<Y_-(v,z)\upphi,\ibf\>\cdot z^{\wt v-1}dz\\
\xlongequal{\eqref{eq155}}&\Res_{z=0}~z^{-1}\upphi(v)dz=\upphi(v).
         \end{align*}
So in $A_n(\Vbb)$ we have $\idt\diamond_R v=v$, and hence $v\diamond_L \idt=v$ by Lem. \ref{lb65}. The same can be said about $u$. Thus, by Prop. \ref{lb66}, in $A_n(\Vbb)$ we have
         \begin{align*}
             u\diamond_L v=u\diamond_L (\ibf\diamond_R v)=(u\diamond_L \ibf)\diamond_R v=u\diamond_R v.
         \end{align*}
     \end{proof}

\subsection{Modules of $\Vbb$ and $A_n(\Vbb)$}

One advantage of $A_n(\Vbb)$ is that its representations are closely related to those of $\Vbb$ (cf. \cite{Zhu-modular-invariance,DLM-Zhu}). In this article, we will not explore a systematic relationship between $\Vbb$-modules and $A_n(\Vbb)$-modules. But we will reprove (or reinterpret) some classical results in order to give the readers an impression that such topics can be neatly dealt with using the machinery of dual fusion products.

\begin{df}
For each weak $\Vbb$-module $\Wbb$, let
\begin{align*}
\Omega_n(\Wbb)=\{w\in\Wbb: Y(v)_kw=0~\forall \text{ homogeneous }v\in\Vbb,k\in\Zbb,k\geq\wt(v)+n\}
\end{align*}
\end{df}

\begin{rem}
Using Jacobi identity, it is easy to see that for each homogeneous $v\in\Vbb$ and $k\in\Zbb$,
\begin{align}
Y(v)_k\Omega_n(\Wbb)\subset\Omega_{n'}(\Wbb)\qquad\text{where }n'=n+\max\{0,\wt v-k-1\}
\end{align}
In particular, $\Omega_n(\Wbb)$ is invariant under $Y_\Wbb(v)_{\wt v-1}$. Moreover, if $\Wbb$ is an admissible $\Vbb$-module, the bracket relation \eqref{eq150} implies that $\Omega_n(\Wbb)$ is $\wtd L(0)$-invariant. So
\begin{align*}
\Omega_n(\Wbb)=\bigoplus_{k\in\Nbb}\Omega_n(\Wbb)\cap \Wbb(k)
\end{align*}
\end{rem}

\begin{rem}\label{lb72}
Suppose that $\Wbb$ is an admissible $\Vbb$-module. Associate $\Vbb,\Wbb',\Wbb$ to the marked points $1,\infty,0$ respectively. Identify $\scr W_{\fq}(\Vbb\otimes\Wbb'\otimes\Wbb)$ with $\Vbb\otimes\Wbb'\otimes\Wbb$ via $\mc U(\zeta-1,1/\zeta,\zeta)$. Then the linear functional
\begin{gather}
\begin{gathered}
\upomega_\Wbb:\Vbb\otimes \Wbb'\otimes\Wbb\rightarrow\Cbb\\
\upomega_\Wbb(v\otimes w'\otimes w)=\<w',Y(v,1)w\>=\sum_{k\in\Zbb}\< w',Y(v)_k w\>
\end{gathered}
\end{gather}
belongs to $\scr T_{\fq}^*(\Vbb\otimes\Wbb'\otimes\Wbb)$. Thus, by Thm. \ref{lb50}, there is a unique weak $\Vbb\times\Vbb$-module morphism
\begin{gather}\label{eq158}
\begin{gathered}
T_\Wbb:\Wbb'\otimes\Wbb\rightarrow\bbs_\fq(\Vbb)\\
\<T_\Wbb(w'\otimes w),v\>=\<w',Y(v,1)w\>\qquad(\forall v\in\Vbb)
\end{gathered}
\end{gather}
\end{rem}

\begin{thm}
Let $\Wbb$ be an admissible $\Vbb$-module with grading operator $\wtd L(0)$. Then there is a (necessarily unique) representation $o$ of $A_n(\Vbb)$ on $\Omega_n(\Wbb)$ satisfying that for each homogeneous $v\in\Vbb$ and $w\in\Omega_n(\Wbb)$,
\begin{align}
o(v)w=Y_\Wbb(v)_{\wt v-1}\cdot w\label{eq151}
\end{align}
\end{thm}

It was proved in \cite{DLM-Zhu} that the same conclusion holds if $\Wbb$ is only assumed to be a weak $\Vbb$-module. We do not know how to prove this fact in our setting.

\begin{proof}
Define linear map $o:\Vbb\rightarrow \End(\Omega_n(\Wbb))$ by \eqref{eq151}. We first show that $o$ descends to a linear map $A_n(\Vbb)\rightarrow\End(\Omega_n(\Wbb))$. Choose any homogeneous $w\in\Omega_n(\Wbb)$. Then $o(v)w$ also has weight $\wtd\wt w$ by \eqref{eq150}. Therefore, for each homogeneous $w'\in\Wbb'$, $\bk{o(v)w,w'}$ is zero if $\wtd\wt w'\neq \wtd\wt w$. Assume $\wtd\wt w'=\wtd\wt w$. Then by \eqref{eq150} we have
\begin{align}
\bk{w',o(v)w}=\bk{w',Y(v,1)w}=\<T_\Wbb(w'\otimes w),v\>  \label{eq153}
\end{align}
If $u\in\Vbb$ is homogeneous and $k\geq \wt(u)+n$, then $w'\otimes w$ is annihilated by $Y_-(u)_k=\idt\otimes Y(u)_k$, and hence $T_\Wbb(w'\otimes w)$ is annihilated by $Y_-(u)_k$. For each homogeneous $v\in \Vbb$,
\begin{align*}
\<Y_+(u)_k T_\Wbb(w'\otimes w),v \>=\<Y(u)_k w',Y(v,1)w \>=\<Y(u)_k w',Y(v)_sw \>
\end{align*}
where $s$ is such that $\wt u-k=\wt v-s$. Since $w\in\Omega_n(\Wbb)$, we have $Y(v)_sw=0$. So the above expression is zero. Thus $T_\Wbb(w'\otimes w)\in\wtd A_n(\Vbb)^*$ by Cor. \ref{lb63}. Moreover, for each $w\in\Wbb,w'\in\Wbb'$ and homogeneous $v\in\Vbb$,
\begin{align*}
&\<L_+(0)T_\Wbb(w'\otimes w),v\>=\<T_\Wbb (L(0)w'\otimes w),v\>=\<w',L(0)Y(v)_{\wt v-1}w \>\\
=&\<w',Y(v)_{\wt v-1}L(0)w \>=\<T_\Wbb (w'\otimes L(0) w),v\>=\<L_-(0)T_\Wbb(w'\otimes w),v\>
\end{align*}
So $T_\Wbb(w'\otimes w)\in A_n(\Vbb)^*$ by \eqref{eq146}. This proves that $o$ descends to $A_n(\Vbb)$.

Choose any $u,v\in\Vbb$. Choose homogeneous $w\in\Wbb,w'\in\Wbb'$. Clearly $\<w',o(u)o(v)w\>$ and $\<w',o(u\diamond v)w\>$ are both $0$ if $\wtd\wt w'\neq \wtd\wt w$. Assume $\wtd\wt w'=\wtd\wt w$. Then
\begin{align*}
&\<w',o(u\diamond v)w\>\xlongequal{\eqref{eq153}} \< T_\Wbb(w'\otimes w),u\diamond_R v\>\xlongequal{\eqref{eq133}} \< Y_-(v)_{\wt v-1}T_\Wbb(w'\otimes w),u\>\\
=&\< T_\Wbb Y_-(v)_{\wt v-1}(w'\otimes w),u\>=\< T_\Wbb(w'\otimes  Y(v)_{\wt v-1}w),u\>\\
=&\< T_\Wbb(w'\otimes  o(v)w),u\>\xlongequal{\eqref{eq153}}\<w',o(u)o(v)w\>.
\end{align*}
This proves $o(u\diamond v)=o(u)o(v)$.
\end{proof}

Now we consider the other direction. 
\begin{df}
Let us define a contravariant functor $\Gamma_n$ from the category of (left) $A_n(\Vbb)$-module to the category of weak $\Vbb$-modules as follows. Let
\begin{align*}
\scr T_{\fq,\infty,n}^*(\Vbb)=\varinjlim_{k\in\Nbb} \scr T_{\fq,k,n}^*(\Vbb)
\end{align*}
For each (left) $A_n(\Vbb)$-module $\Ebb$ we associate a weak $\Vbb$-module
\begin{align}
\Gamma_n(\Ebb)=\Hom_{o(\Vbb)}\big(\Ebb,\scr T_{\fq,\infty,n}^*(\Vbb)\big)
\end{align}
whose elements are all the linear maps $F:\Ebb\rightarrow\scr T_{\fq,\infty,n}^*(\Vbb)$ satisfying
\begin{align}
F \cdot o(v)=Y_-(v)_{\wt v-1}\cdot F
\end{align}
for each homogeneous $v\in\Vbb$. By Thm. \ref{lb43} and Prop. \ref{lb19}, $\Gamma_n(\Ebb)$ is a weak $\Vbb$-module if for each $v\in\Vbb,k\in\Zbb$ we define the action of $Y(v)_k$ on $F$ to be $Y_+(v)_k\cdot F$. We denote this module by $(\Gamma_n(\Ebb),Y_+)$, or simply $\Gamma_n(\Ebb)$ for short.
\end{df}

\begin{rem}
Let $\Wbb$ be an admissible $\Vbb$-module. Let $\Ebb$ be a (left) $A_n(\Vbb)$-module. Let $\varphi:\Ebb\rightarrow \Omega_n(\Wbb)$ be a (left) $A_n(\Vbb)$-module morphism. Then $\varphi$ induces naturally a weak $\Vbb$-module morphism $\wtd\varphi:\Wbb'\rightarrow\Gamma_n(\Ebb)$ defined by
\begin{gather}
\begin{gathered}
\wtd\varphi:\Wbb'\rightarrow \Hom_{o(\Vbb)}(\Ebb,\scr T_{\fq,\infty,n}^*(\Vbb))\\
w'\quad\mapsto\quad\big(~e\in\Ebb\mapsto T_\Wbb(w'\otimes \varphi(e))~\big)
\end{gathered}
\end{gather}

On the other hand, Dong-Li-Mason constructed in \cite{DLM-Zhu} an admissible $\Vbb$-module $\ovl M_n(\Ebb)$ for each $A_n(\Vbb)$-module $\Ebb$, and showed that each morphism $\varphi:\Ebb\rightarrow \Omega_n(\Wbb)$ induces naturally a weak $\Vbb$-module morphism $\ovl\varphi:\ovl M_n(\Ebb)\rightarrow\Wbb$ satisfying certain good properties. $\ovl M_n(\Ebb)$ was constructed by means of induced representations with respect to certain universal algebras. It is reasonable to guess that the contragredient of $\ovl M_n(\Ebb)$ is related to $\Gamma_n(\Ebb)$, and that some dual of $\bbs_\fq(\Vbb)$ (e.g. $\varprojlim_{k,l\in\Nbb}\scr T_{\fq,k,l}(\Vbb)$) is related to a universal algebra of $\Vbb$. We hope to study this topic in future work.
\end{rem}

\subsection{Comparison with the classical construction}\label{lb74}

In this section, we show that our $A_n(\Vbb)$ agrees with the one in \cite{DLM-Zhu}. Our discussion begins with the following observation.

\begin{rem}\label{lb75}
Let $\sigma,\tau$ be as in Lem. \ref{lb73}. Then it is not hard to see that
\begin{subequations}
\begin{gather}
\tau\in H^0(U_+,\scr V_{\fq,n,n}\otimes\omega_{\Pbb^1})\quad\quad\Longleftrightarrow\quad g_k=0 ~\text{ for every } k>\wt v-n-2\\
\sigma\in H^0(U_-,\scr V_{\fq,n,n}\otimes\omega_{\Pbb^1})\quad\Longleftrightarrow\quad f_k=0~\text{ for every } k<\wt v+n
\end{gather}
\end{subequations}
\end{rem}

\begin{pp}
$H^0\big(\Pbb^1,\scr V_{\fq,n,n}\otimes\omega_{\Pbb^1}(\blt 1)\big)\cdot\Vbb$ equals $\wtd O_n(\Vbb)$ where
\begin{align}
\wtd O_n(\Vbb)=\Span_\Cbb\big\{\Res_{z=0}~z^{-2n-2}Y((1+z)^{L(0)+n}u,z)vdz:u,v\in\Vbb  \big\}
\end{align}
In other words, we have $\wtd A_n(\Vbb)=\Vbb/\wtd O_n(\Vbb)$.
\end{pp}

\begin{proof}
Recall that $\zeta$ is the standard coordinate of $\Cbb$. Unlike before, we let $z=\zeta-1$, which is not standard. For each homogeneous $u\in\Vbb$, define a section of $\scr V_\fq\otimes\omega_{\Pbb^1}$ on $\Cbb^\times-\{1\}$:
\begin{align}
\sigma=\MU_\varrho(\zeta)^{-1}(\zeta-1)^{-2n-2}\zeta^{\wt u+n}u d\zeta
\end{align}
By Rem. \ref{lb75}, $\sigma$ belongs to $H^0\big(\Pbb^1, \SV_{\fq,n,n}\otimes \omega_{\Pbb^1}(\bullet 1)\big)$. Clearly
\begin{align}
\sigma\cdot v=\Res_{z=0}~z^{-2n-2}Y((1+z)^{\wt u+n}u,z)vdz
\end{align}
for each $v\in\Vbb$. This proves $\wtd O_n(\Vbb)\subset H^0\big(\Pbb^1, \SV_{\fq,n,n}\otimes \omega_{\Pbb^1}(\bullet 1)\big)\cdot\Vbb$.

On the other hand, by Rem. \ref{lb75}, $H^0\big(\Pbb^1, \SV_{\fq,n,n}\otimes \omega_{\Pbb^1}(\bullet 1)\big)$ is spanned by
\begin{align*}
\sigma= \MU_\varrho(\zeta)^{-1} f\cdot\zeta^{\wt u+n}u d\zeta
\end{align*}
for some homogeneous $u\in \Vbb$ and $f\in H^0(\Pbb^1,\mc O(\blt 1+\blt\infty))$ such that 
\begin{align}
f_k=0\qquad \text{ for every } ~k>-2n-2
\end{align}
if $f=\sum_{k\in\Zbb}f_k\zeta^k$ is the series expansion of $f$ at $\infty$. So $f$ is a linear combination of $\{(\zeta-1)^k:k\leq -2n-2\}$. 

Thus, to prove $H^0\big(\Pbb^1, \SV_{\fq,n,n}\otimes \omega_{\Pbb^1}(\bullet 1)\big)\cdot\Vbb\subset\wtd O_n(\Vbb)$, it suffices to prove $\sigma\cdot v\in\wtd O_n(\Vbb)$ for all $v\in\Vbb$ and all such $\sigma$. By linearity, it suffices to assume $f=(\zeta-1)^k$ where $k\leq -2n-2$. Namely, it suffices to prove
    \begin{equation*}
    \Res_{z=0} ~z^k Y((1+z)^{\wt u+n}u,z)vdz ~~ \in ~~\wtd O_n(\Vbb)
    \end{equation*}
But this can be proved in a similar way as \cite[Lem. 2.1.2]{Zhu-modular-invariance}, as pointed out in the proof of \cite[Lem. 2.1]{DLM-Zhu}.
\end{proof}

\begin{pp}
Let $u,v\in\Vbb$ be homogeneous. Then in $\wtd A_n(\Vbb)$ we have
\begin{subequations}
\begin{gather}
u\diamond_L v=\Res_{z=0}\sum_{m=0}^n (-1)^m \binom{m+n}{n}\frac{(1+z)^{\wt(u)+n}}{z^{m+n+1}}Y(u,z)v dz \label{eq163}\\
u\diamond_R v=\Res_{z=0}\sum_{m=0}^n (-1)^n \binom{m+n}{n}\frac{(1+z)^{\wt(v)+m-1}}{z^{m+n+1}}Y(v,z)u dz  \label{eq164}
\end{gather}
\end{subequations}
\end{pp}

\begin{proof}
Let $U_-$ and $U_+$ be open discs centered at $0$ and $\infty$ respectively and disjoint from $1$. Choose any $\upphi\in\wtd A_n(\Vbb)^*$. By Rem. \ref{lb64} and Prop. \ref{lb20}, $\<\upphi,u\diamond_R v\>=\<Y_-(v)_{\wt v-1}\upphi,u\>$ equals $\<\upphi,\sigma\cdot u\>$ where $\sigma\in H^0(\Pbb^1,\scr V_{\Pbb^1}\otimes\omega_{\Pbb^1}(\blt 1+\blt\infty+\blt 0))$ satisfies
\begin{subequations}
\begin{gather}
\sigma\big|_{U_+}\equiv 0\quad\mod\quad H^0\big(U_+,\scr V_{\fq,n,n}\otimes\omega_{\Pbb^1}\big)\\
\sigma\big|_{U_-}\equiv -\mc U_\varrho(\zeta)^{-1} \zeta^{\wt v-1}\cdot  v d\zeta\quad\mod\quad H^0\big(U_-,\scr V_{\fq,n,n}\otimes\omega_{\Pbb^1}\big)  \label{eq161}
\end{gather}
\end{subequations}
With the help of Lem. \ref{lb73}, we also have that $\<\upphi,u\diamond_L v\>=\<Y_+'(u)_{\wt u-1}\upphi,v\>$ equals $\<\upphi,\tau\cdot v\>$ where $\tau\in H^0(\Pbb^1,\scr V_{\Pbb^1}\otimes\omega_{\Pbb^1}(\blt 1+\blt \infty+\blt 0))$ satisfies
\begin{subequations}
\begin{gather}
\tau\big|_{U_+}\equiv \mc U_\varrho(\zeta)^{-1} \zeta^{\wt u-1}\cdot u   d\zeta\quad\mod\quad H^0\big(U_+,\scr V_{\fq,n,n}\otimes\omega_{\Pbb^1}\big)  \label{eq162}\\
\tau\big|_{U_-}\equiv 0\quad\mod\quad H^0\big(U_-,\scr V_{\fq,n,n}\otimes\omega_{\Pbb^1}\big)
\end{gather}
\end{subequations}
Using Rem. \ref{lb75}, one checks that the following $\sigma$ and $\tau$ satisfy the desired conditions:
\begin{subequations}
\begin{gather}
\sigma=\mc U_\varrho(\zeta)^{-1}\cdot \sum_{m=0}^n(-1)^n \binom{m+n}{n}\frac{\zeta^{\wt(v)+m-1}}{(\zeta-1)^{m+n+1}} \cdot vd\zeta\\
\tau=\mc U_\varrho(\zeta)^{-1}\cdot \sum_{m=0}^n (-1)^m \binom{m+n}{n}\frac{\zeta^{\wt(u)+n}}{(\zeta-1)^{m+n+1}}\cdot ud\zeta
\end{gather}
\end{subequations}
(To check \eqref{eq161} and \eqref{eq162}, one also needs the binomial formula in \cite[Prop. 5.2]{DLM-Zhu}.) $\sigma\cdot u$ and $\tau\cdot v$ are clearly equal to the RHS of \eqref{eq164} and \eqref{eq163} respectively.
\end{proof}

\begin{rem}
In \cite{DLM-Zhu}, the RHS of \eqref{eq163} is given as the definition of the multiplication operation $\diamond$ of $A_n(\Vbb)$. The RHS of \eqref{eq164} appears in \cite[Lem. 2.1-(ii)]{DLM-Zhu}.
\end{rem}

\begin{pp}
For each $v\in\Vbb$, the following relation holds in $\wtd A_n(\Vbb)$:
\begin{gather}
\cbf\diamond_L v-v\diamond_R\cbf=\big(L(0)+L(-1)\big)v
\end{gather}
Thus $A_n(\Vbb)=\wtd A_n(\Vbb)/\{\big(L(0)+L(-1)\big)v:v\in\Vbb\}$.
\end{pp}

\begin{proof}
We know that $Z_n(\Vbb)$ is spanned by all $\cbf\diamond_L v-v\diamond_R\cbf$. Choose any $\upphi\in\wtd A_n(\Vbb)^*$. Then
\begin{align*}
\<\upphi,\cbf\diamond_L v-v\diamond_R\cbf\>=\<(Y_+(\cbf)_1-Y_-(\cbf)_1)\upphi,v \>
\end{align*}
equals $-\<\sigma*_+\upphi+\sigma*_-\upphi,v\>$ where $\sigma\in H^0(\Pbb^1,\scr V_{\Pbb^1}\otimes\omega_{\Pbb^1}(\blt\infty+\blt 0))$ is defined by $\sigma=\mc U_\varrho(\zeta)^{-1}\cbf \cdot \zeta d\zeta$ (recall Lem. \ref{lb73}). So it equals $\<\upphi,\sigma\cdot v\>$ by Rem. \ref{lb64}. Clearly $\sigma\cdot v=\big(L(0)+L(-1)\big)v$.
\end{proof}

\printindex	
\footnotesize
	\bibliographystyle{alpha}

\noindent {\small \sc Yau Mathematical Sciences Center, Tsinghua University, Beijing, China.}

\noindent {\textit{E-mail}}: binguimath@gmail.com\qquad bingui@tsinghua.edu.cn\\

\noindent {\small \sc Yau Mathematical Sciences Center and Department of Mathematics, Tsinghua University, Beijing, China.}

\noindent {\textit{E-mail}}: zhanghao1999math@gmail.com \qquad h-zhang21@mails.tsinghua.edu.cn

\end{document}